\documentclass[11pt]{article}
\usepackage{amssymb}
\usepackage{tikz}
\usepackage{xr}
\usetikzlibrary{arrows, calc, decorations.pathmorphing, backgrounds, positioning, fit, petri, automata}
\definecolor{grey1}{rgb}{0.5,0.5,0.5}
\usepackage[top=2.54cm, bottom=2.54cm, left=2.2cm, right=2.2cm]{geometry}
\usepackage{algorithmicx}
\usepackage{algcompatible}
\usepackage{algorithm}
\usepackage{hyperref}
\hypersetup{
  colorlinks,
  citecolor=blue,
  linkcolor=red,
  urlcolor=magenta}

\usepackage{mathtools}
\usepackage{graphicx}
\usepackage{subfigure}
\usepackage{amsmath,amsthm,bm,bbm}
\usepackage{rotating}
\usepackage{mathrsfs}
\usepackage{chngcntr}
\usepackage{apptools}
\usepackage[titletoc,title]{appendix}
\usepackage{comment}

\usepackage{xcolor}         %for colors
\definecolor{grau}{rgb}{0.8,0.8,0.8}

\newcommand{\chen}[1]{\color{red}}
\usepackage{setspace,lscape,longtable}
\usepackage{amsmath,amsthm,bm}
\usepackage{fullpage}  

\usepackage{titlesec}
\usepackage[english]{babel}
\usepackage{xcolor}         %for colors
\definecolor{grau}{rgb}{0.8,0.8,0.8}

\numberwithin{equation}{section}
\numberwithin{equation}{section}
\theoremstyle{plain}

\newtheorem{theorem}{Theorem}[section]
\newtheorem{lemma}{Lemma}[section]
\newtheorem{corollary}{Corollary}[section]
\theoremstyle{remark}

\newtheorem{remark}{Remark}[section]

\usepackage[comma,authoryear]{natbib}

\bibpunct{(}{)}{;}{a}{,}{,}
\DeclareMathOperator*{\argmin}{arg\,min}

\newcommand{\prob}{{\mathbb{P}}}
\newcommand{\var}{{\mathrm{var}}}
\newcommand{\expect}{\mathbb{E}}
\newcommand{\vect}{\mathrm{vec}}
\newcommand{\vech}{\mathrm{vech}}
\newcommand{\Span}{\mathrm{Span}}
\newcommand{\diag}{\mathrm{diag}}
\newcommand{\vol}{\mathrm{vol}}

\newcommand{\transpose}{^{\mathrm{T}}}
\newcommand{\inverseT}{^{-\mathrm{T}}}

\newcommand{\bLambda}{{\bm{\Lambda}}}
\newcommand{\calA}{{\mathcal{A}}}
\newcommand{\calB}{{\mathcal{B}}}
\newcommand{\calC}{{\mathcal{C}}}

\newcommand{\calE}{{\mathcal{E}}}

\newcommand{\calG}{{\mathcal{G}}}

\newcommand{\calI}{{\mathcal{I}}}
\newcommand{\calJ}{{\mathcal{J}}}

\newcommand{\calL}{{\mathcal{L}}}

\newcommand{\calN}{{\mathcal{N}}}

\newcommand{\calP}{{\mathcal{P}}}

\newcommand{\calS}{{\mathcal{S}}}
\newcommand{\calT}{{\mathcal{T}}}
\newcommand{\calU}{{\mathcal{U}}}

\newcommand{\balpha}{{\boldsymbol{\alpha}}}

\newcommand{\bv}{{\mathbf{v}}}
\newcommand{\bx}{{\mathbf{x}}}
\newcommand{\bA}{{\mathbf{A}}}
\newcommand{\bc}{{\mathbf{c}}}
\newcommand{\bB}{{\mathbf{B}}}
\newcommand{\bC}{{\mathbf{C}}}
\newcommand{\bD}{{\mathbf{D}}}
\newcommand{\bE}{{\mathbf{E}}}
\newcommand{\bF}{{\mathbf{F}}}
\newcommand{\bG}{{\mathbf{G}}}
\newcommand{\bH}{{\mathbf{H}}}
\newcommand{\bM}{{\mathbf{M}}}
\newcommand{\bK}{{\mathbf{K}}}

\newcommand{\bJ}{{\mathbf{J}}}

\newcommand{\bO}{{\mathbf{O}}}
\newcommand{\bP}{{\mathbf{P}}}
\newcommand{\bQ}{{\mathbf{Q}}}
\newcommand{\bR}{{\mathbf{R}}}
\newcommand{\bS}{{\mathbf{S}}}
\newcommand{\bT}{{\mathbf{T}}}
\newcommand{\bU}{{\mathbf{U}}}
\newcommand{\bV}{{\mathbf{V}}}
\newcommand{\bW}{{\mathbf{W}}}
\newcommand{\bX}{{\mathbf{X}}}
\newcommand{\bY}{{\mathbf{Y}}}
\newcommand{\bZ}{{\mathbf{Z}}}
\newcommand{\bw}{{\mathbf{w}}}

\newcommand{\bg}{{\mathbf{g}}}

\newcommand{\bell}{{\boldsymbol{\ell}}}
\newcommand{\bt}{{\mathbf{t}}}
\newcommand{\bu}{{\mathbf{u}}}

\newcommand{\by}{{\mathbf{y}}}

\newcommand{\be}{{\mathbf{e}}}

\newcommand{\bgamma}{{\bm{\gamma}}}
\newcommand{\bGamma}{{\bm{\Gamma}}}
\newcommand{\bbeta}{{\bm{\beta}}}
\newcommand{\bPi}{{\bm{\Pi}}}
\newcommand{\bpi}{{\bm{\pi}}}
\newcommand{\bDelta}{{\bm{\Delta}}}

\newcommand{\bOmega}{{\bm{\Omega}}}

\newcommand{\bSigma}{{\bm{\Sigma}}}

\newcommand{\beps}{{\bm{\varepsilon}}}
\newcommand{\eye}{{\mathbf{I}}}

\newcommand{\btheta}{{\bm{\theta}}}
\newcommand{\bTheta}{{\bm{\Theta}}}

\newcommand{\bmu}{{\bm{\mu}}}

\newcommand{\bPsi}{{\bm{\Psi}}}

\newcommand{\bvarphi}{{\bm{\varphi}}}
\newcommand{\bpsi}{{\bm{\psi}}}

\newcommand{\zero}{{\bm{0}}}
\newcommand{\eps}{\epsilon}

\newcommand{\keywords}[1]{\par\addvspace\baselineskip\noindent\enspace\ignorespaces \textbf{Keywords: }#1}

\usepackage{lipsum}
\usepackage{enumitem}
\setlist[enumerate]{itemsep=0mm}
\setlist[itemize]{itemsep=0mm}

\author{Fangzheng Xie}
\author{Fangzheng Xie\thanks{Department of Statistics, Indiana University}
}
\title{Euclidean Representation of Low-Rank Matrices and Its Statistical Applications}
\begin{document}
\allowdisplaybreaks

\maketitle
% \doublespacing
%\section*{Acknowledgement}

\begin{abstract}
Low-rank matrices are pervasive throughout statistics, machine learning, signal processing, optimization, and applied mathematics. In this paper, we propose a novel and user-friendly Euclidean representation framework for low-rank matrices. Correspondingly, we establish a collection of technical and theoretical tools for analyzing the intrinsic perturbation of low-rank matrices in which the underlying referential matrix and the perturbed matrix both live on the same low-rank matrix manifold. 
Our analyses show that, locally around the referential matrix, the sine-theta distance between subspaces is equivalent to the Euclidean distance between two appropriately selected orthonormal basis, circumventing the orthogonal Procrustes analysis. We also establish the regularity of the proposed Euclidean representation function, which has a profound statistical impact and a meaningful geometric interpretation. 
These technical devices are applicable to a broad range of statistical problems. Specific applications considered in detail include Bayesian sparse spiked covariance model with non-intrinsic loss, efficient estimation in stochastic block models where the block probability matrix may be degenerate, and least-squares estimation in biclustering problems. Both the intrinsic perturbation analysis of low-rank matrices and the regularity theorem may be of independent interest. 
\end{abstract}


\tableofcontents

\section{Introduction} % (fold)
\label{sec:introduction}

% \subsection{Background} % (fold)
% \label{sub:background}

Due to the emergence of high-dimensional data, low-rank matrix models have been extensively studied and broadly applied in statistics, probability, machine learning, optimization, applied mathematics, and various application domains. Statistical analysis of low-rank matrix models also appear in principal component analysis \citep{doi:10.1198/106186006X113430,amini2009,pmlr-v22-vu12,vu2013minimax,ma2013,cai2013sparse,koltchinskii2017new}, covariance matrix estimation \citep{fan2013large,cai2015optimal,johnstone2001distribution,cai2016,donoho2018}, low-rank matrix denoising and completion \citep{candes2009exact,candes2010power,donoho2014,chatterjee2015,cai2018}, random graph inference \citep{sussman2012consistent,rohe2011,athreya2017statistical,abbe2018community,tang2018limit,10.1093/biomet/asaa031}, among others. From the practical perspective, specific domain-oriented applications involving low-rank matrices include collaborative filtering \cite{goldberg1992using}, neural science \citep{eichler2017complete,8570772}, social networks \citep{young2007random,nickel2008random}, and cryo-EM \citep{doi:10.1137/120863642}. 

Spectral methods have been ubiquitous to gain insight into low-rank matrix models in the presence of high-dimensional data. For example, in the stochastic block model, the $K$-means clustering procedure is applied to the rows of the leading eigenvector matrix of the observed adjacency matrix or its normalized Laplacian matrix to discover the underlying community structure \citep{rohe2011,sussman2012consistent,abbe2018community,tang2018limit}. Meanwhile, the theoretical understanding of spectral methods has also been developed based on matrix perturbation analysis \citep{doi:10.1137/0707001,wedin1972perturbation,bhatia2013matrix,Stewart90,doi:10.1093/biomet/asv008,cai2018}. Specifically, given an approximately low-rank matrix $\bSigma_0$ and a perturbation matrix $\bE$ that is comparatively smaller than $\bSigma_0$ in magnitude, matrix perturbation analysis studies how the eigenspaces or singular subspaces of the perturbed matrix $\bSigma:=\bSigma_0 + \bE$ differ from those of the original matrix $\bSigma_0$ in terms of the behavior of the perturbation $\bE$. Notably, when $\bSigma$ and $\bSigma_0$ are symmetric matrices, the famous Davis-Kahan theorem \citep{doi:10.1137/0707001,doi:10.1093/biomet/asv008} asserts that the distance between the subspace spanned by the leading eigenvectors of $\bSigma$ and that of $\bSigma_0$, formally defined via the canonical angles (see Section \ref{sub:notations} below), can be upper bounded by the matrix norms of $\bE$. 
Several extensions and generalizations of the matrix perturbation tools have been developed. \cite{wedin1972perturbation} and \cite{cai2018} later extended the Davis-Kahan theorem to deal with singular subspaces and rectangular matrices. In the context where $\bE$ is a mean-zero random matrix, \cite{https://doi.org/10.1002/rsa.20367} and \cite{o2018random} obtained optimal and sharp results that improve the classical deterministic bounds due to \cite{doi:10.1137/0707001} and \cite{wedin1972perturbation}. Recently, there has also been a collection of works focusing on the 
entrywise perturbation behavior of eigenvectors as well as the two-to-infinity norm of eigenvector perturbation analysis when $\bE$ is deterministic or random \citep{pmlr-v83-eldridge18a,cape2019,10.1093/biomet/asy070,JMLR:v18:16-140,abbe2020}.

% [General setting of the paper]
In this paper, we focus on the intrinsic perturbation analysis of low-rank matrices, in which both the referential matrix $\bSigma_0$ and the perturbed matrix $\bSigma = \bSigma_0 + \bE$ lie on the same low-rank matrix manifold. A key feature of this setup is that the rank of $\bSigma$ is the same as the rank of $\bSigma_0$. 
% both $\bSigma_0$ and $\bSigma$ have the same rank.
% low-rank structure, i.e., $\mathrm{rank}(\bSigma_0) = \mathrm{rank}(\bSigma)$. 
This is slightly different from the classical random perturbation setup where the referential matrix $\bSigma_0$ is perturbed by a mean-zero but potentially full-rank random matrix $\bE$, as the resulting perturbed matrix $\bSigma$ may not necessarily be low-rank. Nevertheless, understanding the intrinsic perturbation of low-rank matrices is of fundamental interest in many statistical problems. For example, in Bayesian statistics, the referential matrix $\bSigma_0$ may correspond to the ground truth of the parameter of interest, and $\bSigma$ is, under the posterior distribution, a random matrix taking values in a low-rank matrix manifold such that $\mathrm{rank}(\bSigma) = \mathrm{rank}(\bSigma_0)$. In the context of intrinsically perturbed low-rank matrices, the classical tools following the Davis-Kahan framework (e.g., those developed in \citealp{doi:10.1137/0707001,wedin1972perturbation,10.1093/biomet/asv008}), although still valid, are less user-friendly to obtain sharp and optimal results in various statistical problems. The main difference is that, in the random perturbation setup, the matrix $\bE$ either has independent mean-zero random variables as its entries (e.g., low-rank matrix denoising and completion, stochastic block model), or is a sum of independent mean-zero random matrices having simple structures (e.g., covariance matrix estimation, canonical correlation analysis). In contrast, in scenarios where the analysis of the difference between two matrices having the same rank is desired (e.g., sparse principal component analysis), the perturbation matrix $\bE$ is structurally more complicated for analysis. Hence, it brings additional technical challenges when the Davis-Kahan framework is applied directly. 

This paper establishes a novel Euclidean representation framework for low-rank matrices and provides a collection of theoretical and technical tools for studying the intrinsic perturbation of low-rank matrices. Both symmetric square matrices and general rectangular matrices are considered. Specifically, leveraging the Cayley parameterization for subspaces \citep{jauch2020}, we propose a matrix-valued function to represent generic
 low-rank matrices using vectors in an open subset of the Euclidean space. Furthermore, built upon the proposed Euclidean representation framework, we show that the intrinsic perturbation of low-rank matrices can be characterized by the behavior of their representing Euclidean vectors. Consequently, the Frobenius sine-theta distance between subspaces (formally defined in Section \ref{sec:preliminaries} below) is locally equivalent to the Frobenius distance between two suitably selected Stiefel matrices spanning the corresponding subspaces, which is user-friendly and circumvents the need for an orthogonal Procrustes analysis. Another fundamental result of the proposed framework is that the collection of low-rank matrices of interest can be viewed as a Euclidean manifold, and our proposed Euclidean representation function serves as a coordinate system for the low-rank matrix manifold. 
  % we establish a homeomorphism between an open subset of the Euclidean space and the collection of low-rank matrices, 
% Hence, we can successfully transform the analysis of non-Euclidean low-rank matrix perturbation into the analysis of Euclidean vectors through the proposed coordinate system. 

Subsequently, we apply the newly developed Euclidean representation framework for low-rank matrices and the accompanying technical devices to several statistical problems in detail. In this paper, we present the applications of the proposed framework to the following statistical problems involving low-rank matrices and obtain sharp and optimal results:
\begin{enumerate}
	\item Bayesian sparse spiked covariance model: Spiked covariance model, initially named by \cite{doi:10.1198/jasa.2009.0121}, is a natural probabilistic model for principal component analysis. We focus on the sparse spiked covariance model where the covariance matrix can be decomposed as the sum of a low-rank matrix and an identity matrix, and the eigenvector matrix of the low-rank component exhibits the so-called row sparsity (formally defined in Section \ref{sub:bayesian_sparse_pca_and_non_intrinsic_loss}). Specifically, we consider a Bayesian model where a sparsity enforcing prior is assigned to the rows of the leading eigenvector matrix and apply the obtained technical tools to obtain the minimax-optimal posterior contraction rate under the spectral norm loss. The main technical challenge is that, unlike the Frobenius norm loss, the spectral norm loss is not equivalent to the intrinsic metric of the model, and the rate-optimal posterior contraction under a non-intrinsic loss is non-trivial. To the best of our knowledge, this is the first non-trivial rate-optimal posterior contraction result under a non-intrinsic loss function for Bayesian sparse spiked covariance model.

	\item Stochastic block model: Statistical analysis of network data has been gaining popularity in statistics, machine learning, physics, and social science. Among various network models, the stochastic block model \citep{HOLLAND1983109} has been serving as a simple yet flexible enough model for network analysis. In the case where the underlying block probability matrix $\bSigma_0$ may be potentially singular, we propose a novel one-step estimator for $\bSigma_0$ based on the proposed Euclidean representation framework for low-rank matrices and apply the obtained technical tools to show that the one-step estimator is asymptotically efficient. Furthermore, the one-step estimator has a smaller mean-squared error asymptotically than the naive maximum likelihood estimator proposed in \cite{bickel2013} when $\bSigma$ is singular. 

	\item Biclustering: Biclustering was originally proposed in \cite{doi:10.1080/01621459.1972.10481214}. Suppose one observes a rectangular data matrix, and both the row and columns of the matrix possess certain cluster structures. When the observed data matrix is binary, it can be viewed as a rectangular extension of the stochastic block model when the data matrix is the off-diagonal block of the adjacency matrix of a bipartite network. We apply the obtained technical tools for rectangular low-rank matrices and establish the asymptotic normality of the least-squares estimator for the block mean matrix when it could be potentially rank-deficient.
\end{enumerate}
We remark that the Euclidean representation framework for low-rank matrices and the corresponding theoretical results can be applied to the analysis of several other statistical contexts, including sparse canonical correlation analysis, cross-covariance matrix estimation, sparse reduced-rank regression, Bayesian denoising of simultaneously low-rank and sparse matrices, among others, to obtain new and optimal results in comparison with existed works. 

The rest of the paper is organized as follows. In Section \ref{sec:preliminaries}, we present the proposed Euclidean representation framework for low-rank matrices after the introduction of basic notations and definitions. Section \ref{sec:main_results} elaborates on our main technical results, including the intrinsic perturbation theorems and the regularity theorem of the proposed Euclidean representation function. In Section \ref{sec:applications}, we apply the proposed framework and the obtained technical tools to Bayesian sparse spiked covariance model, the stochastic block model, and biclustering. Further discussion is included in Section \ref{sec:discussion}. The proofs of the main results are contained in Section \ref{sec:proofs_of_the_main_results}. 
% section introduction (end)

\section{Preliminaries} % (fold)
\label{sec:preliminaries}

\subsection{Notations and definitions} % (fold)
\label{sub:notations}

We use the symbol $:=$ to assign mathematical definitions of quantities. 
For $a, b\in \mathbb{R}$, let $a\wedge b := \min(a, b)$ and $a\vee b := \max(a, b)$. 
For a positive integer $p$, let $[p] := \{1,\ldots,p\}$. Given two positive integers $p, k$, and two functions $\tau_1,\tau_2: [p]\to [k]$, the Hamming distance between $\tau_1$ and $\tau_2$ is defined by
% \[
$d_{\mathrm{H}}(\tau_1, \tau_2) := \sum_{i = 1}^p\mathbbm{1}\{\tau_1(i)\neq \tau_2(i)\}$.
% \]
For two non-negative sequences $(a_n)_{n = 1}^\infty$ and $(b_n)_{n = 1}^\infty$, we use the symbol $a_n\lesssim b_n$ ($a_n\gtrsim b_n$, resp.) to mean that $a_n\leq Cb_n$ ($a_n\geq Cb_n$, resp.) for some constant $C > 0$, and we use the notation $a_n\asymp b_n$ to indicate that $a_n\lesssim b_n$ and $a_n\gtrsim b_n$. The notation $\bA^{\dagger}$ denotes the Moore-Penrose pseudoinverse of an arbitrary matrix $\bA$. 
We use $C, C_0, C_1, C_2, c, c_0, c_1, c_2, \ldots $ to denote generic constants that may change from line to line unless otherwise stated. For any two positive semidefinite matrices $\bA$ and $\bB$ of the same dimension, we use $\bA\succeq \bB$ ($\bA\preceq \bB$, resp.) to indicate that $\bA - \bB$ is positive semidefinite ($\bB - \bA$ is positive semidefinite, resp.). The $r\times r$ identity matrix is denoted by $\eye_r$, the $p$-dimensional zero vector is denoted by $\zero_p$, and the $p\times q$ zero matrix is denoted by $\zero_{p\times q}$. We reserve the symbol $\eye(\cdot)$ without a subscript for the Fisher information matrix of a (regular) statistical model and it should not be confused with the identity matrix. Given two positive integers $p, r$, we denote $\mathbb{O}(p, r) = \{\bU\in\mathbb{R}^{p\times r}:\bU\transpose\bU = \eye_r\}$ the collection of all $p\times r$ Stiefel matrices and write $\mathbb{O}(r):= \mathbb{O}(r, r)$. 
For any $\bU\in\mathbb{O}(p, r)$, we use $\Span(\bU)$ to denote the $r$-dimensional subspace in $\mathbb{R}^p$ spanned by the columns of $\bU$. 
The collection of all $r\times r$ symmetric matrices is denoted by $\mathbb{M}(r)$ and the collection of all $r\times r$ symmetric positive definite matrices is denoted by $\mathbb{M}_+(r)$. For a matrix $\bSigma\in\mathbb{R}^{p_1\times p_2}$ and indices $i\in[p_1],j\in [p_2]$, let $[\bSigma]_{ij}$ denote the element on the $i$th row and $j$th column of $\bSigma$, $[\bSigma]_{i*}$ denote the $i$th row of $\bSigma$, and $[\bSigma]_{j*}$ denote the $j$th column of $\bSigma$. Furthermore, we use $\sigma_1(\bSigma),\ldots,\sigma_{p_1\wedge p_2}(\bSigma)$ to denote the singular values of $\bSigma$ sorted in the non-increasing order, i.e., $\sigma_1(\bSigma)\geq\ldots\geq\sigma_{p_1\wedge p_2}(\bSigma)$. When $\bSigma$ is a $p\times p$ symmetric square matrix, $\lambda_1(\bSigma),\ldots,\lambda_p(\bSigma)$ denote the eigenvalues of $\bSigma$ sorted in the non-increasing order in magnitude, namely, $|\lambda_1(\bSigma)|\geq\ldots\geq|\lambda_p(\bSigma)|$. The spectral norm of a general matrix $\bSigma$, denoted by $\|\bSigma\|_2$, is the largest singular value of $\bSigma$, and the Frobenius norm of $\bSigma$, denoted by $\|\bSigma\|_{\mathrm{F}}$, is defined to be $\|\bSigma\|_{\mathrm{F}} = (\sum_{i = 1}^{p_1}\sum_{j = 1}^{p_2}[\bSigma]_{ij}^2)^{1/2}$. 
For a Euclidean vector $\bx = [x_1,\ldots,x_p]\transpose\in\mathbb{R}^p$, we denote $[\bx]_i:= x_i$, $\|\bx\|_2$ the usual Euclidean norm $\|\bx\|_2 = (\sum_ix_i^2)^{1/2}$, let $B_2(\bx, \eps):=\{\by\in\mathbb{R}^p:\|\by - \bx\|_2 < \eps\}$, and let $\mathrm{diag}(\bx)$ be the $p\times p$ diagonal matrix with $x_i$ being the element on its $i$th row and $i$th column. 

For a $p_1\times p_2$ matrix $\bM$, the operator $\vect(\cdot)$ transform $\bM$ to a $p_1p_2$-dimensional Euclidean vector by stacking the columns of $\bSigma$ consecutively, i.e., 
\begin{align*}
\vect(\bM)
& = [[\bM]_{*1}\transpose,[\bM]_{*2}\transpose,\ldots,[\bM]_{*p_2}]\transpose]\transpose\\
& = [[\bM]_{11},\ldots,[\bM]_{p_11},[\bM]_{12},\ldots,[\bM]_{p_12},\ldots,[\bM]_{1p_2},\ldots,[\bM]_{p_1p_2}]\transpose.
\end{align*}
The operator $\vech(\cdot)$ transform an $r\times r$ square symmetric matrix $\bM$ to an $r(r + 1)/2$-dimensional Euclidean vector by eliminating all its super-diagonal elements, i.e., 
\[
\mathrm{vech}(\bM) = [[\bM]_{11}, [\bM]_{21},\ldots,[\bM]_{r1}, [\bM]_{22},\ldots,[\bM]_{r2},\ldots,[\bM]_{rr}]\transpose.
\]
For any two positive integers $p, q$, we denote $\bK_{pq}$ the $pq\times pq$ commutation matrix such that $\vect(\bM\transpose) = \bK_{pq}\vect(\bM)$ for any $\bM\in \mathbb{R}^{p\times q}$, and denote $\mathbb{D}_p$ the duplication matrix such that $\vect(\bM) = \mathbb{D}_p\vech(\bM)$ for any symmetric $\bM\in\mathbb{R}^{p\times p}$. We refer the readers to \cite{magnus1988linear} for a review of the properties of the commutation matrix $\bK_{pq}$ and the duplication matrix $\mathbb{D}_p$. 
For two matrices $\bA\in\mathbb{R}^{p\times q}$ and $\bB\in\mathbb{R}^{m\times n}$, we use $\bA\otimes\bB$ to denote the Kronecker product of $\bA$ and $\bB$, defined to be the $pm\times qn$ matrix of the form
\[
\bA\otimes \bB = \begin{bmatrix*}
[\bA]_{11}\bB & [\bA]_{12}\bB & \ldots & [\bA]_{1q}\bB\\
[\bA]_{21}\bB & [\bA]_{22}\bB & \ldots & [\bA]_{2q}\bB\\
\vdots & \vdots & & \vdots\\
[\bA]_{p1}\bB & [\bA]_{p2}\bB & \ldots & [\bA]_{pq}\bB
\end{bmatrix*}.
\]

The distance between linear subspaces can be measured in terms of the canonical angles, formally defined as follows. Given two Stiefel matrices $\bU,\bU_0\in\mathbb{O}(p, r)$, let $\sigma_1(\bU_0\transpose\bU)\geq\ldots\geq\sigma_r(\bU_0\transpose\bU)\geq0$ be the singular values of $\bU_0\transpose\bU$. Note the singular values of $\bU_0\transpose\bU$ are unitarily invariant and only depend on $\Span(\bU)$ and $\Span(\bU_0)$. The canonical angles between $\bU_0$ and $\bU$ are defined to be the diagonal entries of 
\[
\Theta(\bU_0, \bU):=\mathrm{diag}\left[\cos^{-1}\left\{\sigma_1(\bU_0\transpose\bU)\right\},\ldots,\cos^{-1}\left\{\sigma_1(\bU_0\transpose\bU)\right\}\right]\in\mathbb{R}^{r\times r}.
\]
Then the spectral sine-theta distance and the Frobenius sine-theta distance between $\Span(\bU_0)$ and $\Span(\bU)$ are defined by $\|\sin\Theta(\bU_0, \bU)\|_2$ and $\|\sin\Theta(\bU_0, \bU)\|_{\mathrm{F}}$, respectively.

% subsection notations (end)

\subsection{Euclidean representation of subspaces} % (fold)
\label{sub:euclidean_representation_of_subspaces}
We first introduce the Cayley parameterization of Subspaces proposed by \cite{jauch2020}, which serves as an intermediate step towards our proposed Euclidean representation framework for low-rank matrices. 

The collection of all $r$-dimensional subspaces in $\mathbb{R}^p$ is of fundamental interest in multivariate statistics. When equipped with an appropriate topology and an atlas, the collection of all $r$-dimensional subspace in $\mathbb{R}^p$ is referred to as the Grassmannian and is denoted by $\calG(p, r)$. Nevertheless, the elements in $\calG(p, r)$ are too abstract and inconvenient for analysis. It is therefore desirable to find a concrete and user-friendly respresentation of general $r$-dimensional subspace in $\mathbb{R}^p$.

Suppose $\mathbb{S}\subset\mathbb{R}^p$ is an $r$-dimensional linear subspace in $\mathbb{R}^p$. It is always possible to find a Stiefel matrix $\bU\in\mathbb{O}(p, r)$ whose columns span $\mathbb{S}$, and one may use $\bU$ as a representer for the subspace $\mathbb{S}$. The disadvantage of this representation is that $\bU$ cannot be uniquely identified by $\mathbb{S}$, since for any orthogonal rotation matrix $\bW\in\mathbb{O}(r)$, $\Span(\bU) = \Span(\bU\bW)$. Such non-identifiability of orthonormal basis brings natural inconvenience for statistical analysis because the Fisher information matrix with a non-identifiable parameterization will be singular. 
% As is known to all, given the subspace $\mathbb{S}$, its orthonormal basis, namely, the columns of $\bU$, can only be identified up to an orthogonal matrix in $\mathbb{O}(r)$, in the sense that the columns of $\bU\bW$ always form another orthonormal basis spanning $\mathbb{S}$ for any $\bW\in\mathbb{O}(r)$. 
% The set of all $r$-dimensional subspaces in $\mathbb{R}^p$ is called the Grassmannian and is denoted by $\calG(p, r)$. Denote $\prob_\calG$ the uniform distribution on the Grassmannian $\calG(p, r)$. For $\prob_\calG$-almost every subspace $\calS\in\calG(p, r)$, 
Thanks to the result of \cite{jauch2020}, almost every $r$-dimensional subspace in $\mathbb{R}^p$ can be uniquely represented by a Stiefel matrix $\bU = [\bQ_1\transpose, \bQ_2\transpose]\transpose$ in $\mathbb{O}(p, r)$ such that $\bQ_1\in\mathbb{R}^{r\times r}$ is symmetric positive definite. Formally, denote
\[
\mathbb{O}_+(p, r) = \left\{
\bU = \begin{bmatrix*}
\bQ_1\\\bQ_2
\end{bmatrix*}\in\mathbb{O}(p, r):\bQ_1\in\mathbb{M}_+(r)
\right\}
\]
and suppose $l:\mathbb{O}_+(p, r)\to\calG(p, r)$ is the map defined by $\ell(\bU) = \mathrm{Span}(\bU)$. By Proposition 3.2 in \cite{jauch2020}, the image of $l$ has probability one with respect to the uniform probability distribution on $\calG(p, r)$. Hence, with probability one, every $r$-dimensional subspace in $\mathbb{R}^p$ can be uniquely represented by a Stiefel matrix in $\mathbb{O}_+(p, r)$, and therefore, finding a suitable representation of subspaces reduces to finding a suitable representation of Stiefel matrices in $\mathbb{O}_+(p, r)$.
% The subset $\mathbb{O}_+(p, r)$ of all $p\times r$ Stiefel matrices with orthonormal columns serves as the basis of the Cayley parameterization of subspaces. 

% We consider the Cayley parameterization introduced in \cite{jauch2020}. 
Let $\bA$ be a $(p - r)\times r$ matrix with $\|\bA\|_2 < 1$, and denote $\bvarphi = \vect(\bA)$. Then for any $\bU\in\mathbb{O}_+(p, r)$, the Cayley parameterization of $\bU$ via $\bvarphi = \vect(\bA)$ is defined as the following map \citep{jauch2020}:
\begin{align}\label{eqn:Cayley_transform}
\bU:\bvarphi\in\mathbb{R}^{(p - r)r}\mapsto \bU(\bvarphi) := (\eye_p + \bX_\bvarphi)(\eye_p - \bX_\bvarphi)^{-1}\eye_{p\times r}\in\mathbb{O}(p, r),
\end{align}
where 
\begin{align}
\bX_\bvarphi = \begin{bmatrix*}
\zero_{r\times r} & -\bA\transpose\\ \bA & \zero_{(p - r)\times(p - r)}
\end{bmatrix*},
% \quad \bB = -\bB\transpose\in\mathrm{Skew}(r), 
\quad 
\text{and}
\quad
\eye_{p\times r} = \begin{bmatrix*}
\eye_r\\ \zero_{(p - r)\times r}
\end{bmatrix*}.
\end{align}
% The Cayley transform (CT) of $\bU = [\bQ_1\transpose, \bQ_2\transpose]\transpose$ is defined by : 
The Cayley parameterization \eqref{eqn:Cayley_transform} above immediately leads to the following explicit expression for the submatrices of $\bU$:
\begin{align*}
\bU & = \begin{bmatrix*}
\bQ_1\\\bQ_2
\end{bmatrix*}
= \begin{bmatrix*}
 (\eye_r - \bA\transpose\bA)(\eye_r + \bA\transpose\bA)^{-1}\\
 2\bA(\eye_r + \bA\transpose\bA)^{-1}
\end{bmatrix*}.
\end{align*}
By Proposition 3.4 in \cite{jauch2020}, the Cayley parameterization $\bU(\cdot)$, viewed as a map from $\{\bvarphi = \vect(\bA):\bA\in\mathbb{R}^{(p - r)\times r},\|\bA\|_2 < 1\}$ to $\mathbb{R}^{p\times r}$, is also differentiable with the Fr\'echet derivative
\begin{align}\label{eqn:Differential_CT}
D\bU(\bvarphi) = 2[\eye_{p\times r}\transpose (\eye_p - \bX_\bvarphi)\inverseT\otimes(\eye_p - \bX_\bvarphi)^{-1}]\bGamma_\bvarphi,
\end{align}
where $\bGamma_\bvarphi$ is a matrix such that $\bGamma_\bvarphi\bvarphi = \vect(\bX_\bvarphi)$. An explicit formula for $\bGamma_\bvarphi$ is also available \citep{jauch2020}:
% $\eye_{p\times r} = \begin{bmatrix*}
% \eye_r & \zero_{r\times (p - r)}
% \end{bmatrix*}\transpose$, 
% \begin{align}\label{eqn:Gamma_matrix}
$\bGamma_\bvarphi = (\eye_{p^2} - \bK_{pp})(\bTheta_1\transpose\otimes\bTheta_2\transpose)$,
% \end{align}
% $\bK_{pp}$ is the commutation matrix satisfying $\bK_{pp}\vect(\bB) = \vect(\bB\transpose)$ for any matrix $\bB\in\mathbb{R}^{p\times p}$, 
where
$\bTheta_1 = \eye_{p\times r}\transpose$, and $\bTheta_2 = [\zero_{(p-r)\times r}, \eye_{p - r}]$. The following theorem is a refined version of Proposition 3.4 in \cite{jauch2020} in terms of a global and dimension-free control of the remainder. 
% We formally summarize this result in the following theorem.
\begin{theorem}
% [Second-order deviation of CT]
\label{thm:second_order_deviation_CT}
Let $\bU:\bvarphi\mapsto\bU(\bvarphi)$ be the Cayley parameterization defined by \eqref{eqn:Cayley_transform}. Denote $D\bU(\bvarphi)$ the Fr\'echet derivative of $\bU$ given by \eqref{eqn:Differential_CT}. Then for all $\bvarphi, \bvarphi_0$, there exists a matrix $\bR_\bU(\bvarphi, \bvarphi_0)\in\mathbb{R}^{p\times r}$, such that
\begin{align*}
\vect\{\bU(\bvarphi) - \bU(\bvarphi_0)\} = D\bU(\bvarphi_0)(\bvarphi - \bvarphi_0) + \vect\{\bR_\bU(\bvarphi, \bvarphi_0)\}, 
\end{align*}
where $\|\bR_\bU(\bvarphi, \bvarphi_0)\|_{\mathrm{F}}\leq 8 \|\bvarphi - \bvarphi_0\|_2^2$. The above equation can be written in the following matrix form:
\begin{align*}
\bU(\bvarphi) - \bU(\bvarphi_0)
& = 2(\eye_p - \bX_{\bvarphi_0})^{-1}(\bX_{\bvarphi} - \bX_{\bvarphi_0})(\eye_p - \bX_{\bvarphi_0})^{-1}\eye_{p\times r} + \bR_\bU(\bvarphi, \bvarphi_0).
\end{align*}
In particular, 
% \begin{align*}
$
\|\bU(\bvarphi) - \bU(\bvarphi_0)\|_{\mathrm{F}}\leq 2\sqrt{2}\|\bvarphi - \bvarphi_0\|_2
$
% \end{align*}
for all $\bvarphi$ and $\bvarphi_0$. 
\end{theorem}
\cite{jauch2020} also showed that the Cayley parameterization is one-to-one, and we refer to the inverse map $\bA:\mathbb{O}_+(p, r)\to \{\bA\in\mathbb{R}^{(p - r)\times r}:\|\bA\|_2 < 1\}$ as the inverse Cayley parameterization. Formally, for any $\bU = [\bQ_1\transpose, \bQ_2\transpose]\transpose\in\mathbb{O}_+(p, r)$ where $\bQ_1\in\mathbb{M}_+(r)$, the inverse Cayley parameterization of $\bU$ is given by
\begin{equation}\label{eqn:inverse_Cayley_transform}
\begin{aligned}
% \bF(\bU) &= (\eye_r - \bQ_1)(\eye_r + \bQ_1)^{-1},\\
% \bB(\bU) &= \frac{1}{2}(\bF(\bU)\transpose - \bF(\bU)),\\
\bA & = 
\bA(\bU)
 % &= 
 =
\bQ_2(\eye_r + \bQ_1)^{-1}
% \frac{1}{2}\bQ_2(\eye_r + \bF(\bU))
.
\end{aligned}
\end{equation}
The following theorem claims that the map $\bA(\cdot)$ is globally Lipschitz continuous. 
\begin{theorem}
% [First-order extrinsic deviation of ICT]
\label{thm:first_order_deviation_ICT}
Let $\bA(\cdot)$ be the inverse Cayley parameterization defined by \eqref{eqn:inverse_Cayley_transform}.
% $\bU = [\bQ_1\transpose, \bQ_2\transpose]\transpose,\bU_0 = [\bQ_{01}\transpose, \bQ_{02}\transpose]\transpose\in\mathbb{O}_+(p, r)$, where $\bQ_1,\bQ_{01}\in\mathbb{M}_+(r)$. 
Then
\begin{align*}
% \|\bF(\bU) - \bF(\bU_0)\|_{\mathrm{F}}
% &\leq 3\|\bQ_1 - \bQ_{01}\|_{\mathrm{F}},\\
% \|\bB(\bU) - \bB(\bU_0)\|_{\mathrm{F}}
% &\leq 3\|\bQ_1 - \bQ_{01}\|_{\mathrm{F}},\\
\|\bA(\bU) - \bA(\bU_0)\|_{\mathrm{F}}
&\leq 2\|\bU - \bU_{0}\|_{\mathrm{F}}
\end{align*}
for all $\bU, \bU_0\in\mathbb{O}_+(p, r)$.
% , 
% where . 
\end{theorem}

% subsection euclidean_representation_of_subspaces (end)

\subsection{Euclidean represention of low-rank matrices and intrinsic perturbation} % (fold)
\label{sub:euclidean_represention_of_low_rank_matrices_and_intrinsic_perturbation}

We now leverage the aforementioned Cayley parameterization of subspaces and establish a Euclidean representation framework for symmetric low-rank matrices. 
Consider a symmetric $p\times p$ matrix $\bSigma$ with $\mathrm{rank}(\bSigma) = r\leq p$. Let $\bSigma$ yield the spectral decomposition
% \[
$\bSigma = \bV\bLambda\bV\transpose$,
% \]
where $\bV \in\mathbb{O}(p, r)$ is the Stiefel matrix of eigenvectors, and $\bLambda = \mathrm{diag}(\lambda_1, \ldots,\lambda_r)$ is the diagonal matrix of non-zero eigenvalues of $\bSigma$ with $|\lambda_1|\geq\ldots\geq|\lambda_r| > 0$. 
% In some concrete statistical applications one typically posits that the rank $r\ll p$. 
In scenarios where the eigenvalues $\lambda_1,\ldots,\lambda_r$ may include multiplicity, the eigenvector matrix $\bV$ may only be determined up to an orthogonal matrix in $\mathbb{O}(r)$. Also note by the aforementioned analysis, for almost every $\Span(\bV)\in\calG(p, r)$, there exists another Stiefel matrix $\bU\in\mathbb{O}_+(p, r)$ such that $\mathrm{Span}(\bV) = \mathrm{Span}(\bU)$. It follows that almost every $\bSigma$ with $\mathrm{rank}(\bSigma) = r$ can be reparameterized as $\bSigma = \bU\bM\bU\transpose$ for a $p\times r$ Stiefel matrix $\bU\in\mathbb{O}_+(p, r)$ and a $r\times r$ symmetric matrix $\bM\in\mathbb{M}(r)$. Then by the result of Section \ref{sub:euclidean_representation_of_subspaces}, there exists a unique $\bA\in\mathbb{R}^{(p - r)\times r}$, $\|\bA\|_2 < 1$, and $\bvarphi = \vect(\bA)$, such that $\bU = \bU(\bvarphi)$, where $\bU(\cdot)$ is the Cayley parameterization defined by \eqref{eqn:Cayley_transform}. 
% In addition, suppose the $r\times r$ positive definite matrix $\bM$ can be written as
% \[
% \bM = \begin{bmatrix*}
% [\bM]_{11} & [\bM]_{21} & \ldots & [\bM]_{r1}\\
% [\bM]_{21} & [\bM]_{22} & \ldots & [\bM]_{r2}\\
% \vdots & \vdots &		 & \vdots\\
% [\bM]_{r1} & [\bM]_{r2} & \ldots & [\bM]_{rr}
% \end{bmatrix*}.
% \]
Let $\bmu = \vech(\bM)$.
 % as the half vectorization of the symmetric matrix $\bM$
 Conversely, the matrix $\bM$ can be viewed as a function of $\bmu$, denoted generically by $\bM(\cdot):\mathbb{R}^{r(r + 1)/2}\to\mathbb{M}(r)$, as the inverse of the map $\bM \mapsto \bmu = \vech(\bM)$.
Therefore, by denoting $\btheta = [\bvarphi\transpose, \bmu\transpose]\transpose$, we can represent almost every $\bSigma$ with $\mathrm{rank}(\bSigma) = r$ through the following matrix-valued function, which is generically denoted by $\bSigma(\cdot)$:
\begin{align}
\label{eqn:Cayley_parameterization_PSD}
\bSigma(\cdot):\mathscr{D}(p, r)\to \mathscr{S}(p, r),\quad\btheta
%  = \begin{bmatrix*}
% \bvarphi\\\bmu
% \end{bmatrix*}
\mapsto
 % := \bSigma = 
% \bU\bM\bU\transpose = 
\bU(\bvarphi)\bM(\bmu)\bU(\bvarphi)\transpose,
\end{align}
where
\begin{align}
\label{eqn:domain_of_parameterization}
\mathscr{D}(p, r) := \left\{\btheta = \begin{bmatrix*}
\vect(\bA)\\
\bmu
\end{bmatrix*}\in\mathbb{R}^{(p - r)r}\times\mathbb{R}^{r(r + 1)/2}:\bA\in\mathbb{R}^{(p - r)\times r},\|\bA\|_2 < 1\right\},
\end{align}
% where $\bM(\bmu)$ is the matrix-valued function given by inverse of the map $\bmu \mapsto \vect\{\bM(\bmu)\}$. 
is the domain of the map $\bSigma(\cdot)$ and 
\begin{align}
\label{eqn:low_rank_matrix_class}
\mathscr{S}(p, r) := \{\bSigma = \bU\bM\bU\transpose:\bU\in\mathbb{O}_+(p, r), \bM\in\mathbb{M}(r)\}
\end{align}
is the collection of $p\times p$ rank-$r$ matrices of interest.
 % and 
% Then $\mathscr{S}(p, r)$ is the range of the map $\bSigma(\cdot)$ and $\mathscr{D}(p, r)$ is the corresponding domain. 

This paper is primarily interested in the intrinsic perturbation analysis between $\bSigma_0$ and $\bSigma$ with $\bSigma, \bSigma_0\in\mathscr{S}(p, r)$, and $\bE:= \bSigma - \bSigma_0$ is assumed to be comparatively smaller than $\bSigma_0$ in magnitude. In many statistical problems, $\bSigma_0$ is the referential matrix of interest, but only the perturbed version $\bSigma$ is accessible. To be more specific, $\bSigma$ typically plays the role of a function of the observed data, namely, an estimator for the unknown $\bSigma_0$. By the aforementioned analysis, $\bSigma$ and $\bSigma_0$ can be represented by some Euclidean vectors $\btheta,\btheta_0\in\mathscr{D}(p, r)$, such that $\bSigma = \bSigma(\btheta)$ and $\bSigma_0 = \bSigma(\btheta_0)$.
% , where we use $\bSigma(\cdot)$ to generically denote the map defined by \eqref{eqn:Cayley_parameterization_PSD}.
 % of the matrix-valued function $\btheta\mapsto \bSigma(\btheta)$, i.e., both the referential matrix and the perturbed matrix lie on the same space of low-rank matrices. Specifically, given a referential low-rank matrix $\bSigma_0$ that can be represented by an Euclidean vector $\btheta_0$ such that $\bSigma_0 = \bSigma(\btheta_0)$, and an intrinsically perturbed low-rank matrix $\bSigma$ that can be represented by another Euclidean vector $\btheta$ such that $\bSigma = \bSigma(\btheta)$, 
 % We focus on the relation between the matrix perturbation $\bE:= \bSigma - \bSigma_0$ and the perturbation of the their respective representing Euclidean vectors, i.e., $\btheta - \btheta_0$. In the context of statistical models, $\bSigma_0$ typically serves as the underlying unobserved population matrix, and $\bSigma$ represents an estimator of $\bSigma$, which is a function of the observed data. 
 In turn, the problem of estimating the unobserved referential matrix $\bSigma_0$ reduces to estimating the  Euclidean representer $\btheta_0$ by an estimator ${\btheta}$, and hence,
% , as will be seen in Section \ref{sec:main_results},
 the perturbation analysis of $\bSigma$ naturally translates to the problem of the perturbation analysis of $\btheta$. 
% Now consider the collection of all low-rank matrices 

We conclude this section with the introduction of the following matrix-valued functions. 
For any $\btheta = [\bvarphi\transpose, \bmu\transpose]\transpose\in\mathscr{D}(p, r)$,
 % where $\bvarphi = \vect(\bA)$ for $(p - r)\times r$ matrix $\bA$ with $\|\bA\|_2 < 1$, 
 let
% Let $\bGamma_\bmu$ be the matrix such that $\vect(\bM) = \bGamma_\bmu\bmu$. 
% For any $\btheta = [\bvarphi\transpose, \bmu\transpose]\transpose\in\mathbb{R}^{d + r(r + 1)/2}$ where $\bvarphi = \vect(\bA)$ for a $(p - r)\times r$ matrix $\bA$ and $\bM\in \mathbb{R}^{r\times r}$ is symmetric with upper triangular entries being elements of $\bmu$, denote the matrix-valued functions
\begin{equation}\label{eqn:covariance_differential}
\begin{aligned}
D_\bvarphi\bSigma(\btheta) & := (\eye_{p^2} + \bK_{pp})\{\bU(\bvarphi)\bM(\bmu)\otimes \eye_p\}D\bU(\bvarphi),\\
D_\bmu\bSigma(\btheta) 	   & := \{\bU(\bvarphi)\otimes \bU(\bvarphi)\}\mathbb{D}_r,\\
D\bSigma(\btheta)		   & := 
\begin{bmatrix*}
 D_\bvarphi\bSigma(\btheta)
% (\eye_{p^2} + \bK_{pp})\{\bU(\bvarphi)\bM\otimes \eye_p\}D\bU(\bvarphi)
 & D_\bmu\bSigma(\btheta)
\end{bmatrix*},
\end{aligned}
\end{equation}
where $D\bU(\bvarphi)$ is the Fr\'echet derivative of the Cayley parameterization defined by \eqref{eqn:Differential_CT}. 
% As will be seen in Section \ref{sec:main_results}, the matrix $D\bSigma(\btheta)$ is exactly the Fr\'echet derivative of the function $\bSigma(\cdot)$. 
When $\btheta$ takes value at the referential Euclidean vector $\btheta_0 = [\bvarphi_0\transpose, \bmu_0\transpose]\transpose$ such that $\bSigma_0 = \bSigma(\btheta_0) = \bU(\bvarphi_0)\bM(\bmu_0)\bU(\bvarphi_0)\transpose$, we simply write $\bU_0 = \bU(\bvarphi_0)$ and $\bM_0 = \bM(\bmu_0)$. 

% subsection euclidean_represention_of_low_rank_matrices_and_intrinsic_perturbation (end)

% section preliminaries (end)

\section{Main results} % (fold)
\label{sec:main_results}

\subsection{Intrinsic perturbation theorems} % (fold)
\label{sub:intrinsic_perturbation_theorems}
We present our first main result in Theorem \ref{thm:CT_deviation_Sigma} below, which translates the perturbation of two matrices $\bSigma(\btheta)$, $\bSigma(\btheta_0)$ on the same manifold $\mathscr{S}(p, r)$ to the perturbation of the corresponding representing Euclidean vectors $\btheta$, $\btheta_0$ through a first-order Taylor expansion device. 
% An application of Theorem \ref{thm:CT_deviation_Sigma} to local asymptotic normality in statistical models is also discussed in Remark \ref{rmk:LAN} subsequently. 
\begin{theorem}
\label{thm:CT_deviation_Sigma}
Under the setup and notations of Sections \ref{sub:euclidean_representation_of_subspaces} and \ref{sub:euclidean_represention_of_low_rank_matrices_and_intrinsic_perturbation}, if $\|\bA_0\|_2 < 1$ and $\|\bA\|_2 < 1$, then there exists a $p\times p$ matrix $\bR(\btheta, \btheta_0)$ depending on $\btheta, \btheta_0$, such that
\begin{equation}
\label{eqn:Taylor_expansion_PSD_vector_form}
\begin{aligned}
\vect\{\bSigma(\btheta) - \bSigma_0\}
& = D\bSigma(\btheta_0)(\btheta - \btheta_0) + \vect\{\bR(\btheta, \btheta_0)\}
% \\
% & = D_\bvarphi\bSigma(\btheta_0)(\bvarphi - \bvarphi_0) + D_\bmu\bSigma(\btheta_0)(\bmu - \bmu_0) + \vect\{\bR(\btheta, \btheta_0)\}
, 
\end{aligned}
\end{equation}
and 
% \[
$\|\bR(\btheta, \btheta_0)\|_{\mathrm{F}}\leq 16(1 + \|\bM_0\|_2)\|\btheta - \btheta_0\|_{\mathrm{2}}^2$.
% \]
In matrix form, Equation \eqref{eqn:Taylor_expansion_PSD_vector_form} can be written as
\begin{equation}
\label{eqn:Taylor_expansion_PSD_matrix_form}
\begin{aligned}
\bSigma(\btheta) - \bSigma_0
& = 2(\eye_p - \bX_{\bvarphi_0})^{-1}(\bX_{\bvarphi} - \bX_{\bvarphi_0})(\eye_p - \bX_{\bvarphi_0})^{-\mathrm{T}}\bSigma_0\\
&\quad - 2\bSigma_0(\eye_p - \bX_{\bvarphi_0})^{-1}(\bX_{\bvarphi} - \bX_{\bvarphi_0})(\eye_p - \bX_{\bvarphi_0})^{-\mathrm{T}}\\
&\quad + \bU_0\{\bM(\bmu) - \bM_0\}\bU_0\transpose
% \\&\quad
 + \bR(\btheta, \btheta_0).
\end{aligned}
\end{equation}
\end{theorem}

Theorem \ref{thm:CT_deviation_Sigma} immediately implies that, locally at $\bSigma(\btheta_0)$, the Frobenius norm of the perturbation $\|\bE\|_{\mathrm{F}} = \|\bSigma(\btheta) - \bSigma(\btheta_0)\|_{\mathrm{F}}$ between $\bSigma(\btheta)$ and $\bSigma(\btheta_0)$ on the same low-rank matrix manifold $\mathscr{S}(p, r)$ can be well controlled by the Euclidean norm $\|\btheta - \btheta_0\|_2$ of the perturbation of their representing Euclidean vector. 
% Theorem \ref{thm:CT_deviation_Sigma} loosely state that 
% locally at $\bSigma(\btheta_0)$, 
% the intrinsic perturbation of low-rank matrices can be well controlled by the perturbation of the corresponding representing Euclidean vectors. This is expected due to the Fr\'echet differentiability of the map $\btheta\mapsto \bSigma(\btheta)$. 
Theorem \ref{thm:intrinsic_deviation_Sigma} below, which is our second main result, asserts that the reverse statement is true: the perturbation of the representing Euclidean vectors can be well controlled by the perturbation of the original matrices locally at $\bSigma_0$.
\begin{theorem}
% [First-Order Intrinsic Deviation of ICT]
\label{thm:intrinsic_deviation_Sigma}
% Let $\bU = [\bQ_1\transpose, \bQ_2\transpose]$ and $\bU_0 = [\bQ_{01}\transpose, \bQ_{02}\transpose]\transpose\in\mathbb{O}_+(p, r)$. Suppose $\bvarphi,\bvarphi_0\in\mathbb{R}^d$ are Euclidean vectors such that $\bU = \bU(\bvarphi)$ and $\bU_0 = \bU(\bvarphi_0)$, where $\bvarphi = \vect(\bA)$, $\bvarphi_0 = \vect(\bA_0)$ for $(p - r)\times r$ matrix $\bA,\bA_0$ with $\|\bA\|_2,\|\bA_0\|_2 < 1$, and $d = (p - r)r$. Let $\bM,\bM_0\in\mathbb{M}_+^{r}$ and $\bmu,\bmu_0$ be the vectors whose elements are the upper diagonal entries of $\bM$ and $\bM_0$, respectively. Denote
% \begin{align*}
% \bSigma(\btheta) = \bU(\bvarphi)\bM\bU(\bvarphi)\transpose
%  % + \sigma^2\eye_p
% ,\quad
% \bSigma(\btheta_0) = \bU(\bvarphi_0)\bM_0\bU(\bvarphi_0)\transpose.
%  % + \sigma^2\eye_p,
% \end{align*}
% where $\sigma^2\geq 0$ is a nuisance parameter. 
Under the setup and notations of Sections \ref{sub:euclidean_representation_of_subspaces} and \ref{sub:euclidean_represention_of_low_rank_matrices_and_intrinsic_perturbation}, if $\|\bA_0\|_2 < 1$, $\|\bA\|_2 < 1$, $\bM_0$ and $\bM$ are positive definite, and 
\begin{align*}
\|\bSigma(\btheta) - \bSigma(\btheta_0)\|_{\mathrm{F}}\leq \frac{(1 - \|\bA_0\|_2^2)^2\lambda_r(\bM_0)}{4\sqrt{2}(1 + \|\bA_0\|_2^2)^2},
\end{align*}
then
\begin{align*}
\|\btheta - \btheta_0\|_2
&\leq 
\left[1 
% + \frac{8(1 + \|\bA_0\|_2^2)^2}{\lambda_{r}(\bM_0)(1 - \|\bA_0\|_2^2)^2}
 + \frac{16\sqrt{2}\{1 + \lambda_{1}(\bM_0)\}(1 + \|\bA_0\|_2^2)}{\lambda_{r}(\bM_0)(1 - \|\bA_0\|_2^2)}\right]
 % \\&\quad\times
\|\bSigma(\btheta) - \bSigma(\btheta_0)\|_{\mathrm{F}}.
\end{align*}
\end{theorem}
When $\bM = \bM_0 = \eye_r$, 
% which corresponds to the perturbation of projection matrices, 
Theorem \ref{thm:intrinsic_deviation_Sigma}, together with the fact that 
$\|\bU\bU\transpose - \bV\bV\transpose\|_{\mathrm{F}} = \sqrt{2}\|\sin\Theta(\bU,\bV)\|_{\mathrm{F}}$ for any $\bU,\bV\in\mathbb{O}(p, r)$, directly leads to the following corollary regarding intrinsic perturbation of subspaces. 
% , which
% asserts that the Frobenius sine-theta distance between subspaces can be lower bounded by the Euclidean distance between their representing vectors. 
 % controls the perturbation of the representing Euclidean vectors by the Frobenius sine-theta distance of the corresponding subspaces. 
\begin{corollary}
\label{corr:intrinsic_deviation_Projection}
% Let $\bU = [\bQ_1\transpose, \bQ_2\transpose]$ and $\bU_0 = [\bQ_{01}\transpose, \bQ_{02}\transpose]\transpose\in\mathbb{O}_+(p, r)$. Suppose $\bvarphi,\bvarphi_0\in\mathbb{R}^d$ are Euclidean vectors such that $\bU = \bU(\bvarphi)$ and $\bU_0 = \bU(\bvarphi_0)$, where $\bvarphi = \vect(\bA)$, $\bvarphi_0 = \vect(\bA_0)$ for $(p - r)\times r$ matrix $\bA,\bA_0$ with $\|\bA\|_2,\|\bA_0\|_2 < 1$, and $d = (p - r)r$. 
% If 
Under the setup and notations of Section \ref{sub:euclidean_representation_of_subspaces},
 % and \ref{sub:euclidean_represention_of_low_rank_matrices_and_intrinsic_perturbation}, 
 if $\|\bA_0\|_2 < 1$, $\|\bA\|_2 < 1$, 
 % $\bM_0 = \bM = \eye_r$ so that $\bSigma = \bU\bU\transpose$ and $\bSigma_0 = \bU_0\bU_0\transpose$, 
 and 
\begin{align*}
% \sqrt{2}
\|\sin\Theta\{\bU(\bvarphi), \bU(\bvarphi_0)\}\|_{\mathrm{F}}
 % = \|\bU\bU\transpose - \bU_0\bU_0\transpose\|_{\mathrm{F}}
 \leq \frac{(1 - \|\bA_0\|_2^2)^2}{8(1 + \|\bA_0\|_2^2)^2},
\end{align*}
then
\begin{align*}
\|\bvarphi - \bvarphi_0\|_2
&\leq 
% \left\{1 + \frac{8(1 + \|\bA_0\|_2^2)^2}{(1 - \|\bA_0\|_2^2)^2} + \frac{16(1 + \|\bA_0\|_2^2)}{(1 - \|\bA_0\|_2^2)}\right\}\|\bU\bU\transpose - \bU_0\bU_0\transpose\|_{\mathrm{F}}\\
% & = 
\sqrt{2}
\left\{1 + \frac{32\sqrt{2}(1 + \|\bA_0\|_2^2)^2}{(1 - \|\bA_0\|_2^2)^2}
 % + \frac{16(1 + \|\bA_0\|_2^2)}{(1 - \|\bA_0\|_2^2)}
 \right\}\|\sin\Theta\{\bU(\bvarphi), \bU(\bvarphi_0)\}\|_{\mathrm{F}}.
\end{align*}
On the other hand, the following reverse inequality always holds for all $\bvarphi$ and $\bvarphi_0$:
\[
\|\sin\Theta\{\bU(\bvarphi), \bU(\bvarphi_0)\}\|_{\mathrm{F}}\leq 4\|\bvarphi - \bvarphi_0\|_2.
\]
\end{corollary}
% The implication of 
Corollary \ref{corr:intrinsic_deviation_Projection} suggests
% is 
that, 
locally around $\bvarphi_0$, the Frobenius sine-theta distance between subspaces $\Span\{\bU(\bvarphi)\}$ and $\Span\{\bU(\bvarphi_0)\}$ is equivalent to the Euclidean distance between their representing Eulicdean vectors $\bvarphi$ and $\bvarphi_0$. Furthermore, by taking Theorem \ref{thm:second_order_deviation_CT} into consideration, we conclude immediately that $\|\sin\Theta\{\bU(\bvarphi),\bU(\bvarphi_0)\}\|_{\mathrm{F}}$ is locally equivalent to $\|\bU(\bvarphi) - \bU(\bvarphi_0)\|_{\mathrm{F}}$.
% , provided that $\|\bA_0\|_2$ is bounded away from $1$. 
This result is formally stated in the following Theorem for ease of reference. 
\begin{theorem}\label{thm:alignment_free_Davis_Kahan}
Under the setup and notations of Section \ref{sub:euclidean_representation_of_subspaces},
 % and \ref{sub:euclidean_represention_of_low_rank_matrices_and_intrinsic_perturbation}, 
 if $\|\bA_0\|_2 < 1$, $\|\bA\|_2 < 1$, 
 % $\bM_0 = \bM = \eye_r$ so that $\bSigma = \bU\bU\transpose$ and $\bSigma_0 = \bU_0\bU_0\transpose$, 
 and 
 \begin{align*}
% \sqrt{2}
\|\sin\Theta\{\bU(\bvarphi), \bU(\bvarphi_0)\}\|_{\mathrm{F}}
 % = \|\bU\bU\transpose - \bU_0\bU_0\transpose\|_{\mathrm{F}}
 \leq \frac{(1 - \|\bA_0\|_2^2)^2}{8(1 + \|\bA_0\|_2^2)^2},
\end{align*}
then
\begin{align*}
\|\bU(\bvarphi) - \bU(\bvarphi_0)\|_{\mathrm{F}}
&\leq 
% \left\{1 + \frac{8(1 + \|\bA_0\|_2^2)^2}{(1 - \|\bA_0\|_2^2)^2} + \frac{16(1 + \|\bA_0\|_2^2)}{(1 - \|\bA_0\|_2^2)}\right\}\|\bU\bU\transpose - \bU_0\bU_0\transpose\|_{\mathrm{F}}\\
% & = 
4\left\{1 + \frac{32\sqrt{2}(1 + \|\bA_0\|_2^2)^2}{(1 - \|\bA_0\|_2^2)^2}
 % + \frac{16(1 + \|\bA_0\|_2^2)}{(1 - \|\bA_0\|_2^2)}
 \right\}\|\sin\Theta\{\bU(\bvarphi), \bU(\bvarphi_0)\}\|_{\mathrm{F}}.
\end{align*}
On the other hand, the following reverse inequality always holds for all $\bvarphi$ and $\bvarphi_0$:
\[
\|\sin\Theta\{\bU(\bvarphi), \bU(\bvarphi_0)\}\|_{\mathrm{F}}\leq \sqrt{2}\|\bU(\bvarphi) - \bU(\bvarphi_0)\|_{\mathrm{F}}.
\]
\end{theorem}
\begin{remark}
Given two Stiefel matrices $\bV$ and $\bV_0$ in $\mathbb{O}(p, r)$, their Frobenius sine-theta distance is equivalent to $\|\bV - \bV_0\bW^*\|_{\mathrm{F}}$, where $\bW^*$ is the solution to the orthogonal Procrustes problem $\min_{\bW\in\mathbb{O}(r)}\|\bV - \bV_0\bW\|_{\mathrm{F}}$ and can be computed explicitly using $\bV_0$ and $\bV$. Formally, by Lemma 1 in \cite{cai2018}, we have
\[
\|\sin\Theta(\bV, \bV_0)\|_{\mathrm{F}}
\leq
\|\bV - \bV_0\bW^*\|_{\mathrm{F}}
\leq
\sqrt{2}\|\sin\Theta(\bV, \bV_0)\|_{\mathrm{F}}.
\]
As pointed out in \cite{cai2018}, it is sometimes more convenient to work with the explicit expression $\|\bV - \bV_0\bW^*\|_{\mathrm{F}}$ based on the representing Stiefel matrices than to work with the original definition of the sine-theta distance between subspaces. Nevertheless, the orthogonal alignment matrix $\bW^*$ may still cause inconvenience for some theoretical analyses.
% , since $\bW^*$ depends on both the underlying referential Stiefel matrix $\bV_0$ and the perturbed matrix $\bV$, which is random in a broad range of statistical contexts. 
In contrast, Theorem \ref{thm:alignment_free_Davis_Kahan} loosely asserts that, in a small neighborhood of $\bV_0$ (or equivalently, a small neighborhood of $\Span(\bV_0)$ with respect to the sine-theta metric), by finding suitable representing Stiefel matrices $\bU(\bvarphi),\bU(\bvarphi_0)\in\mathbb{O}_+(p, r)$ such that $\Span\{\bU(\bvarphi)\} = \Span(\bV)$ and $\Span\{\bU(\bvarphi_0)\} = \Span(\bV_0)$, the Frobenius sine-theta distance between the subspaces $\Span(\bV)$ and $\Span(\bV_0)$ is equivalent to $\|\bU(\bvarphi) - \bU(\bvarphi_0)\|_{\mathrm{F}}$, circumventing the orthogonal alignment matrix $\bW^*$ and facilitating many technical analyses. Therefore, the metric $\|\bU(\bvarphi) - \bU(\bvarphi_0)\|_{\mathrm{F}}$ provides an alignment-free and user-friendly local surrogate for the Frobenius sine-theta distance. 
% In other words, the choice of the form of the orthonormal basis $\bU(\bvarphi)$ allows us to use the Frobenius norm $\|\bU(\bvarphi) - \bU(\bvarphi_0)\|_{\mathrm{F}}$ as a surrogate for the $\sin\Theta$ distance between the corresponding subspaces $\Span(\bV)$ and $\Span(\bV_0)$ locally around $\Span(\bV_0)$.  
\end{remark}

% subsection intrinsic_perturbation_theorems (end)

\subsection{The regularity theorem} % (fold)
\label{sub:the_regularity_theorem}
Theorem \ref{thm:DSigma_singular_value} below is the third main result of this work. It asserts that the map $\bSigma(\cdot):\mathscr{D}(p, r)\to\mathscr{S}(p, r)$ is regular by showing that the Fr\'echet derivative $D\bSigma(\cdot)$ has full column rank, and $\sigma_{\min}\{D\bSigma(\btheta_0)\}$ can be lower bounded using $\sigma_r(\bM_0)$ and $\|\bA_0\|_2$.
% the regularity of the Euclidean representation of low-rank matrices under the setup of Sections \ref{sub:euclidean_representation_of_subspaces} and \ref{sub:euclidean_represention_of_low_rank_matrices_and_intrinsic_perturbation}. 
Subsequently, Remark \ref{rmk:statistical_interpretation_regularity} illustrates the statistical impact of Theorem \ref{thm:DSigma_singular_value}: A statistical model parameterized by $\btheta$ through the map $\bSigma(\cdot)$ has a non-singular Fisher information matrix provided that the Fisher information matrix with regard to $\bSigma$ is non-singular. 
In addition, a geometric perspective of Theorem \ref{thm:DSigma_singular_value} is explained in Remark \ref{rmk:regularity_interpretation}. Specifically, the class of low-rank matrices $\mathscr{S}(p, r)$ can be viewed as a $d$-dimensional manifold in $\mathbb{R}^{p^2}$ with $d = r(r + 1)/2 + (p - r)r$ due to the regularity of the map $\bSigma(\cdot)$.
% , and the map $\bSigma:\mathscr{D}(p, r)\to\mathscr{S}(p, r), \btheta\mapsto\bSigma(\btheta)$ establishes a coordinate chart of $\mathscr{S}(p, r)$. 
% We also provide a sample application of Theorem \ref{thm:DSigma_singular_value} to the covariance matrix estimation in Example \ref{example:application_covariance}. 
\begin{theorem}\label{thm:DSigma_singular_value}
% Let $\bU_0 = \bU(\bvarphi_0)$ be defined by \eqref{eqn:Cayley_transform}, where $\bvarphi_0 = \vect(\bA_0)$, and $\bA_0$ is a $(p - r)\times r$ matrix such that $\|\bA\|_2 < 1$. Let $\btheta_0 = [\bvarphi_0\transpose,\bmu_0\transpose]$, where $\bmu_0$ is a $r(r + 1)/2$-dimensional vector whose elements are upper diagonal entries of a $r\times r$ symmetric matrix $\bM_0$. Let $D\bSigma(\btheta)$ and $D_\bvarphi\bSigma(\btheta)$ be defined by \eqref{eqn:covariance_differential}. 
Under the setup and notations of Sections \ref{sub:euclidean_representation_of_subspaces} and \ref{sub:euclidean_represention_of_low_rank_matrices_and_intrinsic_perturbation}, 
% if $\|\bA_0\|_2 < 1$, $\|\bA\|_2 < 1$, 
% $\bM_0$ and $\bM$ are positive definite, 
% then
% Then
\begin{align*} 
\sigma_{\min}\{D_\bvarphi\bSigma(\btheta_0)\} \geq 
\left\{
\begin{aligned}
&\frac{2\sqrt{2}\sigma_{r}(\bM_0)(1 - \|\bA_0\|_2^2)}{(1 + \|\bA_0\|_2^2)^{2}},&\quad\text{if }r \geq 2,\\
&\frac{2\sqrt{2}\sigma_{r}(\bM_0)}{1 + \|\bA_0\|_2^2}, & \quad\text{if } r = 1.
\end{aligned}
\right.
\end{align*}
Furthermore, 
\begin{align*}
\|\{D\bSigma(\btheta_0)\transpose D\bSigma(\btheta_0)\}^{-1}\|_2
\leq\left\{
\begin{aligned}
& 1 + \frac{(1 + 64\|\bM_0\|_2^2)(1 + \|\bA_0\|_2^2)^{4}}{8\lambda_{r}^2(\bM_0)(1 - \|\bA_0\|_2^2)^2},&\quad\text{if }r\geq 2,\\
&1 + \frac{(1 + 64\|\bM_0\|_2^2)(1 + \|\bA_0\|_2^2)^2}{8\lambda_{r}^2(\bM_0)},&\quad\text{if }r = 1.
\end{aligned}
\right.
\end{align*}
\end{theorem}

\begin{remark}[Statistical impact of Theorem \ref{thm:CT_deviation_Sigma}]
\label{rmk:statistical_interpretation_regularity}
Many multivariate statistical models are parameterized by symmetric matrices.
 % $\bOmega\in\mathbb{R}^{p\times p}$.
Formally, let $\mathscr{C}(p)$ be a collection of $p\times p$ symmetric matrices and $\calP = \{p_\bOmega(\bx):\bOmega\in\mathscr{C}(p)\}$ be a parametric model indexed by $\bOmega$, where $(p_\bOmega)_{\bOmega}$ are density functions with regard to some underlying $\sigma$-finite measure. Suppose independent and identically distributed data $\bx_1,\ldots,\bx_n$ are collected from some distribution $p_{\bOmega}(\cdot)$ with $\bOmega\in\mathscr{C}(p)$. Since $\bOmega$ is symmetric, we can further reduce the free parameters to $\vech(\bOmega)$. Denote
% \[
$
\dot{\bell}_\bOmega(\bx_i) = 
\nabla_{\vech(\bOmega)}\log p_\bOmega(\bx_i)
% \frac{\partial\log p_\bOmega(\bx_i)}{\partial{\vech(\bOmega)}}
$
% \]
the score function with regard to $\vech(\bOmega)$, and 
% \[
$\eye(\bOmega) := \expect\{\dot{\bell}_\bOmega(\bx)\dot{\bell}_\bOmega(\bx)\transpose\}$ 
% \equiv\expect_0\left\{\frac{\partial\log p_\bOmega(\bx_i)}{\partial{\vech(\bOmega)}}\frac{\partial\log p_\bOmega(\bx_i)}{\partial{\vech(\bOmega)}\transpose}\right\}
% \] 
the corresponding Fisher information matrix. We assume that the model $\calP$ is regular, namely, there exists an invertible $p^2\times p^2$ matrix $\bPsi(\bOmega_0)$ such that $\eye(\bOmega_0) = \mathbb{D}_p\transpose \bPsi(\bOmega_0) \mathbb{D}_p$. A classical example is the normal covariance model, where $p_\bOmega(\bx) = \det(2\pi\bOmega)^{-1/2}e^{-(1/2)\bx\transpose\bOmega^{-1}\bx}$, $\bOmega$ is the covariance matrix of interest, and $\mathscr{C}(p)$ is a collection of positive definite matrices. In this model, the Fisher information matrix with regard to $\vech(\bOmega)$ evaluated $\bOmega = \bOmega_0$ is given by (see, e.g., Chapter 10 in \citealp{magnus1988linear})
\[
\eye(\bOmega_0) = \frac{1}{2}\mathbb{D}_p\transpose(\bOmega_0^{-1}\otimes \bOmega_0^{-1})\mathbb{D}_p.
\]

However, if instead one restricts $\bOmega$ onto the spiked matrix class
\[
% $
{\mathscr{C}}(p, r) = \{\bOmega := \bSigma + \eye_p:\bSigma\in\mathscr{S}(p, r)\}
% $
\] 
% (which was originally proposed in \cite{doi:10.1198/jasa.2009.0121}) 
and assumes $\bOmega_0 = \bSigma_0 + \eye_p$ for some $\bSigma_0 \in\mathscr{C}(p, r)$,
% \]
% and assume that $\widetilde{\mathscr{C}}(p, r)\subset\mathscr{C}$ as well as $r < p$,
% namely, $\bOmega$ has a spike structure and can be written as $\bOmega = \bSigma + \eye_p$ for a rank-$r$ matrix $\bSigma\in\mathscr{S}$ with $r < p$, and $\mathscr{S}$ is defined in \eqref{eqn:low_rank_matrix_class}, 
then the spiked matrix $\bOmega$ can be represented by a lower dimensional Euclidean vector $\btheta\in\mathscr{D}(p, r)$. 
% with $\mathscr{D}(p, r)$ being the domain of the function $\btheta\mapsto \bSigma(\btheta)$, defined by \eqref{eqn:domain_of_parameterization}. 
Thus, the statistical submodel under the $\btheta$-parameterization can be written as 
\[
\calG(p, r) = \{p_{\bOmega(\btheta)}(\cdot):\bOmega(\btheta) = \bSigma(\btheta) + \eye_p,\btheta\in\mathscr{D}(p, r)\}.
\]
Denote $\btheta_0\in\mathscr{D}(p, r)$ the vector such that $\bOmega_0 = \bSigma(\btheta_0) + \eye_p$. By Theorem \ref{thm:CT_deviation_Sigma}, the Fisher information matrix with respect to the $\btheta$-parameterization in the submodel $\calG(p, r)$ evaluated at $\btheta = \btheta_0$ is given by
\begin{align*}
\eye(\btheta_0)
% & = D\bSigma(\btheta_0)\transpose \eye(\bOmega_0) D\bSigma(\btheta_0)\\
& = D\bSigma(\btheta_0)\transpose (\mathbb{D}_p^\dagger)\transpose\mathbb{D}_p\transpose \bPsi(\bOmega_0) \mathbb{D}_p\mathbb{D}_p^\dagger D\bSigma(\btheta_0).
\end{align*}
Using the fact that $\mathbb{D}_p\mathbb{D}_p^\dagger = (1/2)(\eye_{p^2} + \bK_{pp})$, $(\mathbb{D}_p\mathbb{D}_p^\dagger)(\eye_{p^2} + \bK_{pp})^2 = (\eye_{p^2} + \bK_{pp})$, and $\mathbb{D}_p\mathbb{D}_p^\dagger(\bU_0\otimes\bU_0)\mathbb{D}_p = (\bU_0\otimes\bU_0)\mathbb{D}_p$ (see, e.g., Chapter 4 in \citealp{magnus1988linear}), we further conclude that
\begin{align*}
\eye(\btheta_0)
 = D\bSigma(\btheta_0)\transpose (\mathbb{D}_p^\dagger)\transpose\mathbb{D}_p\transpose \bPsi(\bOmega_0) \mathbb{D}_p\mathbb{D}_p^\dagger D\bSigma(\btheta_0)
 = D\bSigma(\btheta_0)\transpose \bPsi(\bOmega_0) D\bSigma(\btheta_0).
% & = \begin{bmatrix*}
% (\mathbb{D}_p\mathbb{D}_p^\dagger)(\eye_{p^2} + \bK_{pp})(\bU_0\bM_0\otimes\eye_p)D\bU(\bvarphi_0) & 
% (\mathbb{D}_p\mathbb{D}_p^\dagger)(\bU_0\otimes\bU_0)\mathbb{D}_r
% \end{bmatrix*}\\
% & = \begin{bmatrix*}
% (\eye_{p^2} + \bK_{pp})(\bU_0\bM_0\otimes\eye_p)D\bU(\bvarphi_0) & 
%  (1/2)(\eye_{p^2} + \bK_{pp})(\bU_0\otimes\bU_0)\mathbb{D}_r
% \end{bmatrix*}\\
% & = \begin{bmatrix*}
% (\eye_{p^2} + \bK_{pp})(\bU_0\bM_0\otimes\eye_p)D\bU(\bvarphi_0) & 
%  (\bU_0\otimes\bU_0)(1/2)(\eye_{p^2} + \bK_{pp})\mathbb{D}_r
% \end{bmatrix*}\\
% & = \begin{bmatrix*}
% (\eye_{p^2} + \bK_{pp})(\bU_0\bM_0\otimes\eye_p)D\bU(\bvarphi_0) & 
%  (\bU_0\otimes\bU_0)\mathbb{D}_r
% \end{bmatrix*}
%  = 
% D\bSigma(\btheta_0).
\end{align*}
Hence, by Theorem \ref{thm:DSigma_singular_value}, the Fisher information matrix $\eye(\btheta_0)$ with respect to the $\btheta$-parameterization in the submodel $\calG(p, r)$ is also non-singular provided that $\bPsi(\bOmega_0)$ is non-singular. 

In the aforementioned covariance model, the submodel $\calG(p, r) = \{\mathrm{N}(\zero_p, \bSigma + \eye_r):\bSigma\in\mathscr{S}(p, r)\}$ is also referred to as the spiked covariance model \citep{doi:10.1198/jasa.2009.0121}. The Fisher information matrix with regard to the $\btheta$-parameterization evaluated at $\btheta = \btheta_0$ is 
\[
\eye(\btheta_0) = \frac{1}{2}D\bSigma(\btheta_0)\transpose (\bOmega_0^{-1}\otimes \bOmega_0^{-1}) D\bSigma(\btheta_0).
\]
In particular, $\eye(\btheta_0)$ is always non-singular since $\bOmega_0$ is invertible. Further application of the non-singularity of the Fisher information matrix in the spiked covariance model will be discussed in Section \ref{sub:bayesian_sparse_pca_and_non_intrinsic_loss}. 
\end{remark} 

\begin{remark}[Geometric interpretation of Theorem \ref{thm:CT_deviation_Sigma}]
\label{rmk:regularity_interpretation}
The geometric interpretation of Theorem \ref{thm:DSigma_singular_value} can be loosely stated as follows. The function $\btheta\mapsto \vect\{\bSigma(\btheta)\}$ can be viewed as a function from $\mathbb{R}^d\to\mathbb{R}^{p^2}$, where $d = r(r + 1)/2 + (p - r)r$. Note that $d\leq p^2$ because $r\leq p$, and Theorem \ref{thm:DSigma_singular_value} 
asserts that the Fr\'echet derivative $D\bSigma(\btheta) = \partial\vect\{\bSigma(\btheta)\}/\partial\btheta\transpose$ has full column rank. 
% Therefore, the function $\btheta\mapsto \bSigma(\btheta)$ is an injection, namely, $\bSigma(\btheta_1) \neq \bSigma(\btheta_2)$ whenever $\btheta_1\neq \btheta_2$. 
Now suppose $\btheta = [\theta_1,\ldots,\theta_d]\transpose$ and write
\[
\vect\{\bSigma(\btheta)\}
= \begin{bmatrix*}
\sigma_1(\theta_1,\ldots,\theta_d)\\\vdots
% \\\sigma_d(\theta_1,\ldots,\theta_d)
% \\
% \sigma_{d + 1}(\theta_1,\ldots,\theta_d)\\\vdots
\\\sigma_{p^2}(\theta_1,\ldots,\theta_d)
\end{bmatrix*}.
\]
Without loss of generality, we may assume the Jacobian matrix 
\[
\frac{\partial(\sigma_1,\ldots,\sigma_d)}{\partial(\theta_1,\ldots,\theta_d)}:
=
\begin{bmatrix*}
\frac{\partial\sigma_1}{\partial\theta_1} & \ldots & \frac{\partial\sigma_1}{\partial\theta_d}\\
\vdots & & \vdots\\
\frac{\partial\sigma_d}{\partial\theta_1} & \ldots & \frac{\partial\sigma_d}{\partial\theta_d}
\end{bmatrix*}
\]
is non-singular at $\btheta = \btheta_0$. Consider an extension of $\bSigma(\cdot)$ defined by $\bg:\mathscr{D}(p, r)\times\mathbb{R}^{p^2 - d}\to\mathbb{R}^{p^2}$, where
\[
\bg(\theta_1,\ldots,\theta_d,\theta_{d + 1},\ldots,\theta_{p^2})
 = \begin{bmatrix*}
 \sigma_1(\theta_1,\ldots,\theta_d)\\\vdots
\\\sigma_d(\theta_1,\ldots,\theta_d)
\\
\sigma_{d + 1}(\theta_1,\ldots,\theta_d) + \theta_{d + 1}\\\vdots
\\\sigma_{p^2}(\theta_1,\ldots,\theta_d) + \theta_{p^2}
 \end{bmatrix*}.
\]
Then the Jacobian matrix of $\bg$ with respect to the vector $[\theta_1,\ldots,\theta_d,\theta_{d + 1},\ldots,\theta_p]\transpose$ is also non-singular at $[\btheta_0\transpose, \theta_{d + 1},\ldots,\theta_{p^2}]\transpose$ for any $[\theta_{d+1},\ldots,\theta_{p^2}]\transpose\in\mathbb{R}^{p^2 - d}$. 
% the function $\bSigma:\mathscr{D}\to\mathscr{S},\btheta\mapsto \bSigma(\btheta)$ is both continuous and invertible. An
Therefore, a standard argument based on the inverse mapping theorem implies that
 % $\btheta\mapsto\bSigma(\btheta)$ is a homeomorphism between $\mathscr{D}$ and $\mathscr{S}$. Using the terminology in differential geometry, $(\mathscr{D}, \bSigma(\cdot))$ is a coordinate chart of $\mathscr{S}$, and 
 $\mathscr{S}(p, r)$ is a $d$-dimensional manifold and $\bSigma(\cdot)$ serves as a coordinate system for $\mathscr{S}(p, r)$ (see, e.g., Chapter 5 in \citealp{SpivakManifoldCalculus}). 
\end{remark}

\subsection{Extension to general rectangular matrices} % (fold)
\label{sub:extension_to_general_rectangular_matrices}
Using a similar approach, we can extend the Euclidean representation framework for symmetric low-rank matrices to general and possibly rectangular low-rank matrices, which can be applied to a broader range of problems. Rather than using Euclidean vectors to represent the eigenspaces as an intermediate step, we use Euclidean vectors to represent the corresponding right singular subspaces as follows. 

Suppose $\bSigma$ is a $p_1\times p_2$ matrix with rank $r$ and let $\bSigma = \bV_1\bLambda\bV_2\transpose$ be its singular value decomposition, where $\bV_1\in\mathbb{O}(p_1, r)$, $\bV_2\in\mathbb{O}(p_2, r)$, and $\bLambda = \mathrm{diag}\{\sigma_1(\bSigma),\ldots,\sigma_r(\bSigma)\}$.
 % is a diagonal matrix of singular values of $\bSigma$. 
Assume that $\eye_{p_2\times r}\transpose\bV_2$ is invertible and yields singular value decomposition $\eye_{p_2\times r}\transpose\bV_2 = \bW_1\mathrm{diag}\{\sigma_1(\eye_{p_2\times r}\transpose\bV_2),\ldots,\sigma_r(\eye_{p_2\times r}\transpose\bV_2)\}\bW_2\transpose$, where $\bW_1,\bW_2\in\mathbb{O}(r)$. Similar to Section \ref{sub:euclidean_represention_of_low_rank_matrices_and_intrinsic_perturbation}, we can parameterize the rectangular matrix $\bSigma$ using the following two matrices:
\begin{align*}
\bU = \bV_2\bW_2\bW_1\transpose,\quad \bM = \bV_1\bLambda\bW_2\bW_1\transpose.
\end{align*}
Clearly, $\bSigma = \bM\bU\transpose$, where $\bM\in\mathbb{R}^{p_1\times r}$ is a constraint-free full-rank matrix and $\bU\in\mathbb{O}_+(p_2, r)$. Invoking the Cayley parameterization \eqref{eqn:Cayley_transform}, we can further reparameterize $\bU$ using a $(p_2 - r)\times r$ matrix $\bA$ with $\|\bA\|_2 < 1$, such that $\bU = \bU(\bvarphi)$, where $\bvarphi = \vect(\bA) \in \mathbb{R}^{(p_2 - r)r}$.
% , and
% \[
% \bU(\bvarphi) = (\eye_p + \bX_\bvarphi)(\eye_p - \bX_\bvarphi)^{-1}\eye_{p_2\times r},\quad
% \text{where}\quad
% \bX_\bvarphi = \begin{bmatrix*}
% \zero_{r\times r} & -\bA\transpose\\
% \bA & \zero_{(p_2 - r)\times (p_2 - r)}
% \end{bmatrix*}.
% \]
In particular, we see that $\Span(\bV_2) = \Span(\bU)$. 
% the columns of $\bU$ and those of $\bV_2$ span the same $r$-dimensional subspace in $\mathbb{R}^{p_2}$. 
Now denote $\btheta = [\bvarphi\transpose, \bmu\transpose]\transpose$, where $\bmu = \vect(\bM)$. Then we can view $\bSigma$ generically as a matrix-valued function of $\btheta$:
\begin{align}
\label{eqn:Cayley_transform_rectangular}
\btheta\mapsto \bSigma(\btheta) = \bM(\bmu) \bU(\bvarphi)\transpose,
\end{align}
where $\bM(\cdot):\mathbb{R}^{p_1r}\to\mathbb{R}^{p_1\times r}$ is the inverse of the function $\vect(\cdot):\mathbb{R}^{p_1\times r}\to\mathbb{R}^{p_1r}$. 

We use $\bSigma_0\in\mathbb{R}^{p_1\times p_2}$ to denote the referential matrix of interest. 
% In many statistical applications, it typically corresponds to the true value of the parameter that gives rise to the distribution of the observed data. 
Let $\mathrm{rank}(\bSigma_0) = r \leq \min(p_1, p_2)$. Recall that the right singular vector matrix of $\bSigma_0$ is $\bV_2$. We further assume that $\eye_{p_2\times r}\transpose \bV_2$ is non-singular. Similarly, $\bSigma_0$ can also be represented by a Euclidean vector $\btheta_0$. Write $\btheta_0 = [\bvarphi_0\transpose, \bmu_0\transpose]\transpose\in\mathbb{R}^{(p_2 - r)r}\times \mathbb{R}^{p_1r}$, where $\bvarphi_0 = \vect(\bA_0)$ for a $(p_2 - r)\times r$ matrix $\bA_0$ with $\|\bA_0\|_2 < 1$ and $\bmu_0 = \vect(\bM_0)$ for a $p_1\times r$ full-rank matrix $\bM_0$.
% , and $\btheta_0$ is the inverse image of $\bSigma_0$ under the map $\bSigma(\cdot)$, namely, $\bSigma_0 = \bSigma(\btheta_0)$. 
Let $\bU_0 := \bU(\bvarphi_0)$.
% , where $\bU(\cdot)$ is the Cayley parameterization defined in \eqref{eqn:Cayley_transform}. 
We then define the following matrix-valued functions, extending the functions in \eqref{eqn:covariance_differential} to general rectangular $\bSigma(\cdot)$:
\begin{equation}
\label{eqn:Cayley_transform_rectangular_DSigma}
\begin{aligned}
D_\bvarphi \bSigma(\btheta)&= \bK_{p_2p_1}(\bM \otimes \eye_{p_2}) D\bU(\bvarphi), \\
D_\bmu \bSigma(\btheta)&= \bU(\bvarphi)\otimes \eye_{p_1},\\
D\bSigma(\btheta)&= 
\begin{bmatrix*}
D_\bvarphi \bSigma(\btheta) & D_\bmu \bSigma(\btheta)
\end{bmatrix*}.
\end{aligned}
\end{equation}

Theorem \ref{thm:CT_deviation_Sigma_rectangular} below, which extends Theorem \ref{thm:CT_deviation_Sigma} to general rectangular matrices, asserts that $D\bSigma(\btheta_0)$ defined in \eqref{eqn:Cayley_transform_rectangular_DSigma} is exactly the Fr\'echet derivative of the map $\bSigma(\cdot)$ defined by \eqref{eqn:Cayley_transform_rectangular} evaluated at $\btheta = \btheta_0$.
\begin{theorem}
\label{thm:CT_deviation_Sigma_rectangular}
% Let $\bU_0 = \bU(\bvarphi_0)$ be defined by \eqref{eqn:Cayley_transform}, where $\bvarphi_0 = \vect(\bA_0)$, and $\bA_0\in\mathbb{R}^{(p_2 - r)\times r}$. Let $\btheta_0 = [\bvarphi_0\transpose,\bmu_0\transpose]$, where $\bmu_0$ is a $p_1r$-dimensional vector such that $\bmu_0 = \vect(\bM_0)$ for a $p_1\times r$ matrix $\bM_0$. Let $D_\bvarphi\bSigma(\btheta)$, $D_\bmu\bSigma(\btheta)$, and $D\bSigma(\btheta)$ be defined by \eqref{eqn:Cayley_transform_rectangular_Sigma}, where $\btheta = [\bvarphi\transpose,  \bmu\transpose]\transpose$, where $\bvarphi = \vect(\bA)\transpose$ and $\|\bA\|_2 < 1$. Denote $\bSigma(\btheta) = \bM\bU(\bvarphi)\transpose$, where $\bmu = \vect(\bM)$ for a $p_1\times r$ matrix $\bM$, and similarly $\bSigma_0 = \bSigma(\btheta_0) = \bM_0\bU_0\transpose$. 
Under the setup and notations of Sections \ref{sub:euclidean_representation_of_subspaces} and \ref{sub:extension_to_general_rectangular_matrices}, if $\|\bA_0\|_2 < 1$ and $\|\bA\|_2 < 1$, then
 % for all $\btheta$ with $\|\btheta - \btheta_0\|_2 < 1/(2\sqrt{2})$, 
 there exists a $p_1\times p_2$ matrix $\bR(\btheta, \btheta_0)$ depending on $\btheta, \btheta_0$, such that
\begin{align*}
\vect\{\bSigma(\btheta) - \bSigma_0\}
& = D\bSigma(\btheta_0)(\btheta - \btheta_0) + \vect\{\bR(\btheta, \btheta_0)\}
% \\
% & = D_\bvarphi\bSigma(\btheta_0)(\bvarphi - \bvarphi_0) + D_\bmu\bSigma(\btheta_0)(\bmu - \bmu_0) + \vect\{\bR(\btheta, \btheta_0)\}
,
\end{align*}
where $\bR(\btheta, \btheta_0)$ satisfies 
% \[
$\|\bR(\btheta, \btheta_0)\|_{\mathrm{F}}\leq (4 + 8\|\bM_0\|_2)\|\btheta - \btheta_0\|_{\mathrm{F}}^2$ for all $\btheta, \btheta_0$. In matrix form, the above displayed equation can be written as
% \]
\begin{align*}
\bSigma(\btheta) - \bSigma_0
& = 2\bM_0\bU_0(\eye_{p_2} - \bX_{\bvarphi_0})^{-1}(\bX_{\bvarphi_0} - \bX_\bvarphi)(\eye_{p_2} - \bX_{\bvarphi_0})^{-\mathrm{T}}
\\
&\quad
 + (\bM - \bM_0)\bU_0\transpose + \bR(\btheta, \btheta_0).
\end{align*}
% where
% \[
% \bX_\bvarphi = \begin{bmatrix*}
% \zero_{r\times r} & -\bA\transpose\\
% \bA & \zero_{(p_2 - r)\times (p_2 - r)}
% \end{bmatrix*},\quad\text{and}\quad
% \bX_{\bvarphi_0} = \begin{bmatrix*}
% \zero_{r\times r} & -\bA_0\transpose\\
% \bA_0 & \zero_{(p_2 - r)\times (p_2 - r)}
% \end{bmatrix*}.
% \]
\end{theorem}

Analogously, Theorem \ref{thm:DSigma_singular_value_rectangular} below also extends Theorem \ref{thm:DSigma_singular_value} to general rectangular $\bSigma(\cdot)$. 
\begin{theorem}\label{thm:DSigma_singular_value_rectangular}
Under the setup and notations of Sections \ref{sub:euclidean_representation_of_subspaces} and \ref{sub:extension_to_general_rectangular_matrices}, 
% if $\|\bA_0\|_2 < 1$ and $\|\bA\|_2 < 1$, then 
\[
\|\{D\bSigma(\btheta_0)\transpose D\bSigma(\btheta_0)\}^{-1}\|_2\leq 
\left\{
\begin{aligned}
&1 + \frac{(1 + 8\|\bM_0\|_2^2)(1 + \|\bA_0\|_2^2)^4}{4\sigma_r^2(\bM_0)(1 - \|\bA_0\|_2^2)^2},\quad&\text{if }r \geq 2,\\
&1 + \frac{(1 + 8\|\bM_0\|_2^2)(1 + \|\bA_0\|_2^2)^2}{4\sigma_r^2(\bM_0)},\quad&\text{if }r = 1.
\end{aligned}\right.
\]
\end{theorem}

% subsection extension_to_general_rectangular_matrices (end)

\section{Applications} % (fold)
\label{sec:applications}

\subsection{Bayesian sparse spiked covariance model} % (fold)
\label{sub:bayesian_sparse_pca_and_non_intrinsic_loss}
% [Basic model and brief introduction]
This section presents the analysis of Bayesian sparse spiked covariance model and explores the posterior contraction rate with regard to the spectral norm. 
In general, the posterior contraction rate of a Bayesian model under the intrinsic metric of the sampling model (i.e., the Fisher information metric) can be established following the seminal work of \cite{ghosal2000convergence} and its offsprings. However, because the spectral norm is a non-intrinsic metric as opposed to the Frobenius norm, the rate-optimal posterior contraction in the spectral norm for Bayesian sparse spiked covariance model is a highly non-trivial result, as discussed in \cite{gine2011} and \cite{hoffmann2015}. In this section, leveraging the technical tools developed in Section \ref{sec:main_results}, we show that the posterior contraction rate under the spectral norm is minimax optimal. 

The spiked covariance model was originally named by \cite{doi:10.1198/jasa.2009.0121} and has been explored in several works, including \cite{10.2307/24307692}, \cite{cai2013sparse}, and \cite{donoho2018}. Due to the structural convenience, the spiked covariance model has been used as a natural probabilistic model for principal component analysis (PCA). Formally, let $\by_1,\ldots,\by_n$ be independent and identically distributed $\mathrm{N}(\zero_p, \bOmega)$ random vectors where $\bOmega\in\mathbb{M}_+(p)$. The spiked covariance model posits the following structure on the covariance matrix $\bOmega$:
\begin{align}
\label{eqn:spiked_covariance}
\bOmega = \bV\bLambda\bV\transpose + \eye_p,
\end{align}
where $\bV\in\mathbb{O}(p, r)$ is the matrix of eigenvectors corresponding to the $r$-largest eigenvalues of $\bOmega$, $\bLambda = \mathrm{diag}(\lambda_1,\ldots,\lambda_r)$ is a $r\times r$ diagonal matrix with $\lambda_1\geq\ldots\geq\lambda_r > 0$, and $r \ll p$. 
In a high-dimensional regime where the model dimension $p$ far exceeds the number of samples $n$, \cite{doi:10.1198/jasa.2009.0121} showed that the classical PCA might lead to inconsistent estimates, and certain structural assumptions are needed, e.g., a sparse structure \citep{doi:10.1198/jasa.2009.0121,cai2013sparse} or an effective rank constraint \citep{koltchinskii2017new,koltchinskii2017}. 
Here we focus on the case where the leading eigenvector matrix $\bV$ exhibits the row sparsity. Formally, we define 
the support of $\bV$ as 
% to be the indices of the rows of $\bV$ that are non-zero, namely, 
$\mathrm{supp}(\bV) := \{j\in[p]:[\bV]_{j*}\neq \zero\}$ and assume that $\bV$ satisfies the row sparsity constraint that $|\mathrm{supp}(\bV)|\leq s$. Note that the row sparsity is subspace invariant, i.e., $\mathrm{supp}(\bV) = \mathrm{supp}(\bV\bW)$ for any $\bW\in\mathbb{O}(r)$. 
% that $s\geq r$ due to the fact that $\bV\transpose\bV = \eye_r$. 
Correspondingly, the sparse structure of $\bV$ motivates the development of sparse PCA methods. For an incomplete list of works related to the sparse spiked covariance model and sparse PCA, see \cite{doi:10.1198/106186006X113430,amini2009,doi:10.1198/jasa.2009.0121,pmlr-v22-vu12,vu2013minimax,cai2013sparse,ma2013,berthet2013,lei2015,cai2015optimal}.

% In the case where $p/n\to \gamma$ for some constant $\gamma\in (0, 1]$, \cite{10.2307/24307692} and \cite{donoho2018} studied the asymptotics of the sample eigenstructure. In \cite{koltchinskii2017new} and \cite{koltchinskii2017}, the so-called effective rank assumption is made, in the sense that $\mathrm{tr}(\bOmega)/\|\bOmega\|_2 = o(n)$, which in turn requires that $\lambda_1$ needs to diverge to $\infty$ sufficiently fast. 

When the parameter of interest is the principal subspace $\Span(\bV)$, 
% with a minimal assumption on the sampling distribution, \cite{pmlr-v22-vu12} and 
\cite{vu2013minimax} established the following minimax rate under the Frobenius sine-theta distance: 
\begin{align}
\label{eqn:minimax_rate_SPCA_Frobenius}
\inf_{\widehat{\bV}}\sup_{\bV\in\mathbb{O}(p, r),|\mathrm{supp}(\bV)|\leq s}\expect_\bV\{\|\sin\Theta(\widehat\bV, \bV)\|_{\mathrm{F}}^2\} \asymp \frac{rs + s\log p}{n}.
\end{align}
Furthermore, \cite{cai2015optimal} derived the minimax rate for the principle subspace under the spectral sine-theta distance:
\begin{align}
\label{eqn:minimax_rate_SPCA_spectral}
\inf_{\widehat{\bV}}\sup_{\bOmega\in\Theta_0(s, p, r, \lambda, \tau)}\expect_\bOmega\{\|\sin\Theta(\widehat{\bV}, \bV)\|_2^2\} \asymp \frac{s\log p}{n},
% \frac{(\lambda + 1)s}{\lambda^2 n}\log\frac{ep}{s} \wedge 1,
\end{align}
where $\Theta_0(s, p, r, \lambda, \tau) = \{\bV\bLambda\bV\transpose + \eye_p:\bV\in\mathbb{O}(p, r), |\mathrm{supp}(\bV)|\leq s,\lambda/\tau\leq \lambda_r(\bLambda)\leq \lambda_1(\bLambda)\leq \lambda\}$, and $\lambda,\tau$ are bounded away from $0$ and $\infty$.
% , the minimax rate under the spectral $\sin\Theta$ distance between principal subspaces is exactly $\sqrt{(s\log p)/n}$. 
Note that the minimax rate \eqref{eqn:minimax_rate_SPCA_Frobenius} under the Frobenius sine-theta distance has an extra term $\sqrt{rs/n}$, which is inferior than \eqref{eqn:minimax_rate_SPCA_spectral} when $r\gg \log p$. 

In this section, we focus on the posterior contraction rate of Bayesian sparse spiked covariance model. We are particularly interested in the posterior contraction under the spectral sine-theta distance between principal subspaces when $r\gg \log p$, in which the phase transition phenomenon between the two minimax rates \eqref{eqn:minimax_rate_SPCA_Frobenius} and \eqref{eqn:minimax_rate_SPCA_spectral} occurs. 
% In contrast to the well-developed theory from the frequentist side, the literature regarding Bayesian approaches for sparse spiked covariance model is slightly underexplored. 
\cite{pati2014posterior} first studied the minimax-optimal posterior contraction of $\bOmega$ with sparse priors, assuming the rank $r$ is bounded. Under a more general assumption that $|\mathrm{supp}(\bV)|\leq rs$, \cite{gao2015rate} established the rate-optimal posterior contraction of Bayesian sparse PCA under the Frobenius sine-theta distance. Recently, \cite{xie2018bayesian} and \cite{ning2021spike} focused on the posterior contraction rate under the spectral norm, assuming that $r\log n\lesssim \log p$. 
% Under the same low-rank assumption, \cite{Ning2020} studied the contraction rate under the spectral norm for both the exact posterior distribution and the variational distribution.
% , and proposed computationally efficient algorithms for Bayesian inference with sparse priors. 
% We remark that the assumption $r\log n\lesssim \log p$ greatly simplies the problem
%  % of posterior contraction in spectral norm 
% since the minimax rate \eqref{eqn:minimax_rate_SPCA_spectral} under the spectral $\sin\Theta$ distance coincides with the minimax rate \eqref{eqn:minimax_rate_SPCA_Frobenius} under the Frobenius $\sin\Theta$ distance, when $\lambda_1(\bLambda)$ and $\lambda_r(\bLambda)$ are bounded away from $0$ and $\infty$. Namely, rate-optimal posterior contraction under the Frobenius $\sin\Theta$ distance directly implies rate-optimal posterior contraction under the spectral $\sin\Theta$ distance when $r\log n\lesssim \log p$. 
% [a little introduction to Bayesian sparse pca]
Under the regime $r/\log p\to\infty$, it is unknown whether the rate-optimal posterior contraction under the (non-intrinsic) spectral sine-theta distance is achievable, which is precisely the gap we aim to fill in. 
% In the following discussion, we do not impose rank assumption other than the automatic constraint $r\leq s$. 

% Sampling model for spiked covariance matrix model: A spiked covariance matrix estimation problem \citep{johnstone2001distribution}:
% \begin{align}
% \by_i\iidsim\mathrm{N}(0, \bSigma),\quad \bSigma = \bU\bLambda \bU\transpose{} + \eye_p,\quad \bU\in\mathbb{O}(p,r),\quad\bLambda = \mathrm{diag}(\lambda_1,\cdots,\lambda_r).
% \end{align}
% We denote the true value of $\bOmega$ corresponding to the distribution of the data $\by_1,\ldots,\by_n$ by $\bOmega_0$. 
Recall that in the spiked covariance model \eqref{eqn:spiked_covariance}, the leading eigenvector matrix $\bV$ can only be identified up to an orthogonal matrix in $\mathbb{O}(r)$ in the presence of eigenvalue multiplicity. 
Because, for any covariance matrix $\bOmega$ of the form \eqref{eqn:spiked_covariance}, there exists some permutation matrix $\bPi$ such that $\eye_{p\times r}\transpose(\bPi\bV)$ is non-singular, and the leading eigenvector matrix of $\bPi\bOmega\bPi\transpose$ is exactly $\bPi\bV$. Therefore, without loss of generality, we assume that $\eye_{p\times r}\transpose\bV$ is invertible, namely, the top square block of $\bV$ is non-singular. By the construction in Sections \ref{sub:euclidean_representation_of_subspaces} and \ref{sub:euclidean_represention_of_low_rank_matrices_and_intrinsic_perturbation}, $\bOmega$ can be written as $\bOmega = \bSigma + \eye_p$ for some $\bSigma\in\mathscr{S}(p, r)$, and there exist some $\bU\in\mathbb{O}_+(p, r)$ and $\bM\in\mathbb{M}_+(r)$, such that
\[
\bOmega = \bSigma + \eye_p = \bU\bM\bU\transpose + \eye_p
% ,\quad \bU\in\mathbb{O}_+(p, r),\quad \bM\in\mathbb{M}_+(r)
.
\]
We follow the setup and notations in Sections \ref{sub:euclidean_representation_of_subspaces} and \ref{sub:euclidean_represention_of_low_rank_matrices_and_intrinsic_perturbation}. 
Since $\bSigma$ can be parameterized by an Euclidean vector $\btheta = [\bvarphi\transpose, \bmu\transpose]\transpose\in\mathscr{D}(p, r)$ through the map $\bSigma(\cdot):\mathscr{D}(p, r)\to \mathscr{S}(p, r),\btheta\mapsto \bSigma(\btheta)$, we then use $\bOmega(\cdot)$ to generically denote the induced map $\btheta\mapsto \bOmega(\btheta):=\bSigma(\btheta) + \eye_p$. Furthermore, let $\bOmega_0$ denote the true value of the covariance $\bOmega$ corresponding to the distribution of $\by_1,\ldots,\by_n$, $\btheta_0\in\mathscr{D}(p, r)$ be the inverse image of $\bOmega_0$ under the map $\bOmega(\cdot)$, and $\bSigma_0 := \bSigma(\btheta_0)$. Let $\bA,\bA_0\in\mathbb{R}^{(p - r)\times r}$ be matrices such that $\bvarphi = \vect(\bA), \bvarphi_0 = \vect(\bA_0)$ where $\|\bA\|_2 < 1$, $\|\bA_0\|_2 < 1$, and let $\bM$, $\bM_0\in\mathbb{M}_+(r)$ be positive definite matrices such that $\bmu = \vech(\bM)$ and $\bmu_0 = \vech(\bM_0)$, respectively. 

The advantage of the Cayley parameterization is that $\mathrm{supp}(\bV) = \mathrm{supp}\{\bU(\bvarphi)\}$, and the row sparsity of $\bU(\bvarphi)$ can be directly incorporated to the rows of $\bA$. By the construction of the Cayley parameterization, $\bU(\bvarphi)$ can be written as
\[
\bU(\bvarphi) = \begin{bmatrix*}
(\eye_r - \bA\transpose\bA)(\eye_r + \bA\transpose\bA)^{-1}\\
2\bA(\eye_r + \bA\transpose\bA)^{-1}
\end{bmatrix*}.
\]
It follows that for any $j\in \{r + 1,r + 2,\ldots,p\}$, $[\bU(\bvarphi)]_{j*} = \zero$ if and only if $[\bA]_{(j - r)*}$. Furthermore, $\bU(\bvarphi)$ is subject to the orthonormal constraint $\bU(\bvarphi)\transpose\bU(\bvarphi) = \eye_r$, whereas working with $\bA$ is more convenient. Hence, we consider the following sparsity inducing prior on $\bA$. Let $\pi_p$ be the density of a discrete distribution supported on $\{0,1,2,\ldots,p - r\}$ of the form
\begin{align}\label{eqn:prior_pi_p}
\pi_p(t) = \frac{1}{z_n} n^{-rt}(p - r)^{-at},\quad t = 0,\ldots,p - r
\end{align}
for some constants $a, c> 0$, where 
\[
z_n = \sum_{t = 0}^{p - r}\left\{\frac{1}{n^r(p - r)^a}\right\}^t = \frac{1 - \{n^{-r}(p - r)^{-a}\}^{p - r + 1}}{1 - n^{-r}(p - r)^{-a}}
\]
is the normalizing constant. Based on $\pi_p$, A subset $S\subset[p - r]$ representing the support of $\bA$ is drawn from the following distribution:
\begin{align}\label{eqn:model_selection_prior}
\pi_S(S) = \frac{\pi_p(|S|)}{{p - r\choose |S|}},\quad S\subset[p - r],
\end{align}
where we use $|S|$ to denote the cardinality of a finite set $S$. 
Given $S = \{j_1,j_2,\ldots,j_{|S|}\}\subset[p - r]$, suppose $S^c := [p - r]\backslash S$ can be written as $S^c = \{k_1,k_2,\ldots,k_{|S^c|}\}$, and denote 
\begin{align*}
\bA_S := \begin{bmatrix*}
[\bA]_{j_11} & [\bA]_{j_12} & \ldots & [\bA]_{j_1r}\\
[\bA]_{j_21} & [\bA]_{j_22} & \ldots & [\bA]_{j_2r}\\
\vdots & \vdots &  & \vdots\\
[\bA]_{j_{|S|}1} & [\bA]_{j_{|S|}2} & \ldots & [\bA]_{j_{|S|}r}\\
\end{bmatrix*},\quad
% \text{and}\quad
\bA_{S^c} := \begin{bmatrix*}
[\bA]_{k_11} & [\bA]_{k_12} & \ldots & [\bA]_{k_1r}\\
[\bA]_{k_21} & [\bA]_{k_22} & \ldots & [\bA]_{k_2r}\\
\vdots & \vdots &  & \vdots\\
[\bA]_{k_{|S^c|}1} & [\bA]_{k_{|S^c|}2} & \ldots & [\bA]_{k_{|S^c|}r}\\
\end{bmatrix*}.
\end{align*}
% $\bA_S := [[\bA]_{jk}:j\in S, k\in [r]]_{|S|\times r}$, and $\bA_{S^c} := [[\bA]_{jk}:j\in S^c, k\in [r]]_{|S^c|\times r}$. 
We then define the prior distribution of $\bA$ by
\begin{equation}
\label{eqn:prior_A}
\begin{aligned}
\Pi_\bA(\mathrm{d}\bA)& = \sum_{S\subset[p - r]}\pi_S(S)\left\{\pi_{\bA_S}(\bA_S)\mathrm{d}\bA_S\right\}\left\{\delta_{\zero_{|S^c|\times r}}(\mathrm{d}\bA_{S^c})\right\},\\
\pi_{\bA_S}(\bA_S) & = \frac{\exp\{-2\|\vect(\bA_S)\|_1\}\mathbbm{1}(\|\bA_S\|_2 < 1)\mathrm{d}\bA_S}{
\int_{\|\bA_S\|_2 < 1}\exp(-2\|\vect(\bA_S)\|_1)\mathrm{d}\bA_S}.
\end{aligned}
\end{equation} 
The prior distribution on the entire covariance matrix $\bOmega$ through $\btheta$ is completed by assigning the following prior distribution to $\bM$, which is independent of $\Pi_\bA(\mathrm{d}\bA)$:
\begin{align}
\label{eqn:prior_M}
\begin{aligned}
\pi_\bmu(\bmu)\propto \exp\left( - 2\|\bmu\|_1\right)\mathbbm{1}\{\bM(\bmu) \in \mathbb{M}_+(r)\}.
\end{aligned}
\end{align}
% where $\mathbb{M}_+^r$ is the cone of the all positive definite matrices in $\mathbb{R}^{r\times r}$. 
Then the joint prior distribution on $\btheta = [\bvarphi\transpose,\bmu\transpose]\transpose = [\vect(\bA)\transpose,\bmu\transpose]\transpose\in\mathscr{D}(p, r)$ is defined as the product of the sparsity inducing prior \eqref{eqn:prior_A} on $\bA$ and the prior distribution \eqref{eqn:prior_M} on $\bmu$:
\begin{align}\label{eqn:prior_sparse_pca}
\Pi_\btheta(\mathrm{d}\btheta) = \Pi_\bA(\mathrm{d}\bA)\pi_\bmu(\bmu)\mathrm{d}\bmu.
\end{align}
% $\bvarphi = \vect(\bA)$ and $\btheta = [\bvarphi\transpose, \bmu\transpose]\transpose$. Then the prior specification $\Pi_\btheta$ on $\btheta = [\vect(\bA)\transpose, \bmu\transpose]\transpose$ also induces a prior distribution $\Pi_\bSigma$ on $\bSigma(\btheta)$ over the space $\calC_n = \{\bSigma = \bU\bLambda\bU\transpose + \eye_p:\bU\in\mathbb{O}(p, r), \bLambda = \diag(\lambda_1,\ldots,\lambda_r), \lambda_1 \geq \ldots \geq \lambda_r > 0\}$
Denote $\bY_n = [\by_1,\ldots,\by_n]\in\mathbb{R}^{p\times n}$ the data matrix concatenated by $\by_1,\ldots,\by_n$ and
\[
\ell(\bOmega) = -\frac{n}{2}\log\det(2\pi\bOmega) - \frac{n}{2}\mathrm{tr}(\widehat\bOmega\bOmega^{-1})
\]
the log-likelihood function of $\bOmega$, where $\widehat\bOmega := (1/n)\sum_{i = 1}^n\by_i\by_i\transpose$. We assume the high-dimensionality setup $p/n\to\infty$ so that the sample covariance matrix $\widehat{\bOmega}$ is no longer invertible. Then the posterior distribution of interest given the data matrix $\bY_n$ can be written using the Bayes formula:
\begin{align*}
\Pi_\btheta(\btheta\in A\mid\bY_n)
& = \frac{\int_A\exp\{\ell(\bOmega(\btheta)) - \ell(\bOmega_0)\}\Pi_\btheta(\mathrm{d}\btheta)}{\int\exp\{\ell(\bOmega(\btheta)) - \ell(\bOmega_0)\}\Pi_\btheta(\mathrm{d}\btheta)}.
\end{align*}
where $A$ is any measurable subset of $\mathscr{D}(p, r)$. 

The main result of this section is Theorem \ref{thm:contraction_spectral_norm} below, which asserts that the posterior contraction rate under $\|\sin\Theta\{\bU(\bvarphi),\bU_0\}\|_2$ is minimax optimal. We first present the following assumptions:
\begin{itemize}
	\item[A1] (Row sparsity) $\bU_0$ is jointly $s$-sparse for some $s\geq 2r$, namely, $s = |\mathrm{supp}(\bU_0)|\geq 2r$.
	\item[A2] (Regularity) The spectral norm of $\bA_0$ is bounded away from $1$, i.e., $\sup_{n}\|\bA_0\|_2 < 1$. 
	% \item[A3] The matrix $\bQ_{01}$ satisfies $\lambda_r(\bQ_{01}) = \Theta(1)$. 
	\item[A3] (Bounded spectra) There exists some constants $\underline{\lambda},\overline{\lambda} > 0$ such that 
	\[
	\underline{\lambda}\leq \lambda_r(\bM_0) \leq \lambda_1(\bM_0)\leq \overline{\lambda}.
	\]
	\item[A4] (Fast convergence rate) 
	% $p/n\to\infty $ and 
	% \[
	% \frac
	$
	{(r^2s^2\log n + rs^2\log p)^3}/{n}\to 0
	$
	.
	% \] 
	\item[A5] (Minimum signal strength) The non-zero rows of $\bA_0$ satisfies
	\[
	\frac{\min_{j\in \mathrm{supp}(\bA_0)}\|[\bA_0]_{j*}\|_2}{\sqrt{(rs\log n + s\log p)/n}} \to \infty. 
	\]
\end{itemize}
% \begin{remark}
\begin{remark}
Some remarks regarding assumptions A1-A5 are in order. Assumptions A1 and A3 are standard conditions for the sparse spiked covariance model. Assumption A2 requires that the spectral norm of $\bA_0$ is bounded away from $1$. By the results in Section \ref{sec:main_results}, locally around $\bU_0$, the Frobenius sine-theta distance between $\bU$ and $\bU_0$ is equivalent to the Frobenius norm $\|\bU - \bU_0\|_{\mathrm{F}}$, and hence, $\|\bvarphi - \bvarphi_0\|_2$, up to a constant factor. Furthermore, Theorem \ref{thm:DSigma_singular_value} and Remark \ref{rmk:statistical_interpretation_regularity} indicate that the Fisher information matrix with regard to the $\btheta$-parameterization evaluated at $\btheta = \btheta_0$, given by
\begin{align}\label{eqn:Fisher_information_matrix_spiked_covariance}
\eye(\btheta_0) = \frac{1}{2}D\bSigma(\btheta_0)\transpose(\bOmega_0^{-1}\otimes\bOmega_0^{-1})D\bSigma(\btheta_0),
\end{align}
is asymptotically non-singular, i.e., $\|\eye(\btheta_0)^{-1}\|_2$ is bounded away from $\infty$ when $n\to\infty$. 
% \end{remark}

Assumption A4 claims that the posterior contraction rate with regard to $\|\btheta - \btheta_0\|_1$ is sufficiently fast. Furthermore, roughly speaking, using the technical tools developed in Section \ref{sec:main_results}, we are able to derive a local asymptotic normality expansion of the log-likelihood function $\ell(\bOmega(\btheta))$ as follows:
\begin{align}
\label{eqn:LAN}
\begin{aligned}
\ell(\bOmega(\btheta)) - \ell(\bOmega_0)
& = \frac{n}{2}\vect(\widehat{\bOmega} - \bOmega_0)\transpose(\bOmega_0^{-1}\otimes\bOmega_0^{-1})D\bSigma(\btheta_0)(\btheta - \btheta_0)\\
&\quad - \frac{n}{2}(\btheta - \btheta_{0})\transpose\eye(\btheta_0)(\btheta - \btheta_{0})
% \frac{n}{4}(\btheta - \btheta_{0})\transpose D\bSigma(\btheta_{0})\transpose(\bOmega_0^{-1}\otimes\bOmega_0^{-1})D\bSigma(\btheta_{0})(\btheta - \btheta_{0})\\
% &\quad
 + R_n(\btheta, \btheta_{0}),
\end{aligned}
\end{align}
where the remainder $R_n$ is negligible under
 % certain regularity conditions together with 
Assumption A4. 

Assumption A5 requires that the minimum of the Euclidean norms of the non-zero rows of $\bA_0$ cannot be too small. It is similar to the so-called $\beta$-min condition in the sparse linear regression model (see, e.g., \citealp{buhlmann2011statistics}). In \cite{lei2015}, a similar condition is also required for the exact recovery of $\mathrm{supp}(\bU_0)$ using the Fantope projection and selection method. 
\end{remark}

\begin{theorem}\label{thm:contraction_spectral_norm}
Under the setup and notations in Sections \ref{sub:euclidean_representation_of_subspaces}, \ref{sub:euclidean_represention_of_low_rank_matrices_and_intrinsic_perturbation}, and \ref{sub:bayesian_sparse_pca_and_non_intrinsic_loss},
% Assume the conditions of Theorem \ref{thm:contraction_Frobenius_norm} hold, and further assume that $n\eps_n(\btheta)^3\to 0$ as $n\to\infty$, where $\eps_n(\btheta) = \sqrt{(r^2s_0^2\log n + rs_0^2\log p)/n}$. 
% If $\eps_n^2 = (rs\log n + s\log p)/n \to 0$ and $p\to\infty$ as $n\to\infty$, 
there exists some large constant $M > 0$, such that
\[
\expect_0\left[\Pi_\btheta\left\{\|\sin\Theta(\bU(\bvarphi), \bU_0)\|_{\mathrm{2}} > M\sqrt{\frac{s\log p}{n}}\mathrel{\Bigg|}\bY_n\right\}\right] \to 0.
\]
\end{theorem}
% \begin{remark}
Theorem \ref{thm:contraction_spectral_norm} is a non-trivial result and relies on the asymptotic characterization of the shape of the posterior distribution $\Pi_\btheta(\mathrm{d}\btheta\mid\bY_n)$, which is summarized in Theorem \ref{thm:BvM} below. 
\begin{theorem}\label{thm:BvM}
Assume the setup and notations in Sections \ref{sub:euclidean_representation_of_subspaces}, \ref{sub:euclidean_represention_of_low_rank_matrices_and_intrinsic_perturbation}, and \ref{sub:bayesian_sparse_pca_and_non_intrinsic_loss} hold. 
% Let $\widehat\bOmega = (1/n)\sum_{i = 1}^n\by_i\by_i\transpose$ be the sample covariance matrix, and $\eye(\btheta_0)$ be the Fisher information matrix with regard to the $\btheta$ parameterization given by \eqref{eqn:Fisher_information_matrix_spiked_covariance}. 
For any index set $S\in\calS_0:=\{S\subset[p - r]:\mathrm{supp}(\bA_0)\subset S,|S|\leq \kappa_0s_0\}$, let $\bF_S$ be the matrix such that 
\[
\btheta_S:=\begin{bmatrix*}
\vect(\bA_S)\\\bmu
\end{bmatrix*} = \bF_S\btheta\quad\text{ for any vector }\btheta = 
\begin{bmatrix*}
\vect(\bA)\\\bmu
\end{bmatrix*}.
\]
Further define the following quantities:
\begin{align*}
&\eye_S(\btheta_0) = \bF_S\transpose\eye(\btheta_0)\bF_S,\quad
\widehat{\btheta}_S  = \btheta_{0S} + \eye_S(\btheta_0)^{-1}\bF_S\transpose D\bSigma(\btheta_0)\transpose\vect\{\bOmega_0^{-1}(\widehat\bOmega - \bOmega_0)\bOmega_0^{-1}\},\\
&\widehat{w}_S \propto \frac{
	\pi_p(|S|)\det\{2\pi\eye_S(\btheta_0)^{-1}\}^{1/2}\exp\{(1/2)\widehat{\btheta}_S\transpose\eye_S(\btheta_0)\widehat{\btheta}_S\}
}{{p - r\choose |S|} \int_{\{\|\bA_S\|_2 < 1\}} \exp(-2\|\vect(\bA_S)\|_1)\mathrm{d}\bA_S }\quad\text{such that}\quad\sum_{S\in\calS_0}\widehat{w}_S = 1.
\end{align*}
Let $\Pi_\btheta^\infty(\btheta\in\cdot\mid\bY_n)$ be the following random mixture of normals
\begin{align}
\label{eqn:Pi_infinity}
\Pi^\infty_\btheta(\mathrm{d}\btheta\mid\bY_n) & = \sum_{S\in\calS_0}\widehat{w}_{S}\{\phi(\btheta_{S}\mid\widehat\btheta_{S}, \eye_S(\btheta_0)^{-1})\mathrm{d}\btheta_{S}\}\{\delta_{\zero_{|S^c|\times r}}(\mathrm{d}\btheta_{S^c})\},
\end{align}
Here, $\phi(\bx\mid\bu, \bOmega):=\det(2\pi\bOmega)^{-1/2}e^{-(\bx - \bu)\transpose\bOmega^{-1}(\bx - \bu)/2}$ and $\btheta_{S^c} := \vect(\bA_{S^c})$. Then there exists some constant $\kappa_0 \geq 1$ such that
\begin{align*}
\|\Pi_\btheta(\btheta\in\cdot\mid\bY_n) - \Pi^\infty_\btheta(\btheta\in\cdot\mid\bY_n)\|_{\mathrm{TV}} = o_{\prob_0}(1).
\end{align*}
\end{theorem}
\begin{remark}
By Theorem \ref{thm:DSigma_singular_value} and Remark \ref{rmk:statistical_interpretation_regularity}, the Fisher information matrix $\eye(\btheta_0)$ is strictly positive definite, implying that the submatrix $\eye_S(\btheta_0) = \bF_S\transpose\eye(\btheta_0)\bF_S$ is also strictly positive definite. Hence, leveraging the intrinsic perturbation tools developed in Section \ref{sec:main_results} and Theorem \ref{thm:BvM}, we are able to study the behavior of $\|\sin\Theta\{\bU(\bvarphi), \bU_0\}\|_2$ under the exact posterior distribution $\Pi_\btheta(\btheta\in\cdot\mid\bY_n)$ through the behavior of $\btheta - \btheta_0$ under the limit posterior distribution $\Pi^\infty_\btheta(\btheta\in\cdot\mid\bY_n)$. Theorem \ref{thm:BvM} may be of independent interest as well.
\end{remark}

\subsection{Stochastic Block Model} % (fold)
\label{sub:stochastic_block_model}
In this section, we apply the technical tools developed in Section \ref{sec:main_results} to the parameter estimation in stochastic block models, in which the block probability matrix may not necessarily be full rank. 

The stochastic block model (SBM) is a popular random graph model initially developed by \cite{HOLLAND1983109} for studying social networks. Since then, a rich collection of offsprings and variations of SBM, such as mixed-membership SBM \citep{airoldi2008mixed}, degree-corrected SBM \citep{PhysRevE.83.016107}, hierarchical SBM \citep{7769223}, vertex-contextualized SBM \citep{10.1093/biomet/asx008}, and multi-layer SBM \citep{BOCCALETTI20141}, have been developed, further popularizing the development of statistical network analysis. The readers are referred to \cite{abbe2018community} for a survey of the recent advances in statistical analyses of SBM. 

The SBM can be formally stated as follows in terms of the random adjacency matrix. Given $n$ vertices $[n] = \{1,\ldots,n\}$, a $K\times K$ symmetric block probability matrix $\bSigma\in(0, 1)^{K\times K}$, and a cluster assignment function $\tau:[n]\to[K]$, we say that an $n\times n$ symmetric random matrix $\bA\in\{0, 1\}^{n\times n}$ is the adjacency matrix of a SBM with block probability matrix $\bSigma$ and cluster assignment function $\tau$, denoted by $\bA\sim\mathrm{SBM}(\bSigma, \tau)$, if the random variables $(A_{ij}:1\leq i < j\leq n) := ([\bA]_{ij}:1\leq i < j\leq n)$ are independent, $A_{ij}\sim\mathrm{Bernoulli}([\bSigma]_{\tau(i)\tau(j)})$, $[\bA]_{ji} = [\bA]_{ij}$ for all $i\neq j$, and $[\bA]_{ii} = 0$ for all $i\in [n]$. The likelihood function of $\bSigma,\tau$ is
\begin{align}
\label{eqn:likelihood_SBM}
\calL_{\bA}(\bSigma, \tau) = \prod_{i < j}[\bSigma]_{\tau(i)\tau(j)}^{A_{ij}}(1 - [\bSigma]_{\tau(i)\tau(j)}^{1 - A_{ij}}).
\end{align}
A fundamental task of interest in SBM is the recovery of the cluster assignment function $\tau$ given the observed network encoded in $\bA$, referred to as community detection. The theory and methods for community detection have been studied extensively. In particular, successful community detection algorithms include modularity and likelihood maximization methods \citep{Bickel21068,celisse2012}, spectral clustering \citep{rohe2011,pmlr-v23-chaudhuri12,sussman2012consistent,lei2015}, and semidefinite programming \citep{7298436,7440870}. 
% We refer to \cite{fortunato2010community} for a survey. 

Here we focus on estimating the block probability matrix $\bSigma$, an inference task closely related to community detection. The method of maximum likelihood and its variational approximation have been studied in \cite{bickel2013}. Alternatively, \cite{Bickel21068} proposed to estimate $\bSigma$ based on a strongly consistent estimate of $\tau$ that can be obtained by a modularity maximization method. The estimators proposed in \cite{bickel2013} and \cite{Bickel21068} are asymptotically efficient provided that $\mathrm{rank}(\bSigma) = K$. Nonetheless, as observed in \cite{tang2017asymptotically}, when $\mathrm{rank}(\bSigma) < K$, neither 
% the maximum likelihood estimator nor the modularity-maximization-based 
estimator is asymptotically efficient. This section aims at providing an asymptotically efficient estimator of $\bSigma$ when $\bSigma$ is potentially singular. 

Let $\bA\sim\mathrm{SBM}(\bSigma_0, \tau_0)$, where $\bSigma_0\in(0, 1)^{K\times K}$ is symmetric, and $\tau_0:[n]\to[K]$ is a cluster assignment function. Let $\mathrm{rank}(\bSigma_0) = r\leq K$. 
We follow the setup and notations in Sections \ref{sub:euclidean_representation_of_subspaces} and \ref{sub:euclidean_represention_of_low_rank_matrices_and_intrinsic_perturbation}. 
% We assume that $\mathrm{rank}(\bSigma_0) = r \leq d$, and 
Suppose $\bSigma_0$ has the spectral decomposition $\bSigma_0 = \bV\bS\bV\transpose$, where $\bV\in\mathbb{O}(K, r)$ and $\bS = \mathrm{diag}\{\lambda_1(\bSigma_0),\ldots,\lambda_r(\bSigma_0)\}$. Note that $\bSigma_0$ can only be identified up to a permutation of rows and columns. We assume that any $r$ columns of $\bV$ are linearly independent. This implies that, for any permutation matrix $\bPi\in\mathbb{R}^{K\times K}$, $\bPi\bSigma_0\bPi\transpose\in\mathscr{S}(K, r)$, and hence,
% , where $\lambda_1(\bSigma_0)\geq\ldots\geq\lambda_r(\bSigma_0)$. 
$\bPi\bSigma_0\bPi\transpose$ can be represented by a Euclidean vector in $\mathscr{D}(K, r)$, denoted by $\btheta_{0\bPi} = [\bvarphi_{0\bPi}\transpose, \bmu_{0\bPi}\transpose]\transpose$. Therefore, $\bPi\bSigma_0\bPi\transpose = \bSigma(\btheta_{0\bPi}) = \bU(\bvarphi_{0\bPi})\bM_{0\bPi}\bU(\bvarphi_{0\bPi})\transpose$, where $\bM_{0\bPi}\in\mathbb{R}^{r\times r}$ is symmetric with $\vech(\bM_{0\bPi}) = \bmu_{0\bPi}$, and there exists an $(K - r)\times r$ matrix $\bA_{0\bPi}$ such that $\bvarphi_{0\bPi} = \vect(\bA_{0\bPi})$ and $\|\bA_{0\bPi}\|_2 < 1$. Note that we use the subscript $\bPi$ to suggest that $\btheta_{0\bPi}$ depends on $\bPi$. We further assume that there exists a probability vector $\bpi = [\pi_1,\ldots,\pi_K]\in\mathbb{R}^K$ such that
\[
% $
\frac{1}{n}\sum_{i = 1}^n\mathbbm{1}\{\tau_0(i) = k\} \to \pi_k
% $, $
,\quad
k\in [K]
% $
.
\]

We propose a one-step estimator for $\bSigma_0$ and establish the asymptotic normality under the regime $n\to\infty$. In general, the one-step estimator can be obtained by a single iteration of Newton-Raphson's algorithm for maximizing the log-likelihood function. In a classical parametric model, under certain regularity conditions, the one-step update leads to an asymptotically efficient estimator when the initial guess is $\sqrt{n}$-consistent (see, e.g., Section 5.7 in \citealp{van2000asymptotic}). The same idea has also appeared in \cite{xie2019efficient} for efficient estimation of a more general low-rank random graph. In what follows, we apply the one-step procedure to SBM when $\bSigma_0$ is potentially singular and obtain an efficient estimator. Let $\ell(\btheta, \tau) := \log\calL_\bA(\bSigma(\btheta), \tau)$ be the log-likelihood function under the $\btheta$-parameterization. Given the cluster assignment $\tau$, the score function with respect to $\btheta$ is
\begin{align}\label{eqn:score_function_SBM}
\frac{\partial\ell}{\partial\btheta}(\btheta, \tau) & = \sum_{s = 1}^K\sum_{t = 1}^K\frac{m_{st}(\tau) - n_{st}(\tau)[\bSigma(\btheta)]_{st}}{[\bSigma(\btheta)]_{st}\{1 - [\bSigma(\btheta)]_{st}\}}D\bSigma(\btheta)\transpose\vect(\bE_{st}),
\end{align}
where, for any $s,t\in [K]$,
\[
m_{st}(\tau) := \sum_{i < j\in [n]}A_{ij}\mathbbm{1}\{\tau(i) = s,\tau(j) = t\},\quad
n_{st}(\tau) := \sum_{i < j\in [n]}\mathbbm{1}\{\tau(i) = s,\tau(j) = t\},
\]
and $\bE_{st}$ is a $K\times K$ matrix of all zeros except $1$ at the $(s, t)$th element. In addition, given $\tau$, the Fisher information matrix of SBM with respect to $\btheta$ is
\begin{align}\label{eqn:Fisher_infor_SBM}
\eye(\btheta, \tau) & := \sum_{s = 1}^K\sum_{t = 1}^K\frac{n_{st}(\tau)D\bSigma(\btheta)\transpose\vect(\bE_{st})\vect(\bE_{st})\transpose D\bSigma(\btheta)}{[\bSigma(\btheta)]_{st}\{1 - [\bSigma(\btheta)]_{st}\}}.
\end{align}
Since $\btheta_{0\bPi}\in\mathscr{D}(K, r)$ for any permutation matrix $\bPi$, it follows from Theorem \ref{thm:DSigma_singular_value} that $\eye(\btheta, \tau)$ is invertible in a local neighborhood of $\btheta_{0\bPi}$. 
We then construct the proposed one-step estimator as follows:
\begin{itemize}
	\item[(I)] Let $\widehat{\tau}$ be a strongly consistent estimator of $\tau_0$, namely, there exists a sequence of permutations $(\omega_n)_n$, $\omega_n:[K]\to [K]$ such that
	\[
	\prob_0\left\{  d_{\mathrm{H}}(\widehat{\tau}, \omega_n\circ\tau_0) = 0\right\}\to 1.
	\]
	% where the infimum is taken over all possible permutations $\gamma:[K]\to[K]$. 
	This can be obtained, e.g., by applying the $K$-means procedure to the rows of the leading eigenvector matrix $\widehat{\bV}$ of $\bA$, namely, $\bA\widehat{\bV} = \widehat{\bV}\mathrm{diag}\{\lambda_1(\bA),\ldots,\lambda_r(\bA)\}$ with $\widehat{\bV}\in\mathbb{O}(p, r)$. By Lemma 4 in \cite{tang2017asymptotically}, this results in a strongly consistent estimator of $\tau_0$. 
	\item[(II)] Compute an initial estimator $\widetilde{\bSigma}_n$ that is $n$-consistent, namely, there exists a sequence of $K\times K$ permutation matrices $(\bPi_n)_{n = 1}^\infty$, such that $\|\widetilde{\bSigma}_n - \bPi_n\bSigma_0\bPi_n\transpose\|_{\mathrm{F}} = O_{\prob_0}(n^{-1})$. An example of such an estimator is given by $\widetilde{\bSigma}_n = [[\widetilde{\bSigma}_n]_{st}]_{K\times K}$, where
	\begin{align}\label{eqn:initial_estimator_SBM}
	[\widetilde{\bSigma}_n]_{st} = \frac{\sum_{i,j\in [n]}A_{ij}\mathbbm{1}\{\widehat{\tau}(i) = s, \widehat{\tau}(j) = t\}}
	{
	\sum_{i,j\in [n]}\mathbbm{1}\{\widehat{\tau}(i) = s, \widehat{\tau}(j) = t\}
	},\quad s,t\in [K].
	\end{align}
	\item[(III)] Solve the least squares problem
	\[
	\widetilde{\btheta}_n = \argmin_{\btheta\in\mathscr{D}(K, r)}\|\widetilde{\bSigma}_n - \bSigma(\btheta)\|_{\mathrm{F}}^2.
	\]
	\item[(IV)] Compute the following one-step estimator:
	\begin{align*}
	\widehat{\btheta}_n = \widetilde{\btheta}_n - \eye(\widetilde{\btheta}_n, \widehat{\tau})^{-1}\frac{\partial\ell}{\partial\btheta}(\widetilde\btheta_n, \widehat{\tau}). 
	\end{align*}
\end{itemize}
Theorem \ref{thm:OSE_SBM} below, which is the main result of this section, states the asymptotic normality of the proposed one-step estimator $\widehat{\btheta}_n$.
\begin{theorem}\label{thm:OSE_SBM}
Under the notations and setup above, there exists a sequence of $K\times K$ permutation matrices $(\bPi_n)_{n = 1}^\infty$ such that
\[
{n}\bJ_{\bPi_n}(\btheta_{0\bPi_n})^{1/2}(\widehat{\btheta}_n - \btheta_{0\bPi_n})\overset{\calL}{\to}\mathrm{N}(\zero_d, \eye_d),
\]
where, for any permutation matrix $\bPi\in\mathbb{R}^{K\times K}$ and any $\btheta$ such that $\bSigma(\btheta)\in(0, 1)^{K\times K}$, the matrix $\bJ_{\bPi}(\btheta)$ is  defined by
\begin{align*}
% \label{eqn:limit_Fisher_infor_SBM}
\bJ_{\bPi}(\btheta) &: = \sum_{s = 1}^K\sum_{t = 1}^K\frac{[\bPi\bpi]_s[\bPi\bpi]_t D\bSigma(\btheta)\transpose\vect(\bE_{st})\vect(\bE_{st})\transpose D\bSigma(\btheta)}{2[\bSigma(\btheta)]_{st}\{1 - [\bSigma(\btheta)]_{st}\}},
\end{align*} 
and $d = (K - r)r + r(r + 1)/2$ is the dimension of $\mathscr{D}(K, r)$.
\end{theorem}
\begin{remark}
We briefly compare Theorem \ref{thm:OSE_SBM} with the results of \cite{Bickel21068} and \cite{bickel2013}. When $r = K$, i.e., $\bSigma_0$ is invertible, by Lemma 1 in \cite{bickel2013} and Corollary 1 \cite{Bickel21068}, the initial estimator defined \eqref{eqn:initial_estimator_SBM} is already asymptotically efficient, in the sense that there exists a sequence of permutation matrices $(\bPi_n)_{n = 1}^\infty$ such that
\begin{align*}
n\bG_{\bPi_n}^{1/2}\vech(\widetilde{\bSigma}_n - \bPi_n \bSigma_0\bPi_n\transpose) \overset{\calL}{\to}\mathrm{N}(\zero_{K(K + 1)/2}, \eye_{K(K + 1)/2}),
\end{align*}
where
\begin{align*}
\bG_{\bPi_n} := 
&\sum_{s = 1}^K\frac{[\bPi_n\bpi]_s^2\vech(\bE_{ss})\vech(\bE_{ss})\transpose}{2[\bPi_n\bSigma_0\bPi_n\transpose]_{ss}(1 - [\bPi_n\bSigma_0\bPi_n\transpose]_{ss})}
\\&
 + \sum_{s = 2}^K\sum_{t = 1}^{r - 1}\frac{[\bPi_n\bpi]_s[\bPi_n\bpi]_t\vech(\bE_{st})\vech(\bE_{st})\transpose}{[\bPi_n\bSigma_0\bPi_n\transpose]_{st}(1 - [\bPi_n\bSigma_0\bPi_n\transpose]_{st})}.
\end{align*}
Observe that $\bG_{\bPi_n}$ is exactly the Fisher information matrix under the $\vech(\bSigma)$-parameterization up to a permutation $\bPi_n$ of rows and columns of $\bSigma$. 
We remark that, under the condition that $r = K$, the plug-in $\bSigma(\widehat{\btheta}_n)$ of the one-step estimator $\widehat{\btheta}_n$ is asymptotically equivalent to $\widetilde{\bSigma}_n$ by Theorem \ref{thm:OSE_SBM}. Note that when $r = K$, $\bU_{0\bPi} = \eye_K$ and $D\bSigma(\btheta_{0\bPi}) = \mathbb{D}_K$ for any permutation matrix $\bPi$. Then the matrix $\bJ_{0\bPi_n}(\btheta_{0\bPi_n})$ has the following form:
\begin{align*}
\bJ_{\bPi_n}(\btheta_{0\bPi_n})
& = \sum_{s = 1}^K\frac{[\bPi_n\bpi]_s^2\vech(\bE_{ss})\vech(\bE_{ss})\transpose}{2[\bPi_n\bSigma_0\bPi_n\transpose]_{ss}(1 - [\bPi_n\bSigma_0\bPi_n\transpose]_{ss})}\\
&\quad + \sum_{s = 2}^K\sum_{t = 1}^{r - 1}\frac{[\bPi_n\bpi]_s[\bPi_n\bpi]_t\vech(\bE_{st})\vech(\bE_{st})\transpose}{[\bPi_n\bSigma_0\bPi_n\transpose]_{st}(1 - [\bPi_n\bSigma_0\bPi_n\transpose]_{st})}
\end{align*}
By Theorem \ref{thm:OSE_SBM} and the delta method, the asymptotic covariance matrix of $n\vech\{\bSigma(\widehat{\btheta}_n) - \bPi_n\bSigma_0\bPi_n\transpose\}$ is 
\begin{align*}
% \bH_{\bPi_n}(\btheta_{0\bPi_n})&:=
\mathbb{D}_K^{\dagger}D\bSigma(\btheta_{0\bPi_n})\bJ_{\bPi_n}(\btheta_{0\bPi_n})^{-1}D\bSigma(\btheta_{0\bPi_n})\transpose(\mathbb{D}_K^{\dagger})\transpose
% \\
% & = \mathbb{D}_K^{\dagger}\mathbb{D}_K\bJ(\btheta_{0\bPi_n})^{-1}\mathbb{D}_K\transpose(\mathbb{D}_K^{\dagger})\transpose
& = \bJ_{\bPi_n}(\btheta_{0\bPi_n})^{-1} = \bG_{\bPi_n}^{-1}.
\end{align*}
Namely, the two estimators $\widetilde{\bSigma}_n$ and $\widehat{\bSigma}_n$ have the same asymptotic covariance matrix. 

Furthermore, when $r < K$, Theorem \ref{thm:OSE_SBM} implies that the plug-in estimator $\bSigma(\widehat{\btheta}_n)$ has an asymptotic covariance matrix no greater than that of $\widetilde{\bSigma}$ in spectra. In fact, by the delta method, the asymptotic covariance matrix of $n\vech\{\bSigma(\widehat{\btheta}_n) - \bPi_n\bSigma_0\bPi_n\transpose\}$ is
\begin{align*}
\bH_{\bPi_n}(\btheta_{0\bPi_n}):=\mathbb{D}_{K}^{\dagger} D\bSigma(\btheta_{0\bPi_n})\bJ_{\bPi_n}(\btheta_{0\bPi_n})^{-1} D\bSigma(\btheta_{0\bPi_n})\transpose(\mathbb{D}_K^\dagger)\transpose.
\end{align*}
Let $\mathbf{O}_1\bS\mathbf{O}_2\transpose$ be the singular value decomposition of $D\bSigma(\btheta_{0\bPi_n})$, where $\mathbf{O}_1\in\mathbb{O}(K^2, d)$, $\mathbf{O}_2\in\mathbb{O}(d)$, and $\bS = \mathrm{diag}[\sigma_1\{D\bSigma(\btheta_{0\bPi_n})\},\ldots,\sigma_r\{D\bSigma(\btheta_{0\bPi_n})\}]$
% , and $d = r(r + 1)/2 + (K - r)r$
. By the matrix Cauchy-Schwarz inequality (see, e.g., \citealp{marshall1990matrix}), 
\begin{align*}
&\bO_1\transpose\bD_{\bPi_n}(\btheta_{0\bPi_n})^{-1}\bO_1\succeq 
\left\{\bO_1\transpose\bD_{\bPi_n}(\btheta_{0\bPi_n})\bO_1\right\}^{-1},
\end{align*}
where, for any permutation matrix $\bPi\in\mathbb{R}^{K\times K}$ and $\btheta\in\mathscr{D}(p, r)$, the matrix $\bD_\bPi(\btheta)$ is defined by
\[
\bD_{\bPi}(\btheta) := \sum_{s = 1}^K\sum_{t = 1}^K\frac{[\bPi\bpi]_s[\bPi\bpi]_t \vect(\bE_{st})\vect(\bE_{st})\transpose }{2[\bSigma(\btheta)]_{st}\{1 - [\bSigma(\btheta)]_{st}\}}.
\]
Therefore,
\begin{align*}
\bH_{\bPi_n}(\btheta_{0\bPi_n})& = \mathbb{D}_{K}^{\dagger} D\bSigma(\btheta_{0\bPi_n})\bJ_{\bPi_n}(\btheta_{0\bPi_n})^{-1} D\bSigma(\btheta_{0\bPi_n})\transpose(\mathbb{D}_K^\dagger)\transpose\\
% &
% = \mathbb{D}_{K}^{\dagger} \bO_1\bS\bO_2\transpose\left\{
% \bO_2\bS\bO_1\transpose \bD_{\bPi_n}(\btheta_{0\bPi_n}) \bO_1\bS\bO_2\transpose
% \right\}^{-1} 
% \bO_2\bS\bO_1\transpose(\mathbb{D}_K^\dagger)\transpose\\
% &
% = \mathbb{D}_{K}^{\dagger} \bO_1\bS\bO_2\transpose
% \bO_2\bS^{-1}\{\bO_1\transpose \bD_{\bPi_n}(\btheta_{0\bPi_n}) \bO_1\}^{-1}\bS^{-1}\bO_2\transpose
% \bO_2\bS\bO_1\transpose(\mathbb{D}_K^\dagger)\transpose\\
&
= \mathbb{D}_{K}^{\dagger} \bO_1\{\bO_1\transpose \bD_{\bPi_n}(\btheta_{0\bPi_n}) \bO_1\}^{-1}\bO_1\transpose(\mathbb{D}_K^\dagger)\transpose\\
&
\preceq
\mathbb{D}_K^{\dagger} \bD_{\bPi_n}(\btheta_{0\bPi_n})^{-1} (\mathbb{D}_K^{\dagger})\transpose = 
\bG_{\bPi_n}^{-1},
\end{align*}
where $\bH_{\bPi_n}(\btheta_{0\bPi_n})$ is the asymptotic covariance matrix of $\bSigma(\widehat{\btheta}_n)$
 % based on the one-step estimator $\widehat{\btheta}_n$, 
 and $\bG_{\bPi_n}^{-1}$ is that of the initial estimator $\widetilde{\bSigma}_n$. 
\end{remark}

% subsection stochastic_block_model (end)

\subsection{Biclustering} % (fold)
\label{sub:biclustering}
Biclustering can be viewed as a natural extension of SBM to general and possibly rectangular matrices. It was originally explored in \cite{doi:10.1080/01621459.1972.10481214} and later studied under different contexts, including latent block models with exponential family distributions \citep{brault2020}, community detection in bipartite networks \citep{zhou2020optimal}, matrix completion with biclustering structure, and co-clustering of separately exchangeable nonparametric networks \citep{choi2014}. Much of the existing works have focused on the recovery of the cluster assignment \citep{mariadassou2015,zhou2020optimal,flynn2020,brault2020} or estimating the expected value of the data matrix \citep{gao2016optimal}, with an exception being \cite{brault2020}, who established the asymptotic normality of the maximum likelihood estimator for the parameter of interest. In this section, we extend the idea in Section \ref{sub:stochastic_block_model} to general rectangular data matrices. 

Let $\bY = [y_{ij}]_{m\times n}:=[[\bY]_{ij}]_{m\times n}\in\mathbb{R}^{m\times n}$ be the observed $m\times n$ data matrix with the following structure: the expected value of $\bY$ is a low-rank matrix, i.e., $\bY_0^*:=\expect_0(\bY)$ with $\mathrm{rank}(\bY_0^*) = r \ll \min(m, n)$, and $\bE:= \bY - \bY_0^*\in\mathbb{R}^{m\times n}$ is a mean-zero noise matrix whose entries are independent random variables. We assume that $[\bE]_{ij}$'s are identically distributed mean-zero sub-Gaussian random variables with $\mathrm{var}([\bE]_{ij}) = \sigma^2 > 0$. Under the biclustering setup, the mean matrix $\bY^*_0$ has the following structure: There exist two integers $p_1, p_2 > 0$, representing the numbers of row clusters and column clusters, respectively, two cluster assignment functions $\tau_0:[m]\to[p_1]$ for the rows and $\gamma_0:[n]\to [p_2]$ for the columns, and a block mean matrix $\bSigma_0\in\mathbb{R}^{p_1\times p_2}$ with $\mathrm{rank}(\bSigma_0) = r\leq \min(p_1, p_2)$, such that $[\bY^*_0]_{ij} = [\bSigma_0]_{\tau_0(i)\gamma_0(j)}$, $i\in [m]$ and $j\in [n]$. Alternatively, by taking $\bP_0 \in \{0, 1\}^{m\times p_1}$ and $\bQ_0\in\{0, 1\}^{n\times p_2}$ as row and column cluster assignment matrices such that $[\bP_0]_{is} = \mathbbm{1}\{\tau_0(i) = s\}$, $i\in[m],s\in[p_1]$ and $[\bQ_0]_{jt} = \mathbbm{1}\{\gamma_0(j) = t\}$, $j\in [n],t\in [p_2]$, we can equivalently write $\bY^*_0 = \bP_0\bSigma_0\bQ_0\transpose$. Similar to the treatment in Section \ref{sub:stochastic_block_model}, we focus on estimating the block mean matrix $\bSigma_0$ when it may be potentially rank-deficient under the regime $\min(m, n)\to\infty$. 

% $\bY^* = \bP\bSigma\bQ\transpose$. Here $\bP$ and $\bQ$ are binary matrices such that each column of $\bP$ and $\bQ$ contains only one $1$ on its coordinates, and 

% There exists two matrices $\bP\in\mathbb{R}^{p_1\times p_1}$ and $\bQ\in\mathbb{R}^{p_2\times p_2}$ = \bP\bSigma\bQ\transpose + \bE$, $\bSigma\in\mathbb{R}^{p_1\times p_2}$ is a fixed rectangular matrix with $\mathrm{rank}(\bSigma) = r\leq \min(p_1, p_2)$, $\bZ\in\mathbb{R}^{m\times p_1}$, $\bL\in\mathbb{R}^{n\times p_2}$, $\bE = [e_{ij}]_{m\times n}$ with $e_{ij}$ are independent and identically distributed mean-zero sub-Gaussian random variables, $i\in[m]$, $j\in [n]$. There exists cluster assignment functions $\tau:[m]\to p_1$ and $\gamma:[n]\to p_2$ such that $P_{is} = \mathbbm{1}\{\tau(i) = k\}$, $i\in [m],s\in[p_1]$ and $Q_{jt} = \mathbbm{1}\{\gamma(j) = t\}$, $j\in[n], t\in [p_2]$. 

Suppose $\bSigma_0 = \bV_1\mathrm{diag}\{\sigma_1(\bSigma_0),\ldots,\sigma_r(\bSigma_0)\}\bV_2\transpose$ is the singular value decomposition of $\bSigma_0$, where $\bV_1\in\mathbb{O}(m, r)$ and $\bV_2\in\mathbb{O}(n, r)$. Similar to SBM, the block mean matrix $\bSigma_0$ can only be identified up to a row permutation and a column permutation. Thus, for convenience, we assume that any $r$ columns of $\bV_2$ are linearly independent. Following the notations and setup in Sections \ref{sub:euclidean_representation_of_subspaces} and \ref{sub:extension_to_general_rectangular_matrices}, for any permutation matrices $\bPi_1\in\mathbb{R}^{p_1\times p_1}$ and $\bPi_2\in\mathbb{R}^{p_2\times p_2}$, we can represent $\bPi_1\bSigma_0\bPi_2\transpose$ by a $p_1r + (p_2 - r)r$-dimensional Euclidean vector, denoted by $\btheta_{0\bPi_1\bPi_2}$, such that $\bSigma_0 = \bSigma(\btheta_{0\bPi_1\bPi_2})$, where $\bSigma(\cdot)$ is the map defined by \eqref{eqn:Cayley_transform_rectangular}. Note that $\btheta_{0\bPi_1\bPi_2}$ depends on the permutation matrices $\bPi_1$ and $\bPi_2$.
 % $\bvarphi = \vect(\bA)$ for a $(p_2 - r)\times r$ matrix $\bA$ with $\|\bA\|_2 < 1$, and $\bmu = \vect(\bM)$. We also denote $\btheta_0 = [\bvarphi_0\transpose, \bmu_0\transpose]\transpose$ the true value of $\btheta$ corresponding to the data generating distribution, where $\bvarphi_0 = \vect(\bA_0)$ for a $(p_2 - r)\times r$ matrix $\bA_0$ with $\|\bA_0\|_2 < 1$, and $\bmu_0 = \vect(\bM_0)$ for a $p_1\times r$ matrix $\bM_0\in\mathbb{R}^{p_1\times r}$. 
 % We further denote $\bSigma_0 = \bSigma(\btheta_0)$, $\bU_0 = \bU(\bvarphi_0)$, and the true values of $\bP$, $\bQ$, $\tau$, $\gamma$, $\bY^*$ by $\bP_0$, $\bQ_0$, $\tau_0$, $\gamma_0$, $\bY_0^*$, respectively. The data matrix $\bY$ can thus be written as $\bY = \bY_0^* + \bE = \bP_0 \bSigma_0\bQ_0\transpose + \bE$ under the true distribution. 
% Similar to Section \ref{sub:stochastic_block_model}, 
We further assume that there exist probability vectors $\bw = [w_1,\ldots,w_{p_1}]\transpose$ and $\bpi = [\pi_1,\ldots,\pi_{p_2}]\transpose$, such that \[
\frac{1}{m}\sum_{i = 1}^m\mathbbm{1}\{\tau_0(i) = s\} \to w_s > 0,\quad
\frac{1}{n}\sum_{j = 1}^n\mathbbm{1}\{\gamma_0(j) = t\} \to \pi_t > 0
\]
for all $s\in [p_1]$ and $t\in [p_2]$ as $\min(m, n)\to\infty$. 

Below, we propose a least-squares estimator for $\bSigma_0$ by taking advantage of the technical results in Section \ref{sub:extension_to_general_rectangular_matrices}. 
In preparation for doing so, we need strongly consistent estimators of $\tau_0$ and $\gamma_0$, which can be achieved by the classical spectral clustering method based on the singular vector matrices of $\bY$. 
Formally, let $\widehat{\bU}\in\mathbb{O}(m, r)$ and $\widehat{\bV}\in\mathbb{O}(n, r)$ be the leading $r$ singular vector matrices such that $\bY\widehat{\bV} = \widehat{\bU}\mathrm{diag}\{\sigma_1(\bY),\ldots,\sigma_r(\bY)\}$. Then we apply
the $K$-means clustering procedure (see, for example, \citealp{1056489}) to the rows of $\widehat{\bU}$ and $\widehat{\bV}$, respectively. 
Formally, suppose that the rows of $\widehat{\bU}$ and the rows of $\widehat{\bV}$ are to be assigned into $p_1$ and $p_2$ clusters, respectively. The $K$-means clustering centroids of $\widehat{\bU}$ and $\widehat{\bV}$, represented by an $m\times r$ matrix $\bC(\widehat\bU)$ with $p_1$ distinct rows and an $n\times r$ matrix $\bC(\widehat\bV)$ with $p_2$ distinct rows, are given by
\begin{align*}
&\bC(\widehat\bU) = \argmin_{\bC\in\calC(m, r, p_1)}\| \bC - \widehat\bU\|_{\mathrm{F}},\quad
\bC(\widehat\bV) = \argmin_{\bC\in\calC(n, r, p_2)}\| \bC - \widehat\bV\|_{\mathrm{F}},
\end{align*}
where, for any positive integers $a \geq b$, $\calC(a, r, b) := \{\bC\in\mathbb{R}^{a\times r}:\bC\text{ has }b\text{ distinct rows}\}$. 
Correspondingly, the estimated cluster assignment function $\widehat{\tau}$ for the rows of $\widehat{\bU}$ is defined to be any function $\widehat{\tau}:[m]\to [p_1]$ such that $\widehat{\tau}(i_1) = \widehat{\tau}(i_2)$ if and only if $\bC(\widehat\bU)_{i_1*} = \bC(\widehat\bU)_{i_2*}$ for $i_1,i_2\in[m]$. The estimated cluster assignment function $\widehat{\gamma}$ for the rows of $\widehat{\bV}$ is defined in the same way. Then Theorem \ref{thm:Kmeans_biclustering} below guarantees that the spectral clustering estimates $\widehat{\tau}$ and $\widehat{\gamma}$ are strongly consistent. 

\begin{theorem}\label{thm:Kmeans_biclustering}
Assume the notations and setup in Sections \ref{sub:euclidean_representation_of_subspaces}, \ref{sub:extension_to_general_rectangular_matrices}, and \ref{sub:biclustering} hold. Suppose $\bSigma_0 = \bW_1\bD\bW_2\transpose$ is the singular value decomposition of $\bSigma_0$ with $\bW_1\in\mathbb{O}(p_1, r)$, $\bW_2\in\mathbb{O}(p_2, r)$, and $\bD = \mathrm{diag}\{\sigma_1(\bSigma_0),\ldots,\sigma_r(\bSigma_0)\}$. Denote $\bSigma_{01} = \bW_1\bD^{1/2}$, $\bSigma_{02} = \bW_2\bD^{1/2}$, and assume that there exists some constant $\delta > 0$ such that
\begin{align*}
\min_{s_1\neq s_2}\|[\bSigma_{01}]_{s_1*} - [\bSigma_{01}]_{s_2*}\|_2 \geq\delta,\quad
\min_{t_1\neq t_2}\|[\bSigma_{02}]_{t_1*} - [\bSigma_{02}]_{t_2*}\|_2 \geq\delta.
\end{align*}
If $\log m = o(\sqrt{n})$ and $\log n = o(\sqrt{m})$, then the spectral clustering estimates $\widehat{\tau}$ and $\widehat{\gamma}$ are strongly consistent, i.e., there exists two sequences of permutations $(\omega_m)_m$, $(\iota_n)_n$, such that 
\begin{align}\label{eqn:strong_consistency_biclustering}
\prob_0\left\{d_{\mathrm{H}}(\widehat{\tau}, \omega_m\circ\tau_0) = 0,  d_{\mathrm{H}}(\widehat{\gamma}, \iota_n\circ \gamma_0) = 0\right\}\to 1.
\end{align}
% where the infimum over $\omega$ and $\iota$ are with respect to all permutations $\omega:[p_1]\to[p_1]$ and $\iota:[p_2]\to[p_2]$, respectively.
\end{theorem}
Below, we present the asymptotic normality of the least-squares estimator in Theorem \ref{thm:asymptotic_normality_biclustering}, which is the main result of this section. 
\begin{theorem}\label{thm:asymptotic_normality_biclustering}
Assume the notations and setup in Sections \ref{sub:euclidean_representation_of_subspaces}, \ref{sub:extension_to_general_rectangular_matrices}, and \ref{sub:biclustering} hold. 
% Further assume that any $r$ rows of $\bU_0$ are linearly independent. For any permutation matrices $\bPi_1\in\mathbb{R}^{p_1\times p_1}$ and $\bPi_2\in\mathbb{R}^{p_2\times p_2}$, let $\btheta_{0\bPi_1\bPi_2}$ be the inverse image of $\bPi_1\bSigma_0\bPi_2\transpose$ under the map $\btheta\mapsto \bSigma(\btheta)$ given by \eqref{eqn:Cayley_transform_rectangular}, i.e., $\bSigma(\btheta_{0\bPi_1\bPi_2}) = \bPi_1\bSigma_0\bPi_2\transpose$. 
Let $\widehat{\tau},\widehat{\gamma}$ be strongly consistent estimators of the cluster assignment functions $\tau_0,\gamma_0$ in the sense of \eqref{eqn:strong_consistency_biclustering}. For any $s\in [p_1],t\in [p_2]$, define 
\[
[\widehat{\bSigma}]_{st} = \frac{
	\sum_{i = 1}^m\sum_{j = 1}^ny_{ij}\mathbbm{1}\{\widehat{\tau}(i) = s, \widehat{\gamma}(j) = t\}
}{
	\sum_{i = 1}^m\sum_{j = 1}^n\mathbbm{1}\{\widehat{\tau}(i) = s, \widehat{\gamma}(j) = t\}
}
\]
and let $\widehat{\bSigma} = [[\widehat{\bSigma}]_{st}]_{p_1\times p_2}$. Let
\[
\widehat{\btheta}_{mn} := \argmin_{\btheta\in\mathscr{T}(p_1, p_2, r)}\|\widehat{\bSigma} - \bSigma(\btheta)\|_{\mathrm{F}}^2
\]
be the least-squares estimator, 
where $\mathscr{T}(p_1, p_2, r) = \{\btheta = [\vect(\bA)\transpose, \bmu\transpose]\transpose\in\mathbb{R}^{(p_2 - r)r}\times\mathbb{R}^{p_1r}:\|\bA\|_2 < 1\}$. Then there exist two sequences of permutation matrices $(\bPi_{1m})_m\subset\mathbb{O}(m), (\bPi_{2n})_n\subset\mathbb{O}(n)$, such that
% \[
% \widetilde{\btheta} - \btheta = \{D\bSigma(\btheta)\transpose D\bSigma(\btheta)\}^{-1}D\bSigma(\btheta)\transpose\vect\{\widetilde{\bSigma} - \bSigma(\btheta)\} + o_{\prob_0}\left(\frac{1}{\sqrt{mn}}\right).
% \]
% In addition,
\begin{align*}
\sqrt{mn}\bG(\bPi_{1m},\bPi_{2n})^{-1/2}(\widehat{\btheta}_{mn} - \btheta_{0\bPi_{1m}\bPi_{2n}}) \overset{\calL}{\to}\mathrm{N}(\zero_d, \eye_d),
\end{align*}
where, for any two permutation matrices $\bPi_1\in\mathbb{O}(m)$, $\bPi_2\in\mathbb{O}(n)$,
\begin{align*}
\bG(\bPi_{1},\bPi_{2})
& := 
\{D\bSigma(\btheta_{0\bPi_{1}\bPi_{2}})\transpose
D\bSigma(\btheta_{0\bPi_{1}\bPi_{2}})\}^{-1}\\
&\quad\times D\bSigma(\btheta_{0\bPi_{1}\bPi_{2}})\transpose(\bPi_{2}\otimes\bPi_{1}) \mathrm{diag}\{\sigma^2\vect(\bw\bpi\transpose)\}\\
&\quad\times (\bPi_{2}\transpose\otimes\bPi_{1}\transpose)D\bSigma(\btheta_{0\bPi_{1}\bPi_{2}})\\
&\quad\times 
\{D\bSigma(\btheta_{0\bPi_{1}\bPi_{2}})\transpose D\bSigma(\btheta_{0\bPi_{1}\bPi_{2}})\}^{-1}
,
\end{align*}
and $d = p_1r + (p_2 - r)r$ is the dimension of the parameter space $\mathscr{T}(p_1, p_2, r)$.
\end{theorem}

\begin{remark}
In Theorem \ref{thm:asymptotic_normality_biclustering}, the asymptotic normality of the least-squares estimator $\widehat{\btheta}_{mn}$ does not require a parametric form of the distribution of the entries of $\bY$, or equivalently, that of the entries of $\bE$. When the entries of $\bE$ are independent and identically distributed $\mathrm{N}(0, \sigma^2)$ random variables, the least-squares estimator $\widehat{\btheta}_{mn}$ coincides with the maximum likelihood estimator and is therefore asymptotically efficient. When the likelihood function of $\bY$ is available and is non-Gaussian, one can also follow the idea of Section \ref{sub:stochastic_block_model} and implement the one-step procedure initialized at $\widehat{\btheta}_{mn}$. Under certain regularity conditions (e.g., the conditions required by Theorem 5.45 in \citealp{van2000asymptotic}), the one-step estimator will be asymptotically efficient as well. 
\end{remark}

% subsection applications (end)

\section{Discussion} % (fold)
\label{sec:discussion}

In this paper, we present a novel Euclidean representation framework for low-rank matrices and, correspondingly, develop a collection of technical devices for studying the intrinsic perturbation of low-rank matrices, i.e., when the referential matrix and the perturbed matrix have the same rank. These technical tools are then subsequently applied to three concrete statistical problems in detail, namely, the rate-optimal posterior contraction of Bayesian sparse spiked covariance model under the spectral sine-theta distance (a non-intrinsic loss), the one-step estimator for SBM and its asymptotic efficiency, and least-squares estimation in biclustering. The applications of the obtained technical devices in their respective statistical contexts lead to new and optimal results, demonstrating the usefulness of the proposed framework. 

As mentioned in Section \ref{sec:introduction}, besides the three concrete applications discussed in detail in this paper, there are several other potential applications of the current framework, including sparse canonical correlation analysis, cross-covariance matrix estimation, sparse reduced-rank regression, and Bayesian denoising of simultaneously low-rank and sparse matrices. These applications are naturally connected to general rectangular low-rank matrices. As observed in \cite{cai2018}, the unilateral perturbation bound for the right singular subspace of a low-rank rectangular matrix can be sharper than the spectral/Frobenius norm of the perturbation matrix itself, and Wedin's sine-theta theorem may lead to sub-optimal results. 
As an illustrative example, we briefly discuss how the proposed framework could lead to 
% unilateral  perturbation bound for singular subspace when both the referential matrix and the perturbed matrix have the same rank, which could be further applied to derive, e.g., 
a unilateral posterior contraction rate for singular subspaces in Bayesian denoising of sparse and low-rank matrix models. For simplicity, we assume $\mathrm{rank}(\bSigma(\btheta)) = \mathrm{rank}(\bSigma_0) = 1$ and let $\bu, \bu_0$ be the right singular vectors of $\bSigma$ and $\bSigma_0$, respectively. 
Then simple algebra shows that $D_\bvarphi\bSigma(\btheta_0)\transpose D_\bmu\bSigma(\btheta_0) = \zero$. 
Therefore, the matrix $D\bSigma(\btheta_0)\transpose D\bSigma(\btheta_0)$, which is the Fisher information matrix of the low-rank matrix denoising model, has a block diagonal structure:
\[
D\bSigma(\btheta_0)\transpose D\bSigma(\btheta_0) = \begin{bmatrix*}
D_\bvarphi\bSigma(\btheta_0)\transpose D_\bvarphi\bSigma(\btheta_0) & \zero \\
\zero & D_\bmu\bSigma(\btheta_0)\transpose D_\bmu\bSigma(\btheta_0)
\end{bmatrix*}
\]
Correspondingly, the asymptotic shape of the marginal posterior distribution of $\bvarphi$ only depends on $D_\bvarphi\bSigma(\btheta_0)\transpose D_\bvarphi\bSigma(\btheta_0)$. By Corollary \ref{corr:intrinsic_deviation_Projection}, the unilateral posterior contraction for the right singular vector, namely, $\|\sin\Theta(\bu, \bu_0)\|_2$, can be obtained by a direct analysis of the asymptotic marginal posterior distribution of $\bvarphi$ using a technique similar to that developed in Section \ref{sub:bayesian_sparse_pca_and_non_intrinsic_loss}. We defer the technical details to future works.

% section discussion (end)

\section{Proofs of the main results} % (fold)
\label{sec:proofs_of_the_main_results}

\subsection{Proofs for Section \ref{sub:euclidean_representation_of_subspaces}} % (fold)
\label{sub:proofs_for_section_sub:euclidean_representation_of_subspaces}

\begin{proof}[Proof of Theorem \ref{thm:second_order_deviation_CT}]
First observe that by definition of $\bGamma_\bvarphi$, we have,
\begin{align*}
\|\bGamma_\bvarphi\|_2 = \max_{\|\bvarphi\|_2 = 1}\|\bGamma_\bvarphi\bvarphi\|_2 = \max_{\|\bvarphi\|_2 = 1}\|\vect(\bX_\bvarphi)\|_2 = \max_{\|\bvarphi\|_2 = 1}(2\|\vect(\bA)\|_2^2)^{1/2} = \sqrt{2}.
\end{align*}
% Therefore, 
% \[
% \|\bX_\bvarphi - \bX_{\bvarphi_0}\|_{\mathrm{F}} \leq \|\bGamma_\bvarphi\|_2\|\bvarphi - \bvarphi_0\|_2 < 1/2
% \]
% whenever $\|\bvarphi - \bvarphi_0\| < 1/(2\sqrt{2})$. 
Because $\bX_\bvarphi\transpose = -\bX_\bvarphi$, we also have
\begin{align*}
\|(\eye_p + \bX_\bvarphi)^{-1}\|_2
& = 
\|(\eye_p - \bX_\bvarphi)^{-1}\|_2 = \lambda_{\min}^{-1/2}\left((\eye_p + \bX_\bvarphi)(\eye_p - \bX_\bvarphi)\right)\\
& = \lambda_{\min}^{-1/2}(\eye_p + \bX_\bvarphi\bX_\bvarphi\transpose)\leq \lambda_{\min}^{-1/2}(\eye_p) = 1
\end{align*}
for any $\bvarphi$. Write
\begin{align*}
\|(\eye_p - \bX_\bvarphi)^{-1} - (\eye_p - \bX_{\bvarphi_0})^{-1}\|_{\mathrm{F}}
& = \|(\eye_p - \bX_\bvarphi)^{-1}(\bX_\bvarphi - \bX_{\bvarphi_0})(\eye_p - \bX_{\bvarphi_0})^{-1}\|_{\mathrm{F}}
\\
& \leq \|(\eye_p - \bX_\bvarphi)^{-1}\|_2\|\bX_\bvarphi - \bX_{\bvarphi_0}\|_{\mathrm{F}}\|(\eye_p - \bX_{\bvarphi_0})^{-1}\|_2\\
&\leq \|\bX_\bvarphi - \bX_{\bvarphi_0}\|_{\mathrm{F}}.
\end{align*}
% Therefore, 
Furthermore, by matrix algebra,
\begin{align*}
(\eye_p - \bX_\bvarphi)^{-1} - (\eye_p - \bX_{\bvarphi_0})^{-1}
& = (\eye_p - \bX_\bvarphi)^{-1}(\bX_\bvarphi - \bX_{\bvarphi_0})(\eye_p - \bX_{\bvarphi_0})^{-1}\\
& = (\eye_p - \bX_{\bvarphi_0})^{-1}(\bX_\bvarphi - \bX_{\bvarphi_0})(\eye_p - \bX_{\bvarphi_0})^{-1}
\\
&\quad + \{(\eye_p - \bX_\bvarphi)^{-1} - (\eye_p - \bX_{\bvarphi_0})^{-1}\}(\bX_\bvarphi - \bX_{\bvarphi_0})(\eye_p - \bX_{\bvarphi_0})^{-1}.
\end{align*}
% \begin{align*}
% (\eye_p - \bX_\bvarphi)^{-1} 
% & = [(\eye_p - \bX_{\bvarphi_0})\{\eye_p - (\eye_p - \bX_{\bvarphi_0})^{-1}(\bX_\bvarphi - \bX_{\bvarphi_0})\}]^{-1}\\
% & = \{\eye_p - (\eye_p - \bX_{\bvarphi_0})^{-1}(\bX_\bvarphi - \bX_{\bvarphi_0})\}^{-1}(\eye_p - \bX_{\bvarphi_0})^{-1}\\
% & = (\eye_p - \bX_{\bvarphi_0})^{-1} + \sum_{m = 1}^{\infty}\{(\eye_p - \bX_{\bvarphi_0})^{-1}(\bX_\bvarphi - \bX_{\bvarphi_0})\}^m(\eye_p - \bX_{\bvarphi_0})^{-1}.
% \end{align*}
Denote
\[
\bT(\bvarphi, \bvarphi_0) := 
\{(\eye_p - \bX_\bvarphi)^{-1} - (\eye_p - \bX_{\bvarphi_0})^{-1}\}(\bX_\bvarphi - \bX_{\bvarphi_0})(\eye_p - \bX_{\bvarphi_0})^{-1}.
% \sum_{m = 2}^{\infty}\{(\eye_p - \bX_{\bvarphi_0})^{-1}(\bX_\bvarphi - \bX_{\bvarphi_0})\}^m(\eye_p - \bX_{\bvarphi_0})^{-1}.
\]
% and hence,
% \begin{align*}
% \|(\eye_p - \bX_\bvarphi)^{-1} - (\eye_p - \bX_{\bvarphi_0})^{-1} \|_{\mathrm{F}}
% &\leq \sum_{m = 1}^\infty \|\bX_\bvarphi - \bX_{\bvarphi_0}\|_{\mathrm{F}} = \frac{\|\bX_\bvarphi - \bX_{\bvarphi_0}\|_{\mathrm{F}}}{1 - \|\bX_\bvarphi - \bX_{\bvarphi_0}\|_{\mathrm{F}}}.
% \end{align*}
It follows that
\[
(\eye_p - \bX_\bvarphi)^{-1}  -  (\eye_p - \bX_{\bvarphi_0})^{-1} =
(\eye_p - \bX_{\bvarphi_0})^{-1}(\bX_\bvarphi - \bX_{\bvarphi_0})(\eye_p - \bX_{\bvarphi_0})^{-1}
 + \bT(\bvarphi, \bvarphi_0),
\]
and
\[
\|\bT(\bvarphi,\bvarphi_0)\|_{\mathrm{F}}\leq \|\bX_{\bvarphi} - \bX_{\bvarphi_0}\|_{\mathrm{F}}^2.
\]
Therefore, we can write
\begin{align*}
\bU(\bvarphi) - \bU(\bvarphi_0)
& = \{(\eye_p + \bX_{\bvarphi})(\eye_p - \bX_{\bvarphi})^{-1} - (\eye_p + \bX_{\bvarphi_0})(\eye_p - \bX_{\bvarphi_0})^{-1}\}\eye_{p\times r}\\
& = (\eye_p + \bX_{\bvarphi_0} + \bX_\bvarphi - \bX_{\bvarphi_0})\{(\eye_p - \bX_{\bvarphi})^{-1} - (\eye_p - \bX_{\bvarphi_0})^{-1}\}\eye_{p\times r}\\
&\quad + (\bX_\bvarphi - \bX_{\bvarphi_0})(\eye_p - \bX_{\bvarphi_0})^{-1}\eye_{p\times r}\\
& = (\eye_p + \bX_{\bvarphi_0})(\eye_p - \bX_{\bvarphi_0})^{-1}(\bX_\bvarphi - \bX_{\bvarphi_0})(\eye_p - \bX_{\bvarphi_0})^{-1}\eye_{p\times r}\\
&\quad + (\eye_p + \bX_{\bvarphi_0})\bT(\bvarphi, \bvarphi_0)\eye_{p\times r}\\
&\quad + (\bX_{\bvarphi} - \bX_{\bvarphi_0})\{(\eye_p - \bX_\bvarphi)^{-1}  -  (\eye_p - \bX_{\bvarphi_0})^{-1}\}\eye_{p\times r}\\
&\quad + (\eye_p - \bX_{\bvarphi_0})(\eye_p - \bX_{\bvarphi_0})^{-1}(\bX_\bvarphi - \bX_{\bvarphi_0})(\eye_p - \bX_{\bvarphi_0})^{-1}\eye_{p\times r}\\
& = 2(\eye_p - \bX_{\bvarphi_0})^{-1}(\bX_\bvarphi - \bX_{\bvarphi_0})(\eye_p - \bX_{\bvarphi_0})^{-1}\eye_{p\times r}
% \\&\quad 
+ 
\bR_\bU(\bvarphi, \bvarphi_0),
\end{align*}
where
\begin{align*}
\bR_\bU(\bvarphi, \bvarphi_0)
& := (\eye_p + \bX_{\bvarphi_0})\bT(\bvarphi, \bvarphi_0)\eye_{p\times r}\\
&\quad + (\bX_{\bvarphi} - \bX_{\bvarphi_0})\{(\eye_p - \bX_\bvarphi)^{-1}  -  (\eye_p - \bX_{\bvarphi_0})^{-1}\}\eye_{p\times r}.
\end{align*}
% and when $\|\bvarphi - \bvarphi_0\|_2 < 1/(2\sqrt{2})$,
Using the aforementioned results, we further compute
\begin{align*}
\|\bR_\bU(\bvarphi, \bvarphi_0)\|_{\mathrm{F}}
&\leq \|\eye_p + \bX_{\bvarphi_0}\|_2\|\bT(\bvarphi, \bvarphi_0)\|_{\mathrm{F}} + \|\bX_\bvarphi - \bX_{\bvarphi_0}\|_{\mathrm{F}}^2\\
&\leq (1 + 2\|\bA\|_2)\|\bX_{\bvarphi} - \bX_{\bvarphi_0}\|_{\mathrm{F}}^2 + \|\bX_{\bvarphi} - \bX_{\bvarphi_0}\|_{\mathrm{F}}^2\\
&\leq 8\|\bvarphi - \bvarphi_0\|_2^2.
\end{align*}
The proof is thus completed. 
\end{proof} 

\begin{proof}[Proof of Theorem \ref{thm:first_order_deviation_ICT}]
% Since $\|(\eye + \bQ_{01})^{-1}(\bQ_1 - \bQ_{01})\|_2 < \alpha\|(\eye_r + \bQ_{01})\|_2\leq 1/2$, we obtain by the matrix series that
% \begin{align*}
% (\eye_r + \bQ_1)^{-1}
% & = \left[(\eye_r + \bQ_{01})\left\{\eye_r + (\eye_r + \bQ_{01})^{-1}(\bQ_1 - \bQ_{01})\right\}\right]^{-1}\\
% & = \left\{\eye_r - (\eye_r + \bQ_{01})^{-1}(\bQ_{01} - \bQ_{1})\right\}^{-1}(\eye_r + \bQ_{01})^{-1}\\
% & = \left[\eye_r + \sum_{m = 1}^\infty\{(\eye_r + \bQ_{01})^{-1}(\bQ_{01} - \bQ_{1})\}^m\right](\eye_r + \bQ_{01})^{-1}\\
% & = (\eye_r + \bQ_{01})^{-1} + \sum_{m = 1}^\infty\{(\eye_r + \bQ_{01})^{-1}(\bQ_1 - \bQ_{01})\}^m(\eye_r + \bQ_{01})^{-1}. 
% \end{align*}
Observe that
\begin{align*}
\|(\eye_r + \bQ_1)^{-1} - (\eye_r + \bQ_{01})^{-1}\|_{\mathrm{F}}
& = \|(\eye_r + \bQ_1)^{-1}\{(\eye_r + \bQ_{01}) - (\eye_r + \bQ_{1})\}(\eye_r + \bQ_{01})^{-1}\|_{\mathrm{F}}\\
&\leq \|(\eye_r + \bQ_1)^{-1}\|_2 \|\bQ_1 - \bQ_{01}\|_{\mathrm{F}} \|(\eye_r + \bQ_{01})^{-1}\|_2\\
&\leq \|\bQ_1 - \bQ_{01}\|_{\mathrm{F}}
\end{align*}
because $\eye_r + \bQ_1\succeq \eye_r$ and $\eye_r + \bQ_{01}\succeq \eye_r$. 
Therefore,
\begin{align*}
\|\bA(\bU) - \bA(\bU_0)\|_{\mathrm{F}}
% &\leq \frac{1}{2}\|(\bQ_2 - \bQ_{02})\{\eye_r + \bF(\bU_0) - \bF(\bU_0) + \bF(\bU)\}\|_{\mathrm{F}}\\
% &\quad + \frac{1}{2}\|\bQ_{02}\{\bF(\bU) - \bF(\bU_0)\}\|_{\mathrm{F}}\\
&\leq \|\bQ_2 - \bQ_{02}\|_{\mathrm{F}}\|(\eye_r + \bQ_1)^{-1}\|_2
\\&\quad
 + \|\bQ_{02}\|_2\|(\eye_r + \bQ_1)^{-1} - (\eye_r + \bQ_{01})^{-1}\|_{\mathrm{F}}\\
&\leq \|\bQ_2 - \bQ_{02}\|_{\mathrm{F}} + \|\bQ_1 - \bQ_{01}\|_{\mathrm{F}}\\
&\leq 2\|\bU - \bU_{0}\|_{\mathrm{F}}.
\end{align*}
The proof is thus completed. 
\end{proof}

% subsection proofs_for_section_sub:euclidean_representation_of_subspaces (end)

\subsection{Proofs for Section \ref{sub:intrinsic_perturbation_theorems}} % (fold)
\label{sub:proofs_for_section_sub:intrinsic_perturbation_theorems}

\begin{proof}[Proof of Theorem \ref{thm:CT_deviation_Sigma}]
First observe that the following matrix decomposition holds:
\begin{align*}
% \bSigma - \bSigma_0
&\bSigma(\btheta) - \bSigma(\btheta_{0})\\
&\quad = \bU_{0}(\bM - \bM_0)\bU_{0}\transpose + \bU_{0}\bM_0\{\bU(\bvarphi) - \bU_{0}\}\transpose + \{\bU(\bvarphi) - \bU_{0}\}\bM_0\bU_{0}\transpose + \bR_{\bSigma}(\btheta, \btheta_{0}),
\end{align*}
where the remainder
\begin{align*}
R_\bSigma(\btheta, \btheta_{0}) & = \bU(\bvarphi)(\bM - \bM_0)\{\bU(\bvarphi) - \bU_{0}\}\transpose + \{\bU(\bvarphi) - \bU_{0}\}\bM_0\{\bU(\bvarphi) - \bU_{0}\}\transpose\\
&\quad + \{\bU(\bvarphi) - \bU_0\}(\bM - \bM_0)\bU_0\transpose
\end{align*}
satisfies 
\begin{align*}
\|R_\bSigma(\btheta, \btheta_{0})\|_{\mathrm{F}}
&\leq 2\|\bM - \bM_0\|_{\mathrm{F}}\|\bU(\bvarphi) - \bU_0\|_{\mathrm{F}} + \|\bM_0\|_2\|\bU(\bvarphi) - \bU_0\|_{\mathrm{F}}^2\\
&\leq \|\bM - \bM_0\|_2^2 + \|\bU(\bvarphi) - \bU_0\|_{\mathrm{F}}^2 + \|\bM_0\|_2\|\bU(\bvarphi) - \bU_0\|_{\mathrm{F}}^2\\
&\leq 2\|\bmu - \bmu_0\|_2^2 + \left(1 + \|\bM_0\|_2\right)\|\bU(\bvarphi) - \bU_0\|_{\mathrm{F}}^2.
\end{align*}
By Theorem \ref{thm:second_order_deviation_CT}, 
\begin{align*}
% \|\bU(\bvarphi) - \bU_0\|_{\mathrm{F}}
% & \leq \|D\bU(\bvarphi_0)\|_2\|\bvarphi - \bvarphi_0\|_2 + 8\|\bvarphi - \bvarphi_0\|_2^2,\\
\|\bU(\bvarphi) - \bU_0\|_{\mathrm{F}}
& \leq 2\sqrt{2}\|\bvarphi - \bvarphi_0\|_2
\end{align*}
for all $\bvarphi$ and $\bvarphi_0$. 
% Note that 
% \begin{align*}
% \|D\bU(\bvarphi_0)\|_2 &\leq 2\|(\eye_p - \bX_0)^{-\mathrm{T}}\|_2\|(\eye_p - \bX_0)^{-1}\|_2\|\bGamma_\bvarphi\|_2\\
% &\leq 2\sqrt{2}\|(\eye_p - \bX_0)^{-1}\|_2^2\leq 2\sqrt{2}. 
% \end{align*}
Therefore,
\[
\|\bR_\bSigma(\btheta, \btheta_0)\|_{\mathrm{F}}\leq 8(1 + \|\bM_0\|_2)\|\btheta - \btheta_0\|_2^2.
\]
% when $\|\btheta - \btheta_0\|_2\leq 1/(2\sqrt{2})$. 
Furthermore, using Theorem \ref{thm:second_order_deviation_CT} again, we obtain
\begin{align*}
\vect\{\bU(\bvarphi) - \bU_0\} = 
D\bU(\bvarphi_0)(\bvarphi - \bvarphi_0) + \vect\{\bR_\bU(\bvarphi, \bvarphi_0)\},
\end{align*}
where $\|\bR_\bU(\bvarphi, \bvarphi_0)\|_{\mathrm{F}}\leq 8\|\bvarphi - \bvarphi_0\|_2^2$ for all $\bvarphi, \bvarphi_0$. In matrix form, we have
\begin{align*}
\bU(\bvarphi) - \bU_0 = 2(\eye_p - \bX_0)^{-1}(\bX_\bvarphi - \bX_0)(\eye_p - \bX_0)^{-1}\eye_{p\times r} + \bR_\bU(\bvarphi, \bvarphi_0).
\end{align*}
Hence we finally obtain
\begin{align*}
\vect\{\bSigma(\btheta) - \bSigma_0\}
& = D_\bmu\bSigma(\btheta_0)(\bmu - \bmu_0) + (\eye_{p^2} + \bK_{pp})(\bU_0\bM_0\otimes \eye_p)\vect\{\bU(\bvarphi )- \bU_0\}\\
&\quad + \vect\{\bR_\bSigma(\btheta, \btheta_0)\}\\
& = D\bSigma(\btheta_0)(\btheta - \btheta_0) + (\eye_{p}^2 + \bK_{pp})(\bU_0\bM_0\otimes\eye_p)\vect\{\bR_\bU(\bvarphi, \bvarphi_0)\}\\
&\quad + \vect\{\bR_\bSigma(\btheta, \btheta_0)\}\\
& = D\bSigma(\btheta_0)(\btheta - \btheta_0) + \vect\{\bR(\btheta, \btheta_0)\},
\end{align*}
where
\[
\bR(\btheta, \btheta_0) = (\eye_{p}^2 + \bK_{pp})(\bU_0\bM_0\otimes\eye_p)\vect\{\bR_\bU(\bvarphi, \bvarphi_0)\} + \vect\{\bR_\bSigma(\btheta, \btheta_0)\},
\]
and the proof is completed by observing that
\begin{align*}
\|\bR(\btheta, \btheta_0)\|_{\mathrm{F}}
&\leq 2\|\bM_0\|_2\|\bR_\bU(\bvarphi, \bvarphi_0)\|_{\mathrm{F}} + \|\bR_\bSigma(\btheta, \btheta_0)\|_{\mathrm{F}}\leq 16(1 + \|\bM_0\|_2)\|\btheta - \btheta_0\|_2^2.
\end{align*}
\end{proof}

Before proving Theorem \ref{thm:intrinsic_deviation_Sigma}, we introduce the following intermediate Lemma claiming that the perturbation of projection matrices can be controlled by the corresponding Euclidean representing vectors.

\begin{lemma}
% [First-Order Intrinsic Deviation of ICT]
\label{thm:First-order intrinsic deviation of ICT}
Under the setup of Section \ref{sub:euclidean_representation_of_subspaces}, if $\|\bA\|_2,\|\bA_0\|_2 < 1$, then
\begin{align*}
\|\bU(\bvarphi) - \bU(\bvarphi_0)\|_{\mathrm{F}}
&\leq \frac{2\|\bU(\bvarphi)\bU(\bvarphi)\transpose - \bU(\bvarphi_0)\bU(\bvarphi_0)\transpose\|_{\mathrm{F}}}{\min\{\lambda_{r}(\bQ_{01}), \lambda_r(\bQ_1)\}},\\
\|\bvarphi - \bvarphi_0\|_2
&\leq \frac{4\|\bU(\bvarphi)\bU(\bvarphi)\transpose - \bU(\bvarphi_0)\bU(\bvarphi_0)\transpose\|_{\mathrm{F}}}{\min\{\lambda_r(\bQ_{01}), \lambda_r(\bQ_1)\}}.
\end{align*}
% Furthermore, if 
% \[
% \|\bU\bU\transpose - \bU_0\bU_0\transpose\|_{\mathrm{F}}\leq2\left(\frac{1 + \|\bA_0\|_2^2}{1 - \|\bA_0\|_2^2}\right)^2,
% \]
% then
% \begin{align*}
% \|\bU - \bU_0\|_{\mathrm{F}}
% &\leq \frac{2\|\bU\bU\transpose - \bU_0\bU_0\transpose\|_{\mathrm{F}}}{\min\{\lambda_{r}(\bQ_{01}), \lambda_r(\bQ_1)\}},\\
% \|\bvarphi - \bvarphi_0\|_2
% &\leq \frac{2\|\bU\bU\transpose - \bU_0\bU_0\transpose\|_{\mathrm{F}}}{\lambda_r(\bQ_{01})\min\{\lambda_r(\bQ_{01}), \lambda_r(\bQ_1)\}}.
% \end{align*}
\end{lemma}
\begin{proof}
% [\bf Proof of Theorem \ref{thm:First-order intrinsic deviation of ICT}]
For convenience we denote $\bU = \bU(\bvarphi)$ and $\bU_0 = \bU(\bvarphi_0)$. First note that
\begin{align*}
\|\bU\bU\transpose - \bU_0\bU_0\transpose\|_{\mathrm{F}}^2
& = \left\|
\begin{bmatrix*}
\bQ_1^2 - \bQ_{01}^2 & \bQ_1\bQ_2\transpose - \bQ_{01}\bQ_{02}\transpose\\
\bQ_2\bQ_1 - \bQ_{02}\bQ_{01} & \bQ_2\bQ_2\transpose - \bQ_{02}\bQ_{02}\transpose
\end{bmatrix*}
\right\|_{\mathrm{F}}^2\\
& = \|\bQ_1^2 - \bQ_{01}^2\|_{\mathrm{F}} + 2\|\bQ_2\bQ_1 - \bQ_{02}\bQ_{01}\|_{\mathrm{F}}^2 + \|\bQ_2\bQ_2\transpose - \bQ_{02}\bQ_{02}\transpose\|_{\mathrm{F}}^2.
\end{align*}
For any positive definite matrix $\bQ\in\mathbb{R}^{r\times r}$, denote $\bTheta(\bQ) = \bQ^{1/2}$ the matrix square root function evaluated at $\bQ$. Then the matrix differential technique yields
\[
\frac{\partial\vect\{\bTheta(\bQ)\}}{\partial\vect(\bQ)\transpose} = (\bQ^{1/2}\otimes \eye_r + \eye_r\otimes \bQ^{1/2})^{-1}\preceq \frac{1}{2\sigma_{\min}(\bQ^{1/2})}\eye_{r^2} =  \frac{1}{2\sigma_{\min}^{1/2}(\bQ)}\eye_{r^2}. 
\]
By the mean-value inequality,
\begin{align*}
\|\bQ_1 - \bQ_{01}\|_{\mathrm{F}} 
 & = \|\bTheta(\bQ_1^2) - \bTheta(\bQ_{01}^2)\|_{\mathrm{F}}\\
 &\leq \max_{t\in [0, 1]}\left\|\frac{\partial\vect\{\bTheta(\bQ)\}}{\partial\vect(\bQ)\transpose}\mathrel{\Bigg|}_{\bQ = (t\bQ_{01}^2 + (1 - t)\bQ_1^2)}\right\|_2\|\bQ_1^2 - \bQ_{01}^2\|_{\mathrm{F}}\\
 & = \frac{1}{2}\max_{t\in [0,1]}\lambda_{\min}^{-1/2}\{t\bQ_{01}^2 + (1 - t)\bQ_1^2\}\|\bQ_1^2 - \bQ_{01}^2\|_{\mathrm{F}}\\
 &\leq \frac{\|\bQ_1^2 - \bQ_{01}^2\|_{\mathrm{F}}}{2\min\{\lambda_{r}(\bQ_{01}), \lambda_r(\bQ_1)\}}
 \leq \frac{\|\bU\bU\transpose - \bU_0\bU_0\transpose\|_{\mathrm{F}}}{2\min\{\lambda_{r}(\bQ_{01}), \lambda_r(\bQ_1)\}}.
\end{align*}
For $\bQ_2$, we have
\begin{align*}
\bQ_2 - \bQ_{02} & = \bQ_2\bQ_{01}\bQ_{01}^{-1} - \bQ_2\bQ_1\bQ_{01}^{-1} + \bQ_2\bQ_1\bQ_{01}^{-1} - \bQ_{02}\bQ_{01}\bQ_{01}^{-1}\\
& = \bQ_2(\bQ_{01} - \bQ_1)\bQ_{01}^{-1} + (\bQ_2\bQ_1 - \bQ_{02}\bQ_{01})\bQ_{01}^{-1}.
\end{align*}
Using the previous result, we further write
\begin{align*}
\|\bQ_2 - \bQ_{02}\|_{\mathrm{F}} 
& \leq \|\bQ_2\|_2\|\bQ_{01} - \bQ_1\|_{\mathrm{F}}\|\bQ_{01}^{-1} \|_2 + \|\bQ_2\bQ_1 - \bQ_{02}\bQ_{01}\|_{\mathrm{F}}\|\bQ_{01}^{-1}\|_2\\
& \leq \frac{\|\bU\bU\transpose - \bU_0\bU_0\transpose\|_{\mathrm{F}}\|\bQ_{01}^{-1}\|_2}{2\min\{\lambda_r(\bQ_{01}), \lambda_r(\bQ_1)\}} + \|\bU\bU\transpose - \bU_0\bU_0\transpose\|_{\mathrm{F}}\|\bQ_{01}^{-1}\|_2
 % = \frac{2\|\bU\bU\transpose - \bU_0\bU_0\transpose\|_{\mathrm{F}}}{\lambda_r(\min\{\lambda_r(\bQ_{01}), \lambda_r(\bQ_1)\}}
 .
\end{align*}
Hence, 
\begin{align*}
\|\bU - \bU_0\|_{\mathrm{F}}
& = (\|\bQ_1 - \bQ_{01}\|_{\mathrm{F}}^2 + \|\bQ_2 - \bQ_{02}\|_{\mathrm{F}}^2)^{1/2}\\
&\leq\|\bQ_1 - \bQ_{01}\|_{\mathrm{F}} + \|\bQ_2 - \bQ_{02}\|_{\mathrm{F}}\\
&\leq \frac{\|\bU\bU\transpose - \bU_0\bU_0\transpose\|_{\mathrm{F}}\|\bQ_{01}^{-1}\|_2}{\min\{\lambda_r(\bQ_{01}), \lambda_r(\bQ_1)\}} + \|\bU\bU\transpose - \bU_0\bU_0\transpose\|_{\mathrm{F}}\|\bQ_{01}^{-1}\|_2\\
&\leq \frac{2\|\bU\bU\transpose - \bU_0\bU_0\transpose\|_{\mathrm{F}}\|\bQ_{01}^{-1}\|_2}{\min\{\lambda_r(\bQ_{01}), \lambda_r(\bQ_1)\}}.
\end{align*}
This completes the proof of the first inequality. For the second inequality, note that
\begin{align*}
\|(\eye_r + \bQ_1)^{-1} - (\eye_r + \bQ_{01})^{-1}\|_{\mathrm{F}}
& = \|(\eye_r + \bQ_1)^{-1}\{(\eye_r + \bQ_{01}) - (\eye_r + \bQ_{1})\}(\eye_r + \bQ_{01})^{-1}\|_{\mathrm{F}}\\
&\leq \|(\eye_r + \bQ_1)^{-1}\|_2 \|\bQ_1 - \bQ_{01}\|_{\mathrm{F}} \|(\eye_r + \bQ_{01})^{-1}\|_2\\
&\leq \|\bQ_1 - \bQ_{01}\|_{\mathrm{F}}.
\end{align*}
Therefore, by Theorem \ref{thm:first_order_deviation_ICT},
\begin{align*}
\|\bvarphi - \bvarphi_0\|_2 & = \|\bA - \bA_0\|_{\mathrm{F}}
% \\
% & = \left\|\frac{1}{2}\bQ_2\{\eye_r + (\eye_r - \bQ_1)(\eye_r + \bQ_1)^{-1}\} - \frac{1}{2}\bQ_{02}\{\eye_r + (\eye_r - \bQ_{01})(\eye_r + \bQ_{01})^{-1}\}\right\|_{\mathrm{F}}\\
% & 
% = \left\|\bQ_2(\eye_r + \bQ_1)^{-1} - \bQ_{02}(\eye_r + \bQ_{01})^{-1}\right\|_{\mathrm{F}}\\
% & = \left\|(\bQ_2 - \bQ_{02})(\eye_r + \bQ_1)^{-1}\|_{\mathrm{F}} + \|\bQ_{02}\{(\eye_r + \bQ_1)^{-1} - (\eye_r + \bQ_{01})^{-1}\}\right\|_{\mathrm{F}}\\
% &\leq \left\|\bQ_2 - \bQ_{02}\|_{\mathrm{F}}\|(\eye_r + \bQ_1)^{-1}\|_2 + \|\{(\eye_r + \bQ_1)^{-1} - (\eye_r + \bQ_{01})^{-1}\}\right\|_{\mathrm{F}}\|\bQ_{02}\|_2\\
% &
% \leq \|\bQ_2 - \bQ_{02}\|_{\mathrm{F}} + \|\bQ_1 - \bQ_{01}\|_{\mathrm{F}}
\leq 2\|\bU - \bU_0\|_{\mathrm{F}}.
\end{align*}
The proof is completed by combining the obtained upper bound for $\|\bU - \bU_0\|_{\mathrm{F}}$. 
\end{proof}

\begin{proof}[Proof of Theorem \ref{thm:intrinsic_deviation_Sigma}]
By Weyl's inequality, we have
\begin{align*}
|\lambda_r(\bQ_1) - \lambda_r(\bQ_{01})| & = \frac{|\lambda^2_r(\bQ_1) - \lambda^2_r(\bQ_{01})|}{\lambda_r(\bQ_1) + \lambda_r(\bQ_{01})}
\leq \frac{\|\bQ_1^2 - \bQ_{01}^2\|_{\mathrm{F}}}{\lambda_r(\bQ_{01})}\leq \frac{\|\bU\bU\transpose - \bU_0\bU_0\transpose\|_{\mathrm{F}}}{\lambda_r(\bQ_{01})}.
\end{align*}
Note that
\begin{align*}
\lambda_r(\bQ_{01}) & = \lambda_r\left\{(\eye_r - \bA_0\transpose\bA_0)(\eye_r + \bA_0\transpose\bA_0)^{-1}\right\}
 % = \frac{1 - \sigma_1^2(\bA_0)}{1 + \sigma_1^2(\bA_0)}.
 = \frac{1 - \|\bA_0\|_2^2}{1 + \|\bA_0\|_2^2}. 
\end{align*}
By the Davis-Kahan theorem (see, e.g., Theorem 2 in \citealp{10.1093/biomet/asv008}),
\begin{align*}
|\lambda_r(\bQ_1) - \lambda_r(\bQ_{01})|\leq \frac{\|\bU\bU\transpose - \bU_0\bU_0\transpose\|_{\mathrm{F}}}{\lambda_r(\bQ_{01})}\leq 2\sqrt{2}\left(\frac{1 + \|\bA_0\|_2^2}{1 - \|\bA_0\|_2^2}\right)\frac{\|\bSigma(\btheta) - \bSigma(\btheta_0)\|_{\mathrm{F}}}{\lambda_{r}(\bM_0)}.
\end{align*}
Therefore,
\begin{align*}
\lambda_r(\bQ_1)
&\geq\lambda_r(\bQ_{01}) - \frac{2\sqrt{2}(1 + \|\bA_0\|_2^2)\|\bSigma(\btheta) - \bSigma(\btheta_0)\|_{\mathrm{F}}}{(1 - \|\bA_0\|_2^2)\lambda_{r}(\bM_0)}\\
& = \frac{1 - \|\bA_0\|_2^2}{1 + \|\bA_0\|_2^2} - \frac{2\sqrt{2}(1 + \|\bA_0\|_2^2)\|\bSigma(\btheta) - \bSigma(\btheta_0)\|_{\mathrm{F}}}{(1 - \|\bA_0\|_2^2)\lambda_{r}(\bM_0)}\\
&\geq \frac{1 - \|\bA_0\|_2^2}{2(1 + \|\bA_0\|_2^2)} = \frac{\lambda_r(\bQ_{01})}{2}.
\end{align*}
Hence, by Lemma \ref{thm:First-order intrinsic deviation of ICT} and the Davis-Kahan theorem, we have
\begin{align*}
\|\bU - \bU_0\|_{\mathrm{F}}
& \leq \frac{4\|\bU\bU\transpose - \bU_0\bU_0\transpose\|_{\mathrm{F}}}{\lambda_r(\bQ_{01})}
\leq \frac{8\sqrt{2}\|\bSigma(\btheta) - \bSigma(\btheta_0)\|_{\mathrm{F}}}{\lambda_r(\bQ_{01})\lambda_r(\bM_0)}
\\&
 = \frac{8\sqrt{2}(1 + \|\bA_0\|_2^2)}{\lambda_r(\bM_0)(1 - \|\bA_0\|_2^2)}\|\bSigma(\btheta) - \bSigma(\btheta_0)\|_{\mathrm{F}},\\
\|\bvarphi - \bvarphi_0\|_2
&\leq \frac{8\|\bU\bU\transpose - \bU_0\bU_0\transpose\|_{\mathrm{F}}}{\lambda_r(\bQ_{01})}
 \leq \frac{16\sqrt{2}\|\bSigma(\btheta) - \bSigma(\btheta_0)\|_{\mathrm{F}}}{\lambda_r(\bQ_{01})\lambda_r(\bM_0)}
 \\&
 = \frac{16\sqrt{2}(1 + \|\bA_0\|_2^2)}{\lambda_r(\bM_0)(1 - \|\bA_0\|_2^2)}\|\bSigma(\btheta) - \bSigma(\btheta_0)\|_{\mathrm{F}}.
\end{align*}
For the matrix $\bM$ and the vector $\bmu$, we have,
\begin{align*}
\|\bmu - \bmu_0\|_2 &\leq \|\bM - \bM_0\|_{\mathrm{F}} = \| \bU\transpose\bSigma(\btheta)\bU - \bU_0\transpose\bSigma(\btheta_0)\bU_0 \|_{\mathrm{F}}\\
& \leq \|\bU\transpose\{\bSigma(\btheta) - \bSigma(\btheta_0)\}\bU\|_{\mathrm{F}} + \|\bU\transpose\bSigma_0(\bU - \bU_0)\|_{\mathrm{F}} + \|(\bU - \bU_0)\transpose\bSigma_0\bU_0\|_{\mathrm{F}}\\
& \leq \|\bSigma(\btheta) - \bSigma(\btheta_0)\|_{\mathrm{F}} + 2\lambda_1(\bM_0)\|\bU - \bU_0\|_{\mathrm{F}}\\
& \leq \left\{1 + \frac{16\sqrt{2}\lambda_{1}(\bM_0)(1 + \|\bA_0\|_2^2)}{\lambda_{r}(\bM_0)(1 - \|\bA_0\|_2^2)}\right\}\|\bSigma(\btheta) - \bSigma(\btheta_0)\|_{\mathrm{F}}.
\end{align*}
Therefore, we conclude that
\begin{align*}
\|\btheta - \btheta_0\|_2
&\leq \|\bvarphi - \bvarphi_0\|_2 + \|\bmu - \bmu_0\|_2\\
&\leq \left[1 + \frac{16\sqrt{2}\{1 + \lambda_{1}(\bM_0)\}(1 + \|\bA_0\|_2^2)^2}{\lambda_{r}(\bM_0)(1 - \|\bA_0\|_2^2)^2}
 % + \frac{16\lambda_{1}(\bM_0)(1 + \|\bA_0\|_2^2)}{\lambda_{r}(\bM_0)(1 - \|\bA_0\|_2^2)}
 \right]\|\bSigma(\btheta) - \bSigma(\btheta_0)\|_{\mathrm{F}}.
\end{align*}
\end{proof}

% subsection proofs_for_section_sub:intrinsic_perturbation_theorems (end)

\subsection{Proof of Theorem \ref{thm:DSigma_singular_value}} % (fold)
\label{sub:proof_of_theorem_thm:dsigma_singular_value}

The proof of Theorem \ref{thm:DSigma_singular_value} is involved and relies on the following two technical lemmas, the proofs of which are deferred to the Supplementary Material. 

\begin{lemma}\label{lemma:DSigma_lower_bound}
Under the setup and notations in Sections \ref{sub:euclidean_representation_of_subspaces} and \ref{sub:euclidean_represention_of_low_rank_matrices_and_intrinsic_perturbation}, 
\begin{align*}
&\bSigma_0^2\otimes\eye_p - \bSigma_0\otimes\bSigma_0 + \eye_p\otimes\bSigma_0^2\\
&\quad\succeq \lambda_{r}^2(\bM_0)\{\bU_0\bU_0\transpose\otimes(\eye_p - \bU_0\bU_0\transpose) + (\eye_p - \bU_0\bU_0\transpose)\otimes \bU_0\bU_0\transpose\},\\
&\bSigma_0^2\otimes (\eye_p - \bU_0\bU_0\transpose) + (\eye_p - \bU_0\bU_0\transpose)\otimes\bSigma_0^2\\
&\quad\succeq \lambda_{r}^2(\bM_0)\{\bU_0\bU_0\transpose\otimes(\eye_p - \bU_0\bU_0\transpose) + (\eye_p - \bU_0\bU_0\transpose)\otimes \bU_0\bU_0\transpose\}.
\end{align*}
\end{lemma}

\begin{lemma}\label{lemma:Cmatrix_singular_value}
Let $\bA_0\in\mathbb{R}^{(p - r)\times r}$ with $\|\bA_0\|_2 < 1$, and define
\begin{align*}
\bC_{11} &= (\eye_r + \bA_0\transpose\bA_0)^{-1},& &\bC_{12} = -(\eye_r + \bA_0\transpose\bA_0)^{-1}\bA_0\transpose,\\
\bC_{21} &= \bA_0(\eye_r + \bA_0\transpose\bA_0)^{-1},& &\bC_{22} = \eye_{p - r} - \bA_0(\eye_r + \bA_0\transpose\bA_0)^{-1}\bA_0\transpose.
\end{align*}
Let 
\[
\bC_0 = \begin{bmatrix*}
\bC_{11} & \bC_{12}\\\bC_{21} & \bC_{22}
\end{bmatrix*},\quad\text{and}\quad \bU_0 = \bC_0^{-\mathrm{T}}\bC_0\eye_{p\times r}
% = \bC_0^{-\mathrm{T}}\bC_0\begin{bmatrix*}
% \eye_r\\\zero_{(p - r)\times r}
% \end{bmatrix*}
.
\]
% and $\bU_{0\perp}$ be the orthogonal complement of $\bU_0$ such that $[\bU_0,\bU_{0\perp}]\in\mathbb{O}(p)$. 
\begin{itemize}
	\item[(i)] For any vector $\bA\in\mathbb{R}^{(p - r)\times r}$ and $\bvarphi := \vect(\bA)$,
	\begin{align*}
	&\vect(\bX_\bvarphi)\transpose
	(\bC_0\transpose\otimes\bC_0\transpose)\{\bU_0\bU_0\otimes(\eye_p - \bU_{0}\bU_0\transpose)\}(\bC_0\otimes\bC_0)
	\vect(\bX_\bvarphi)\\
	&\quad = 
	\vect(\bX_\bvarphi)\transpose
	(\bC_0\transpose\otimes\bC_0\transpose)\{(\eye_p - \bU_{0}\bU_0\transpose)\otimes\bU_0\bU_0\transpose\}(\bC_0\otimes\bC_0)
	\vect(\bX_\bvarphi)\\
	&\quad = \|\bC_{22}\transpose\bA\bC_{11} - \bC_{12}\transpose\bA\transpose\bC_{21}\|_{\mathrm{F}}^2,
	\end{align*}
	where
	\[
	\bX_{\bvarphi}:= \begin{bmatrix*}
	\zero_{r\times r} & -\bA\transpose\\
	\bA & \zero_{(p - r)\times (p - r)}
	\end{bmatrix*}.
	\]
	\item[(ii)] $\bC_{11}\otimes \bC_{22} - (\bC_{21}\transpose\otimes \bC_{12}\transpose)\bK_{(p - r)r}$ has the following lower bound in spectra:
	\[
	\sigma_{\min}\{\bC_{11}\otimes \bC_{22} - (\bC_{21}\transpose\otimes \bC_{12}\transpose)\bK_{(p - r)r}\}\geq 
	\left\{
	\begin{aligned}
	&\frac{1 - \|\bA_0\|_2^2}{(1 + \|\bA_0\|_2^2)^2},\quad &\text{if }r > 1,\\
	&\frac{1}{1 + \|\bA_0\|_2^2},\quad &\text{if }r = 1.
	\end{aligned}
	\right.
	\]
\end{itemize}
\end{lemma}

\begin{proof}[Proof of first assertion of Theorem \ref{thm:DSigma_singular_value}]
Let $\bA$ be any $(p-r)\times r$ matrix and $\bvarphi: = \vect(\bA)$. Denote
\[
\bX_\bvarphi = \begin{bmatrix*}
\zero_{r\times r} & -\bA\transpose\\ \bA & \zero_{(p - r)\times (p - r)}
\end{bmatrix*}.
\]
Write $\bX_0 = \bX_{\bvarphi_0}$ and $\bC_0 = (\eye_p - \bX_0)^{-1}$. Let $\bA_0$ be the corresponding $(p - r)\times r$ matrix such that $\bvarphi_0 = \vect(\bA_0)$. Then for any $\bvarphi = \vect(\bA)\in\mathbb{R}^d$, we have $\bGamma_{\bvarphi}\bvarphi = \vect(\bX_\bvarphi)$ by the definition of $\bGamma_\bvarphi$. Therefore,
\begin{align*}
(\bU_0\bM_0\otimes \eye_r)D\bU(\bvarphi_0)\bvarphi 
& = 2(\bU_0\bM_0\otimes \eye_r)[\eye_{p\times r}\transpose(\eye_p - \bX_0)\inverseT\otimes (\eye - \bX_0)^{-1}]\vect(\bX_\bvarphi)\\
& = 2[\bU_0\bM_0\eye_{p\times r}\transpose (\eye_p - \bX_0)\inverseT \otimes (\eye - \bX_0)^{-1}]\vect(\bX_\bvarphi)\\
& = 2\vect\left\{ (\eye_p - \bX_0)^{-1}\bX_\bvarphi (\eye_p - \bX_0)^{-1}\eye_{p\times r}\bM_0\bU_0\transpose \right\}\\
% & = 2\vect\left\{ (\eye_p - \bX_0)^{-1}\bX_\bvarphi (\eye_p - \bX_0)\inverseT (\eye_p + \bX_0)(\eye_p - \bX_0)^{-1}\eye_{p\times r}\bM_0\bU_0\transpose \right\}\\
& = 2\vect\left\{ (\eye_p - \bX_0)^{-1}\bX_\bvarphi (\eye_p - \bX_0)\inverseT \bU_0\bM_0\bU_0\transpose \right\}.
\end{align*}
% By Theorem 3.1 (vi) of \cite{magnus1979}, the eigenvalues of $\bK_{pp}$ are $\pm1$'s, and we see that 
By definition of the commutation matrix $\bK_{pp}$, 
$\bK_{pp}\vect(\bM) = \vect(\bM)$ for any symmetric $\bM\in\mathbb{R}^{p\times p}$ and $\bK_{pp}\vect(\bM) = -\vect(\bM)$ for any skew-symmetric $\bM\in\mathbb{R}^{p\times p}$. Hence, for any $p\times p$ matrix $\bM$, we have, $(\eye_{p^2} + \bK_{pp})\vect(\bM) = \vect(\bM + \bM\transpose)$, and hence,
\begin{align*}
D_\bvarphi\bSigma(\btheta_0)\bvarphi
& = (\eye_{p^2} + \bK_{pp})(\bU_0\bM_0\otimes \eye_r)D\bU(\bvarphi_0)\bvarphi \\
& = 2(\eye_{p^2} + \bK_{pp})\vect\left\{ (\eye_p - \bX_0)^{-1}\bX_\bvarphi (\eye_p - \bX_0)\inverseT \bU_0\bM_0\bU_0\transpose \right\}\\
& = 2\vect\left(\bC_0\bX_\bvarphi\bC_0\transpose\bSigma_0 - \bSigma_0\bC_0\bX_\bvarphi\bC_0\transpose\right)\\
& = 2(\bSigma_0\bC_0\otimes \bC_0 - \bC_0\otimes \bSigma_0\bC_0)\vect(\bX_\bvarphi).
\end{align*}
% where $\bB_0 = \bU_0\bM_0\bU_0\transpose$. 
Let $\bU_{0\perp}$ be the orthogonal complement of $\bU_0$ such that $[\bU_0, \bU_{0\perp}]\in\mathbb{O}(p)$. By Lemma \ref{lemma:DSigma_lower_bound},
\begin{align*}
&\|D_\bvarphi\bSigma(\btheta_0)\bvarphi\|_2^2\\
&\quad = 4\vect(\bX_\bvarphi)\transpose(\bSigma_0\bC_0\otimes \bC_0 - \bC_0\otimes \bSigma_0\bC_0)\transpose (\bSigma_0\bC_0\otimes \bC_0 - \bC_0\otimes \bSigma_0\bC_0)\vect(\bX_\bvarphi)\\
% &\quad = 4\vect(\bX_\bvarphi)\transpose(\bC_0\transpose\bSigma_0^2\bC_0\otimes \bC_0\transpose\bC_0 - 2\bC_0\transpose \bSigma_0\bC_0\otimes \bC_0\transpose\bSigma_0\bC_0 + \bC_0\transpose\bC_0\otimes\bC_0\transpose\bSigma_0^2\bC_0)\vect(\bX_\bvarphi)\\
&\quad = 4\vect(\bX_\bvarphi)\transpose
\left\{(\bC_0\transpose\otimes\bC_0\transpose)(\bSigma_0^2\otimes\eye_p - 2\bSigma_0\otimes\bSigma_0 + \eye_p\otimes \bSigma_0^2)(\bC_0\otimes\bC_0)\right\}
\vect(\bX_\bvarphi)\\
&\quad\geq 4\lambda_{r}^2(\bM_0)\vect(\bX_\bvarphi)\transpose
(\bC_0\transpose\otimes\bC_0\transpose)
(\bU_0\bU_0\transpose\otimes\bU_{0\perp}\bU_{0\perp}\transpose + \bU_{0\perp}\bU_{0\perp}\transpose\otimes\bU_0\bU_0\transpose)\\
&\quad\quad\times
(\bC_0\otimes\bC_0)
\vect(\bX_\bvarphi)\\
&\quad = 4\lambda_{r}^2\vect(\bX_\bvarphi)\transpose
(\bC_0\transpose\otimes\bC_0\transpose)\{(\eye_p - \bU_{0}\bU_0\transpose)\otimes\bU_0\bU_0\transpose\}(\bC_0\otimes\bC_0)
\vect(\bX_\bvarphi)\\
&\quad\quad + 4\lambda_{r}^2\vect(\bX_\bvarphi)\transpose
(\bC_0\transpose\otimes\bC_0\transpose)\{\bU_0\bU_0\transpose\otimes(\eye_p - \bU_{0}\bU_0\transpose)\}
(\bC_0\otimes\bC_0)
\vect(\bX_\bvarphi).
\end{align*}
Write $\bC_0$ in the block form
\[
\bC_0 = \begin{bmatrix*}
\bC_{11} & \bC_{12} \\ \bC_{21} & \bC_{22}
\end{bmatrix*}
= 
\begin{bmatrix*}
(\eye_r + \bA_0\transpose\bA_0)^{-1} & -(\eye_r + \bA_0\transpose\bA_0)^{-1}\bA_0\transpose\\
\bA_0(\eye_r + \bA_0\transpose\bA_0)^{-1} & \eye_{p - r} - \bA_0(\eye_r + \bA_0\transpose\bA_0)^{-1}\bA_0\transpose
\end{bmatrix*}
\]
according to Appendix C in \cite{jauch2020}. Invoking Lemma \ref{lemma:Cmatrix_singular_value} (i), we further obtain the following lower bound for $\|D_\bvarphi\bSigma(\btheta_0)\bvarphi\|_2^2$:
\begin{align*}
\|D_\bvarphi\bSigma(\btheta_0)\bvarphi\|_2^2
& \geq 8\lambda_{r}^2(\bM_0)\|\bC_{22}\transpose\bA\bC_{11} - \bC_{12}\transpose\bA\transpose\bC_{21}\|_{\mathrm{F}}^2\\
& = 8\lambda_{r}^2(\bM_0)\|(\bC_{11}\transpose\otimes \bC_{22}\transpose)\vect(\bA) - (\bC_{21}\transpose\otimes \bC_{12}\transpose)\vect(\bA\transpose)\|_2^2\\
& = 8\lambda_{r}^2(\bM_0)\|(\bC_{11}\transpose\otimes \bC_{22}\transpose)\vect(\bA) - (\bC_{21}\transpose\otimes \bC_{12}\transpose)\bK_{(p - r)r}\vect(\bA)\|_2^2\\
& \geq 8\lambda_{r}^2(\bM_0)\sigma_{\min}^2\left\{(\bC_{11}\otimes \bC_{22}) - (\bC_{21}\transpose\otimes \bC_{12}\transpose)\bK_{(p - r)r}\right\}\|\bvarphi\|_{2}^2,
\end{align*}
which immediately implies that
\[
\sigma_{\min}\{D_\bvarphi\bSigma(\btheta_0)\} \geq 2\sqrt{2}\sigma_r(\bM_0)\sigma_{\min}\left\{(\bC_{11}\otimes \bC_{22}) - (\bC_{21}\transpose\otimes \bC_{12}\transpose)\bK_{(p - r)r}\right\}.
\]
By Lemma \ref{lemma:Cmatrix_singular_value} (ii), we finally obtain
\begin{align*}
\sigma_{\min}\{D_\bvarphi\bSigma(\btheta_0)\}
&\geq \left\{
\begin{aligned}
&\frac{2\sqrt{2}\sigma_{r}(\bM_0)(1 - \|\bA_0\|_2^2)}{(1 + \|\bA_0\|_2^2)^2},\quad&\text{if }r > 1,\\
&\frac{2\sqrt{2}\sigma_{r}(\bM_0)}{1 + \|\bA_0\|_2^2},\quad&\text{if }r = 1.
\end{aligned}
\right.
\end{align*}
completing the proof of the first assertion regarding $\sigma_{\min}\{D_\bvarphi\bSigma(\btheta_0)\}$. 
\end{proof}

\begin{proof}[Proof of second assertion of Theorem \ref{thm:DSigma_singular_value}]
Write
\[
D\bSigma(\btheta_0)\transpose D\bSigma(\btheta_0) = \begin{bmatrix*}
\bJ_1 & \bJ_2\\
\bJ_2\transpose & \bJ_3
\end{bmatrix*}: = \begin{bmatrix*}
D_\bvarphi\bSigma(\btheta_0)\transpose D_\bvarphi\bSigma(\btheta_0) & 
D_\bvarphi\bSigma(\btheta_0)\transpose D_\bmu\bSigma(\btheta_0)\\
D_\bmu\bSigma(\btheta_0)\transpose D_\bvarphi\bSigma(\btheta_0) &
D_\bmu\bSigma(\btheta_0)\transpose D_\bmu\bSigma(\btheta_0)
\end{bmatrix*}
.
\]
Note that the Schur complement of the block $D_\bmu\bSigma(\btheta_0)\transpose D_\bmu\bSigma(\btheta_0)$ of the entire matrix $D\bSigma(\btheta_0)\transpose D\bSigma(\btheta_0)$ is given by
\begin{align*}
\bJ_1 - \bJ_2\bJ_3^{-1}\bJ_2\transpose
& = D_\bvarphi\bSigma(\btheta_0)\transpose D_\bvarphi\bSigma(\btheta_0)\\
&\quad - D_\bvarphi\bSigma(\btheta_0)\transpose D_\bmu\bSigma(\btheta_0)\{D_\bmu\bSigma(\btheta_0)\transpose D_\bmu\bSigma(\btheta_0)\}^{-1}D_\bmu\bSigma(\btheta_0)\transpose D_\bvarphi\bSigma(\btheta_0),
\end{align*}
We now assume that $\bJ_1 - \bJ_2\bJ_3^{-1}\bJ_2\transpose$ is invertible. Then by the block matrix inversion formula,
\begin{align*}
&\{D\bSigma(\btheta_0)\transpose D\bSigma(\btheta_0)\}^{-1}\\
&\quad = \begin{bmatrix*}
\bJ_1 & \bJ_2\\
\bJ_2\transpose & \bJ_3
\end{bmatrix*}^{-1}
% \\&
 = \begin{bmatrix*}
(\bJ_1 - \bJ_2\bJ_3^{-1}\bJ_2\transpose)^{-1} & -(\bJ_1 - \bJ_2\bJ_3^{-1}\bJ_2\transpose)\bJ_2\bJ_3^{-1}\\
-\bJ_3^{-1}\bJ_2\transpose(\bJ_1 - \bJ_2\bJ_3^{-1}\bJ_2\transpose)^{-1} & \bJ_3^{-1} + \bJ_3^{-1}\bJ_2\transpose(\bJ_1 - \bJ_2\bJ_3^{-1}\bJ_2\transpose)^{-1}\bJ_2\bJ_3^{-1}
\end{bmatrix*},
\end{align*}
By construction, 
\[
\|\bJ_3^{-1}\|_2 = \|\{D_\bmu\bSigma(\btheta_0)\transpose D_\bmu\bSigma(\btheta_0)\}^{-1}\|_2 = \|\{\mathbb{D}_r\transpose(\eye_r\otimes\eye_r)\mathbb{D}_r\}^{-1}\|\leq 1,
\]
where the last inequality is due to Theorem 4.4 in \cite{magnus1988linear}. In addition, 
\begin{align*}
\|\bJ_2\|_2 &\leq \|D\bU(\bvarphi_0)\transpose(\bM_0\bU_0\transpose\otimes \eye_p)(\eye_{p^2} + \bK_{pp})(\bU_0\otimes\bU_0)\mathbb{D}_r\|\\
&\leq 2\|D\bU(\bvarphi_0)\|_2\|\bM_0\|_2\|\mathbb{D}_r\|_2\leq 8\|\bM_0\|_2\|\eye_{p\times r}(\eye_p - \bX_{\bvarphi_0})^{-1}\|_2\|(\eye_p - \bX_{\bvarphi_0})^{-1}\|_2\\
&\leq 8\|\bM_0\|_2.
\end{align*}
Thus, by Lemma 3.4 of \cite{10.1112/blms/bds080}, we see that
\begin{align}\label{eqn:Schur_complement_upper_spectral_bound}
\begin{aligned}
\|\{D\bSigma(\btheta_0)\transpose D\bSigma(\btheta_0)\}^{-1}\|_2
&\leq \|(\bJ_1 - \bJ_2\bJ_3^{-1}\bJ_2\transpose)^{-1}\|_2\\
&\quad + \|\bJ_3^{-1} + \bJ_3^{-1}\bJ_2\transpose(\bJ_1 - \bJ_2\bJ_3^{-1}\bJ_2\transpose)^{-1}\bJ_2\bJ_3^{-1}\|_2\\
&\leq \|(\bJ_1 - \bJ_2\bJ_3^{-1}\bJ_2\transpose)^{-1}\|_2\\
&\quad + \|\bJ_3^{-1}\|_2 + \|\bJ_3^{-1}\|_2^2\|\bJ_2\|_2^2\|(\bJ_1 - \bJ_2\bJ_3^{-1}\bJ_2\transpose)^{-1}\|_2\\
&\leq 1 + (1 + 64\|\bM_0\|_2^2)\|(\bJ_1 - \bJ_2\bJ_3^{-1}\bJ_2\transpose)^{-1}\|_2.
\end{aligned}
\end{align}
Therefore, it is sufficient to provide a lower bound for the smallest eigenvalue of $\bJ_1 - \bJ_2\bJ_3^{-1}\bJ_2\transpose$.
For any $\bvarphi$, we follow the computation above and write
\begin{align*}
&\bvarphi\transpose(\bJ_1 - \bJ_2\bJ_3^{-1}\bJ_2\transpose)\bvarphi\\
&\quad = \bvarphi\transpose
D_\bvarphi\bSigma(\btheta_0)\transpose D_\bvarphi\bSigma(\btheta_0)\bvarphi \\
&\quad\quad - \bvarphi\transpose D_\bvarphi\bSigma(\btheta_0)\transpose D_\bmu\bSigma(\btheta_0)\{ D_\bmu\bSigma(\btheta_0)\transpose D_\bmu\bSigma(\btheta_0)\}^{-1}D_\bmu\bSigma(\btheta_0)\transpose D_\bvarphi\bSigma(\btheta_0)
\bvarphi\\
&\quad = 4\vect(\bX_\bvarphi)\transpose
\left\{(\bC_0\transpose\otimes\bC_0\transpose)(\bSigma_0^2\otimes\eye_p - 2\bSigma_0\otimes\bSigma_0 + \eye_p\otimes \bSigma_0^2)(\bC_0\otimes\bC_0)\right\}
\vect(\bX_\bvarphi)\\
&\quad\quad - 4\vect(\bX_\bvarphi)(\bC_0\transpose\otimes\bC_0\transpose)(\bSigma_0\otimes\eye_p - \eye_p\otimes\bSigma)\\
&\quad\quad\qquad
\times
D_\bmu\bSigma(\btheta_0)\{ D_\bmu\bSigma(\btheta_0)\transpose D_\bmu\bSigma(\btheta_0)\}^{-1}D_\bmu\bSigma(\btheta_0)\\
&\quad\quad\qquad
\times
(\bSigma_0\otimes\eye_p - \eye_p\otimes\bSigma)
(\bC_0\otimes\bC_0)\vect(\bX_\bvarphi).
\end{align*}
Note that
\begin{align*}
&D_\bmu\bSigma(\btheta_0)\{ D_\bmu\bSigma(\btheta_0)\transpose D_\bmu\bSigma(\btheta_0)\}^{-1}D_\bmu\bSigma(\btheta_0)\\
&\quad = (\bU_0\otimes\bU_0)\mathbb{D}_r\{\mathbb{D}_r\transpose(\eye_r\otimes\eye_r)\mathbb{D}_r\}^{-1}\mathbb{D}_r\transpose(\bU_0\otimes\bU_0)\transpose
\end{align*}
is a projection matrix. Since $\mathbb{D}_r\in\mathbb{R}^{r^2\times r(r + 1)/2}$ has full column rank and $r^2\geq r(r + 1)/2$, then $\mathbb{D}_r\{\mathbb{D}_r\transpose(\eye_r\otimes\eye_r)\mathbb{D}_r\}^{-1}\mathbb{D}_r\preceq \eye_{r^2}$, and hence, 
\begin{align*}
D_\bmu\bSigma(\btheta_0)\{ D_\bmu\bSigma(\btheta_0)\transpose D_\bmu\bSigma(\btheta_0)\}^{-1}D_\bmu\bSigma(\btheta_0)
% & = (\bU_0\otimes\bU_0)\bGamma_\bmu\{\bGamma_\bmu\transpose(\eye_r\otimes\eye_r)\bGamma_\bmu\}^{-1}\bGamma_\bmu\transpose(\bU_0\otimes\bU_0)\transpose\\
& \preceq (\bU_0\otimes\bU_0)(\bU_0\otimes\bU_0)\transpose.
\end{align*}
Therefore, we invoke Lemmas \ref{lemma:DSigma_lower_bound} and \ref{lemma:Cmatrix_singular_value} and proceed to compute
\begin{align*}
&\bvarphi\transpose(\bJ_1 - \bJ_2\bJ_3^{-1}\bJ_2\transpose)\bvarphi\\
% &\quad = 4\vect(\bX_\bvarphi)\transpose
% \left\{(\bC_0\transpose\otimes\bC_0\transpose)(\bSigma_0^2\otimes\eye_p - 2\bSigma_0\otimes\bSigma_0 + \eye_p\otimes \bSigma_0^2)(\bC_0\otimes\bC_0)\right\}
% \vect(\bX_\bvarphi)\\
% &\quad\quad - 4\vect(\bX_\bvarphi)(\bC_0\transpose\otimes\bC_0\transpose)(\bSigma_0\otimes\eye_p - \eye_p\otimes\bSigma)\\
% &\quad\quad\qquad
% \times
% D_\bmu\bSigma(\btheta_0)\{ D_\bmu\bSigma(\btheta_0)\transpose D_\bmu\bSigma(\btheta_0)\}^{-1}D_\bmu\bSigma(\btheta_0)\\
% &\quad\quad\qquad
% \times
% (\bSigma_0\otimes\eye_p - \eye_p\otimes\bSigma_0)
% (\bC_0\otimes\bC_0)\vect(\bX_\bvarphi)
% \\
&\quad \geq 4\vect(\bX_\bvarphi)\transpose
\left\{(\bC_0\transpose\otimes\bC_0\transpose)(\bSigma_0^2\otimes\eye_p - 2\bSigma_0\otimes\bSigma_0 + \eye_p\otimes \bSigma_0^2)(\bC_0\otimes\bC_0)\right\}
\vect(\bX_\bvarphi)\\
&\quad\quad - 4\vect(\bX_\bvarphi)(\bC_0\transpose\otimes\bC_0\transpose)
(\bSigma_0\otimes\eye_p - \eye_p\otimes\bSigma_0)(\bU_0\otimes\bU_0)(\bU_0\otimes\bU_0)\transpose\\
&\quad\quad\quad\times(\bSigma_0\otimes\eye_p - \eye_p\otimes\bSigma_0)
% \\
% &\quad\quad\quad\times
(\bC_0\otimes\bC_0)\vect(\bX_\bvarphi)\\
&\quad = 4\vect(\bX_\bvarphi)\transpose
\left\{(\bC_0\transpose\otimes\bC_0\transpose)(\bSigma_0^2\otimes\eye_p - 2\bSigma_0\otimes\bSigma_0 + \eye_p\otimes \bSigma_0^2)(\bC_0\otimes\bC_0)\right\}
\vect(\bX_\bvarphi)\\
&\quad\quad - 4\vect(\bX_\bvarphi)(\bC_0\transpose\otimes\bC_0\transpose)
(\bSigma_0^2\otimes\bU_0\bU_0\transpose - 2\bSigma_0\otimes \bSigma_0 + \bU_0\bU_0\transpose\otimes \bSigma_0^2)\\
&\quad\quad\quad\times
(\bC_0\otimes\bC_0)\vect(\bX_\bvarphi)\\
&\quad = 4\vect(\bX_\bvarphi)\transpose
(\bC_0\transpose\otimes\bC_0\transpose)\left\{\bSigma_0^2\otimes(\eye_p - \bU_0\bU_0\transpose) + (\eye_p - \bU_0\bU_0\transpose)\otimes \bSigma_0^2\right\}\\
&\quad\quad\times (\bC_0\otimes\bC_0)
\vect(\bX_\bvarphi)\\
&\quad\geq 4\lambda_{r}^2(\bM_0)\vect(\bX_\bvarphi)\transpose(\bC_0\transpose\otimes\bC_0\transpose)(
 \bU_0\bU_0\transpose\otimes \bU_{0\perp}\bU_{0\perp}\transpose + 
\bU_{0\perp}\bU_{0\perp}\transpose\otimes\bU_0\bU_0\transpose
)\\
&\quad\quad\times (\bC_0\otimes\bC_0)\vect(\bX_\bvarphi)\\
&\quad = 8\lambda_{r}^2(\bM_0)\|\bC_{22}\transpose\bA\bC_{11} - \bC_{12}\transpose\bA\transpose\bC_{21}\|_{\mathrm{F}}^2\\
&\quad \geq 8\lambda_{r}^2(\bM_0)\sigma_{\min}^2\left\{(\bC_{11}\otimes \bC_{22}) - (\bC_{21}\transpose\otimes \bC_{12}\transpose)\bK_{(p - r)r}\right\}\|\bvarphi\|_{2}^2.
\end{align*}
It follows from Lemma \ref{lemma:Cmatrix_singular_value} (ii) that
\begin{align*}
\lambda_{\min}(\bJ_1 - \bJ_2\bJ_3^{-1}\bJ_2\transpose)
\geq\left\{
\begin{aligned}
&\frac{8\lambda_{r}^2(\bM_0)(1 - \|\bA_0\|_2^2)}{(1 + \|\bA_0\|_2^2)^2},\quad&\text{if }r > 1,\\
&\frac{8\lambda_{r}^2(\bM_0)}{1 + \|\bA_0\|_2^2},\quad&\text{if }r = 1.
\end{aligned}
\right.
\end{align*}
Therefore, using \eqref{eqn:Schur_complement_upper_spectral_bound},
\begin{align*}
\|\{D\bSigma(\btheta_0)\transpose D\bSigma(\btheta_0)\}^{-1}\|_2
\leq \left\{
\begin{aligned}
&1 + \frac{(1 + 64\|\bM_0\|_2^2)(1 + \|\bA_0\|_2^2)^2}{8\lambda_{r}^2(\bM_0)(1 - \|\bA_0\|_2^2)},\quad&\text{if }r > 1,\\
&1 + \frac{(1 + 64\|\bM_0\|_2^2)(1 + \|\bA_0\|_2^2)}{8\lambda_{r}^2(\bM_0)},\quad&\text{if }r = 1,
\end{aligned}
\right.
\end{align*}
and the proof is thus completed.
\end{proof}

% subsection proof_of_theorem_thm:dsigma_singular_value (end)

% section proofs_of_the_main_results (end)

\section*{Supplement}

The supplementary material  includes the proofs of Lemmas \ref{lemma:DSigma_lower_bound} and \ref{lemma:Cmatrix_singular_value}, Theorems \ref{thm:CT_deviation_Sigma_rectangular} and \ref{thm:DSigma_singular_value_rectangular}, and the results in Section \ref{sec:applications}. 

\clearpage

\begin{center}
  \begin{Large}
    \textbf{Supplementary Material for ``Euclidean Representation of Low-Rank Matrices and Its Statistical Applications''}
  \end{Large}
\end{center}
\appendix
\counterwithin{lemma}{section}
\counterwithin{theorem}{section}

The Supplementary Material contains the proofs of Lemmas \ref{lemma:DSigma_lower_bound} and \ref{lemma:Cmatrix_singular_value}, Theorems \ref{thm:CT_deviation_Sigma_rectangular} and \ref{thm:DSigma_singular_value_rectangular}, and the results in Section \ref{sec:applications}.

\section{Proofs of Technical Lemmas in Section \ref{sec:proofs_of_the_main_results}} % (fold)
\label{sec:proofs_of_technical_lemmas_in_Proofs}

\begin{proof}[Proof of Lemma \ref{lemma:DSigma_lower_bound}]
let $\bU_{0\perp}$ be the orthogonal complement of $\bU_0$ such that $\bW_0:= [\bU_0,\bU_{0\perp}]\in\mathbb{O}(p)$, and denote 
\[
\widetilde\bM_0 = \eye_{p\times r}\bM_0\eye_{p\times r}\transpose = \begin{bmatrix*}
\bM_0 & \\ & \zero_{(p - r)\times(p - r)}
\end{bmatrix*}.
\]
Clearly, one can take $\bW_0 = \bC_0\inverseT\bC_0$. Suppose $\bM_0 = \bV_0\bLambda_0\bV_0\transpose$ is the spectral decomposition of $\bM_0$. Then
\begin{align*}
\widetilde{\bM}_0 = \eye_{p\times r}\bV_0\bLambda_0\bV_0\transpose\eye_{p\times r}\transpose
= \widetilde{\bV}_0\widetilde{\bLambda}_0\widetilde{\bV}_0\transpose,
\end{align*}
where
\[
\widetilde{\bV}_0 := \begin{bmatrix*}
\bV_0 & \\ & \eye_{p - r}
\end{bmatrix*},\quad\text{and}\quad
\widetilde{\bLambda}_0 = \begin{bmatrix*}
\bLambda_0 & \\ & \zero_{(p - r)\times (p - r)}
\end{bmatrix*}.
\]
For convenience let $\lambda_{0k}:=\lambda_k(\bM_0)$, $k\in [r]$. For the first assertion, write
\begin{align*}
&\bSigma_0^2\otimes\eye_p - 2\bSigma_0\otimes\bSigma_0 + \eye_p\otimes \bSigma_0^2\\
&\quad = (\bW_0\otimes\bW_0)(\widetilde\bM_0^2\otimes\eye_p - 2\widetilde\bM_0\otimes\widetilde\bM_0 + \eye_p\otimes\widetilde\bM_0^2)(\bW_0\transpose\otimes\bW_0\transpose)\\
&\quad = (\bW_0\otimes\bW_0)
\begin{bmatrix*}
\bM_0^2\otimes\eye_p - 2\bM_0\otimes\widetilde\bM_0 + \eye_r\otimes\widetilde\bM_0^2 & \\
& \eye_{(p - r)}\otimes \widetilde{\bM}_0^2
\end{bmatrix*}
(\bW_0\transpose\otimes\bW_0\transpose)
\\
% &\quad = (\bW_0\otimes\bW_0)
% \begin{bmatrix*}
% \bV_0\bLambda_0^2\bV_0\transpose\otimes\widetilde{\bV}_0\widetilde{\bV}_0\transpose - 2\bV_0\bLambda_0\bV_0\transpose\otimes\widetilde{\bV}_0\widetilde{\bLambda}_0\widetilde{\bV}_0\transpose + \bV_0\bV_0\transpose\otimes\widetilde{\bV}_0\widetilde{\bLambda}_0\widetilde{\bV}_0\transpose & \\
% & \eye_{(p - r)}\otimes \widetilde{\bM}_0^2
% \end{bmatrix*}
% (\bW_0\transpose\otimes\bW_0\transpose)
% \\
% &\quad = (\bW_0\otimes\bW_0)
% \begin{bmatrix*}
% (\bV_0\otimes\widetilde\bV_0)
% (\bLambda_0\otimes\eye_p - 2\bLambda_0\otimes\widetilde{\bLambda}_0 + \eye_r\otimes\widetilde\bLambda_0)
% (\bV_0\otimes\widetilde\bV_0)\transpose
% % \bV_0\bLambda_0^2\bV_0\transpose\otimes\widetilde{\bV}_0\widetilde{\bV}_0\transpose - 2\bV_0\bLambda_0\bV_0\transpose\otimes\widetilde{\bV}_0\widetilde{\bLambda}_0\widetilde{\bV}_0\transpose + \bV_0\bV_0\transpose\otimes\widetilde{\bV}_0\widetilde{\bLambda}_0\widetilde{\bV}_0\transpose
%  & \\
% & \eye_{(p - r)}\otimes \widetilde{\bV}_0\widetilde{\bLambda}_0^2\widetilde{\bV}_0\transpose
% \end{bmatrix*}
% (\bW_0\transpose\otimes\bW_0\transpose)\\
&\quad
= (\bW_0\otimes\bW_0)(\widetilde{\bV}_0\otimes\widetilde{\bV}_0)
% \begin{bmatrix*}
% \bV_0\otimes \widetilde{\bV}_0 & \\ & \eye_{p - r}\otimes\widetilde{\bV}_0
% \end{bmatrix*}
\begin{bmatrix*}
% (\bV_0\otimes\widetilde\bV_0)
\bLambda_0^2\otimes\eye_p - 2\bLambda_0\otimes\widetilde{\bLambda}_0 + \eye_r\otimes\widetilde\bLambda_0^2
 & \\
& \eye_{(p - r)}\otimes \widetilde{\bLambda}_0^2
\end{bmatrix*}\\
&\quad\times
(\widetilde{\bV}_0\transpose\otimes\widetilde{\bV}_0\transpose)
(\bW_0\transpose\otimes\bW_0\transpose)\\
&\quad = (\bW_0\otimes\bW_0)(\widetilde{\bV}_0\otimes\widetilde{\bV}_0)
\begin{bmatrix*}
\lambda_{01}\eye_p - \widetilde\bLambda_0 &  &  &  &  & \\
& \ddots & & & & \\
&        & \lambda_{0r}\eye_p - \widetilde\bLambda_0 & & & \\ 
&		 & 												 & \widetilde\bLambda_0 & & \\
&		 & 												 &  &\ddots & \\
&		 & 												 &  & & \widetilde\bLambda_0
\end{bmatrix*}^2\\
&\quad\quad\times
(\widetilde{\bV}_0\transpose\otimes\widetilde{\bV}_0\transpose)
(\bW_0\transpose\otimes\bW_0\transpose).
\end{align*}
Since for each $k\in [r]$,
\begin{align*}
(\lambda_{0k}\eye_p - \widetilde\bLambda_0)^2 & 
= \begin{bmatrix*}
(\lambda_{0k} - \lambda_{01})^2 & & & & & \\
& \ddots & & & & \\
& & (\lambda_{0k} - \lambda_{0r})^2 & & & \\
& & & \lambda_{0k}^2 & & \\
& & & & \ddots & \\
& & & & & \lambda_{0k}^2
\end{bmatrix*}
\succeq\lambda_{0r}^2\begin{bmatrix*}
\zero_{r\times r} & \\ & \eye_{p - r}
\end{bmatrix*},
\end{align*}
and
\[
\widetilde\bLambda_0^2\succeq \lambda_{0r}^2\begin{bmatrix*}
\eye_r & \\ & \zero_{(p - r)\times(p - r)}
\end{bmatrix*},
\]
we can further write
\begin{align*}
&\bSigma_0^2\otimes\eye_p - 2\bSigma_0\otimes\bSigma_0 + \eye_p\otimes \bSigma_0^2\\
&\quad \succeq
\lambda_{0r}^2(\bW_0\otimes\bW_0)(\widetilde{\bV}_0\otimes\widetilde{\bV}_0)
\begin{bmatrix*}
\eye_r\otimes\begin{bmatrix*}
\zero_{r\times r} & \\ & \eye_{p - r}
\end{bmatrix*} & \\
& \eye_{p - r}\otimes\begin{bmatrix*}
\eye_r & \\
& \zero_{(p - r)\times(p - r)}
\end{bmatrix*}
\end{bmatrix*}\\
&\quad\quad\times
(\widetilde{\bV}_0\transpose\otimes\widetilde{\bV}_0\transpose)
(\bW_0\transpose\otimes\bW_0\transpose).
\end{align*}
Now we focus on the matrix on the right-hand side of the previous display. Write
\begin{align*}
&(\widetilde{\bV}_0\otimes\widetilde{\bV}_0)
\begin{bmatrix*}
\eye_r\otimes\begin{bmatrix*}
\zero_{r\times r} & \\ & \eye_{p - r}
\end{bmatrix*} & \\
& \eye_{p - r}\otimes\begin{bmatrix*}
\eye_r & \\
& \zero_{(p - r)\times(p - r)}
\end{bmatrix*}
\end{bmatrix*}
(\widetilde{\bV}_0\transpose\otimes\widetilde{\bV}_0\transpose)\\
&\quad
= \begin{bmatrix*}
\bV_0\otimes\widetilde{\bV}_0 & \\ & \eye_{p - r}\otimes\widetilde{\bV}_0
\end{bmatrix*}
\begin{bmatrix*}
\eye_r\otimes\begin{bmatrix*}
\zero & \\ & \eye_{p - r}
\end{bmatrix*} & \\
& \eye_{p - r}\otimes\begin{bmatrix*}
\eye_r & \\
& \zero
\end{bmatrix*}
\end{bmatrix*}
\begin{bmatrix*}
\bV_0\transpose\otimes\widetilde{\bV}_0\transpose & \\ & \eye_{p - r}\otimes\widetilde{\bV}_0\transpose
\end{bmatrix*}\\
&\quad
= \begin{bmatrix*}
\bV_0\bV_0\transpose\otimes\begin{bmatrix*}
\bV_0 & \\ & \eye_{p - r}
\end{bmatrix*}\begin{bmatrix*}
\zero & \\ & \eye_{p - r}
\end{bmatrix*}\begin{bmatrix*}
\bV_0\transpose & \\ & \eye_{p - r}
\end{bmatrix*} & \\
& \eye_{p - r}\otimes\begin{bmatrix*}
\bV_0 & \\ & \eye_{p - r}
\end{bmatrix*}\begin{bmatrix*}
\eye_r & \\
& \zero
\end{bmatrix*}
\begin{bmatrix*}
\bV_0\transpose & \\ & \eye_{p - r}
\end{bmatrix*}
\end{bmatrix*}\\
&\quad = \begin{bmatrix*}
\eye_r\otimes\begin{bmatrix*}
\zero_{r\times r} & \\ & \eye_{p - r}
\end{bmatrix*} & \\
& \eye_{p - r}\otimes\begin{bmatrix*}
\eye_r & \\
& \zero_{(p - r)\times(p - r)}
\end{bmatrix*}
\end{bmatrix*}.
\end{align*}
Therefore,
\begin{align*}
&\bSigma_0^2\otimes\eye_p - 2\bSigma_0\otimes\bSigma_0 + \eye_p\otimes \bSigma_0^2\\
&\quad \succeq\lambda_{0r}^2
(\bW_0\otimes\bW_0)(\widetilde{\bV}_0\otimes\widetilde{\bV}_0)
\begin{bmatrix*}
\eye_r\otimes\begin{bmatrix*}
\zero_{r\times r} & \\ & \eye_{p - r}
\end{bmatrix*} & \\
& \eye_{p - r}\otimes\begin{bmatrix*}
\eye_r & \\
& \zero_{(p - r)\times(p - r)}
\end{bmatrix*}
\end{bmatrix*}\\
&\quad\quad\times
(\widetilde{\bV}_0\transpose\otimes\widetilde{\bV}_0\transpose)
(\bW_0\transpose\otimes\bW_0\transpose)\\
&\quad = \lambda_{0r}^2(\bW_0\otimes\bW_0)
\begin{bmatrix*}
\eye_r\otimes\begin{bmatrix*}
\zero_{r\times r} & \\ & \eye_{p - r}
\end{bmatrix*} & \\
& \eye_{p - r}\otimes\begin{bmatrix*}
\eye_r & \\
& \zero_{(p - r)\times(p - r)}
\end{bmatrix*}
\end{bmatrix*}
(\bW_0\transpose\otimes\bW_0\transpose)\\
&\quad = \lambda_{0r}^2\begin{bmatrix*}
\bU_0\otimes \bW_0 & \bU_{0\perp}\otimes\bW_0
\end{bmatrix*}
\begin{bmatrix*}
\eye_r\otimes\begin{bmatrix*}
\zero_{r\times r} & \\ & \eye_{p - r}
\end{bmatrix*} & \\
& \eye_{p - r}\otimes\begin{bmatrix*}
\eye_r & \\
& \zero_{(p - r)\times(p - r)}
\end{bmatrix*}
\end{bmatrix*}\\
&\quad\quad\times
\begin{bmatrix*}
\bU_0\transpose\otimes \bW_0\transpose\\
\bU_{0\perp}\transpose\otimes \bW_0\transpose
\end{bmatrix*}\\
&\quad = \lambda_{0r}^2(\bU_0\otimes\bW_0)\left\{\eye_r\otimes\begin{bmatrix*}
\zero & \\ & \eye_{p - r}
\end{bmatrix*}\right\}(\bU_0\transpose\otimes\bW_0\transpose)\\
&\quad\quad +\lambda_{0r}^2
(\bU_{0\perp}\otimes\bW_0)\left\{\eye_{p - r}\otimes\begin{bmatrix*}
\eye_r & \\ & \zero
\end{bmatrix*}\right\}(\bU_{0\perp}\transpose\otimes \bW_0\transpose)\\
&\quad = \lambda_{0r}^2(\bU_0\bU_0\transpose\otimes \bU_{0\perp}\bU_{0\perp}\transpose + 
\bU_{0\perp}\bU_{0\perp}\transpose\otimes\bU_0\bU_0\transpose)\\
&\quad = \lambda_{0r}^2\{\bU_0\bU_0\transpose\otimes (\eye_p - \bU_{0}\bU_{0}\transpose) + (\eye_p - \bU_0\bU_0\transpose)\otimes\bU_0\bU_0\transpose\}.
\end{align*}
For the matrix $\bSigma_0^2\otimes (\eye_p - \bU_0\bU_0\transpose) + (\eye_p - \bU_0\bU_0\transpose)\otimes\bSigma_0^2$, we write
\begin{align*}
&\bSigma_0^2\otimes (\eye_p - \bU_0\bU_0\transpose) + (\eye_p - \bU_0\bU_0\transpose)\otimes\bSigma_0^2\\
&\quad = \bW_0\widetilde{\bM}_0\bW_0\transpose\otimes \bW_0\begin{bmatrix*}
\zero_r & \\ & \eye_{p - r}
\end{bmatrix*}\bW_0\transpose + 
\bW_0\begin{bmatrix*}
\zero_r & \\ & \eye_{p - r}
\end{bmatrix*}\bW_0\transpose\otimes \bW_0\widetilde{\bM}_0\bW_0\transpose\\
&\quad = (\bW_0\otimes\bW_0)\left\{
\widetilde{\bM}_0^2\otimes \begin{bmatrix*}
\zero_r & \\ & \eye_{p - r}
\end{bmatrix*} + \begin{bmatrix*}
\zero_r & \\ & \eye_{p - r}
\end{bmatrix*}\otimes \widetilde{\bM}_0^2
\right\}(\bW_0\transpose\otimes\bW_0\transpose)\\
&\quad = (\bW_0\otimes\bW_0)(\widetilde{\bV}_0\otimes\widetilde{\bV}_0)
\left\{
\widetilde{\bLambda}_0^2\otimes \begin{bmatrix*}
\zero_r & \\ & \eye_{p - r}
\end{bmatrix*} + \begin{bmatrix*}
\zero_r & \\ & \eye_{p - r}
\end{bmatrix*}\otimes \widetilde{\bLambda}_0^2
\right\}\\
&\quad\quad\times (\widetilde{\bV}_0\otimes\widetilde{\bV}_0)\transpose(\bW_0\transpose\otimes\bW_0\transpose)\\
&\quad = 
(\bW_0\otimes\bW_0)(\widetilde{\bV}_0\otimes\widetilde{\bV}_0)
% \left\{
\begin{bmatrix*}
\bLambda_0^2\otimes \begin{bmatrix*}
\zero_r & \\ & \eye_{p - r}
\end{bmatrix*} & \\
 & \eye_{p - r}\otimes \begin{bmatrix*}
\bLambda_0^2 & \\ & \zero_{(p - r)\times(p - r)}
\end{bmatrix*}
\end{bmatrix*}\\
&\quad\quad\times
(\widetilde{\bV}_0\otimes\widetilde{\bV}_0)\transpose(\bW_0\transpose\otimes\bW_0\transpose)\\
&\quad\succeq
\lambda_{0r}^2
(\bW_0\otimes\bW_0)(\widetilde{\bV}_0\otimes\widetilde{\bV}_0)
% \left\{
\begin{bmatrix*}
\eye_r\otimes \begin{bmatrix*}
\zero_r & \\ & \eye_{p - r}
\end{bmatrix*} & \\
 & \eye_{p - r}\otimes \begin{bmatrix*}
\eye_r & \\ & \zero_{(p - r)\times(p - r)}
\end{bmatrix*}
\end{bmatrix*}\\
&\quad\quad\times
(\widetilde{\bV}_0\otimes\widetilde{\bV}_0)\transpose(\bW_0\transpose\otimes\bW_0\transpose)
\end{align*}
Now we focus on the matrix on the right-hand side of the previous display. Write
\begin{align*}
&(\widetilde{\bV}_0\otimes\widetilde{\bV}_0)
\begin{bmatrix*}
\eye_r\otimes\begin{bmatrix*}
\zero_{r\times r} & \\ & \eye_{p - r}
\end{bmatrix*} & \\
& \eye_{p - r}\otimes\begin{bmatrix*}
\eye_r & \\
& \zero_{(p - r)\times(p - r)}
\end{bmatrix*}
\end{bmatrix*}
(\widetilde{\bV}_0\transpose\otimes\widetilde{\bV}_0\transpose)\\
&\quad
= \begin{bmatrix*}
\bV_0\otimes\widetilde{\bV}_0 & \\ & \eye_{p - r}\otimes\widetilde{\bV}_0
\end{bmatrix*}
\begin{bmatrix*}
\eye_r\otimes\begin{bmatrix*}
\zero & \\ & \eye_{p - r}
\end{bmatrix*} & \\
& \eye_{p - r}\otimes\begin{bmatrix*}
\eye_r & \\
& \zero
\end{bmatrix*}
\end{bmatrix*}
\begin{bmatrix*}
\bV_0\transpose\otimes\widetilde{\bV}_0\transpose & \\ & \eye_{p - r}\otimes\widetilde{\bV}_0\transpose
\end{bmatrix*}\\
&\quad
= \begin{bmatrix*}
\bV_0\bV_0\transpose\otimes\begin{bmatrix*}
\bV_0 & \\ & \eye_{p - r}
\end{bmatrix*}\begin{bmatrix*}
\zero & \\ & \eye_{p - r}
\end{bmatrix*}\begin{bmatrix*}
\bV_0\transpose & \\ & \eye_{p - r}
\end{bmatrix*} & \\
& \eye_{p - r}\otimes\begin{bmatrix*}
\bV_0 & \\ & \eye_{p - r}
\end{bmatrix*}\begin{bmatrix*}
\eye_r & \\
& \zero
\end{bmatrix*}
\begin{bmatrix*}
\bV_0\transpose & \\ & \eye_{p - r}
\end{bmatrix*}
\end{bmatrix*}\\
&\quad = \begin{bmatrix*}
\eye_r\otimes\begin{bmatrix*}
\zero_{r\times r} & \\ & \eye_{p - r}
\end{bmatrix*} & \\
& \eye_{p - r}\otimes\begin{bmatrix*}
\eye_r & \\
& \zero_{(p - r)\times(p - r)}
\end{bmatrix*}
\end{bmatrix*}.
\end{align*}
Therefore,
\begin{align*}
&\bSigma_0^2\otimes(\eye_p - \bU_0\bU_0\transpose) + (\eye_p - \bU_0\bU_0\transpose)\otimes \bSigma_0^2\\
&\quad \succeq \lambda_{0r}^2(\bW_0\otimes\bW_0)(\widetilde{\bV}_0\otimes\widetilde{\bV}_0)
\begin{bmatrix*}
\eye_r\otimes\begin{bmatrix*}
\zero_{r\times r} & \\ & \eye_{p - r}
\end{bmatrix*} & \\
& \eye_{p - r}\otimes\begin{bmatrix*}
\eye_r & \\
& \zero_{(p - r)\times(p - r)}
\end{bmatrix*}
\end{bmatrix*}\\
&\quad\quad\times
(\widetilde{\bV}_0\transpose\otimes\widetilde{\bV}_0\transpose)
(\bW_0\transpose\otimes\bW_0\transpose)\\
&\quad = \lambda_{0r}^2(\bW_0\otimes\bW_0)
\begin{bmatrix*}
\eye_r\otimes\begin{bmatrix*}
\zero_{r\times r} & \\ & \eye_{p - r}
\end{bmatrix*} & \\
& \eye_{p - r}\otimes\begin{bmatrix*}
\eye_r & \\
& \zero_{(p - r)\times(p - r)}
\end{bmatrix*}
\end{bmatrix*}
(\bW_0\transpose\otimes\bW_0\transpose)\\
&\quad = \lambda_{0r}^2\begin{bmatrix*}
\bU_0\otimes \bW_0 & \bU_{0\perp}\otimes\bW_0
\end{bmatrix*}
\begin{bmatrix*}
\eye_r\otimes\begin{bmatrix*}
\zero_{r\times r} & \\ & \eye_{p - r}
\end{bmatrix*} & \\
& \eye_{p - r}\otimes\begin{bmatrix*}
\eye_r & \\
& \zero_{(p - r)\times(p - r)}
\end{bmatrix*}
\end{bmatrix*}\\
&\quad\quad\times
\begin{bmatrix*}
\bU_0\transpose\otimes \bW_0\transpose\\
\bU_{0\perp}\transpose\otimes \bW_0\transpose
\end{bmatrix*}\\
&\quad = \lambda_{0r}^2(\bU_0\otimes\bW_0)\left\{\eye_r\otimes\begin{bmatrix*}
\zero & \\ & \eye_{p - r}
\end{bmatrix*}\right\}(\bU_0\transpose\otimes\bW_0\transpose)\\
&\quad\quad
+
\lambda_{0r}^2(\bU_{0\perp}\otimes\bW_0)\left\{\eye_{p - r}\otimes\begin{bmatrix*}
\eye_r & \\ & \zero
\end{bmatrix*}\right\}(\bU_{0\perp}\transpose\otimes \bW_0\transpose)\\
&\quad = \lambda_{0r}^2(\bU_0\bU_0\transpose\otimes \bU_{0\perp}\bU_{0\perp}\transpose) + 
\lambda_{0r}^2(\bU_{0\perp}\bU_{0\perp}\transpose\otimes\bU_0\bU_0\transpose)\\
&\quad = \lambda_{0r}^2\{\bU_0\bU_0\transpose\otimes (\eye_p - \bU_{0}\bU_{0}\transpose) + 
(\eye_p - \bU_{0}\bU_{0}\transpose)\otimes\bU_0\bU_0\transpose\},
\end{align*}
and the proof is thus completed. 
\end{proof}

\begin{proof}[Proof of Lemma \ref{lemma:Cmatrix_singular_value}]
Let $\bU_{0\perp}$ be the orthogonal complement of $\bU_0$ such that $[\bU_0, \bU_{0\perp}]\in\mathbb{O}(p)$, and one can therefore take $\bW_0 = \bC_0^{-\mathrm{T}}\bC_0$. For any $\bA\in\mathbb{R}^{(p - r)\times r}$ and $\bvarphi = \vect(\bA)$,
\begin{align*}
&\vect(\bX_\bvarphi)\transpose
(\bC_0\transpose\otimes\bC_0\transpose)\{\bU_0\bU_0\transpose\otimes (\eye_p - \bU_0\bU_0\transpose)\}(\bC_0\otimes\bC_0)
\vect(\bX_\bvarphi)\\
&\quad = \vect(\bX_\bvarphi)\transpose
(\bC_0\transpose\otimes\bC_0\transpose)(\bU_0\otimes\bU_{0\perp})(\bU_0\transpose\otimes\bU_{0\perp}\transpose)(\bC_0\otimes\bC_0)
\vect(\bX_\bvarphi)\\
&\quad = \|(\bU_0\transpose\otimes\bU_{0\perp}\transpose)(\bC_0\otimes\bC_0)
\vect(\bX_\bvarphi)\|_2^2\\
&\quad = \left\|
\begin{bmatrix*}
\zero_{r\times r} & \eye_{p - r}
\end{bmatrix*}\bC_0\transpose\bX_\bvarphi\bC_0\begin{bmatrix*}
\eye_r\\\zero_{(p - r)\times r}
\end{bmatrix*}
\right\|_{\mathrm{F}}^2\\
&\quad = 
\left\|
\begin{bmatrix*}
\zero_{r\times r} & \eye_{p - r}
\end{bmatrix*}
\begin{bmatrix*}
\bC_{11}\transpose & \bC_{21}\transpose \\
\bC_{12}\transpose & \bC_{22}\transpose
\end{bmatrix*}
\begin{bmatrix*}
\zero_{r\times r} & -\bA\transpose \\ \bA & \zero_{(p - r)\times (p - r)}
\end{bmatrix*}
\begin{bmatrix*}
\bC_{11} & \bC_{12} \\ \bC_{21} & \bC_{22}
\end{bmatrix*}
\begin{bmatrix*}
\eye_r \\ \zero_{(p - r)\times r}
\end{bmatrix*}
\right\|_{\mathrm{F}}^2\\
&\quad = \|\bC_{22}\transpose\bA\bC_{11} - \bC_{12}\transpose\bA\transpose\bC_{21}\|_{\mathrm{F}}^2,
\end{align*}
and similarly,
\begin{align*}
&\vect(\bX_\bvarphi)\transpose
(\bC_0\transpose\otimes\bC_0\transpose)\{(\eye_p - \bU_0\bU_0\transpose)\otimes\bU_0\bU_0\transpose\}(\bC_0\otimes\bC_0)
\vect(\bX_\bvarphi)\\
&\quad = \vect(\bX_\bvarphi)\transpose
(\bC_0\transpose\otimes\bC_0\transpose)(\bU_{0\perp}\otimes\bU_0)(\bU_{0\perp}\transpose\otimes\bU_0\transpose)(\bC_0\otimes\bC_0)
\vect(\bX_\bvarphi)\\
&\quad = \|(\bU_{0\perp}\transpose\otimes\bU_{0}\transpose)(\bC_0\otimes\bC_0)
\vect(\bX_\bvarphi)\|_2^2\\
&\quad = \left\|
\begin{bmatrix*}
\eye_r & \zero_{r\times(p - r)}
\end{bmatrix*}\bC_0\transpose\bX_\bvarphi\bC_0\begin{bmatrix*}
\zero_{r\times r} \\ \eye_{p - r}
\end{bmatrix*}
\right\|_{\mathrm{F}}^2\\
&\quad = 
\left\|
\begin{bmatrix*}
\eye_r & \zero_{r\times(p - r)}
\end{bmatrix*}
\begin{bmatrix*}
\bC_{11}\transpose & \bC_{21}\transpose \\
\bC_{12}\transpose & \bC_{22}\transpose
\end{bmatrix*}
\begin{bmatrix*}
\zero_{r\times r} & -\bA\transpose \\ \bA & \zero_{(p - r)\times (p - r)}
\end{bmatrix*}
\begin{bmatrix*}
\bC_{11} & \bC_{12} \\ \bC_{21} & \bC_{22}
\end{bmatrix*}
\begin{bmatrix*}
\zero_{r\times r} \\ \eye_{p - r}
\end{bmatrix*}
\right\|_{\mathrm{F}}^2\\
&\quad = \|\bC_{21}\transpose\bA\bC_{12} - \bC_{11}\transpose\bA\transpose\bC_{22}\|_{\mathrm{F}}^2\\
&\quad = \|\bC_{22}\transpose\bA\bC_{11} - \bC_{12}\transpose\bA\transpose\bC_{21}\|_{\mathrm{F}}^2,
\end{align*}
and the proof of assertion (i) is completed. In the rest of the proof, we focus on assertion (ii), which is slightly involved. We consider two scenarios separately, i.e., $p \geq 2r$ and $p < 2r$. 

\vspace*{2ex}
\noindent\textbf{Case I: $p - r \geq r$. }Let $\bA_0 = \bU_{\bA_0}\bS_{\bA_0}\bV_{\bA_0}\transpose$ be the singular value decomposition of ${\bA_0}$, where $\bU_{\bA_0}\in\mathbb{O}(p - r, r)$, $\bV_{\bA_0}\in\mathbb{O}(r)$, and $\bS_{\bA_0} = \mathrm{diag}\{\sigma_1(\bA_0),\ldots,\sigma_r(\bA_0)\}$. 
Suppose $\bU_{{\bA_0}\perp}\in\mathbb{O}(p - r, p - 2r)$ spans the orthogonal complement of $\mathrm{Span}(\bU_{\bA_0})$, i.e., $\bU_{\bA_0}\transpose\bU_{{\bA_0}\perp} = \zero_{r \times (p - 2r)}$. Note that by assumption, $0 < \sigma_r(\bA_0) \leq \sigma_1(\bA_0) < 1$. Using the property of commutation matrices \citep{magnus1979} that
\[
(\bC_{21}\transpose\otimes \bC_{12}\transpose)\bK_{(p - r)r} = 
\bK_{r(p - r)}(\bC_{12}\transpose\otimes \bC_{21}\transpose),\quad
\bK_{r(p - r)}(\bU_{\bA_0}\otimes\bV_{\bA_0}) = (\bV_{\bA_0}\otimes \bU_{\bA_0})\bK_{rr},
\]
we write
\begin{align*}
&\bC_{11}\otimes\bC_{22} - (\bC_{21}\transpose\otimes \bC_{12}\transpose)\bK_{(p - r)r}\\
&\quad = \bC_{11}\otimes\bC_{22} - \bK_{r(p - r)}(\bC_{12}\transpose\otimes \bC_{21}\transpose)\\
&\quad = \bV_{\bA_0}(\eye_r + \bS_{\bA_0}^2)^{-1}\bV_{\bA_0}\transpose\otimes \{\eye_{p - r} - \bU_{\bA_0}\bS_{\bA_0}(\eye_r + \bS_{\bA_0}^2)^{-1}\bS_{\bA_0}\bU_{\bA_0}\transpose\}\\
&\quad\quad + \bK_{r(p - r)}\{\bU_{\bA_0}\bS_{\bA_0}(\eye_r + \bS_{\bA_0}^2)^{-1}\bV_{\bA_0}\transpose\otimes \bV_{\bA_0}(\eye_r + \bS_{\bA_0}^2)^{-1}\bS_{\bA_0}\bU_{\bA_0}\transpose\}\\
&\quad = \bV_{\bA_0}(\eye_r + \bS_{\bA_0}^2)^{-1}\bV_{\bA_0}\transpose\otimes (\bU_{\bA_0}\bU_{\bA_0}\transpose + \bU_{{\bA_0}\perp}\bU_{{\bA_0}\perp}\transpose)\\
&\quad\quad - \bV_{\bA_0}(\eye_r + \bS_{\bA_0}^2)^{-1}\bV_{\bA_0}\transpose\otimes\{\bU_{\bA_0}\bS_{\bA_0}(\eye_r + \bS_{\bA_0}^2)^{-1}\bS_{\bA_0}\bU_{\bA_0}\transpose\}\\
&\quad\quad + \bK_{r(p - r)}(\bU_{\bA_0}\otimes \bV_{\bA_0})\{\bS_{\bA_0}(\eye_r + \bS_{\bA_0}^2)^{-1} \otimes (\eye_r + \bS_{\bA_0}^2)^{-1}\bS_{\bA_0}\}(\bV_{\bA_0}\transpose\otimes\bU_{\bA_0}\transpose)\\
&\quad = (\bV_{\bA_0}\otimes \bU_{\bA_0})[(\eye_r + \bS_{\bA_0}^2)^{-1}\otimes\{\eye_r - \bS_{\bA_0}(\eye_r + \bS_{\bA_0}^2)^{-1}\bS_{\bA_0}\}](\bV_{\bA_0}\transpose\otimes \bU_{\bA_0}\transpose)
% \{\bU_\bA\bU_\bA\transpose + \bU_{\bA\perp}\bU_{\bA\perp}\transpose - \bU_\bA\bS_\bA(\eye_r + \bS_\bA^2)^{-1}\bS_\bA\bU_\bA\transpose\}
\\
&\quad\quad + (\bV_{\bA_0}\otimes \bU_{{\bA_0}\perp})\{(\eye_r + \bS_{\bA_0}^2)^{-1}\otimes\eye_{p - r}\}(\bV_{\bA_0}\transpose\otimes\bU_{{\bA_0}\perp}\transpose)\\
&\quad\quad + (\bV_{\bA_0}\otimes \bU_{\bA_0})\bK_{rr}\{(\eye_r + \bS_{\bA_0}^2)^{-1}\bS_{\bA_0} \otimes \bS_{\bA_0}(\eye_r + \bS_{\bA_0}^2)^{-1}\}(\bV_{\bA_0}\transpose\otimes\bU_{\bA_0}\transpose)\\
&\quad
= (\bV_{\bA_0}\otimes\bU_{\bA_0})\widetilde{\bR}(\bV_{\bA_0}\transpose\otimes \bU_{\bA_0}\transpose)\\
&\quad\quad + 
(\bV_{\bA_0}\otimes \bU_{{\bA_0}\perp})\{(\eye_r + \bS_{\bA_0}^2)^{-1}\otimes\eye_{p - r}\}(\bV_{\bA_0}\transpose\otimes\bU_{{\bA_0}\perp}\transpose),
% &\quad\succeq
%  (\bV_\bA\otimes \bU_\bA)[(\eye_r + \bS_\bA^2)^{-1}\otimes\{\eye_r - \bS_\bA(\eye_r + \bS_\bA^2)^{-1}\bS_\bA\}](\bV_\bA\transpose\otimes \bU_\bA\transpose)\\
% &\quad\quad + \bK_{r(p - r)}(\bV_\bA\otimes \bU_\bA)\{(\eye_r + \bS_\bA^2)^{-1}\bS_\bA \otimes \bS_\bA(\eye_r + \bS_\bA^2)\}(\bU_\bA\transpose\otimes\bV_\bA\transpose)\\
% &
\end{align*}
where
\begin{align*}
\widetilde{\bR}
& := (\eye_r + \bS_{\bA_0}^2)^{-1}\otimes\{\eye_r - \bS_{\bA_0}(\eye_r + \bS_{\bA_0}^2)^{-1}\bS_{\bA_0}\}\\
&\quad + 
 \bK_{rr}\{(\eye_r + \bS_{\bA_0}^2)^{-1}\bS_{\bA_0} \otimes \bS_{\bA_0}(\eye_r + \bS_{\bA_0}^2)^{-1}\}\\
& = (\eye_r + \bS_{\bA_0}^2)^{-1}\otimes\{\eye_r - \bS_{\bA_0}(\eye_r + \bS_{\bA_0}^2)^{-1}\bS_{\bA_0}\}\\
&\quad - (\eye_r + \bS_{\bA_0}^2)^{-1}\bS_{\bA_0} \otimes \bS_{\bA_0}(\eye_r + \bS_{\bA_0}^2)^{-1}\\
&\quad + (\bK_{rr} + \eye_{r^2})\{(\eye_r + \bS_{\bA_0}^2)^{-1}\bS_{\bA_0} \otimes \bS_{\bA_0}(\eye_r + \bS_{\bA_0}^2)^{-1}\}\\
& = (\eye_r + \bS_{\bA_0}^2)^{-1}\otimes(\eye_r + \bS_{\bA_0}^2)^{-1} - (\eye_r + \bS_{\bA_0}^2)^{-1}\bS_{\bA_0} \otimes \bS_{\bA_0}(\eye_r + \bS_{\bA_0}^2)^{-1}\\
&\quad + (\bK_{rr} + \eye_{r^2})\{(\eye_r + \bS_{\bA_0}^2)^{-1}\bS_{\bA_0} \otimes \bS_{\bA_0}(\eye_r + \bS_{\bA_0}^2)^{-1}\}.
\end{align*}
Note that $[\bV_{\bA_0}\otimes\bU_{\bA_0}, \bV_{\bA_0}\otimes\bU_{\bA_0\perp}]\in\mathbb{O}(r(p - r))$. It follows that
\begin{align*}
&\sigma_{\min}\{\bC_{11}\transpose\otimes\bC_{22}\transpose - (\bC_{21}\transpose\otimes \bC_{12}\transpose)\bK_{(p - r)r}\}
% \\&\quad
 = \min\left[
\sigma_{\min}(\widetilde{\bR}), \sigma_{\min}\{(\eye_r + \bS_{\bA_0}^2)^{-1}\otimes\eye_{p - r}\}
\right].
\end{align*}
We now provide a lower bound for the smallest singular value of $\widetilde{\bR}$. When $r = 1$, $\bK_{rr} = \bK_{11} = 1$, $\bS_{\bA_0} = \|\bA_0\|_2\in\mathbb{R}^{1\times 1}$, and hence
\[
\widetilde{\bR} = \frac{1}{(1 + \bS_{\bA_0}^2)^2} + \frac{\bS_{\bA_0}^2}{(1 + \bS_{\bA_0}^2)^2} = \frac{1}{1 + \|\bA\|_2^2}.
\]
When $r > 1$, we consider the following approach. Denote the diagonal matrices
\[
\bD_1 = (\eye_r + \bS_{\bA_0}^2)^{-1},\quad
% \bD_2 = \eye_r - \bS_\bA(\eye_r + \bS_\bA^2)^{-1}\bS_\bA,\quad
\bD_2 = (\eye_r + \bS_{\bA_0}^2)^{-1}\bS_{\bA_0}.
\]
Note that by the property of the commutation matrix, 
\begin{align*}
&
(\bK_{rr} + \eye_{r^2})\{(\eye_r + \bS_{\bA_0}^2)^{-1}\bS_{\bA_0} \otimes \bS_{\bA_0}(\eye_r + \bS_\bA^2)^{-1}\}\\
&\quad
 = (\bK_{rr} + \eye_{r^2})(\bD_2\otimes\bD_2) = (\bD_2\otimes\bD_2)(\bK_{rr} + \eye_{r^2}).
\end{align*}
Therefore, $(\bK_{rr} + \eye_{r^2})(\bD_2\otimes\bD_2)$ is symmetric and is a product two positive semidefinite matrices $\bK_{rr} + \eye_{r^2}$ and $\bD_2\otimes\bD_2$ because $\bK_{rr}$ is symmetric and only has eigenvalues in $\{\pm1\}$ \citep{magnus1979}. Therefore, by Corollary 11 in \cite{1695991}, we have
\[
\lambda_{\min}\{(\bK_{rr} + \eye_{r^2})(\bD_2\otimes\bD_2)\}
\geq \lambda_{\min}(\bK_{rr} + \eye_{r^2})\lambda_{\min}(\bD_2\otimes\bD_2) = 0.
\]
This further implies that $(\bK_{rr} + \eye_{r^2})(\bD_2\otimes\bD_2)$ is positive semidefinite and $\widetilde{\bR}$ is symmetric. Note that $\bD_1\succeq \bD_2$ because $\sigma_k(\bA_0)\in[0, 1)$, and that $\widetilde{\bR}$ is positive semidefinite. Hence, to provide a lower bound for the smallest singular value of $\widetilde{\bR}$, it is sufficient to provide a lower bound for the smallest eigenvalue of 
\begin{align*}
&(\eye_r + \bS_{\bA_0}^2)^{-1}\otimes (\eye_r + \bS_{\bA_0}^2)^{-1} - (\eye_r + \bS_{\bA_0}^2)^{-1}\bS_{\bA_0} \otimes \bS_{\bA_0}(\eye_r + \bS_{\bA_0}^2)^{-1}\\
&\quad = \bD_1\otimes \bD_1 - \bD_2\otimes\bD_2.
\end{align*}
Write
\begin{align*}
&\bD_1\otimes \bD_1 - \bD_2\otimes\bD_2\\
&\quad = \mathrm{diag}\left\{\frac{1}{1 + \sigma_1^2({\bA_0})},\ldots,\frac{1}{1 + \sigma_r^2({\bA_0})}\right\}
\otimes \mathrm{diag}\left\{\frac{1}{1 + \sigma_1^2({\bA_0})},\ldots,\frac{1}{1 + \sigma_r^2({\bA_0})}\right\}\\
&\quad\quad - \mathrm{diag}\left\{\frac{\sigma_1({\bA_0})}{1 + \sigma_1^2({\bA_0})},\ldots,\frac{\sigma_r({\bA_0})}{1 + \sigma_r^2({\bA_0})}\right\}\otimes \mathrm{diag}\left\{\frac{\sigma_1({\bA_0})}{1 + \sigma_1^2({\bA_0})},\ldots,\frac{\sigma_r({\bA_0})}{1 + \sigma_r^2({\bA_0})}\right\}\\
&\quad = \begin{bmatrix*}
\frac{1}{1 + \sigma_1^2(\bA_0)} \bD_1 - \frac{\sigma_1(\bA_0)}{1 + \sigma_1^2(\bA_0)}\bD_2 &  & \\
 & \vdots & \\
 & & \frac{1}{1 + \sigma_r^2(\bA_0)}  \bD_1 - \frac{\sigma_r(\bA_0)}{1 + \sigma_r^2(\bA_0)}\bD_2
\end{bmatrix*}\succeq \zero_{r^2\times r^2}.
\end{align*}
It follows that
\begin{align*}
\sigma_{\min}(\widetilde\bR)
& = \sigma_{\min}\left\{
\bD_1\otimes \bD_1- \bD_2\otimes \bD_2 + (\bK_{rr} + \eye_{r^2})(\bD_2\otimes\bD_2)
\right\}\\
& = \lambda_{\min}\left\{
\bD_1\otimes \bD_1- \bD_2\otimes \bD_2 + (\bK_{rr} + \eye_{r^2})(\bD_2\otimes\bD_2)
\right\}\\
&\geq \lambda_{\min}\left\{
\bD_1\otimes \bD_1- \bD_2\otimes \bD_2)
\right\}\\
& = \min_{k\in [r]}\lambda_{\min}\left\{\frac{1}{1 + \sigma_k^2(\bA_0)}\bD_1 - \frac{\sigma_k(\bA_0)}{1 + \sigma_k^2(\bA_0)}\bD_2\right\}\\
& = \min_{k,l\in [r]}\frac{1 - \sigma_k(\bA_0)\sigma_l(\bA_0)}{\{1 + \sigma_k^2(\bA_0)\}\{1 + \sigma_l^2(\bA_0)\}}.
\end{align*}
We finally conclude that when $r > 1$, 
\begin{align*}
&\sigma_{\min}\{\bC_{11}\otimes \bC_{22} - (\bC_{21}\transpose\otimes\bC_{12}\transpose)\bK_{(p - r)r}\}\\
&\quad = \lambda_{\min}\{\bC_{11}\otimes \bC_{22} - (\bC_{21}\transpose\otimes\bC_{12}\transpose)\bK_{(p - r)r}\}\\
&\quad = \min\left\{
\sigma_{\min}(\widetilde\bR), \sigma_{\min}(\bD_1\otimes \eye_{p - r})
\right\}\\
&\quad\geq \min\left[
\min_{k,l\in [r]}\frac{1 - \sigma_k(\bA_0)\sigma_l(\bA_0)}{\{1 + \sigma_k^2(\bA_0)\}\{1 + \sigma_l^2(\bA_0)\}},
\min_{k\in [r]}\frac{1}{1 + \sigma_k^2(\bA_0)}
\right]\\
&\quad = \frac{1 - \|\bA_0\|_2^2}{(1 + \|\bA_0\|_2^2)^2},
\end{align*}
and when $r = 1$, we directly obtain
\begin{align*}
\sigma_{\min}\{\bC_{11}\otimes \bC_{22} - (\bC_{21}\transpose\otimes\bC_{12}\transpose)\bK_{(p - r)r}\}
% \\
% &\quad = \lambda_{\min}\{\bC_{11}\transpose\otimes \bC_{22}\transpose - (\bC_{21}\transpose\otimes\bC_{12}\transpose)\bK_{(p - r)r}\}\\
% &\quad
% & = \min\left\{
% \sigma_{\min}(\widetilde\bR), \sigma_{\min}(\bD_1\otimes \eye_{p - r})
% \right\}
% \\
% &\quad\geq \min\left[
% \min_{k,l\in [r]}\frac{1 - \sigma_k(\bA_0)\sigma_l(\bA_0)}{\{1 + \sigma_k^2(\bA_0)\}\{1 + \sigma_l^2(\bA_0)\}},
% \min_{k\in [r]}\frac{1}{1 + \sigma_k^2(\bA_0)}
% \right]\\
% &\quad
& = \frac{1}{1 + \|\bA_0\|_2^2}.
\end{align*}

\vspace*{2ex}
\noindent\textbf{Case II: $p - r < r$. }This situation occurs only if $r > 1$. Let $\bA_0 = \bU_{\bA_0}\bS_{\bA_0}\bV_{\bA_0}\transpose$ be the singular value decomposition of ${\bA_0}$, where $\bU_{\bA_0}\in\mathbb{O}(p - r)$, $\bV_{\bA_0}\in\mathbb{O}(r, p - r)$, and $\bS_{\bA_0} = \mathrm{diag}\{\sigma_1(\bA_0),\ldots,\sigma_{p - r}(\bA_0)\}$. Suppose $\bV_{{\bA_0}\perp}\in\mathbb{O}(r, 2r - p)$ spans the orthogonal complement of $\mathrm{Span}(\bV_{\bA_0})$, i.e., $\bV_{\bA_0}\transpose\bV_{{\bA_0}\perp} = \zero_{p - r \times (2r - p)}$. Note that by assumption, $0 < \sigma_r(\bA_0) \leq \sigma_1(\bA_0) < 1$. Observe that
\begin{align*}
\bC_{11} & =(\eye_r + \bA_0\transpose\bA_0)^{-1}
 = 
\begin{bmatrix*}
\bV_{\bA_0} & \bV_{\bA_0\perp}
\end{bmatrix*}\begin{bmatrix*}
(\eye_{p - r} + \bS_{\bA_0}^2)^{-1} & \\ & \eye_{2r - p}
\end{bmatrix*}\begin{bmatrix*}
\bV_{\bA_0}\transpose\\\bV_{\bA_0\perp}\transpose
\end{bmatrix*},
\end{align*}
and
\begin{align*}
\bC_{22} &= \eye_{p - r} - \bA_0(\eye_r + \bA_0\transpose\bA_0)^{-1}\bA_0\transpose\\
&
% \quad
 = \eye_{p - r} - \bU_{\bA_0}\bS_{\bA_0}\bV_{\bA_0}\transpose
\begin{bmatrix*}
\bV_{\bA_0} & \bV_{\bA_0\perp}
\end{bmatrix*}\begin{bmatrix*}
(\eye_{p - r} + \bS_{\bA_0}^2)^{-1} & \\ & \eye_{2r - p}
\end{bmatrix*}\begin{bmatrix*}
\bV_{\bA_0}\transpose\\\bV_{\bA_0\perp}\transpose
\end{bmatrix*}
\bV_{\bA_0}\bS_{\bA_0}\bU_{\bA_0}\transpose\\
&
% \quad
 = \eye_{p - r} - \bU_{\bA_0}\bS_{\bA_0}(\eye_{p - r} + \bS_{\bA_0}^2)^{-1}\bS_{\bA_0}\bU_{\bA_0}\transpose
 	   = \bU_{\bA_0}(\eye_{p - r} + \bS_{\bA_0}^2)^{-1}\bU_{\bA_0}\transpose.
\end{align*}
Also,
\begin{align*}
\bC_{12}\transpose & = -\bA_0(\eye_r + \bA_0\transpose\bA_0)^{-1} = -\bU_{\bA_0}\bS_{\bA_0}(\eye_r + \bS_{\bA_0}^2)^{-1}\bV_{\bA_0}\transpose,\\
\bC_{21}\transpose & = (\eye_r + \bA_0\transpose\bA_0)^{-1}\bA_0\transpose = \bV_{\bA_0}(\eye_r + \bS_{\bA_0}^2)^{-1}\bS_{\bA_0}\bU_{\bA_0}\transpose.
\end{align*}
Using the property of commutation matrices \citep{magnus1979} that
\begin{align*}
&(\bC_{21}\transpose\otimes \bC_{12}\transpose)\bK_{(p - r)r} = 
\bK_{r(p - r)}(\bC_{12}\transpose\otimes \bC_{21}\transpose),\\
&\bK_{r(p - r)}(\bU_{\bA_0}\otimes\bV_{\bA_0}) = (\bV_{\bA_0}\otimes \bU_{\bA_0})\bK_{(p - r)(p - r)},
\end{align*}
we write
\begin{align*}
&\bC_{11}\otimes\bC_{22} - (\bC_{21}\transpose\otimes \bC_{12}\transpose)\bK_{(p - r)r}\\
&\quad = \bC_{11}\otimes\bC_{22} - \bK_{r(p - r)}(\bC_{12}\transpose\otimes \bC_{21}\transpose)\\
&\quad = 
% \bV_{\bA_0}(\eye_r + \bS_{\bA_0}^2)^{-1}\bV_{\bA_0}\transpose
\begin{bmatrix*}
\bV_{\bA_0} & \bV_{\bA_0\perp}
\end{bmatrix*}\begin{bmatrix*}
(\eye_{p - r} + \bS_{\bA_0}^2)^{-1} & \\ & \eye_{2r - p}
\end{bmatrix*}\begin{bmatrix*}
\bV_{\bA_0}\transpose\\\bV_{\bA_0\perp}\transpose
\end{bmatrix*}
\otimes 
\bU_{\bA_0}(\eye_{p - r} + \bS_{\bA_0}^2)^{-1}\bU_{\bA_0}\transpose
% \{\eye_{p - r} - \bU_{\bA_0}\bS_{\bA_0}(\eye_r + \bS_{\bA_0}^2)^{-1}\bS_{\bA_0}\bU_{\bA_0}\transpose\}
\\
&\quad\quad + \bK_{r(p - r)}\{\bU_{\bA_0}\bS_{\bA_0}(\eye_{p - r} + \bS_{\bA_0}^2)^{-1}\bV_{\bA_0}\transpose\otimes \bV_{\bA_0}(\eye_{p - r} + \bS_{\bA_0}^2)^{-1}\bS_{\bA_0}\bU_{\bA_0}\transpose\}\\
&\quad = \{\bV_{\bA_0}(\eye_{p - r} + \bS_{\bA_0}^2)^{-1}\bV_{\bA_0}\transpose + \bV_{\bA_0\perp}\bV_{\bA_0\perp}\transpose\}\otimes 
\{\bU_{\bA_0}(\eye_{p - r} + \bS_{\bA_0}^2)^{-1}\bU_{\bA_0}\transpose\}\\
&\quad\quad + \bK_{r(p - r)}(\bU_{\bA_0}\otimes \bV_{\bA_0})\{\bS_{\bA_0}(\eye_{p - r} + \bS_{\bA_0}^2)^{-1} \otimes (\eye_{p - r} + \bS_{\bA_0}^2)^{-1}\bS_{\bA_0}\}\\
&\quad\quad\quad\times(\bV_{\bA_0}\transpose\otimes\bU_{\bA_0}\transpose)\\
&\quad = (\bV_{\bA_0}\otimes \bU_{\bA_0})\{(\eye_{p - r} + \bS_{\bA_0}^2)^{-1}\otimes(\eye_{p - r} + \bS_{\bA_0}^2)^{-1}\}(\bV_{\bA_0}\transpose\otimes \bU_{\bA_0}\transpose)
% \{\bU_\bA\bU_\bA\transpose + \bU_{\bA\perp}\bU_{\bA\perp}\transpose - \bU_\bA\bS_\bA(\eye_r + \bS_\bA^2)^{-1}\bS_\bA\bU_\bA\transpose\}
\\
&\quad\quad + (\bV_{\bA_0\perp}\otimes \bU_{{\bA_0}})\{\eye_{2r - p}\otimes(\eye_{p - r} + \bS_{\bA_0}^2)^{-1}\}(\bV_{\bA_0\perp}\transpose\otimes\bU_{{\bA_0}}\transpose)\\
&\quad\quad + (\bV_{\bA_0}\otimes \bU_{\bA_0})\bK_{(p - r)(p - r)}\{(\eye_{p - r} + \bS_{\bA_0}^2)^{-1}\bS_{\bA_0} \otimes \bS_{\bA_0}(\eye_{p - r} + \bS_{\bA_0}^2)^{-1}\}\\
&\quad\quad\quad\times (\bV_{\bA_0}\transpose\otimes\bU_{\bA_0}\transpose)\\
&\quad
= (\bV_{\bA_0}\otimes\bU_{\bA_0})\widetilde{\bR}(\bV_{\bA_0}\transpose\otimes \bU_{\bA_0}\transpose)\\
&\quad\quad + 
(\bV_{\bA_0\perp}\otimes \bU_{{\bA_0}})\{\eye_{2r - p}\otimes(\eye_{p - r} + \bS_{\bA_0}^2)^{-1}\}(\bV_{\bA_0\perp}\transpose\otimes\bU_{{\bA_0}}\transpose),
% &\quad\succeq
%  (\bV_\bA\otimes \bU_\bA)[(\eye_r + \bS_\bA^2)^{-1}\otimes\{\eye_r - \bS_\bA(\eye_r + \bS_\bA^2)^{-1}\bS_\bA\}](\bV_\bA\transpose\otimes \bU_\bA\transpose)\\
% &\quad\quad + \bK_{r(p - r)}(\bV_\bA\otimes \bU_\bA)\{(\eye_r + \bS_\bA^2)^{-1}\bS_\bA \otimes \bS_\bA(\eye_r + \bS_\bA^2)\}(\bU_\bA\transpose\otimes\bV_\bA\transpose)\\
% &
\end{align*}
where
\begin{align*}
\widetilde{\bR}
& = (\eye_{p - r} + \bS_{\bA_0}^2)^{-1}\otimes(\eye_{p - r} + \bS_{\bA_0}^2)^{-1}\\
&\quad + 
 \bK_{(p - r)(p - r)}\{(\eye_{p - r} + \bS_{\bA_0}^2)^{-1}\bS_{\bA_0} \otimes \bS_{\bA_0}(\eye_{p - r} + \bS_{\bA_0}^2)^{-1}\}\\
& = (\eye_{p - r} + \bS_{\bA_0}^2)^{-1}\otimes (\eye_{p - r} + \bS_{\bA_0}^2)^{-1} - (\eye_{p - r} + \bS_{\bA_0}^2)^{-1}\bS_{\bA_0} \otimes \bS_{\bA_0}(\eye_{p - r} + \bS_{\bA_0}^2)^{-1}\\
&\quad + (\bK_{(p - r)(p - r)} + \eye_{(p - r)^2})\{(\eye_{p - r} + \bS_{\bA_0}^2)^{-1}\bS_{\bA_0} \otimes \bS_{\bA_0}(\eye_{p - r} + \bS_{\bA_0}^2)^{-1}\}
% \\
% & = (\eye_r + \bS_{\bA_0}^2)^{-1}\otimes(\eye_r + \bS_{\bA_0}^2)^{-1} - (\eye_r + \bS_{\bA_0}^2)^{-1}\bS_{\bA_0} \otimes \bS_{\bA_0}(\eye_r + \bS_{\bA_0}^2)^{-1}\\
% &\quad + (\bK_{rr} + \eye_{r^2})\{(\eye_r + \bS_{\bA_0}^2)^{-1}\bS_{\bA_0} \otimes \bS_{\bA_0}(\eye_r + \bS_{\bA_0}^2)^{-1}\}
.
\end{align*}
Note that $[\bV_{\bA_0}\otimes\bU_{\bA_0}, \bV_{\bA_0\perp}\otimes\bU_{\bA_0}]\in\mathbb{O}(r(p - r))$. It follows that
\begin{align*}
&\sigma_{\min}\{\bC_{11}\otimes\bC_{22} - (\bC_{21}\transpose\otimes \bC_{12}\transpose)\bK_{(p - r)r}\}
% \\&\quad
 = 
\sigma_{\min}(\widetilde{\bR})\wedge \sigma_{\min}\{\eye_{2r - p}\otimes(\eye_{p - r} + \bS_{\bA_0}^2)^{-1}\}.
\end{align*}
Similar to the case where $p - r\geq r$, we also provide a lower bound for the smallest singular value of $\widetilde{\bR}$. Denote the diagonal matrices
\[
\bD_1 = (\eye_{p - r} + \bS_{\bA_0}^2)^{-1},\quad
% \bD_2 = \eye_r - \bS_\bA(\eye_r + \bS_\bA^2)^{-1}\bS_\bA,\quad
\bD_2 = (\eye_{p - r} + \bS_{\bA_0}^2)^{-1}\bS_{\bA_0}.
\]
Note that by the property of the commutation matrix, 
\begin{align*}
&
(\bK_{(p - r)(p - r)} + \eye_{(p - r)^2})\{(\eye_{p - r} + \bS_{\bA_0}^2)^{-1}\bS_{\bA_0} \otimes \bS_{\bA_0}(\eye_{p - r} + \bS_{\bA_0}^2)^{-1}\}
\\&\quad
 = (\bK_{(p - r)(p - r)} + \eye_{(p - r)^2})(\bD_2\otimes\bD_2) = (\bD_2\otimes\bD_2)(\bK_{(p - r)(p - r)} + \eye_{(p - r)^2}).
\end{align*}
Therefore, $(\bK_{(p - r)(p - r)} + \eye_{(p - r)^2})(\bD_2\otimes\bD_2)$ is symmetric and is a product two positive semidefinite matrices $\bK_{(p - r)(p - r)} + \eye_{(p - r)^2}$ and $\bD_2\otimes\bD_2$% because $\bK_{(p - r)(p - r)}$ is symmetric and only has eigenvalues in $\{\pm1\}$ \citep{magnus1979}
. Again, by Corollary 11 in \cite{1695991}, we have
\[
\lambda_{\min}\{(\bK_{(p - r)(p - r)} + \eye_{(p - r)^2})(\bD_2\otimes\bD_2)\}
\geq \lambda_{\min}(\bK_{(p - r)(p - r)} + \eye_{(p - r)^2})\lambda_{\min}(\bD_2\otimes\bD_2) = 0.
\]
This further implies that $(\bK_{(p - r)(p - r)} + \eye_{(p - r)^2})(\bD_2\otimes\bD_2)$ is positive semidefinite and $\widetilde{\bR}$ is symmetric. Note that $\bD_1\succeq \bD_2$ because $\sigma_k(\bA_0)\in[0, 1)$, and that $\widetilde{\bR}$ is positive semidefinite. Hence, to provide a lower bound for the smallest singular value of $\widetilde{\bR}$, it is sufficient to provide a lower bound for the smallest eigenvalue of 
\begin{align*}
&(\eye_{p - r} + \bS_{\bA_0}^2)^{-1}\otimes (\eye_{p - r} + \bS_{\bA_0}^2)^{-1} - (\eye_{p - r} + \bS_{\bA_0}^2)^{-1}\bS_{\bA_0} \otimes \bS_{\bA_0}(\eye_{p - r} + \bS_{\bA_0}^2)^{-1}\\
&\quad = \bD_1\otimes \bD_1 - \bD_2\otimes\bD_2.
\end{align*}
Write
\begin{align*}
&\bD_1\otimes \bD_1 - \bD_2\otimes\bD_2\\
&\quad = \mathrm{diag}\left\{\frac{1}{1 + \sigma_1^2({\bA_0})},\ldots,\frac{1}{1 + \sigma_{p - r}^2({\bA_0})}\right\}
\otimes \mathrm{diag}\left\{\frac{1}{1 + \sigma_1^2({\bA_0})},\ldots,\frac{1}{1 + \sigma_{p - r}^2({\bA_0})}\right\}\\
&\quad\quad - \mathrm{diag}\left\{\frac{\sigma_1({\bA_0})}{1 + \sigma_1^2({\bA_0})},\ldots,\frac{\sigma_{p - r}({\bA_0})}{1 + \sigma_{p - r}^2({\bA_0})}\right\}\otimes \mathrm{diag}\left\{\frac{\sigma_1({\bA_0})}{1 + \sigma_1^2({\bA_0})},\ldots,\frac{\sigma_{p - r}({\bA_0})}{1 + \sigma_{p - r}^2({\bA_0})}\right\}\\
&\quad = \begin{bmatrix*}
\frac{1}{1 + \sigma_1^2(\bA_0)} \bD_1 - \frac{\sigma_1(\bA_0)}{1 + \sigma_1^2(\bA_0)}\bD_2 &  & \\
 & \vdots & \\
 & & \frac{1}{1 + \sigma_{p - r}^2(\bA_0)}  \bD_1 - \frac{\sigma_{p - r}(\bA_0)}{1 + \sigma_{p - r}^2(\bA_0)}\bD_2
\end{bmatrix*}\succeq \zero_{{p - r}^2\times {p - r}b^2}.
\end{align*}
It follows that
\begin{align*}
\sigma_{\min}(\widetilde\bR)
& = \sigma_{\min}\left\{
\bD_1\otimes \bD_1- \bD_2\otimes \bD_2 + (\bK_{(p - r)(p - r)} + \eye_{(p - r)^2})(\bD_2\otimes\bD_2)
\right\}\\
& = \lambda_{\min}\left\{
\bD_1\otimes \bD_1- \bD_2\otimes \bD_2 + (\bK_{(p - r)(p - r)} + \eye_{(p - r)^2})(\bD_2\otimes\bD_2)
\right\}\\
&\geq \lambda_{\min}\left\{
\bD_1\otimes \bD_1- \bD_2\otimes \bD_2)
\right\}\\
& = \min_{k\in [p - r]}\lambda_{\min}\left\{\frac{1}{1 + \sigma_k^2(\bA_0)}\bD_1 - \frac{\sigma_k(\bA_0)}{1 + \sigma_k^2(\bA_0)}\bD_2\right\}\\
& = \min_{k,l\in [p - r]}\frac{1 - \sigma_k(\bA_0)\sigma_l(\bA_0)}{\{1 + \sigma_k^2(\bA_0)\}\{1 + \sigma_l^2(\bA_0)\}}.
\end{align*}
We finally conclude that
\begin{align*}
&\sigma_{\min}\{\bC_{11}\otimes \bC_{22} - (\bC_{21}\transpose\otimes\bC_{12}\transpose)\bK_{(p - r)r}\}\\
&\quad = \lambda_{\min}\{\bC_{11}\otimes \bC_{22} - (\bC_{21}\transpose\otimes\bC_{12}\transpose)\bK_{(p - r)r}\}\\
&\quad = \min\left\{
\sigma_{\min}(\widetilde\bR), \sigma_{\min}(\eye_{2r - p}\otimes\bD_1)
\right\}\\
&\quad\geq \min\left[
\min_{k,l\in [p - r]}\frac{1 - \sigma_k(\bA_0)\sigma_l(\bA_0)}{\{1 + \sigma_k^2(\bA_0)\}\{1 + \sigma_l^2(\bA_0)\}},
\min_{k\in [p - r]}\frac{1}{1 + \sigma_k^2(\bA_0)}
\right]\\
&\quad = \frac{1 - \|\bA_0\|_2^2}{(1 + \|\bA_0\|_2^2)^2}.
\end{align*}
The proof is thus completed. 
\end{proof}

% section proofs_of_technical_lemmas_in_Proofs (end)

\section{Proofs for Section \ref{sub:extension_to_general_rectangular_matrices}} % (fold)
\label{sec:proofs_for_section_sub:extension_to_general_rectangular_matrices}

\begin{proof}[Proof of Theorem \ref{thm:CT_deviation_Sigma_rectangular}]
First observe that simple algebra leads to the following matrix decomposition
\begin{align*}
\bSigma - \bSigma_0
& = \bSigma(\btheta) - \bSigma(\btheta_{0})
% \\&
 = (\bM - \bM_0)\bU_{0}\transpose + \bM_{0}\{\bU(\bvarphi) - \bU_{0}\}\transpose + \bR_{\bSigma}(\btheta, \btheta_{0}),
\end{align*}
where the remainder
% \begin{align*}
$\bR_\bSigma(\btheta, \btheta_{0}) = (\bM - \bM_0)\{\bU(\bvarphi) - \bU_{0}\}\transpose$
satisfies 
\begin{align*}
\|R_\bSigma(\btheta, \btheta_{0})\|_{\mathrm{F}}
&\leq \|\bM - \bM_0\|_{\mathrm{F}}\|\bU(\bvarphi) - \bU_0\|_{\mathrm{F}}\leq \frac{1}{2}\|\bmu - \bmu_0\|_2^2 + \frac{1}{2}\|\bU(\bvarphi) - \bU_0\|_{\mathrm{F}}^2.
\end{align*}
By Theorem \ref{thm:second_order_deviation_CT},
\begin{align*}
\|\bU(\bvarphi) - \bU_0\|_{\mathrm{F}}
& \leq 2\sqrt{2}\|\bvarphi - \bvarphi_0\|_2
 % + C\|\bvarphi - \bvarphi_0\|_2^2
 . 
\end{align*}
% Note that for $\bvarphi$ such that $\|\bvarphi - \bvarphi_0\|_2\leq 1/(2\sqrt{2})$. 
% \begin{align*}
% \|D\bU(\bvarphi_0)\|_2 &\leq 2\|(\eye_p - \bX_0)^{-\mathrm{T}}\|_2\|(\eye_p - \bX_0)^{-1}\|_2\|\bGamma_\bvarphi\|_2\\
% &\leq 2\sqrt{2}\|(\eye_p - \bX_0)^{-1}\|_2^2\leq 2\sqrt{2}. 
% \end{align*}
Hence, 
$\|\bR_\bSigma(\btheta, \btheta_0)\|_{\mathrm{F}}\leq 4\|\btheta - \btheta_0\|_2^2$% when $\|\btheta - \btheta_0\|_2\leq 1/(2\sqrt{2})$
. Furthermore, using Theorem \ref{thm:second_order_deviation_CT} again, we obtain
\begin{align*}
\vect\{\bU(\bvarphi) - \bU_0\} = 
D\bU(\bvarphi_0)(\bvarphi - \bvarphi_0) + \vect\{\bR_\bU(\bvarphi, \bvarphi_0)\},
\end{align*}
where $\|\bR_\bU(\bvarphi, \bvarphi_0)\|_{\mathrm{F}}\leq 8\|\bvarphi - \bvarphi_0\|_2^2$
 for all $\bvarphi, \bvarphi_0$. In matrix form, we have
\begin{align*}
\bU(\bvarphi) - \bU_0 = 2(\eye_p - \bX_0)^{-1}(\bX_\bvarphi - \bX_0)(\eye_p - \bX_0)^{-1}\eye_{p\times r} + \bR_\bU(\bvarphi, \bvarphi_0).
\end{align*}
Hence we finally obtain
\begin{align*}
\vect\{\bSigma(\btheta) - \bSigma_0\}
& = D_\bmu\bSigma(\btheta_0)(\bmu - \bmu_0) +  \bK_{p_2p_1}(\bM_0\otimes \eye_{p_2})\vect\{\bU(\bvarphi )- \bU_0\}\\
&\quad + \vect\{\bR_\bSigma(\btheta, \btheta_0)\}\\
& = D\bSigma(\btheta_0)(\btheta - \btheta_0) +  \bK_{p_2p_1}(\bM_0\otimes\eye_{p_2})\vect\{\bR_\bU(\bvarphi, \bvarphi_0)\}\\
&\quad + \vect\{\bR_\bSigma(\btheta, \btheta_0)\}\\
& = D\bSigma(\btheta_0)(\btheta - \btheta_0) + \vect\{\bR(\btheta, \btheta_0)\},
\end{align*}
where
\begin{align*}
\|\bR(\btheta, \btheta_0)\|_{\mathrm{F}}
&\leq \|\bM_0\|_2\|\bR_\bU(\bvarphi, \bvarphi_0)\|_{\mathrm{F}} + \|\bR_\bSigma(\btheta, \btheta_0)\|_{\mathrm{F}}\leq (4 + 8\|\bM_0\|_2)\|\btheta - \btheta_0\|_2^2.
\end{align*}
\end{proof}

\begin{proof}[Proof of Theorem \ref{thm:DSigma_singular_value_rectangular}]
Let $\btheta_0 = [\bvarphi_0\transpose, \bmu_0\transpose]\transpose$, where $\bvarphi_0 = \vect(\bA_0)$ for some $\bA_0\in\mathbb{R}^{(p_2 - r)\times r}$ and $\bmu_0 = \vect(\bM_0)$ for some $\bM_0\in\mathbb{R}^{p_1\times r}$. Let $\bU_0 := \bU(\bvarphi_0)$. Then
\begin{align*}
&D\bSigma(\btheta_0)\transpose D\bSigma(\btheta_0)\\
&\quad = \begin{bmatrix*}
\bJ_1 & \bJ_2\\
\bJ_2\transpose & \bJ_3
\end{bmatrix*}
: = \begin{bmatrix*}
D_\bvarphi\bSigma(\btheta_0)\transpose D_\bvarphi\bSigma(\btheta_0) & D_\bvarphi\bSigma(\btheta_0)\transpose D_\bmu\bSigma(\btheta_0)\\
D_\bmu\bSigma(\btheta_0)\transpose D_\bvarphi\bSigma(\btheta_0) & D_\bmu\bSigma(\btheta_0)\transpose D_\bmu\bSigma(\btheta_0)
\end{bmatrix*}\\
&\quad = \begin{bmatrix*}
D\bU(\bvarphi)\transpose (\bM_0\transpose\bM_0\otimes \eye_{p_2}) D\bU(\bvarphi) & 
D\bU(\bvarphi)\transpose (\bM_0\transpose\otimes\eye_{p_2})\bK_{p_1p_2}(\bU_0\otimes\eye_{p_1})
\\
(\bU_0\transpose\otimes\eye_{p_1})\bK_{p_2p_1}(\bM_0\otimes\eye_{p_2})D\bU(\bvarphi) & 
\eye_r\otimes \eye_{p_1}
\end{bmatrix*}.
\end{align*}
The Schur complement of $\bJ_3$ of the entire matrix $D\bSigma(\btheta_0)\transpose D\bSigma(\btheta_0)$ is given by
\begin{align*}
&\bJ_1 - \bJ_2\bJ_3^{-1}\bJ_2\transpose\\
&\quad = D\bU(\bvarphi_0)\transpose (\bM_0\transpose\bM_0\otimes \eye_{p_2}) D\bU(\bvarphi_0)\\
&\quad\quad - 
D\bU(\bvarphi_0)\transpose (\bM_0\transpose\otimes \eye_{p_2}) \bK_{p_1p_2} (\bU_0\bU_0\transpose\otimes\eye_{p_1}) \bK_{p_2p_1} (\bM_0\otimes \eye_{p_2}) D\bU(\bvarphi_0).
\end{align*}
Similar to the proof of Theorem \ref{thm:DSigma_singular_value}, we provide a lower bound for the smallest eigenvalue of $\bJ_1 - \bJ_2\bJ_3^{-1}\bJ_2\transpose$. 
Denote 
\begin{align*}
\bX_0 = \bX_{\bvarphi_0}
 = \begin{bmatrix*}
 \zero_{r\times r} & - \bA_0\transpose{}\\\bA_0 & \zero_{(p_2 - r)\times (p_2 - r)}
 \end{bmatrix*},\quad
\bC_0 = (\eye_{p_2} - \bX_{\bvarphi_0})^{-1},\quad \bW_0 = \bC_0^{-\mathrm{T}}\bC_0.
\end{align*}
For any nonzero vector $\bvarphi\in\mathbb{R}^{(p_2 - r)r}$, write
\begin{align*}
&\bvarphi\transpose(\bJ_1 - \bJ_2\bJ_3^{-1}\bJ_2\transpose)\bvarphi\\
&\quad = \bvarphi D\bU(\bvarphi)\transpose (\bM_0\transpose\bM_0\otimes \eye_{p_2}) D\bU(\bvarphi)\bvarphi\\
&\quad\quad - \{\bK_{p_2p_1}(\bM_0\otimes\eye_{p_2}) D\bU(\bvarphi_0)\bvarphi\}\transpose(\bU_0\bU_0\transpose\otimes\eye_{p_1})\{\bK_{p_2p_1}(\bM_0\otimes\eye_{p_2}) D\bU(\bvarphi_0)\bvarphi\}\\
&\quad = 4\{(\eye_{p_2\times r}\transpose\bC_0\transpose\otimes\bC_0)\vect(\bX_{\bvarphi})\}\transpose (\bM_0\transpose\bM_0\otimes \eye_{p_2}) \{(\eye_{p_2\times r}\transpose\bC_0\transpose\otimes\bC_0)\vect(\bX_{\bvarphi})\}\\
&\quad\quad - 4\{\bK_{p_2p_1}(\bM_0\otimes\eye_{p_2})(\eye_{p_2\times r}\transpose\bC_0\transpose\otimes\bC_0)\vect(\bX_{\bvarphi})\}\transpose (\bU_0\bU_0\transpose\otimes \eye_{p_1})\\
&\quad\quad\quad\times \{\bK_{p_2p_1}(\bM_0\otimes\eye_{p_2})(\eye_{p_2\times r}\transpose\bC_0\transpose\otimes\bC_0)\vect(\bX_{\bvarphi})\}\\
&\quad = 4\{\vect(\bC_0\bX_{\bvarphi}\bC_0\eye_{p_2\times r})\}\transpose (\bM_0\transpose\bM_0\otimes \eye_{p_2}) \{\vect(\bC_0\bX_{\bvarphi}\bC_0\eye_{p_2\times r})\}\\
&\quad\quad - 4\{\bK_{p_2p_1}(\bM_0\otimes\eye_{p_2})\vect(\bC_0\bX_{\bvarphi}\bC_0\eye_{p_2\times r})\}\transpose (\bU_0\bU_0\transpose\otimes \eye_{p_1})\\
&\quad\quad\quad\times \{\bK_{p_2p_1}(\bM_0\otimes\eye_{p_2})\vect(\bC_0\bX_{\bvarphi}\bC_0\eye_{p_2\times r})\}\\
&\quad = 4\{\vect(\bC_0\bX_{\bvarphi}\bC_0\eye_{p_2\times r})\}\transpose (\bM_0\transpose\bM_0\otimes \eye_{p_2}) \{\vect(\bC_0\bX_{\bvarphi}\bC_0\eye_{p_2\times r})\}\\
&\quad\quad - 4\{\vect(\bC_0\bX_{\bvarphi}\bC_0\eye_{p_2\times r})\}\transpose (\bM_0\transpose\otimes\eye_{p_2})\bK_{p_1p_2}(\bU_0\bU_0\transpose\otimes \eye_{p_1})\bK_{p_2p_1}(\bM_0\otimes \eye_{p_2})
\\
&\quad\quad\quad\times 
\{\vect(\bC_0\bX_{\bvarphi}\bC_0\eye_{p_2\times r})\}\\
&\quad = 4\{\vect(\bC_0\bX_{\bvarphi}\bC_0\eye_{p_2\times r})\}\transpose (\bM_0\transpose\bM_0\otimes \eye_{p_2}) \{\vect(\bC_0\bX_{\bvarphi}\bC_0\eye_{p_2\times r})\}\\
&\quad\quad - 4\{\vect(\bC_0\bX_{\bvarphi}\bC_0\eye_{p_2\times r})\}\transpose (\bM_0\transpose\bM_0\otimes\bU_0\bU_0\transpose)
% \\&\quad\quad\times 
\{\vect(\bC_0\bX_{\bvarphi}\bC_0\eye_{p_2\times r})\}\\
&\quad = 4\{\vect(\bC_0\bX_{\bvarphi}\bC_0\eye_{p_2\times r})\}\transpose \{\bM_0\transpose\bM_0\otimes (\eye_{p_2} - \bU_0\bU_0\transpose)\} \{\vect(\bC_0\bX_{\bvarphi}\bC_0\eye_{p_2\times r})\}\\
&\quad = 4\vect(\bX_{\bvarphi})\transpose(\bC_0\eye_{p_2\times r}\otimes \bC_0\transpose)\{\bM_0\transpose\bM_0\otimes (\eye_{p_2} - \bU_0\bU_0\transpose)\}(\eye_{p_2\times r}\transpose\bC_0\transpose\otimes\bC_0)\vect(\bX_\bvarphi)\\
&\quad = 4\vect(\bX_{\bvarphi})(\bC_0\otimes\bC_0\transpose)\{\eye_{p_2\times r}\bM_0\transpose\bM_0\eye_{p_2\times r}\transpose\otimes (\eye_{p_2} - \bU_0\bU_0\transpose)\}(\bC_0\transpose\otimes\bC_0)\vect(\bX_{\bvarphi})\\
&\quad = 4\vect(\bX_{\bvarphi})(\bC_0\transpose\otimes\bC_0\transpose)\{\bU_0\bM_0\transpose\bM_0\bU_0\transpose\otimes (\eye_{p_2} - \bU_0\bU_0\transpose)\}(\bC_0\otimes\bC_0)\vect(\bX_{\bvarphi}).
\end{align*}
Let $\bM_0\transpose\bM_0 = \bV\bS\bV\transpose$ be the spectral decomposition of $\bM_0\transpose\bM_0$, where $\bV\in\mathbb{O}(r)$, and 
% $\bS$ is the diagonal matrix of the eigenvalues of $\bM_0\transpose\bM_0$, 
$\bS = \mathrm{diag}\{\sigma_1(\bM_0)^2,\ldots,\sigma_r(\bM_0)^2\}$. Denote 
\[
\widetilde{\bV} = \begin{bmatrix*}
\bV & \\ & \eye_{p_2 - r}
\end{bmatrix*}.
\]
Then we compute
\begin{align*}
&\bU_0\bM_0\transpose\bM_0\bU_0\transpose\otimes (\eye_{p_2} - \bU_0\bU_0\transpose)\\
&\quad = \bW_0\widetilde{\bV}\begin{bmatrix*}
\bS & \\ & \zero_{(p_2 - r)\times (p_2 - r)}
\end{bmatrix*}\widetilde{\bV}\transpose\bW_0\transpose \otimes \bW_0\widetilde{\bV}\begin{bmatrix*}
\zero_{r\times r} & \\ & \eye_{p_2 - r}
\end{bmatrix*}\widetilde{\bV}\transpose\bW_0\transpose\\
&\quad = (\bW_0\otimes\bW_0)(\widetilde{\bV}\otimes\widetilde{\bV})
\begin{bmatrix*}
\bS\otimes\begin{bmatrix*}
\zero_{r\times r} & \\ & \eye_{p_2 - r}
\end{bmatrix*}
 & \\ & \zero_{(p_2 - r)^2\times (p_2 - r)^2}
\end{bmatrix*}
(\widetilde{\bV}\otimes\widetilde{\bV})\transpose(\bW_0\otimes\bW_0)\\
&\quad \succeq \sigma_{r}^2(\bM_0)(\bW_0\otimes\bW_0)
\begin{bmatrix*}
\bV\otimes\widetilde{\bV} & \\ & \eye_{p_2 - r}\otimes\widetilde{\bV}
\end{bmatrix*}
\begin{bmatrix*}
\eye_r\otimes\begin{bmatrix*}
\zero_{r\times r} & \\ & \eye_{p_2 - r}
\end{bmatrix*}
 & \\ & \zero_{(p_2 - r)^2\times (p_2 - r)^2}
\end{bmatrix*}\\
&\quad\quad\times
\begin{bmatrix*}
\bV\transpose\otimes\widetilde{\bV}\transpose & \\ & \eye_{p_2 - r}\otimes\widetilde{\bV}\transpose
\end{bmatrix*}(\bW_0\transpose\otimes\bW_0\transpose)\\
&\quad = \sigma_r^2(\bM_0)(\bW_0\otimes\bW_0)\begin{bmatrix*}
\eye_r & \\ & \zero_{(p_2 - r)\times(p_2 - r)}
\end{bmatrix*}\otimes \begin{bmatrix*}
\zero_{r\times r} & \\ & \eye_{p_2 - r}
\end{bmatrix*}
(\bW_0\transpose\otimes\bW_0\transpose)\\
&\quad = \sigma_r^2(\bM_0)\{\bU_0\bU_0\otimes (\eye_{p_2} - \bU_0\bU_0\transpose)\}.
\end{align*}
Therefore, the quadratic form $\bvarphi\transpose(\bJ_1 - \bJ_2\bJ_3^{-1}\bJ_2\transpose)\bvarphi$ can be lower bounded:
\begin{align*}
&\bvarphi\transpose(\bJ_1 - \bJ_2\bJ_3^{-1}\bJ_2\transpose)\bvarphi\\
&\quad \geq 4\sigma_r^2(\bM_0)\vect(\bX_\bvarphi)\transpose(\bC_0\transpose\otimes \bC_0\transpose)\{\bU_0\bU_0\transpose\otimes(\eye_{p_2} - \bU_0\bU_0\transpose)\}(\bC_0\otimes\bC_0)\vect(\bX_\bvarphi)\\
&\quad = 4\sigma_r^2(\bM_0)\|(\bU_0\transpose\bC_0 \otimes \bU_{0\perp}\transpose\bC_0)\vect(\bX_\bvarphi)\|_2^2\\
&\quad = 4\sigma_r^2(\bM_0)\left\|\left\{(\eye_{p_2\times r}\transpose\bC_0\transpose\bC_0^{-1}\bC_0)\otimes \begin{bmatrix*}
\zero_{(p_2 - r)\times r} & \eye_{p_2 - r}
\end{bmatrix*}\bC_0\transpose\bC_0^{-1}\bC_0\right\}\vect(\bX_\bvarphi)\right\|_2^2\\
&\quad = 4\sigma_r^2(\bM_0)\left\|\left\{(\eye_{p_2\times r}\transpose\bC_0\transpose)\otimes \begin{bmatrix*}
\zero_{(p_2 - r)\times r} & \eye_{p_2 - r}
\end{bmatrix*}\bC_0\transpose\right\}\vect(\bX_\bvarphi)\right\|_2^2,
\end{align*}
where $\bU_{0\perp}$ is the orthogonal complement of $\bU_0$, \emph{i.e.}, $\bU_{0\perp} = \bC_0^{-\mathrm{T}}\bC_0\bTheta_2\transpose$, and $\bTheta_2 = [\zero_{(p_2 - r)\times r},\eye_{p_2 - r}]$. Write $\bC_0$ in the block matrix form
\[
\bC_0 = \begin{bmatrix*}
\bC_{11} & \bC_{12} \\ \bC_{21} & \bC_{22}
\end{bmatrix*}
= \begin{bmatrix*}
(\eye_r + \bA_0\transpose\bA_0)^{-1} & - (\eye_r + \bA_0\transpose\bA_0)^{-1}\bA_0\transpose\\
\bA_0(\eye_r + \bA_0\transpose\bA_0)^{-1} & \eye_{p_2 - r} - \bA_0(\eye_r + \bA_0\transpose\bA_0)^{-1}\bA_0\transpose
\end{bmatrix*}.
\]
It follows that
\begin{align*}
\bvarphi\transpose(\bJ_1 - \bJ_2\bJ_3^{-1}\bJ_2\transpose)\bvarphi
& \geq 4\sigma_r^2(\bM_0)
\left\|\left\{\begin{bmatrix*}
\bC_{11}\transpose & \bC_{21}\transpose
\end{bmatrix*}\otimes \begin{bmatrix*}
\bC_{12}\transpose & \bC_{22}\transpose
\end{bmatrix*}\right\} \vect(\bX_\bvarphi)\right\|_2^2\\
& = 4\sigma_r^2(\bM_0)
\left\|
\begin{bmatrix*}
\bC_{12}\transpose & \bC_{22}\transpose
\end{bmatrix*}\begin{bmatrix*}
\zero_{r\times r} & -\bA\transpose \\
\bA & \zero_{(p_2 - r)\times(p_2 - r)}
\end{bmatrix*}\begin{bmatrix*}
\bC_{11} \\ \bC_{21}
\end{bmatrix*}
\right\|_{\mathrm{F}}^2\\
& = 4\sigma_r^2(\bM_0)\|\bC_{22}\transpose\bA\bC_{11} - \bC_{12}\transpose\bA\transpose\bC_{21}\|\\
& = 4\sigma_r^2(\bM_0) \|\{(\bC_{11}\otimes \bC_{22}) - (\bC_{21}\transpose\otimes \bC_{12}\transpose)\bK_{(p_2 - r)r}\}\vect(\bA)\|_{\mathrm{F}}^2.
\end{align*}
We finally invoke Lemma \ref{lemma:Cmatrix_singular_value} to conclude that
\begin{align*}
\bvarphi\transpose(\bJ_1 - \bJ_2\bJ_3^{-1}\bJ_2\transpose)\bvarphi
&\geq \left\{
\begin{aligned}
&\frac{4\sigma_r^2(\bM_0)(1 - \|\bA_0\|_2^2)^2\|\bvarphi\|_2^2}{(1 + \|\bA_0\|_2^2)^4},&\quad\text{if }r \geq 2,\\
&\frac{4\sigma_r^2(\bM_0)\|\bvarphi\|_2^2}{(1 + \|\bA_0\|_2^2)^2},&\quad\text{if }r = 1,
\end{aligned}
\right.
\end{align*}
which further implies that
\begin{align*}
\|(\bJ_1 - \bJ_2\bJ_3^{-1}\bJ_2\transpose)^{-1}\|_2\leq\left\{
\begin{aligned}
&\frac{(1 + \|\bA_0\|_2^2)^4}{4\sigma_r^2(\bM_0)(1 - \|\bA_0\|_2^2)^2},&\quad\text{if }r \geq 2,\\
&\frac{(1 + \|\bA_0\|_2^2)^2}{4\sigma_r^2(\bM_0)},&\quad\text{if }r = 1.
\end{aligned}
\right.
\end{align*}
This shows that $\bJ_1 - \bJ_2\bJ_3^{-1}\bJ_2\transpose$ is invertible. Since $\bJ_3$ is also invertible, the property of the Schur complement immediately implies that $D\bSigma(\btheta_0)\transpose D\bSigma(\btheta_0)$ is invertible. Furthermore, by the block matrix inversion formula,
\begin{align*}
&\{D\bSigma(\btheta_0)\transpose D\bSigma(\btheta_0)\}^{-1}\\
% & = \begin{bmatrix*}
% \bJ_1 & \bJ_2\\
% \bJ_2\transpose & \bJ_3
% \end{bmatrix*}^{-1}
&\quad = \begin{bmatrix*}
(\bJ_1 - \bJ_2\bJ_3^{-1}\bJ_2\transpose)^{-1} & -(\bJ_1 - \bJ_2\bJ_3^{-1}\bJ_2\transpose)^{-1}\bJ_2\bJ_3^{-1}\\
-\bJ_3^{-1}\bJ_2\transpose(\bJ_1 - \bJ_2\bJ_3^{-1}\bJ_2\transpose)^{-1} & \bJ_3^{-1} + \bJ_3^{-1}\bJ_2\transpose(\bJ_1 - \bJ_2\bJ_3^{-1}\bJ_2\transpose)^{-1}\bJ_2\bJ_3^{-1}
\end{bmatrix*}.
\end{align*}
By construction, $\|\bJ_3^{-1}\|_2 = 1$ and
\begin{align*}
\|\bJ_2\|_2 &\leq \|D\bU(\bvarphi_0)\transpose(\bM_0\transpose\otimes \eye_p) \bK_{p_2p_1}(\bU_0\otimes\eye_{p_1})\|\leq \|D\bU(\bvarphi_0)\|_2\|\bM_0\|_2\\
&\leq 2\sqrt{2}\|\bM_0\|_2\|\eye_{p\times r}(\eye_p - \bX_{\bvarphi_0})^{-1}\|_2\|(\eye_p - \bX_{\bvarphi_0})^{-1}\|_2\leq 2\sqrt{2}\|\bM_0\|_2.
\end{align*}
Thus, by Lemma 3.4 of \cite{10.1112/blms/bds080}, we see that
\begin{align*}
\|\{D\bSigma(\btheta_0)\transpose D\bSigma(\btheta_0)\}^{-1}\|_2
&\leq \|(\bJ_1 - \bJ_2\bJ_3^{-1}\bJ_2\transpose)^{-1}\|_2\\
&\quad + \|\bJ_3^{-1} + \bJ_3^{-1}\bJ_2\transpose(\bJ_1 - \bJ_2\bJ_3^{-1}\bJ_2\transpose)^{-1}\bJ_2\bJ_3^{-1}\|_2\\
&\leq \|(\bJ_1 - \bJ_2\bJ_3^{-1}\bJ_2\transpose)^{-1}\|_2\\
&\quad + \|\bJ_3^{-1}\|_2 + \|\bJ_3^{-1}\|_2^2\|\bJ_2\|_2^2\|(\bJ_1 - \bJ_2\bJ_3^{-1}\bJ_2\transpose)^{-1}\|_2\\
&\leq 1 + (1 + 8\|\bM_0\|_2^2)\|(\bJ_1 - \bJ_2\bJ_3^{-1}\bJ_2\transpose)^{-1}\|_2.
\end{align*}
The proof is completed by combining the upper bound for $\|(\bJ_1 - \bJ_2\bJ_3^{-1}\bJ_2\transpose)^{-1}\|_2$. 
\end{proof}

% section proofs_for_section_sub:extension_to_general_rectangular_matrices (end)

\section{Proofs for Section \ref{sub:bayesian_sparse_pca_and_non_intrinsic_loss}} % (fold)
\label{sec:proofs_for_section_sub:bayesian_sparse_pca_and_non_intrinsic_loss}

In this section we prove the main result of Section \ref{sub:bayesian_sparse_pca_and_non_intrinsic_loss}, namely, Theorem \ref{thm:contraction_spectral_norm}. The proof is lengthy and is partitioned into several subsections. The sketch of the proof can be loosely summarized as the following steps:
\begin{enumerate}
	\item Prior concentration (Section \ref{sub:prior_concentration}). We provide an lower bound for the prior probability that $\bOmega(\btheta)$ is inside a small neighborhood of $\bOmega_0$, i.e., $\Pi_\btheta\{\|\bOmega(\btheta) - \bOmega_0\|_{\mathrm{F}} \leq \eta_n\}$, where $(\eta_n)_{n = 1}^\infty$ is a sequence converging to $0$.
	\item Posterior consistency (Section \ref{sub:posterior_sparsity}). We prove that with posterior probability going to $1$, the intrinsic dimension cannot be too large, namely, $\expect_0\{\Pi_\btheta(\btheta:|S_\bA|\leq \kappa_0 s\mid\bY_n)\} = o(1)$ for some constant $\kappa_0 > 0$. 
	\item Construction of certain test functions (Section \ref{sub:construction_of_test_functions}). This step is needed in order to obtain the rate-optimal posterior contraction under the Frobenius norm following the general framework of \cite{ghosal2007convergence}. 
	\item Posterior contraction under the Frobenius norm (Section \ref{sub:posterior_contraction_under_frobenius_norm}). This is immediate once the previous steps are completed, but also serves as an intermediate step to the posterior contraction under the spectral norm. 
	\item Local asymptotic normality (Section \ref{sub:local_asymptotic_normality}). We expand the log-likelihood function locally at $\btheta_0$ under the posterior sparsity restriction via a Taylor expansion argument, which can be viewed as a variant of the local asymptotic normality (see, e.g., Chapter 7 in \citealp{van2000asymptotic}).
	\item Distributional approximation (Section \ref{sub:distributional_approximation}). Leveraging the local asymptotic normality result established in Section \ref{sub:local_asymptotic_normality}, we prove Theorem \ref{thm:BvM}, i.e., the asymptotic characterization of the shape of the posterior distribution $\Pi_\btheta(\btheta\in\cdot\mid\bY_n)$ using a random mixture of normal distributions. 
	\item Posterior contraction under the spectral norm (Section \ref{sub:posterior_contraction_under_spectral_norm}). Finally, we prove the rate-optimal posterior contraction of $\Span(\bU)$ under the spectral sine-theta distance using the asymptotic distributional approximation result obtained in Theorem \ref{thm:BvM}. 
\end{enumerate}
Now denote, $S_0 = \mathrm{supp}(\bA_0)$, $s_0 = |S_0|$, and
\[
\gamma(|S|) = \int_{\|\bA_S\|_2 < 1}\exp(-2\|\vect(\bA_S)\|_1)\mathrm{d}\bA_S.
\]
We begin the proof with the following upper and lower bounds for $\gamma(|S|)$:
\begin{align}\label{eqn:gamma_S_upper_bound}
\gamma(|S|) &\leq \int_{\mathbb{R}^{|S|\times r}}\exp(-2\|\vect(\bA_S)\|_1)\mathrm{d}\bA_S = 1,\\
\label{eqn:gamma_S_lower_bound}
\gamma(|S|) &\geq \exp\left\{-\frac{1}{2}r|S|\log(r|S|) - (2 - \log 2)r|S|\right\},\quad |S| \geq 1,
\end{align}
where the lower bound can be derived as follows:
\begin{align*}
\gamma(|S|) 
& \geq \int_{\|\bA_S\|_{\mathrm{F}} < 1}\exp(-2\|\vect(\bA_S)\|_1)\mathrm{d}\bA_S\\
&\geq\int_{\|\vect(\bA_S)\|_\infty\leq (r|S|)^{-1/2}}\exp(-2\|\vect(\bA_S)\|_1)\mathrm{d}\bA_S\\
& = \left\{\int_{-(r|S|)^{-1/2}}^{(r|S|)^{-1/2}}e^{-2x}\mathrm{d}x\right\}^{r|S|}
\geq \left\{\frac{2e^{-2}}{(r|S|)^{1/2}}\right\}^{r|S|}\\
& = \exp\left\{-\frac{1}{2}r|S|\log(r|S|) - (2 - \log 2)r|S|\right\}.
\end{align*}
Also, observe that for sufficiently large $n$, $z_n\in [1/2, 2]$. 

\subsection{Prior concentration} % (fold)
\label{sub:prior_concentration}
This subsection focuses on proving the following lemma that describes the prior concentration behavior of $\Pi_\btheta(\cdot)$:
\begin{lemma}\label{lemma:prior_concentration}
Under the prior specification and setup in Section \ref{sub:bayesian_sparse_pca_and_non_intrinsic_loss}, if $(\eta_n)_{n = 1}^\infty$ is a sequence such that $\eta_n/\|\bOmega_0\|_2\to 0$ and $n\eta_n^2\to\infty$, then
\[
\Pi_\btheta\{\|\bOmega(\btheta) - \bOmega_0\|_{\mathrm{F}} < \eta_n\}\geq
\exp\left(-C_0rs_0\log n + Cs_0\log p\right)
\]
for some constant $C_0 = C(\|\bOmega_0\|_2) > 0$ that only depends on $\|\bOmega_0\|_2$, and some absolute constant $C > 0$. 
\end{lemma}
Before proving Lemma \ref{lemma:prior_concentration}, we need the following auxiliary lemma from \cite{pati2014posterior}.
\begin{lemma}[Lemma 9.1 in the Supplement of \citealp{pati2014posterior}]
\label{lemma:evidence_lower_bound}
Let $(\eta_n)_{n = 1}^\infty$ be a sequence converging to $0$ with $n\eta_n^2\to\infty$. Then there exist a constant $C > 0$ and a sequence of events $(\Xi_n)_{n = 1}^\infty$ with $\prob_0(\Xi_n) \to 1$ such that over the event $\Xi_n$, 
\[
D_n\geq \exp\left\{-Cn\eta_n^2\log(2\|\bOmega_0\|_2)\right\}\Pi_\btheta\left\{\btheta:\|\bOmega(\btheta) - \bOmega_0\|_{\mathrm{F}} < \eta_n\right\}.
\]
\end{lemma}
\begin{proof}[Proof of Lemma \ref{lemma:prior_concentration}]
% For simplicity we shall use $C > 0$ to denote constants that may change from line to line but are universal and independent of $n$, the hyperparameter, and the spectra of the true covariance matrix $\bOmega_0$. 
Let $\bvarphi_0 = \vect(\bA_0) = \bU^{-1}(\bU_0)$ be such that $\bU(\bvarphi_0) = \bU_0$, where $\bU(\bvarphi)$ is the Cayley transform of $\bvarphi\in\mathbb{R}^d$, $\bmu_0$ is the vector formed by taking the upper diagonal entries of $\bM_0$, and $\btheta_0 = [\bvarphi_0\transpose, \bmu_0\transpose]\transpose$. 
First observe that by Theorem \ref{thm:second_order_deviation_CT}, for any $\bvarphi\in B_2(\bvarphi_0, \eps)$ with sufficiently small $\eps > 0$,
\[
\|\bU(\bvarphi) - \bU_0\|_{\mathrm{F}}\leq 2\|D\bU(\bvarphi_0)\|_{\mathrm{2}}\|\bvarphi - \bvarphi_0\|_2\leq 4\sqrt{2}\|\bvarphi - \bvarphi_0\|_2.
\]
Therefore, for $\bvarphi\in B_2(\bvarphi_0, \eta_n/[32\|\bOmega_0\|_2])$ and $\bmu\in B_2(\bmu_0, \eta_n/8)$ with sufficiently large $n$, we have
\[
\|\btheta - \btheta_0\|_2\leq \|\bvarphi - \bvarphi_0\|_2 + \|\bmu - \bmu_0\|_2< \frac{\eta_n}{32\|\bOmega_0\|_2} + \frac{\eta_n}{8}\leq \frac{\eta_n}{4},
\]
and then, for sufficiently large $n$, we obtain
\begin{align*}
\|\bOmega(\btheta) - \bOmega_0\|_{\mathrm{F}}
& = \|\bSigma(\btheta) - \bSigma_0\|_{\mathrm{F}}\\
&\leq 2\|\bM_0\|_2\|\bU(\bvarphi) - \bU_0\|_{\mathrm{F}} + \|\bM - \bM_0\|_{\mathrm{F}} + 2\|\bM - \bM_0\|_{\mathrm{F}}\|\bU(\bvarphi) - \bU_0\|_{\mathrm{F}}\\
&\quad + \|\bM_0\|_2\|\bU(\bvarphi) - \bU_0\|_{\mathrm{F}}^2\\
&\leq 8\sqrt{2}\|\bOmega_0\|_2\|\bvarphi - \bvarphi_0\|_2 + \sqrt{2}\|\bmu - \bmu_0\|_2 + 32\|\btheta - \btheta\|_2^2\\
&\quad + \|\bOmega_0\|_2\|\bU(\bvarphi) - \bU_0\|_{\mathrm{F}}^2\\
&\leq 16\|\bOmega_0\|_2\|\bA - \bA_0\|_2 + 2\|\bmu - \bmu_0\|_2 < \eta_n.
\end{align*}
Now we can estimate the prior mass $\Pi_\btheta\{\|\bOmega(\btheta) - \bOmega_0\|_{\mathrm{F}} < \eta_n\}$ from below:
\begin{align*}
\Pi_\btheta\{\|\bOmega(\btheta) - \bOmega_0\|_{\mathrm{F}} < \eta_n\}
% &\geq \Pi_\bmu\left(\|\bmu - \bmu_0\|_{\mathrm{2}} < \frac{\eta_n}{8}\right)\Pi_\bvarphi\left(\|\bvarphi - \bvarphi_0\|_2 < \frac{\eta_n}{32\|\bOmega_0\|_2}\right)\\
&\geq \Pi_\bmu\left(\|\bmu - \bmu_0\|_{\mathrm{2}} < \frac{\eta_n}{8}\right)\Pi_\bA\left(\|\bA - \bA_0\|_{\mathrm{F}} < \frac{\eta_n}{32\|\bOmega_0\|_2}\right),
\end{align*}
where $\Pi_\bmu(\mathrm{d}\bmu) = \pi_\bmu(\bmu)\mathrm{d}\bmu$. 
Denote $\Pi^{L}_\bmu(\mathrm{d}\bmu)$ the Laplace distribution on $\bmu$ given by
\[
\Pi^{L}_\bmu(\mathrm{d}\bmu) = \exp(-2\|\bmu\|_1)\mathrm{d}\bmu,\quad\bmu\in\mathbb{R}^{r(r + 1)/2}.
\]
Clearly, $\Pi_\bmu(\mathrm{d}\bmu)$ is the normalized restriction of $\Pi^L_\bmu(\mathrm{d}\bmu)$ on $\bM(\bmu)\in\mathbb{M}_+(r)$, where $\bmu=\vech\{\bM(\bmu)\}$. 
Now let $\bmu\in B_2(\bmu_0, \eta_n/8)$. Then for sufficiently large $n$,
\[
\|\bM(\bmu) - \bM(\bmu_0)\|_{\mathrm{F}}\leq \sqrt{2}\|\bmu - \bmu_0\|_2 < \frac{\sqrt{2}\eta_n}{8}\to 0.
\]
Since $\bM_0$ is already strictly positive definite with $\lambda_r(\bM_0)$ bounded away from $0$, it follows that $\bM(\bmu)$ is also positive definite. Now we proceed to provide a lower bound the first factor as follows for sufficiently large $n$:
\begin{align*}
\Pi_\bmu\left(\|\bmu - \bmu_0\|_{\mathrm{2}} < \frac{\eta_n}{8}\right)
& = \frac{\Pi^L_\bmu\{\|\bmu - \bmu_0\|_{\mathrm{2}} < \eta_n/8, \bM(\bmu)\in\mathbb{M}_+(r)\}}{\Pi^L_\bmu\{\bM(\bmu)\in\mathbb{M}_+(r)\}}\\
& = \frac{\Pi^L_\bmu(\|\bmu - \bmu_0\|_{\mathrm{2}} < \eta_n/8)}{\Pi^L_\bmu\{\bM(\bmu)\in\mathbb{M}_+(r)\}}\\
& \geq \Pi^L_\bmu\left(\|\bmu - \bmu_0\|_{\mathrm{2}} < \frac{\eta_n}{8}\right)\\
&\geq \vol\left\{B_2\left(\bmu_0, \frac{\eta_n}{8}\right)\right\}\exp\left\{-2\max_{\|\bmu - \bmu_0\|_2\leq 1}(\|\bmu_0\|_1 + \|\bmu - \bmu_0\|_1)\right\}\\
&\geq \vol\left\{B_2\left(\zero_{r(r + 1)/2}, 1\right)\right\}\left(\frac{\eta_n}{8}\right)^{r(r + 1)/2}\\
&\quad\times \exp\left\{-{2\sqrt{\frac{r(r + 1)}{2}}}(\|\bmu_0\|_2 + 1)\right\}\\
& \gtrsim\frac{1}{\sqrt{\pi r(r + 1)/2}}\left(\frac{\sqrt{2\pi e}\eta_n}{8\sqrt{r(r + 1)/2}}\right)^{r(r + 1)/2}\exp\left(-2r\|\bM_0\|_{\mathrm{F}}\right)\\
& \geq\exp\left( - Cr^2\left|\log\frac{\eta_n}{r}\right| - 2r^{3/2}\|\bOmega_0\|_2\right)\\
& \geq\exp\left( - Crs\left|\log\frac{\eta_n}{r}\right| - 2rs\|\bOmega_0\|_2\right).
\end{align*}
We now focus on the first factor. Note that for any row index $j\in [p]$, $j > r$, $[\bU_0]_{j*} = \zero_r$ if and only if $[\bA(\bU_0)]_{(j - r)*} = \zero_r$. 
Given $S$ drawn from $\pi_S(S)$, denote $\Pi_{\bA_S}^{L}(\mathrm{d}\bA_S)$ the Laplace distribution on $\vect(\bA_S)$, i.e.,
\[
\Pi_{\bA_S}^L(\mathrm{d}\bA_S) = \exp(-2\|\vect(\bA_S)\|_1)\mathrm{d}\bA_S.
\]
Clearly, $\Pi_{\bA_S}$ is the normalized restriction of $\Pi_{\bA_S}^L$ on $\{\bA_S\in \mathbb{R}^{|S|\times r}:\|\bA_S\|_2 < 1\}$. Furthermore, given $S = S_0$ drawn from $\pi_S(S)$, for any $\bA_{S_0}\in \{\|\bA_{S_0} - \bA_{0S_0}\|_{\mathrm{F}} < \eta_n/(32\|\bOmega_0\|_2)\}$ with $\eta_n\to 0$, we have
\[
\|\bA_{S_0}\|_2 \leq \|\bA_{0S_0}\|_2 + \|\bA_{S_0} - \bA_{0S_0}\|_{\mathrm{F}}\leq \sup_{n\geq 1}\|\bA_{0}\|_2 + o(1) < 1. 
\]
This implies that
\[
\left\{\bA_{S_0}:\|\bA_{S_0} - \bA_{0S_0}\|_{\mathrm{F}} < \frac{\eta_n}{32\|\bOmega_0\|_2}\right\}\subset\{\bA_{S_0}:\|\bA_{S_0}\|_2 < 1\}.
\]
Then for sufficiently large $n$, we provide the following lower bound the first factor by restricting $S$ to be $S_0$:
\begin{align*}
&\Pi_\bA\left(\|\bA - \bA_0\|_{\mathrm{F}} < \frac{\eta_n}{32\|\bOmega_0\|_2}\right)\\
&\quad\geq \Pi_S(S_0)\Pi_{\bA_{S_0}}\left(\|\bA_{S_0} - \bA_{0S_0}\|_{\mathrm{F}} < \frac{\eta_n}{32\|\bOmega_0\|_2}\right)\\
&\quad= \Pi_S(S_0)\frac{\Pi_{\bA_{S_0}}^L\left\{\|\bA_{S_0} - \bA_{0S_0}\|_{\mathrm{F}} < {\eta_n}/{(32\|\bOmega_0\|_2)}, \|\bA_{S_0}\|_2 < 1\right\}}
{\Pi_{\bA_{S_0}}^L\left(\bA_{S_0}:\|\bA_{S_0}\|_2 < 1\right)}\\
&\quad \geq \Pi_S(S_0){\Pi_{\bA_{S_0}}^L\left(\|\bA_{S_0} - \bA_{0S_0}\|_{\mathrm{F}} < \frac{\eta_n}{32\|\bOmega_0\|_2}\right)}\\
&\quad\geq \frac{\pi_p(s_0)}{{p - r\choose s_0}}\vol\left\{B_2(\zero_{s_0r}, 1)\right\}\left(\frac{\eta_n}{32\|\bOmega_0\|_2}\right)^{s_0r}\\
&\quad\quad\times \exp\left\{-2\max_{\|\bA_{S_0} - \bA_{0S_0}\|_{\mathrm{F}} < 1}(\|\vect(\bA_{0S_0})\|_1 + \|\vect(\bA_{S_0} - \bA_{0S_0})\|_1\right\}\\
&\quad\geq \frac{\pi_p(s_0)}{{p - r\choose s_0}}\frac{1}{2\sqrt{s_0r\pi}}\left(\frac{\sqrt{2\pi e}\eta_n}{32\|\bOmega_0\|_2\sqrt{s_0r}}\right)^{s_0r}
% \\&\quad\quad\times 
\exp\left\{-2\sqrt{s_0r}\left(\|\bA_0\|_{\mathrm{F}} + 1\right)\right\}\\
&\quad\geq \frac{\pi_p(s_0)}{{p - r\choose s_0}}\frac{1}{2\sqrt{s_0r\pi}}\left(\frac{\sqrt{2\pi e}\eta_n}{32\|\bOmega_0\|_2\sqrt{s_0r}}\right)^{s_0r}
% \\&\quad\quad\times 
\exp\left(-4\sqrt{s_0r^2}\right)\\
&\quad\geq \frac{\pi_p(s_0)}{{p - r\choose s_0}}\exp\left\{-Cs_0r\left|\log\frac{\eta_n}{\|\bOmega_0\|_2\sqrt{s_0r}}\right| - 4s_0r\right\}.
\end{align*}
Since for sufficiently large $n$, $z_n\geq 1/2$ and
\begin{align*}
\frac{\pi_p(s_0)}{{p - r\choose s_0}}&\geq \frac{1}{2}n^{-rs_0}(p - r)^{-as_0}\left(\frac{s_0}{p - r}\right)^{2s_0}
% \\&
\geq \frac{1}{2}n^{-rs_0}(p - r)^{-as_0}(p - r)^{-2s_0}\\
% &\geq \frac{1}{2}\exp\left(-rs_0\log n - s_0\log p\right)\\
% &\geq \frac{1}{2}\exp\left(-rs_0\log n - s_0\log p - s_0\log s_0\right)\\
&\geq \frac{1}{2}\exp(-rs_0\log n - cs_0\log p)
\end{align*}
for some constant $c > 0$, it follows that
\begin{align*}
&\Pi_\bA\left(\|\bA - \bA_0\|_{\mathrm{F}} < \frac{\eta_n}{32\|\bOmega_0\|_2}\right)\\
&\quad\geq\exp\left\{-Crs_0\left|\log\frac{\eta_n}{\|\bOmega_0\|_2\sqrt{s_0r}}\right| - C(rs_0\log n +s_0\log p)\right\}.
\end{align*}
Hence, using the fact that $n\eta_n^2 \to \infty$, we conclude that
\begin{align*}
\Pi_\btheta\left\{\|\bOmega(\btheta) - \bOmega_0\|_{\mathrm{F}} < \eta_n\right\}
\geq\exp\left\{-C(\|\bOmega_0\|_2)rs_0 - Cs_0\log p\right\}
% \geq\exp\left\{-\frac{(1 + 2\beta)\sqrt{2}}{\tau_{\max}}(r + \|\bSigma_0\|_2)s_0 - C(s_0 - r)\log(p - r) - Cs_0r|\log \eta_n| - Cs_0r|\log(\|\bSigma_0\|_2\tau_{\max}s_0r) \right\}
\end{align*}
for some constant $C(\|\bOmega_0\|_2) > 0$ that only depends on $\|\bOmega_0\|_2$. The proof is thus completed. 
\end{proof}

% subsection prior_concentration (end)

\subsection{Posterior sparsity} % (fold)
\label{sub:posterior_sparsity}
In this subsection, we aim at establishing Lemma \ref{lemma:posterior sparsity} regarding the posterior sparsity of $\bA$ given the observed data, which in turn depends on Lemma \ref{lemma:prior_sparsity} that characterizes the prior sparsity of $\bA$. 
\begin{lemma}\label{lemma:prior_sparsity}
Under the setup and the prior specification in Section \ref{sub:bayesian_sparse_pca_and_non_intrinsic_loss}, for any constant $\kappa \geq 1$, there exists some constant $C > 0$ such that
\[
\Pi_\btheta(\btheta:|\mathrm{supp}(\bA)| > \kappa s_0)\lesssim \exp\{-C\kappa (rs_0\log n + s_0\log p)\}.
\]
\end{lemma}
\begin{proof}
[Proof of Lemma \ref{lemma:prior_sparsity}]
Write
\begin{align*}
&\Pi_\btheta(\btheta:|\mathrm{supp}(\bA)| > \kappa s_0)\\
&\quad= \sum_{|S| = \lfloor\kappa s_0\rfloor}^{p - r}\pi_p(|S|)
= \frac{1}{z_n}\sum_{t = \lfloor\kappa s_0\rfloor}^{p - r}\exp\{-rt\log n - at\log (p - r)\}\\
&\quad\leq 2\sum_{t = \lfloor\kappa s_0\rfloor}^{p - r}\exp(-rs_0\log n - Cs_0\log p)\\
&\quad\lesssim \exp\left(-\kappa rs_0\log n - \frac{C}{2}\kappa s_0 \log p\right)\sum_{t = \lfloor\kappa s_0\rfloor}^{p - r}\exp\left(-\frac{C}{2}s_0\log p\right)\\
&\quad\lesssim \exp\{-C\kappa (rs_0\log n + s_0 \log p)\}
\end{align*}
for some absolute constant $C > 0$. The proof is thus completed. 
\end{proof}

\begin{lemma}\label{lemma:posterior sparsity}
Under the setup and the prior specification in Section \ref{sub:bayesian_sparse_pca_and_non_intrinsic_loss}, there exists some constant $\kappa_0 \geq 1$ depending on $\|\bOmega_0\|_2$, such that
\begin{align*}
\expect_0\left[\Pi_\btheta\left\{\btheta:|\mathrm{supp}(\bA)| > \kappa_0
% \left(1 + \frac{r\log n}{\log p}\right)
s_0\mathrel{\Big|}\bY_n\right\}\right]\to 0.
\end{align*}
\end{lemma}

\begin{proof}
[Proof of Lemma \ref{lemma:posterior sparsity}]
Let 
\[
\Xi_n := \{D_n \geq \exp\{-Cn\eta_n^2\log(2\|\bOmega_0\|_2)\}\Pi_\btheta\{\btheta:\|\bOmega(\btheta) - \bOmega_0\|_{\mathrm{F}} < \eta_n\}\},
\] where $\eta_n = \sqrt{(rs\log n + s\log p)/n}$. By Lemma \ref{lemma:evidence_lower_bound}, $\prob_0(\Xi_n^c)\to 0$, and by Lemma \ref{lemma:prior_concentration}, 
\[
\Pi_\btheta\{\|\bOmega(\btheta) - \bOmega_0\|_{\mathrm{F}} < \eta_n\}\geq \exp(- C_0rs_0\log n + Cs_0\log p)
\]
for some constant $C_0$ depending on $\|\bOmega_0\|_2$.
Therefore, over the event $\Xi_n$, we have
\begin{align*}
D_n&\geq \exp\{-Cn\eta_n^2\log(2\|\bOmega_0\|_2) - C_0rs_0\log n - Cs_0\log p\}\\
&\geq\exp\{-C_0(rs_0\log n + s_0\log p)\} = \exp(-C_0n\eta_n^2)
\end{align*}
for some constant $C_0$ depending on $\|\bOmega_0\|_2$. Hence, by Lemma \ref{lemma:prior_sparsity} and the Fubini's theorem, we have,
\begin{align*}
&\expect_0\left[\Pi_\btheta\left\{\btheta:|\mathrm{supp}(\bA)| > \kappa_0
s_0\mathrel{\big|}\bY_n\right\}\right]\\
&\quad \leq \expect_0\left[\Pi_\btheta\left\{\btheta:|\mathrm{supp}(\bA)| > \kappa_0
s_0\mathrel{\big|}\bY_n\right\}\mathbbm{1}(\Xi_n)\right] + \prob_0(\Xi_n^c)\\
&\quad \leq \exp(C_0n\eta_n^2)\expect_0\left[\int_{\{\btheta:|\mathrm{supp}(\bA)| > \kappa_0s_0\}}\exp\{\ell(\bOmega(\btheta)) - \ell(\bOmega_0)\}\Pi_\btheta(\mathrm{d}\btheta)\right] + \prob_0(\Xi_n^c)\\
&\quad = \exp(C_0n\eta_n^2)\int_{\{\btheta:|\mathrm{supp}(\bA)| > \kappa_0s_0\}}\expect_0\left[\exp\{\ell(\bOmega(\btheta)) - \ell(\bOmega_0)\}\right]\Pi_\btheta(\mathrm{d}\btheta) + \prob_0(\Xi_n^c)\\
&\quad = \exp(C_0n\eta_n^2)\Pi_\btheta\left\{\btheta:|\mathrm{supp}(\bA)| > \kappa_0s_0\right\} + \prob_0(\Xi_n^c)\\
&\quad \lesssim  \exp\left\{C_0n\eta_n^2 - C\kappa_0\left(rs_0\log n + s_0\log p\right)\right\} + \prob_0(\Xi_n^c)\\
&\quad =  \exp\left(C_0n\eta_n^2 - C\kappa_0n\eta_n^2\right) + o(1).
\end{align*}
Therefore, we conclude that 
\[
\expect_0\left[\Pi_\btheta\left\{\btheta:|\mathrm{supp}(\bA)| > \kappa_0 s_0\mathrel{\Big|}\bY_n\right\}\right]\to 0
\]
by taking $\kappa_0$, possibly depending on $\|\bOmega_0\|_2$, to be sufficiently large. 
\end{proof}

% subsection posterior_sparsity (end)

\subsection{Construction of test functions} % (fold)
\label{sub:construction_of_test_functions}
In this section, we construct a test function that will be useful for deriving posterior contraction under the Frobenius norm through Lemma \ref{lemma:existence_global_test} below. 
\begin{lemma}\label{lemma:existence_global_test}
Assume the random vectors $\by_1,\ldots,\by_n$ follows $\mathrm{N}_p(\zero_p, \bOmega)$ independently, where $\bOmega = \bU\bM\bU\transpose + \eye_p$, $\bU\in\mathbb{O}_+(p, r)$, and $\bM\in\mathbb{M}_+(r)$. Let $\bOmega_0 = \bU_0\bM_0\bU_0\transpose + \eye_p$, where $\bU_0\in\mathbb{O}_+(p, r)$ with $|\mathrm{supp}\{\bA(\bU_0)\}|\leq s_0$ and $\bM_0\in\mathbb{M}_+(r)$. If $(\eps_n)_{n = 1}^\infty$ is a sequence converging to $0$, then for any $\kappa \geq 1$ and $M > 4$, there exists a sequence of test functions $(\phi_n)_{n = 1}^\infty$ such that 
\begin{align*}
\expect_{\bOmega_0}(\phi_n)&\leq 3\exp\left\{(2C + 4)\kappa s_0 \log p - \frac{CM^2n\eps_n^2}{4\|\bOmega_0\|_2^2}\right\},\\
\sup_{\bOmega\in H_1}\expect_{\bOmega}(1 - \phi_n)&\leq \exp\left(2C\kappa s_0 - \frac{CMn\eps_n^2}{8\|\bOmega_0\|_2^2}\right),
\end{align*}
where 
\begin{align*}
H_1 = \{\bOmega = \bU\bM\bU\transpose + \eye_p&:\|\bOmega - \bOmega_0\|_{\mathrm{F}} > M\eps_n, \bU\in\mathbb{O}_+(p, r),\bM\in\mathbb{M}_+(r),\\
&\quad|\mathrm{supp}\{\bA(\bU)\}| \leq \kappa s\}
\end{align*}
and $C$ is some absolute constant.
\end{lemma}
The proof of Lemma \ref{lemma:existence_global_test} relies on the oracle testing lemma from \cite{gao2015rate} below.
\begin{lemma}[\citealp{gao2015rate}]
\label{lemma:existence_local_test}
Let the random vectors $\by_1,\ldots,\by_n$ follow $\mathrm{N}_d(\zero_d, \bOmega)$ independently, where $\bOmega\in\mathbb{R}^{d\times d}$. If $(\eps_n)_{n = 1}^\infty$ is a sequence converging to $0$, then for any $M>0$ and $d\times d$ covariance matrices $\bOmega^{(1)}$ and $\bOmega^{(2)}$, there exists a test function $\phi_n$ such that
\begin{align}
\expect_{\bOmega^{(1)}}(\phi_n)&\leq \exp\left(Cd-\frac{CM^2n\eps_n^2}{4\|\bOmega^{(1)}\|_2^2}\right)+2\exp\left(Cd-C\sqrt{M}n\right)\nonumber,\\
\sup_{\{\bOmega^{(2)}:\|\bOmega^{(2)}-\bOmega^{(1)}\|_2>M\eps_n\}}\expect_{\bOmega^{(2)}}(1-\phi_n)&\leq \exp\left[Cd - \frac{CMn\eps_n^2}{4}\left\{1\vee\frac{M}{(\sqrt{M}+2)^2\|\bOmega^{(1)}\|_2^2}\right\}\right]\nonumber
\end{align}
with some absolute constant $C>0$.
\end{lemma}

\begin{proof}[Proof of Lemma \ref{lemma:existence_global_test}]
The proof of Lemma \ref{lemma:existence_global_test} is very similar to that of Lemma 5.4 in \cite{gao2015rate} and is included here for completeness. Decompose $H_1$ by
\[
H_1 \subset\bigcup_{S:|S| \leq \kappa  s_0 } H_{1S},
\]
where
\begin{align*}
H_{1S} = \{\bOmega = \bU\bM\bU\transpose + \eye_p&:\bU\in\mathbb{O}_+(p, r), \bM\in\mathbb{M}_+(r),\|\bOmega - \bOmega_0\|_{\mathrm{F}} > M\eps_n,\\
&\quad S = \mathrm{supp}\{\bA(\bU)\}\}.
\end{align*}
Let $\bar{S}: = S\cup S_0$, where $S_0 = \mathrm{supp}\{\bA(\bU_0)\}$, and let $\bar{s}: = |\bar{S}|$. Clearly, $\bar{s}\leq (\kappa + 1)s_0$ and
\[
\|\bOmega - \bOmega_0\|_{\mathrm{F}} = \|\bar{\bOmega} - \bar{\bOmega}_0\|_{\mathrm{F}},
\]
where
\[
\bar{\bOmega} = \bU(\bA_{\bar{S}})\bM\bU(\bA_{\bar{S}})\transpose + \eye_{\bar{s} + r},\quad 
\bar{\bOmega}_0 = \bU(\bA_{0\bar{S}})\bM_0\bU(\bA_{0\bar{S}})\transpose + \eye_{\bar{s} + r}.
\]
% and $\bU = [\bu_1,\ldots,\bu_r]$, $\bu_{k\bar{S}} = [u_{kj}:k\in \bar{S}]\transpose$, $\bU_0 = [\bu_{01},\ldots,\bu_{0r}]$, $\bu_{0k\bar{S}} = [u_{0kj}:k\in \bar{S}]\transpose$. 
For each ${S}\subset[p]$, denote $\by_{iS} = [y_{ij}:j\in S]\transpose$ for $i = 1,\ldots,n$. 
By Lemma \ref{lemma:existence_local_test}, for each $S$ and $M > 4$, there exists a sequence of tests $(\phi_{nS})_{n = 1}^\infty$, where $\phi_{nS}$ is a measurable function of $\{\by_{1\mathrm{supp}\{\bU(\bA_{\bar{S}})\}},\ldots,\by_{n\mathrm{supp}\{\bU(\bA_{\bar{S}})\}}\}$, such that
\begin{align*}
\expect_{\bar{\bOmega}_0}(\phi_{nS})&\leq \exp\left\{C(\kappa + 1)s_0 - \frac{CM^2n\eps_n^2}{4\|\bOmega_0\|_2^2}\right\} + 2\exp\{C(\kappa + 1)s_0 - C\sqrt{M}n\}\\
&\leq 3\exp\left(2C\kappa s_0-\frac{CM^2n\eps_n^2}{4\|\bOmega_0\|_2^2}\right),
\end{align*}
and
\begin{align*}
\sup_{\bar{\bOmega}\in \bar{H}_{1S}}\expect_{\bar{\bOmega}}(1 - \phi_n)
&\leq \exp\left[C(\kappa + 1)s_0 - \frac{CMn\eps_n^2}{4}\left\{1\vee \frac{M}{(\sqrt{M} + 2)^2\|\bOmega_0\|_2^2}\right\}\right]\\
&\leq \exp\left(2C\kappa s_0-\frac{CMn\eps_n^2}{8\|\bOmega_0\|_2^2}\right),
\end{align*} 
where 
\[
\bar{H}_{1S} = \left\{\bar{\bOmega} = \bU(\bA_{\bar{S}})\bM\bU(\bA_{\bar{S}})\transpose + \eye_{\bar{s} + r}: \bM\in\mathbb{M}_+(r), \|\bar{\bOmega} - \bar{\bOmega}_0\|_{\mathrm{F}} > M\eps_n\right\}.
\]
Hence we can combine tests by taking $\phi_n = \max_{S}\phi_{nS}$ and apply the union bound to obtain
\begin{align*}
\expect_0(\phi_n)&\leq \sum_{s = 1}^{\lceil \kappa s_0\rceil}{p - r\choose s}3\exp\left(2C\kappa s_0 - \frac{CM^2n\eps_n^2}{4\|\bOmega_0\|_2^2}\right)\\
&\leq 3\kappa s_0\exp\left(3\kappa s_0\log p\right)\exp\left(2C\kappa s_0 - \frac{CM^2n\eps_n^2}{4\|\bOmega_0\|_2^2}\right)\\
&\leq 3\exp\left\{(2C + 4)\kappa s_0 \log p -  \frac{CM^2n\eps_n^2}{4\|\bOmega_0\|_2^2}\right\},
\end{align*}
and
\begin{align*}
\sup_{\bOmega\in H_1}\expect_{\bOmega}(1 - \phi_n)&\leq \sup_{S:|S|<\kappa s_0}\sup_{\bar{\bOmega}\in \bar{H}_{1S}}\expect_{\bar{\bOmega}}(1 - \phi_{nS})
% \\&
\leq \exp\left(2C\kappa s_0 - \frac{CMn\eps_n^2}{8\|\bOmega_0\|_2^2}\right).
\end{align*}
The proof is thus completed. 
\end{proof}

% subsection construction_of_test_functions (end)

\subsection{Posterior contraction under Frobenius norm} % (fold)
\label{sub:posterior_contraction_under_frobenius_norm}

\begin{theorem}\label{thm:contraction_Frobenius_norm}
Under the prior specification and setup in Section \ref{sub:bayesian_sparse_pca_and_non_intrinsic_loss}, there exists some large constant $M_0 > 0$ (possibly depending on $\|\bOmega_0\|_2$), such that
\[
\expect_0\left\{\Pi_\btheta\left(\|\bOmega(\btheta) - \bOmega_0\|_{\mathrm{F}} > M\sqrt{\frac{rs_0\log n + s_0\log p}{n}}\mathrel{\Big|}\bY_n\right)\right\} \to 0.
\]
\end{theorem}

\begin{proof}
% [\bf Proof of Theorem \ref{thm:contraction_Frobenius_norm}]
Denote $\eps_n = \sqrt{(rs_0\log n + s_0\log p)/n}$. We first decompose the expected posterior probability by
\begin{align*}
&\expect_0[\Pi_\btheta\{\|\bOmega(\btheta) - \bOmega_0\|_{\mathrm{F}} > M\eps_n\mid\bY_n\}]\\
&\quad\leq 
\expect_0\left[\Pi_\btheta\left\{\|\bOmega(\btheta) - \bOmega_0\|_{\mathrm{F}} > M\eps_n, |\mathrm{supp}(\bA)| \leq \kappa_0
% \left(1 + \frac{r\log n}{\log p}\right)
s_0 \mathrel{\Big|}\bY_n\right\}\right]\\
&\quad\quad+ \expect_0\left[\Pi_\btheta\left\{\btheta:|\mathrm{supp}(\bA)| > \kappa_0
% \left(1 + \frac{r\log n}{\log p}\right)
s_0\mid\bY_n\right\}\right]
\end{align*}
where $\kappa_0$ is set to be large enough such that the second term on the right-hand side is $o(1)$ according to Lemma \ref{lemma:posterior sparsity}. It suffices to focus on the first term consequently. Let $\Xi_n = \{D_n\geq \exp(-C_0n\eps_n^2)\}$, where $C_0$ is a constant depending on $\|\bOmega_0\|_2$ such that $\prob_0(\Xi_n^c) = o(1)$ according to Lemma \ref{lemma:evidence_lower_bound} and Lemma \ref{lemma:prior_concentration}. Take $\phi_n$ to be the test function given by Lemma \ref{lemma:existence_global_test}. Then we can decompose the first term on the right-hand side of the previous display by
\begin{align*}
&\expect_0\left[\Pi_\btheta\left\{\|\bOmega(\btheta) - \bOmega_0\|_{\mathrm{F}} > M\eps_n, |\mathrm{supp}(\bA)| \leq \kappa_0
% \left(1 + \frac{r\log n}{\log p}\right)
s_0 \mathrel{\Big|}\bY_n\right\}\right]\\
&\quad\leq \expect_0\left[\Pi_\btheta\left\{\|\bOmega(\btheta) - \bOmega_0\|_{\mathrm{F}} > M\eps_n, |\mathrm{supp}(\bA)| \leq \kappa_0
% \left(1 + \frac{r\log n}{\log p}\right)
s_0 \mathrel{\Big|}\bY_n\right\}\mathbbm{1}(\Xi_n)(1 - \phi_n)\right]\\
&\quad\quad + \expect_0(\phi_n) + \prob_0(\Xi_n^c).
\end{align*}
Since the third term on the right-hand side is upper bounded by
\[
3\exp\left\{(2C + 4)\kappa_0
% \left(1 + \frac{r\log n}{\log p}\right)
s_0\log p - \frac{CM^2n\eps_n^2}{4\|\bOmega_0\|_2^2}\right\} = o(1)
\]
by Lemma \ref{lemma:existence_global_test} with a sufficiently large $ M > 0$, and the second term is also $o(1)$ by Lemma \ref{lemma:evidence_lower_bound} and Lemma \ref{lemma:prior_concentration}, it suffices to show that the first term is also $o(1)$. Denote
\begin{align*}
H_1 = \{\bOmega = \bU\bM\bU\transpose + \eye_p&:\bU\in\mathbb{O}_+(p, r), \bM\in\mathbb{M}_+(r), \|\bOmega - \bOmega_0\|_{\mathrm{F}} > M\eps_n,\\
&\quad |\mathrm{supp}\{\bA(\bU)\}| \leq \kappa_0
% \left(1 + \frac{r\log n}{\log p}\right)
s_0\}. 
\end{align*}
Then by Lemma \ref{lemma:existence_global_test}, the Fubini's theorem, and the definition of $\Xi_n$, 
\begin{align*}
&\expect_0\left[\Pi_\btheta\left\{\|\bOmega(\btheta) - \bOmega_0\|_{\mathrm{F}} > M\eps_n, |\mathrm{supp}(\bA)| \leq \kappa_0
% \left(1 + \frac{r\log n}{\log p}\right)
s_0 \mathrel{\Big|}\bY_n\right\}\mathbbm{1}(\Xi_n)(1 - \phi_n)\right] \\
&\quad 
\leq \exp(C_0n\eps_n^2)\expect_0\left[(1 - \phi_n)\int_{\{\btheta:\bOmega(\btheta)\in H_1\}}\exp\{\ell(\bOmega(\btheta)) - \ell(\bOmega_0)\}\Pi_\btheta(\mathrm{d}\btheta)\right]\\
&\quad = 
\exp(C_0n\eps_n^2)\int_{\btheta:\bOmega(\btheta)\in H_1}\expect_0\left[(1 - \phi_n)\exp\{\ell(\bOmega(\btheta)) - \ell(\bOmega_0)\}\right]\Pi_\btheta(\mathrm{d}\btheta)\\
&\quad  = 
\exp(C_0n\eps_n^2)\int_{\{\btheta:\bOmega(\btheta)\in H_1\}}\expect_{\bOmega}\left\{(1 - \phi_n)\right\}\Pi(\mathrm{d}\bOmega)\\
&\quad\leq\exp(C_0n\eps_n^2)\sup_{\bOmega\in H_1}\expect_{\bOmega}\left\{(1 - \phi_n)\right\}\\
&\quad\leq \exp\left\{C_0n\eps_n^2 + 2C\kappa_0
% \left(1 + \frac{r\log n}{\log p}\right)
s_0 - \frac{CMn\eps_n^2}{8\|\bOmega_0\|_2^2}\right\} = o(1)
\end{align*}
by taking $M$ to be suffciently large enough. The proof is thus completed. 
\end{proof}

% subsection posterior_contraction_in_frobenius_norm (end)

\subsection{Local asymptotic normality} % (fold)
\label{sub:local_asymptotic_normality}
In this subsection, we establish the local asymptotic normality of the spiked covariance model under the sparsity constraint through Theorem \ref{thm:LAN} below. Some preliminaries are needed in order to proof this theorem. Define
\begin{align}\label{eqn:concentration_set}
\calA_{n} = \left\{\btheta:
\|\bOmega(\btheta) - \bOmega_0\|_{\mathrm{F}} < M\eps_n
, |\mathrm{supp}(\bA)|\leq \kappa_0 s_0
\right\},
\end{align}
where $\eps_n = \sqrt{(rs\log n + s\log p)/n}$. By Theorem \ref{thm:contraction_Frobenius_norm} and Lemma \ref{lemma:posterior sparsity}, there exists some constant $M > 0$ and $\kappa_0 > 0$, possibly depending on $\|\bOmega_0\|_2$, such that
\[
\expect_0\{\Pi_\btheta(\btheta\in\calA_{n}\mid\bY_n)\}\to 1
\]
% Now assume that $\|(\eye_r + \bQ_{01})^{-1}\|_2 < \beta$ for some constant $\beta > 0$. 
Under the assumption that $\sup_{n\geq 1}(\bA_0)$ is bounded away from $1$, by Theorem \ref{thm:intrinsic_deviation_Sigma}
\begin{align*}
\calA_{n}&\subset \{\btheta: \|\btheta - \btheta_0\|_2\leq M'\eps_n, |\mathrm{supp}(\bA)|\leq \kappa_0s_0\}.
\end{align*}
for some large constant $M' > 0$. Note that with a slight abuse of notation, we may use $M$ to denote a generic constant that is sufficiently large such that we can write
\begin{align*}
\calA_{n}&\subset \{\btheta: \|\btheta - \btheta_0\|_2\leq M\eps_n, |\mathrm{supp}(\bA)|\leq \kappa_0s_0\}
\end{align*}
and $\calA_n$ still satisfies $\expect_0\{\Pi_\btheta(\btheta\in\calA_n\mid\bY_n)\} \to 1$. For any $S\subset[p - r]$ with $|S|\leq \kappa_0 s_0$, let
\[
\calA_{n}(S) = \left\{\btheta:\|\btheta - \btheta_{0}\|_1\leq M\sqrt{\frac{r^2s_0^2\log n + rs_0^2\log p}{n}}, \mathrm{supp}(\bA) = S
\right\}
\]
for some large constant $M > 0$. It follows that
\[
\calA_n \subset\calB_n:= \bigcup_{S:|S|\leq \kappa_0 s_0}\calA_{n}(S)
\]
This is because for all $\btheta = [\vect(\bA)\transpose, \bmu\transpose]\transpose\in\calA_n$ with $\mathrm{supp}(\bA) = S$, $|S|\leq \kappa_0s_0$, we have
\begin{align*}
\|\btheta - \btheta_{0}\|_1&\leq \sqrt{(|S| + |S_0|)r + 2r^2 + r(r + 1)}\|\btheta - \btheta_{0}\|_2
\\
&
\lesssim \sqrt{(\kappa_0 + 1)rs_0}\sqrt{\frac{rs_0\log n + s_0\log p}{n}}
% \\&
\lesssim \sqrt{\frac{r^2s_0^2\log n + rs_0^2\log p}{n}}.
\end{align*}
% Also, we know that
% \[
% \|D\bU(\bvarphi)\|_{2}\leq 2\| \eye_{p\times r}\transpose(\eye_p - \bX_\bvarphi)\inverseT\otimes (\eye_p - \bX_\bvarphi)^{-1} \|_2\|\bGamma\|_{2}\leq 2\sqrt{2}
% \]
% for all $\bvarphi\in\mathbb{R}^d$.
\begin{theorem}\label{thm:LAN}
Under the prior specification and setup in Section \ref{sub:bayesian_sparse_pca_and_non_intrinsic_loss}, the log-likelihood function of $\btheta$ yields the following local asymptotic normality expansion:
\begin{align*}
\ell(\bOmega(\btheta)) - \ell(\bOmega_0)
 = 
 &\frac{n}{2}\vect\left(\widehat\bOmega - \bOmega_0\right)\transpose(\bOmega_0^{-1}\otimes\bOmega_0^{-1})D\bSigma(\btheta_{0})(\btheta - \btheta_{0})\\
 & - \frac{n}{4}(\btheta - \btheta_{0})\transpose D\bSigma(\btheta_{0})\transpose(\bOmega_0^{-1}\otimes\bOmega_0^{-1})D\bSigma(\btheta_{0})(\btheta - \btheta_{0})\\
 & + R_n(\btheta, \btheta_{0}),
\end{align*}
where $\widehat\bOmega = (1/n)\sum_{i = 1}^n\by_i\by_i\transpose$ denotes the sample covariance matrix and the remainder $R_n$ satisfies
\[
\sup_{\btheta\in\calB_{n}}|R_n(\btheta, \btheta_{0})| = o_{\prob_0}(1).
\] 
\end{theorem}
The key to the proof of the local asymptotic normality expansion in Theorem \ref{thm:LAN} is the following lemma that controls the stochastic remainder in the Taylor expansion of the log-likelihood function. For convenience denote
\begin{align*}
\bR_\bU(\bvarphi, \bvarphi_{0}) & = \bU(\bvarphi) - \bU(\bvarphi_{0}) - 2(\eye_p - \bX_{0})^{-1}(\bX_\bvarphi - \bX_{0})(\eye_p - \bX_{0})^{-1}\eye_{p\times r},\\
\bR_\bOmega(\btheta, \btheta_{0}) & = \bU(\bvarphi)(\bM - \bM_0)\{\bU(\bvarphi) - \bU_{0}\}\transpose + \{\bU(\bvarphi) - \bU_{0}\}\bM_0\{\bU(\bvarphi) - \bU_{0}\}\transpose\\
&\quad + \{\bU(\bvarphi) - \bU_{0}\}(\bM - \bM_0)\bU_{0}\transpose,\\
\bR_1(\bOmega, \bOmega_0) & = \bOmega_0^{-1}(\bOmega_0 - \bOmega)\bOmega_0^{-1}\sum_{m = 1}^\infty\{(\bOmega_0 - \bOmega)\bOmega_0^{-1}\}^m,\quad \|\bOmega - \bOmega_0\|_2 < 1
\end{align*}
\begin{lemma}\label{lemma:stochastic_remainder}
Under the prior specification and setup in Section \ref{sub:bayesian_sparse_pca_and_non_intrinsic_loss}, the following stochastic remainders are asymptotically uniformly negligible:
\begin{align}
% \begin{aligned}
&\sup_{\btheta\in\calB_{n}(S)} \left|2n\vect(\widehat\bOmega - \bOmega_0)\transpose (\bOmega_0^{-1}\otimes\bOmega_0^{-1})\vect\{\bR_\bU(\bvarphi, \bvarphi_{0})\bM_0\bU_0\transpose\}\right|
% \nonumber\\
\label{eqn:stochastic_remainder_I}
% &\quad
 = o_{\prob_0}(1),\\
\label{eqn:stochastic_remainder_II}
&
% \sup_{S\in\calS(\kappa_n)}
\sup_{\btheta\in\calB_{n}(S)}\left|n\vect(\widehat\bOmega - \bOmega_0)\transpose(\bOmega_0^{-1}\otimes\bOmega_0^{-1}) \vect\{\bR_\bOmega(\btheta, \btheta_{0})\}\right| = o_{\prob_0}(1),\\
\label{eqn:stochastic_remainder_III}
&\sup_{\btheta\in\calB_{n}(S)} \left|n\vect(\widehat\bOmega - \bOmega_0)\transpose \vect\{\bR_1(\bOmega(\btheta), \bOmega_0)\}\right| = o_{\prob_0}(1),
% \end{aligned}
\end{align}
where $\widehat\bOmega = (1/n)\sum_{i = 1}^n\by_i\by_i\transpose$ denotes the sample covariance matrix.
\end{lemma}

\begin{proof}
[Proof of Lemma \ref{lemma:stochastic_remainder}]
The proof is based on reducing the dimension of the deterministic remainders $\bR_\bU$, $\bR_\bSigma$, and $\bR_1$ because for $\btheta$ and $\btheta_{0}$, the instrinsic dimension is much smaller than the ambient dimension due to the sparsity. We first fix $S\in\calS(\kappa_0s_0)$.
Let $\bar{S} = S\cup S_0$, where $S_0 = S_{\bA_{0}}$ and $\btheta_{0} = [\vect(\bA_{0})\transpose, \bmu_0\transpose]\transpose$, and let $\bar{s} = |\bar{S}|$. Denote
\begin{align*}
\bA = \bP_S\begin{bmatrix*}
\bA_{\bar{S}}\\ \zero
\end{bmatrix*},
\quad
\bvarphi_{\bar{S}} = 
\vect(\bA_{\bar{S}})
,\quad
\btheta_{\bar{S}} = 
\begin{bmatrix*}
\bvarphi_{\bar{S}}\\
\bmu
\end{bmatrix*}
\end{align*}
for a suitable permutation matrix $\bP_S$. 
Similarly, denote
\begin{align*}
\bA_{0} = \bP_S\begin{bmatrix*}
\bA_{0\bar{S}}\\ \zero
\end{bmatrix*},
\quad
\bvarphi_{0\bar{S}} = 
% \begin{bmatrix*}
% \bb\\
\vect(\bA_{0\bar{S}})
% \end{bmatrix*}
,\quad
\btheta_{0\bar{S}} = 
\begin{bmatrix*}
\bvarphi_{0\bar{S}}\\
\bmu
\end{bmatrix*}.
\end{align*}
By definition of the Cayley parameterization $\bvarphi\mapsto \bU(\bvarphi)$,
\begin{align*}
\bU(\bvarphi) &
% = \begin{bmatrix*}
% (\eye_r - \bA\transpose \bA)(\eye_r + \bA\transpose \bA)^{-1}\\
% \bA(\eye_r + \bA\transpose \bA)^{-1}
% \end{bmatrix*}
% \\&
=
\begin{bmatrix*}
(\eye_r - \bA_{\bar{S}}\transpose \bA_{\bar{S}})(\eye_r + \bA_{\bar{S}}\transpose \bA_{\bar{S}})^{-1}\\
\bP_S\begin{bmatrix*}
\bA_{\bar{S}}(\eye_r + \bA\transpose \bA)^{-1}\\
\zero
\end{bmatrix*}
\end{bmatrix*}
% \\&
= \begin{bmatrix*}
\eye_r & \\ & \bP_S
\end{bmatrix*}\begin{bmatrix*}
\bU(\bvarphi_{\bar{S}})\\\zero
\end{bmatrix*}
\end{align*}
where $\bU(\bvarphi_{\bar{S}})$ is the Cayley parameterization that maps the vector $\bvarphi_{\bar{S}}$ to $\in\mathbb{O}(\bar{s} + r, r)$. Write $\bQ_S = \diag(\eye_r, \bP_S)$. Similarly, we can also write $\bU_{0} = \bQ_S[\bU_{0\bar{S}}\transpose, \zero]\transpose$, where $\bU_{0\bar{S}} = \bU(\bvarphi_{0\bar{S}})$. The permutation matrix $\bQ_S$ will be useful in this proof. For $\bR_\bU$, write
\begin{align*}
\eye_p - \bX_{0}
& = \begin{bmatrix*}
\eye_r & & \\
	   & \eye_{\bar{s}} & \\
	   &				& \eye_{p - (\bar{s} + r)}
\end{bmatrix*} - \bQ_S\begin{bmatrix*}
\zero & -\bA_{0\bar{S}}\transpose & \zero\\
\bA_{0\bar{S}} & \zero & \zero\\
\zero & \zero & \zero
\end{bmatrix*}\bQ_S\transpose
= \bQ_S\begin{bmatrix*}
\eye - \bX_{0\bar{S}} & \\
& \eye_{p - (\bar{s} + r)} \\
\end{bmatrix*}\bQ_S\transpose,
\end{align*}
where 
\[
\bX_{0\bar{S}} := \bX_{\bvarphi_{0\bar{S}}} = \begin{bmatrix*}
\zero & -\bA_{0\bar{S}}\transpose\\
-\bA_{0\bar{S}} & \zero
\end{bmatrix*}.
\]
Similarly, we also have
\begin{align*}
\bX_\bvarphi - \bX_{0} = \bQ_S\begin{bmatrix*}
\bX_{\bvarphi_{\bar{S}}} - \bX_{0\bar{S}} & \zero \\
\zero & \zero
\end{bmatrix*}\bQ_S\transpose,
\quad\text{where}\quad
\bX_{\bvarphi_{\bar{S}}} = \begin{bmatrix*}
\zero & -\bA_{\bar{S}}\transpose\\
\bA_{\bar{S}} & \zero
\end{bmatrix*}.
\end{align*}

\vspace*{2ex}\noindent
$\blacksquare$
We first consider $\bR_\bU(\bvarphi, \bvarphi_{0})$. Write $\bR_\bU(\bvarphi, \bvarphi_{0})$ in the following block form with a zero matrix in the lower block:
\begin{align*}
&\bR_\bU(\bvarphi, \bvarphi_{0})\\
&\quad = \bQ_S\begin{bmatrix*}
\bU(\bvarphi_{\bar{S}}) - \bU(\bvarphi_{0\bar{S}})\\\zero
\end{bmatrix*}
- \bQ_S\begin{bmatrix*}
(\eye - \bX_{0\bar{S}})^{-1}(\bX_{\bvarphi_{\bar{S}}} - \bX_{0\bar{S}})(\eye - \bX_{0\bar{S}})^{-1}\eye_{(\bar{s} + r)\times r}\\\zero
\end{bmatrix*}\\
&\quad := \bQ_S\begin{bmatrix*}
\bR_{\bU}(\bvarphi_{\bar{S}}, \bvarphi_{0\bar{S}})\\\zero
\end{bmatrix*},
\end{align*}
where we have used the fact that $\bQ_S\eye_{p\times r} = \eye_{p\times r}$. Therefore, 
\begin{align*}
&\vect(\widehat\bOmega - \bOmega_0)\transpose(\bOmega_0^{-1}\otimes\bOmega_0^{-1})\vect\{\bR_\bU(\bvarphi, \bvarphi_{0})\bM_0\bU_0\transpose\}
\\
&\quad = \mathrm{tr}\left\{(\widehat\bOmega - \bOmega_0)\bOmega_0^{-1}\bR_\bU(\bvarphi, \bvarphi_{0})\bM_0\bU_{0}\transpose\bOmega_0^{-1}\right\}
 % + \mathrm{tr}\left\{(\widehat\bOmega - \bOmega_0)\bOmega_0^{-1}\bU_{0}\bM_0\bR_\bU(\bvarphi, \bvarphi_{0})\transpose\bOmega_0^{-1}\right\}
 % \\&\quad
 = \mathrm{tr}\left\{(\widehat\bOmega - \bOmega_0)\bOmega_0^{-1}\bR_\bU(\bvarphi, \bvarphi_{0})\tilde\bM_0\bU_{0}\transpose\right\}
 % + \mathrm{tr}\left\{\left(\widehat\bOmega - \bOmega_0\right)\bU_{0}\tilde\bM_0\bR_\bU(\bvarphi, \bvarphi_{0})\transpose\bOmega_0^{-1}\right\}
 \\
&\quad = \mathrm{tr}\left\{\bU_{0}\transpose(\widehat\bOmega - \bOmega_0)\bOmega_0^{-1}\bR_\bU(\bvarphi, \bvarphi_{0})\tilde\bM_0\right\}
 % + \mathrm{tr}\left\{\bR_\bU(\bvarphi, \bvarphi_{0})\transpose\bOmega_0^{-1}\left(\widehat\bOmega - \bOmega_0\right)\bU_{0}\tilde\bM_0\right\}
 ,
\end{align*}
where $\tilde\bM_0 = \bM_0(\bM_0 + \eye)^{-1}$. 
Write $\widehat\bOmega$ and $\bOmega_0$ in the block forms
\[
\widehat\bOmega = \bQ_S\begin{bmatrix*}
\widehat\bOmega_{\bar{S}} & \widehat\bOmega_{12}\\ \widehat\bOmega_{21} & \widehat\bOmega_{22}
\end{bmatrix*}\bQ_S\transpose.
\]
and
\begin{align*}
\bOmega_0 &= \bQ_S\left(\begin{bmatrix*}
\bU_{0\bar{S}}\\\zero
\end{bmatrix*}\bM_0
\begin{bmatrix*}
\bU_{0\bar{S}}\transpose & \zero
\end{bmatrix*} + \eye_p\right)\bQ_S\transpose
% \\&
= \bQ_S\begin{bmatrix*}
\bU_{0\bar{S}}\bM_0\bU_{0\bar{S}}\transpose + \eye_{\bar{s} + r} & \zero\\
\zero & \eye_{p - (\bar{s} + r)}
\end{bmatrix*}\bQ_S\transpose\\
& = \bQ_S
\begin{bmatrix*}
\bOmega_{0\bar{S}} & \zero\\
\zero & \eye_{p - (\bar{s} + r)}
\end{bmatrix*}\bQ_S\transpose,
\end{align*}
where $\bOmega_{0\bar{S}} = \bU_{0\bar{S}}\bM_0\bU_{0\bar{S}}\transpose + \eye_{\bar{s} + r}$. 
It follows that
\begin{align*}
&\bU_{0}\transpose(\widehat\bOmega - \bOmega_0)\bOmega_0^{-1}\bR_\bU(\bvarphi, \bvarphi_{0})\tilde\bM_0\\
&\quad = \begin{bmatrix*}
\bU_{0\bar{S}}\transpose & \zero 
\end{bmatrix*}\bQ_S\transpose \bQ_S\begin{bmatrix*}
\widehat\bOmega_{\bar{S}} - \bOmega_{0\bar{S}} & \widehat\bOmega_{12}\\
\widehat\bOmega_{21} & \widehat\bOmega_{22} - \eye_{p - (\bar{s} + r)}
\end{bmatrix*}
\begin{bmatrix*}
\bOmega_{0\bar{S}}^{-1} & \zero\\ \zero & \eye_{p - (\bar{s} + r)}
\end{bmatrix*}
\begin{bmatrix*}
\bR_\bU(\bvarphi_{\bar{S}}, \bvarphi_{0\bar{S}})\\ \zero
\end{bmatrix*}\tilde\bM_0\\
&\quad = \bU_{0\bar{S}}\transpose(\widehat\bOmega_{\bar{S}} - \bOmega_{0\bar{S}})\bOmega_{0\bar{S}}^{-1}\bR_\bU(\bvarphi_{\bar{S}}, \bvarphi_{0\bar{S}})\tilde\bM_0.
% \\
% &\bR_\bU(\bvarphi, \bvarphi_{0})\transpose\bSigma_0^{-1}\left(\widehat\bSigma - \bSigma_0\right)\bU_{0}\tilde\bM_0\\
% &\quad = \bR_\bU(\bvarphi_{\bar{S}}, \bvarphi_{0\bar{S}})\transpose\bSigma_{0\bar{S}}^{-1}(\widehat\bSigma_{\bar{S}} - \bSigma_{0\bar{S}})\bU_{0\bar{S}}\tilde\bM_0.
\end{align*}
By the random matrix theory (see, for example, Section 5.4.1. in \citealp{vershynin2010introduction}), for any $t > 0$, 
\begin{align*}
\prob_0\left(\|\widehat\bOmega_{\bar{S}} - \bOmega_{0\bar{S}}\|_2 > C_0\sqrt{\frac{\kappa_0s_0 + t^2}{n}}\right) \leq 2\exp(-ct^2)
\end{align*}
for some absolute constant $c > 0$ and some constant $C_0 > 0$ that depends on $\|\bSigma_0\|_2$, where $\kappa_n = \kappa_0 s_0$ and $\bSigma_{0\bar{S}} = \bU_{0\bar{S}}\bM_0\bU_{0\bar{S}}\transpose + \eye_{\bar{s} + r}$. Therefore, with $t^2 = (2\kappa_0/c)(rs_0\log n + s_0\log p)$, we have, 
\begin{equation}
\label{eqn:RMT_sample_covariance}
\begin{aligned}
&\prob_0\left(
\sup_{S\in\calS(\kappa_0s_0)}\|\widehat\bOmega_{\bar{S}} - \bOmega_{0\bar{S}}\|_2 > C_0\sqrt{\frac{rs_0\log n + s_0\log p}{n}}\right)\\
&\quad\leq 2{p - r\choose \kappa_0s_0}\exp\left\{-2\kappa_0(rs_0\log n  + s_0\log p)\right\}\\
&\quad\leq2\exp\{-\kappa_0(rs_0\log n  + s_0\log p)\}
\end{aligned}
\end{equation}
Denote $\|\bA\|_* = \sum_i\sigma_i(\bA)$ the nuclear norm of a matrix. Therefore, by H\"older's inequality, the equivalence between nuclear norm and Frobenius norm, and Theorem \ref{thm:second_order_deviation_CT}, the left-hand side of \eqref{eqn:stochastic_remainder_I} is upper bounded by
\begin{align*}
&\sup_{S\in\calS(\kappa_0s_0)}\sup_{\btheta\in \calA_{n}(S)}2n\|\widehat\bOmega_{\bar{S}} - \bOmega_{0\bar{S}}\|_2\|\bU_{0\bar{S}}\tilde\bM_0\bR_\bU(\bvarphi_{\bar{S}}, \bvarphi_{0\bar{S}})\transpose\bOmega_{0\bar{S}}^{-1}\|_*\\
&\quad\leq 2n\sqrt{r}\sup_{S\in\calS(\kappa_0s_0)}\|\widehat\bOmega_{\bar{S}} - \bOmega_{0\bar{S}}\|_2
% \sup_{S\in\calS(\kappa_0s_0)}
\sup_{S\in\calS(\kappa_0s_0)}\sup_{\btheta\in \calA_{n}(S)}\|\bU_{0\bar{S}}\tilde\bM_0\bR_\bU(\bvarphi_{\bar{S}}, \bvarphi_{0\bar{S}})\transpose\bOmega_{0\bar{S}}^{-1}\|_{\mathrm{F}}\\
&\quad\lesssim n\sqrt{\frac{r(rs_0\log n + s_0\log p)}{n}} \sup_{S\in\calS(\kappa_0s_0)}\sup_{\btheta\in \calA_{n}(S)}\|\bR_\bU(\bvarphi_{\bar{S}}, \bvarphi_{0\bar{S}})\|_{\mathrm{F}}\\
&\quad\lesssim n\sqrt{\frac{r(rs_0\log n + s_0\log p)}{n}}\|\btheta - \btheta_0\|_1^2\\
&\quad\lesssim \sqrt{\frac{(r^2s^2\log n + rs^2\log p)^3}{n}} = o(1)
\end{align*}
with probability greater than $1 - 2\exp(-\kappa_0n\eps_n^2)\to 1$. This shows that the left-hand side of \eqref{eqn:stochastic_remainder_I} is asymptotically negligible by Assumption A4.  

\vspace*{2ex}\noindent
$\blacksquare$
For $\bR_\bOmega$, we have, using the permutation matrix $\bQ_S$,
\begin{align*}
\bU(\bvarphi)(\bM - \bM_0)\{\bU(\bvarphi) - \bU_{0}\}\transpose
& = \bQ_S\begin{bmatrix*}
\bU(\bvarphi_{\bar{S}})(\bM - \bM_0)\{\bU(\bvarphi_{\bar{S}}) - \bU_{0\bar{S}}\}\transpose & \zero \\
\zero & \zero
\end{bmatrix*}\bQ_S\transpose,\\
\{\bU(\bvarphi) - \bU_{0}\}\bM_0\{\bU(\bvarphi) - \bU_{0}\}\transpose
& = \bQ_S\begin{bmatrix*}
\{\bU(\bvarphi_{\bar{S}}) - \bU_{0\bar{S}}\}\bM_0\{\bU(\bvarphi_{\bar{S}}) - \bU_{0\bar{S}}\}\transpose & \zero \\
\zero & \zero
\end{bmatrix*}\bQ_S\transpose,\\
\{\bU(\bvarphi) - \bU_{0}\}(\bM - \bM_0)\bU_{0}\transpose
& = \bQ_S\begin{bmatrix*}
\{\bU(\bvarphi_{\bar{S}}) - \bU_{0\bar{S}}\}(\bM - \bM_0)\bU_{0\bar{S}}\transpose & \zero \\
\zero & \zero
\end{bmatrix*}\bQ_S\transpose,
\end{align*}
which implies that
\begin{align*}
\bR_\bOmega(\btheta, \btheta_{0})
& = \bQ_S\begin{bmatrix*}
\bR_\bOmega(\btheta_{\bar{S}}, \btheta_{0\bar{S}})
& \zero\\
\zero & \zero
\end{bmatrix*}\bQ_S\transpose,
\end{align*}
where
\begin{align*}
\bR_\bOmega(\btheta_{\bar{S}}, \btheta_{0\bar{S}})
& = \bU(\bvarphi_{\bar{S}})(\bM - \bM_0)\{\bU(\bvarphi_{\bar{S}}) - \bU_{0\bar{S}}\}\transpose
\\&\quad
 + \{\bU(\bvarphi_{\bar{S}}) - \bU_{0\bar{S}}\}\bM_0\{\bU(\bvarphi_{\bar{S}}) - \bU_{0\bar{S}}\}\transpose
\\&\quad
 + \{\bU(\bvarphi_{\bar{S}}) - \bU_{0\bar{S}}\}(\bM - \bM_0)\bU_{0\bar{S}}\transpose.
\end{align*}
Clearly, $\mathrm{rank}\{\bR_\bOmega(\btheta_{\bar{S}}, \btheta_{0\bar{S}})
\}\leq 3r$. Hence, the left-hand side of \eqref{eqn:stochastic_remainder_II} can be upper bounded similarly using H\"older's inequality \eqref{eqn:RMT_sample_covariance} by
\begin{align*}
&\sup_{\calB_{n}}\left|n\vect(\widehat\bOmega - \bOmega_0)\transpose(\bOmega_0^{-1}\otimes \bOmega_0^{-1})\vect\{\bR_\bOmega(\btheta, \btheta_{0})\}\right|\\
&\quad = 
\sup_{\calB_{n}}\left|n\mathrm{tr}\left\{(\widehat\bOmega - \bOmega_0)\bOmega_0^{-1}\bR_\bOmega(\btheta, \btheta_{0})\bOmega_0^{-1}\right\}\right|\\
&\quad = \sup_S\sup_{\calA_{n}(S)}\left|n\mathrm{tr}\left[
\begin{bmatrix*}
\widehat\bOmega_{\bar{S}} - \bOmega_{0\bar{S}} & \widehat\bOmega_{12}\\
\widehat\bOmega_{21} & \widehat\bOmega_{22} - \eye_{p - (\bar{s} + r)}
\end{bmatrix*}
\begin{bmatrix*}
\bOmega_{0\bar{S}}^{-1}\bR_\bOmega(\btheta_{\bar{S}}, \btheta_{0\bar{S}})\bOmega_{0\bar{S}}^{-1}
& \zero\\
\zero & \zero
\end{bmatrix*}
\right]\right|\\
&\quad = \sup_S\sup_{\calA_{n}(S)}\left|n\mathrm{tr}\left[
\begin{bmatrix*}
\widehat\bOmega_{\bar{S}} - \bOmega_{0\bar{S}} & \widehat\bOmega_{12}\\
\widehat\bOmega_{21} & \widehat\bOmega_{22} - \eye_{p - (\bar{s} + r)}
\end{bmatrix*}
\begin{bmatrix*}
\bOmega_{0\bar{S}}^{-1}\bR_\bOmega(\btheta_{\bar{S}}, \btheta_{0\bar{S}})\bOmega_{0\bar{S}}^{-1}
\\
\zero 
\end{bmatrix*}\begin{bmatrix*}
\eye_{\bar{s} + r} & \zero
\end{bmatrix*}
\right]\right|\\
&\quad = \sup_S\sup_{\calA_{n}(S)}\left|n\mathrm{tr}\left[\begin{bmatrix*}
\eye_{\bar{s} + r} & \zero
\end{bmatrix*}
\begin{bmatrix*}
\widehat\bOmega_{\bar{S}} - \bOmega_{0\bar{S}} & \widehat\bOmega_{12}\\
\widehat\bOmega_{21} & \widehat\bOmega_{22} - \eye_{p - (\bar{s} + r)}
\end{bmatrix*}
\begin{bmatrix*}
\bOmega_{0\bar{S}}^{-1}\bR_\bOmega(\btheta_{\bar{S}}, \btheta_{0\bar{S}})\bOmega_{0\bar{S}}^{-1}
\\
\zero 
\end{bmatrix*}
\right]\right|\\
% &\quad\leq \sup_{\calB_{n}(S)}\left|n\mathrm{tr}\left\{
% \begin{bmatrix*}
% \widehat\bOmega_{\bar{S}} - \bOmega_{0\bar{S}} & \widehat\bOmega_{12}\\
% \widehat\bOmega_{21} & \widehat\bOmega_{22} - \eye_{p - (\bar{s} + r)}
% \end{bmatrix*}
% \begin{bmatrix*}
% \bOmega_{0\bar{S}}^{-1}\bU(\bvarphi_{\bar{S}})\\\zero
% \end{bmatrix*}
% (\bM - \bM_0)\begin{bmatrix*}
% \{\bU(\bvarphi_{\bar{S}}) - \bU_{0\bar{S}}\}\transpose\bOmega_{0\bar{S}}^{-1} & \zero
% \end{bmatrix*}
% \right\}\right|\\
% &\quad\quad 
% + \sup_{\calB_{n}(S)}\left|n\mathrm{tr}\left\{
% \begin{bmatrix*}
% \widehat\bOmega_{\bar{S}} - \bOmega_{0\bar{S}} & \widehat\bOmega_{12}\\
% \widehat\bOmega_{21} & \widehat\bOmega_{22} - \eye_{p - (\bar{s} + r)}
% \end{bmatrix*}
% \begin{bmatrix*}
% \bOmega_{0\bar{S}}^{-1}\{\bU(\bvarphi_{\bar{S}}) - \bU_{0\bar{S}}\}\\\zero
% \end{bmatrix*}
% \bM_0\begin{bmatrix*}
% \{\bU(\bvarphi_{\bar{S}}) - \bU_{0\bar{S}}\}\transpose\bOmega_{0\bar{S}}^{-1} & \zero
% \end{bmatrix*}
% \right\}\right|\\
% &\quad\quad 
% + \sup_{\calB_{n}(S)}\left|n\mathrm{tr}\left\{
% \begin{bmatrix*}
% \widehat\bOmega_{\bar{S}} - \bOmega_{0\bar{S}} & \widehat\bOmega_{12}\\
% \widehat\bOmega_{21} & \widehat\bOmega_{22} - \eye_{p - (\bar{s} + r)}
% \end{bmatrix*}
% \begin{bmatrix*}
% \bOmega_{0\bar{S}}^{-1}\{\bU(\bvarphi_{\bar{S}}) - \bU_{0\bar{S}}\}\\\zero
% \end{bmatrix*}
% (\bM - \bM_0)\begin{bmatrix*}
% \bU_{0\bar{S}}\transpose\bOmega_{0\bar{S}}^{-1} & \zero
% \end{bmatrix*}
% \right\}\right|\\
&\quad \leq \sup_S\sup_{\calA_{n}(S)}\left|n\mathrm{tr}\left\{(\widehat\bOmega_{\bar{S}} - \bOmega_{0\bar{S}})
 \bOmega_{0\bar{S}}^{-1}\bR_\bOmega(\btheta_{\bar{S}}, \btheta_{0\bar{S}})\bOmega_{0\bar{S}}^{-1}\right\}\right|
 \\
&\quad = \sup_S\sup_{\calA_{n}(S)}\left|n\mathrm{tr}\left\{\bOmega_{0\bar{S}}^{-1}(\widehat\bOmega_{\bar{S}} - \bOmega_{0\bar{S}})
 \bOmega_{0\bar{S}}^{-1}\bR_\bOmega(\btheta_{\bar{S}}, \btheta_{0\bar{S}})\right\}\right|
 \\ 
%  &\quad\quad + 
%  \sup_{S}\sup_{\calA_{n}(S)}\left|\mathrm{tr}\left[\{\bU(\bvarphi_{\bar{S}}) - \bU_{0\bar{S}}\}\transpose\bOmega_{0\bar{S}}^{-1}(\widehat\bOmega_{\bar{S}} - \bOmega_{0\bar{S}})\bOmega_{0\bar{S}}^{-1}
%  \{\bU(\bvarphi_{\bar{S}}) - \bU_{0\bar{S}}\}\bM_0\right]\right|\\
% &\quad\quad + \sup_{S}\sup_{\calA_{n}(S)}\left|\mathrm{tr}\left[\bU_{0\bar{S}}\transpose\bOmega_{0\bar{S}}^{-1}\left(\widehat\bOmega_{\bar{S}} - \bOmega_{0\bar{S}}\right)\bOmega_{0\bar{S}}^{-1}\{\bU(\bvarphi_{\bar{S}}) - \bU_{0\bar{S}}\}(\bM - \bM_0)
% \right]\right|\\
&\quad\leq n\sup_S\|\bOmega_{0\bar{S}}^{-1}\|_2^2\|\widehat\bOmega_{\bar{S}} - \bOmega_{0\bar{S}}\|_2\sup_{S}\sup_{\btheta\in\calA_n(S)}\|\bR_\bOmega(\btheta_{\bar{S}}, \btheta_{0\bar{S}})\|_*\\
&\quad\leq n\sup_S\|\widehat\bOmega_{\bar{S}} - \bOmega_{0\bar{S}}\|_2\sup_{S}\sup_{\btheta\in\calA_n(S)}\sqrt{3r}\|\bR_\bOmega(\btheta_{\bar{S}}, \btheta_{0\bar{S}})\|_{\mathrm{F}}\\
&\quad\lesssim n\sqrt{\frac{r(rs_0\log n + s_0\log p)}{n}}\|\btheta - \btheta_{0}\|_1^2\lesssim \sqrt{
	\frac{(r^2s^2\log n + rs^2\log p)^3}{n}
} = o(1)
\end{align*}
with probability greater than $1 - 2\exp(-\kappa_0n\eps_n^2)\to 1$. Hence the left-hand side of \eqref{eqn:stochastic_remainder_II} is also $o_{\prob_0}(1)$ by Assumption A4. 

\vspace*{2ex}\noindent
$\blacksquare$
For $\bR_1$, we follow the same spirit and
let $\bOmega_{\bar{S}} = \bU(\bvarphi_{\bar{S}})\bM\bU(\bvarphi_{\bar{S}}) + \eye_{\bar{s} + r}$ and $\bOmega_{0\bar{S}} = \bU_{0\bar{S}}\bM_0\bU_{0\bar{S}}\transpose + \eye_{\bar{s} + r}$. Denote $\bOmega = \bOmega(\btheta)$. It follows that
\begin{align*}
(\bOmega_0 - \bOmega)\bOmega_0^{-1} 
& = \bQ_S
\begin{bmatrix*}
\bOmega_{0\bar{S}} - \bOmega_{\bar{S}} & \zero \\
\zero & \zero 
\end{bmatrix*}\bQ_S\transpose \bQ_S\begin{bmatrix*}
\bOmega_{0\bar{S}}^{-1} & \zero \\ \zero & \eye_{p - (\bar{s} + r)}
\end{bmatrix*}\bQ_S\transpose\\
& = \bQ_s\begin{bmatrix*}
(\bOmega_{0\bar{S}} - \bOmega_{\bar{S}})\bOmega_{0\bar{S}}^{-1} & \zero \\ \zero & \zero
\end{bmatrix*}\bQ_S\transpose
\end{align*}
Therefore, 
\begin{align*}
\bR_1(\bOmega, \bOmega_0)
& = \bQ_S\begin{bmatrix*}
\bOmega_{0\bar{S}}^{-1}(\bOmega_{0\bar{S}} - \bOmega_{\bar{S}})\bOmega_{0\bar{S}}^{-1} & \zero \\ \zero & \zero
\end{bmatrix*}\bQ_S\transpose\sum_{m = 1}^\infty \bQ_s\begin{bmatrix*}
\{(\bOmega_{0\bar{S}} - \bOmega_{\bar{S}})\bOmega_{0\bar{S}}^{-1}\}^m & \zero \\ \zero & \zero
\end{bmatrix*}\bQ_S\transpose\\
& = 
\bQ_S\begin{bmatrix*}
\bOmega_{0\bar{S}}^{-1}(\bOmega_{0\bar{S}} - \bOmega_{\bar{S}})\bOmega_{0\bar{S}}^{-1}\sum_{m = 1}^\infty\{(\bOmega_{0\bar{S}} - \bOmega_{\bar{S}})\bOmega_{0\bar{S}}^{-1}\}^m & \zero \\ \zero & \zero
\end{bmatrix*}\bQ_S\transpose\\
& := \bQ_S\begin{bmatrix*}
\bR_1(\bOmega_{\bar{S}}, \bOmega_{0\bar{S}}) & \zero \\ \zero & \zero
\end{bmatrix*}\bQ_S\transpose.
\end{align*}
Let $\bR_1(\bOmega_{\bar{S}}, \bOmega_{0\bar{S}})$ yield singular value decomposition $\bW_1\bS\bW_2\transpose$. 
% Clearly, $\mathrm{rank}(\bS) \leq 2r$, and f
Following the same reasoning, we have,
\begin{align*}
&\sup_{\calB_{n}}\left|n\vect(\widehat\bOmega - \bOmega_0)\vect\{\bR_1(\bOmega(\btheta), \bOmega_0)\}\right|\\
&\quad = \sup_{S}\sup_{\calA_{n}(S)}n\left|\mathrm{tr}\left\{
\bQ_S\begin{bmatrix*}
\widehat\bOmega_{\bar{S}} - \bOmega_{0\bar{S}} & \widehat\bOmega_{12}\\
\widehat\bOmega_{21} & \widehat\bOmega_{22} - \eye_{p - (\bar{s} + r)}
\end{bmatrix*}\bQ_S\transpose\bQ_S
\begin{bmatrix*}
\bW_1\\\zero
\end{bmatrix*}\bS\begin{bmatrix*}
\bW_2\transpose & \zero
\end{bmatrix*}
\bQ_S\transpose\right]\right|\\
&\quad = \sup_{S}\sup_{\calA_{n}(S)}n\left|\mathrm{tr}\left\{
\begin{bmatrix*}
\bW_2\transpose & \zero
\end{bmatrix*}\begin{bmatrix*}
\widehat\bOmega_{\bar{S}} - \bOmega_{0\bar{S}} & \widehat\bOmega_{12}\\
\widehat\bOmega_{21} & \widehat\bOmega_{22} - \eye_{p - (\bar{s} + r)}
\end{bmatrix*}
\begin{bmatrix*}
\bW_1\\\zero
\end{bmatrix*}\bS
\right]\right|\\
&\quad = \sup_{S}\sup_{\calA_{n}(S)}n\left|\mathrm{tr}\left\{\bW_2\transpose(
\widehat\bOmega_{\bar{S}} - \bOmega_{0\bar{S}})
\bW_1\right\}\bS\right|\\
&\quad\lesssim n\sqrt{s_0}\sup_{S\in\calS(\kappa_0s_0)}\|\widehat\bOmega_{\bar{S}} - \bOmega_{0\bar{S}}\|_2
\sup_{S}\sup_{\btheta\in\calA_{n}(S)}\|\bR_1(\bOmega_{\bar{S}}, \bOmega_{0\bar{S}})\|_{\mathrm{F}}\\
&\quad\lesssim \sqrt{ns_0(rs_0\log n + s_0\log p)}\|\bOmega(\btheta) - \bOmega_0\|_{\mathrm{F}}^2\lesssim  
\sqrt{
	\frac{(r^2s^2\log n + rs^2\log p)^3}{n}
} = o(1)
\end{align*}
with probability at least $1 - 2\exp(-\kappa_0n\eps_n^2)\to 1$ for some $c > 0$ by Assumption A4. This shows that the left-hand side of \eqref{eqn:stochastic_remainder_III} is also $o_{\prob_0}(1)$ and the proof is thus completed. 
\end{proof}

\begin{proof}[Proof of Theorem \ref{thm:LAN}]
We first consider the Taylor expansion of $\ell$ as a function of $\bOmega$ when $\|\bOmega - \bOmega_0\|_{\mathrm{F}}$ is sufficiently small. By definition, 
\begin{align*}
\ell(\bOmega) - \ell(\bOmega_0) & = \frac{n}{2}\mathrm{tr}\left\{\widehat\bOmega(\bOmega_0^{-1} - \bOmega^{-1})\right\} + \frac{n}{2}\log\det(\bOmega^{-1}\bOmega_0)\\
& = \frac{n}{2}\mathrm{tr}\left\{(\widehat\bOmega - \bOmega_0)(\bOmega_0^{-1} - \bOmega^{-1})\right\}\\
&\quad + \frac{n}{2}\mathrm{tr}(\eye_p - \bOmega_0\bOmega^{-1}) + \frac{n}{2}\log\det\{(\bOmega^{-1} - \bOmega_0^{-1})\bOmega_0 + \eye_p\}\\
& = \frac{n}{2}\mathrm{tr}\left\{(\widehat\bOmega - \bOmega_0)(\bOmega_0^{-1} - \bOmega^{-1})\right\} + \frac{n}{2}\mathrm{tr}\{\bOmega_0^{1/2}(\bOmega_0^{-1} - \bOmega^{-1})\bOmega_0^{1/2}\}\\
&\quad + \frac{n}{2}\log\det\{\eye_p - \bOmega_0^{1/2}(\bOmega_0^{-1} - \bOmega^{-1})\bOmega_0^{1/2}\}.
\end{align*}
Let $h_k = \lambda_k\{\bOmega_0^{1/2}(\bOmega_0^{-1} - \bOmega^{-1})\bOmega_0^{1/2}\}$, $k = 1,\ldots,2r$. Clearly, 
\begin{align*}
\max_{1\leq k\leq 2r}|h_k|&\leq \|\bOmega_0^{1/2}\|_2\|\|\bOmega_0^{-1} - \bOmega^{-1}\|_2\|\bOmega_0^{1/2}\|_2\\
&\leq \|\bOmega_0\|_2\|\bOmega^{-1}(\bOmega - \bOmega_0)\bOmega_0^{-1}\|_2\\
&\leq \|\bOmega_0\|_2\|\bOmega - \bOmega_0\|_2.
\end{align*}
Furthermore, using the Taylor expansion technique with the integral remainder (see, for example, Lemma 6.2 in \citealp{gao2016}), 
\begin{align*}
&\frac{n}{2}\mathrm{tr}\{\bOmega_0^{1/2}(\bOmega_0^{-1} - \bOmega^{-1})\bOmega_0^{1/2}\}
 + \frac{n}{2}\log\det\{\eye_p - \bOmega_0^{1/2}(\bOmega_0^{-1} - \bOmega^{-1})\bOmega_0^{1/2}\}\\
&\quad = \frac{n}{2}\sum_{k = 1}^{2r}\{h_k + \log(1 - h_k)\}
% \\&\quad
 = -\frac{n}{4}\sum_{k = 1}^{2r}h_k^2 - \frac{n}{2}\sum_{k = 1}^{2r}\int_0^{h_k}\frac{(h_k - s)^2}{(1 - s)^3}\mathrm{d}s\\
&\quad = -\frac{n}{4}\|\bOmega_0^{1/2}(\bOmega^{-1} - \bOmega_0^{-1})\bOmega_0^{1/2}\|_{\mathrm{F}}^2 - \frac{n}{2}\sum_{k = 1}^{2r}\int_0^{h_k}\frac{(h_k - s)^2}{(1 - s)^3}\mathrm{d}s.
\end{align*}
We now analyze the linear term $\mathrm{tr}\{(\widehat\bOmega - \bOmega_0)(\bOmega_0^{-1} - \bOmega^{-1})\}$ and the quadratic term $\|\bOmega_0^{1/2}(\bOmega_0^{-1} - \bOmega^{-1})\bOmega_0^{1/2}\|_{\mathrm{F}}^2$ separately. By the matrix series expansion, 
\begin{align*}
\bOmega_0^{-1} - \bOmega^{-1} & = \bOmega_0^{-1}(\bOmega - \bOmega_0)(\bOmega - \bOmega_0 + \bOmega_0)^{-1}\\
& = \bOmega_0^{-1}(\bOmega - \bOmega_0)\bOmega_0^{-1}\{(\bOmega - \bOmega_0)\bOmega_0^{-1} + \eye_p\}^{-1}\\
& = \bOmega_0^{-1}(\bOmega - \bOmega_0)\bOmega_0^{-1} + \bR_1(\bOmega, \bOmega_0),
\end{align*}
where the remainder
\begin{align*}
\bR_1(\bOmega, \bOmega_0) = \bOmega_0^{-1}(\bOmega_0 - \bOmega)\bOmega_0^{-1}\sum_{m = 1}^\infty\{(\bOmega_0 - \bOmega)\bOmega_0^{-1}\}^m
\end{align*}
satisfies $\|\bR_1(\bOmega, \bOmega_0)\|_{\mathrm{F}}\lesssim\|\bOmega - \bOmega_0\|_{\mathrm{F}}^2$. 
The vectorization form of the previous equation can be written as
\begin{align*}
\vect(\bOmega_0^{-1} - \bOmega^{-1}) = (\bOmega_0^{-1}\otimes \bOmega_0^{-1})\vect(\bOmega - \bOmega_0) + \vect\{\bR_1(\bOmega, \bOmega_0)\}.
\end{align*} 
Now we consider parameterize $\bOmega$ by $\bOmega = \bOmega(\btheta)$. It follows from Theorem \ref{thm:CT_deviation_Sigma} that
\[
\|\bR_1(\bOmega(\btheta), \bOmega_0)\|_{\mathrm{F}}\lesssim\|\bOmega(\btheta) - \bOmega_0\|_{\mathrm{F}}^2
\lesssim \|\btheta - \btheta_{0}\|_1^2.
\]
Following the proof of Theorem \ref{thm:CT_deviation_Sigma}, we obtain the following matrix decomposition
\begin{align*}
\bOmega(\btheta) - \bOmega_0
& = \bSigma(\btheta) - \bSigma(\btheta_{0})\\
& = \bU_{0}(\bM - \bM_0)\bU_{0}\transpose + \bU_{0}\bM_0\{\bU(\bvarphi) - \bU_{0}\}\transpose + \{\bU(\bvarphi) - \bU_{0}\}\bM_0\bU_{0}\transpose\\
&\quad + \bR_{\bOmega}(\btheta, \btheta_{0}),
\end{align*}
where the remainder
\begin{align*}
R_\bOmega(\btheta, \btheta_{0}) & = \bU(\bvarphi)(\bM - \bM_0)\{\bU(\bvarphi) - \bU_{0}\}\transpose + \{\bU(\bvarphi) - \bU_{0}\}\bM_0\{\bU(\bvarphi) - \bU_{0}\}\transpose\\
&\quad + \{\bU(\bvarphi) - \bU_{0}\}(\bM - \bM_0)\bU_0\transpose
\end{align*}
satisfies $\|R_\bOmega(\btheta, \btheta_{0})\|_{\mathrm{F}}\lesssim \|\btheta - \btheta_{0}\|_{2}^2\lesssim \|\btheta - \btheta_{0}\|_1^2$.
In the vectorization form, we can write
\begin{align*}
\vect(\bSigma - \bSigma_0)
& = D_\bmu\bSigma(\btheta_{0})(\bmu - \bmu_0) + (\eye_{p^2} + \bK_{pp})(\bU_{0}\bM_0\otimes \eye_p)\vect\{\bU(\bvarphi) - \bU_{0}\}\\
&\quad + \vect\{\bR_{\bOmega}(\btheta, \btheta_{0})\}.
\end{align*}
In addition, by Theorem \ref{thm:second_order_deviation_CT}, we have, 
\begin{align*}
\bU(\bvarphi) - \bU_{0} = 2(\eye_p - \bX_{0})^{-1}(\bX_\bvarphi - \bX_{0})(\eye_p - \bX_{0})^{-1}\eye_{p\times r} + \bR_\bU(\bvarphi, \bvarphi_{0}),
\end{align*}
where $\bU_{0} = \bU(\bvarphi_{0})$, and $\bR_\bU(\bvarphi, \bvarphi_{0})$ satisfies $\|\bR_\bU(\bvarphi, \bvarphi_{0})\|_{\mathrm{F}}\lesssim \|\btheta - \btheta_{0}\|_1^2$. The vector version of the previous display can be written as
\begin{align*}
\vect\{\bU(\bvarphi) - \bU_{0}\} = D\bU(\bvarphi_{0})(\bvarphi - \bvarphi_{0}) + \vect\{\bR_\bU(\bvarphi, \bvarphi_{0})\},
\end{align*}
where the Fr\'echet derivative $D\bU$ is defined by \eqref{eqn:Differential_CT}. Recall that $D\bSigma(\btheta_{0})$ is defined by \eqref{eqn:covariance_differential}. It follows from the above derivations that
\begin{align*}
\vect\{\bOmega(\btheta) - \bOmega_0\} & = D\bSigma(\btheta_{0})(\btheta - \btheta_{0})\\
&\quad + (\eye_{p^2} + \bK_{pp})(\bU_{0}\bM_0\otimes \eye_p)\vect\{\bR_\bU(\bvarphi, \bvarphi_{0})\}
       + \vect\{\bR_\bOmega(\btheta, \btheta_{0})\}.
\end{align*}
This means that over $\calB_{n}$, we can have well control of the Frobenius norm deviation $\|\bOmega(\btheta) - \bOmega_0\|_{\mathrm{F}}$:
For any $\btheta\in\calB_{n}$,
\begin{align*}
\|\bOmega(\btheta) - \bOmega_0\|_{\mathrm{F}} & = \|\vect\{\bOmega(\btheta)\} - \vect(\bOmega_0)\|_2\\
& \leq \|D\bSigma(\btheta_{0})\|_2\|\btheta -  \btheta_{0}\|_2 + 2\|\bM_0\|_2\|\bR_\bU(\bvarphi, \bvarphi_{0})\|_{\mathrm{F}}
% \\&\quad
 + \|\bR_\bOmega(\btheta, \btheta_{0})\|_{\mathrm{F}}\\
&\lesssim \|\btheta - \btheta_{0}\|_1\lesssim \sqrt{\frac{r^2s^2\log n + rs^2\log p}{n}}.
\end{align*}
% where we have assume that $\eps_n(\btheta)\to 0$ and $n\eps_n(\btheta)^3\to 0$. 
Hence for the precision matrix $\bOmega^{-1}$, we have
\begin{align*}
\vect(\bOmega_0^{-1} - \bOmega^{-1}) & = (\bOmega_0^{-1}\otimes \bOmega_0^{-1})D\bSigma(\btheta_{0})(\btheta - \btheta_{0})\\
&\quad + (\bOmega_0^{-1}\otimes \bOmega_0^{-1})(\eye_{p^2} + \bK_{pp})(\bU_{0}\bM_0\otimes \eye_p)\vect\{\bR_\bU(\bvarphi, \bvarphi_{0})\}\\
&\quad + (\bOmega_0^{-1}\otimes \bOmega_0^{-1})\vect\{\bR_\bOmega(\btheta, \btheta_{0})\}
% \\&\quad
 + \vect\{\bR_1(\bOmega(\btheta), \bOmega_0)\},
\end{align*}
with the remainder $\bR_1$ satisfies $\|\bR_1(\bOmega(\btheta), \bOmega_0)\|_{\mathrm{F}}\lesssim \|\bOmega(\btheta) - \bOmega_0\|_{\mathrm{F}}^2\lesssim \|\btheta - \btheta_{0}\|_1^2$. 
Denote
\begin{align*}
\mathbf{r}_{\bOmega^{-1}}(\btheta, \btheta_{0})
&: = (\bOmega_0^{-1}\otimes \bOmega_0^{-1})(\eye_{p^2} + \bK_{pp})(\bU_{0}\bM_0\otimes \eye_p)\vect\{\bR_\bU(\bvarphi, \bvarphi_{0})\}\\
&\quad + (\bOmega_0^{-1}\otimes \bOmega_0^{-1})\vect\{\bR_\bOmega(\btheta, \btheta_{0})\}
% \\&\quad
 + \vect\{\bR_1(\bOmega(\btheta), \bOmega_0)\}\\
& = (\eye_{p^2} + \bK_{pp})(\bOmega_0^{-1}\otimes \bOmega_0^{-1})\vect\{\bR_\bU(\bvarphi, \bvarphi_{0})\bM_0\bU_0\transpose\}\\
&\quad + (\bOmega_0^{-1}\otimes \bOmega_0^{-1})\vect\{\bR_\bOmega(\btheta, \btheta_{0})\}
% \\&\quad
 + \vect\{\bR_1(\bOmega(\btheta), \bOmega_0)\}.
\end{align*}
Clearly, $\|\mathbf{r}_{\bOmega^{-1}}(\btheta, \btheta_{0})\|_2\lesssim \|\btheta - \btheta_{0}\|_1^2$ by the properties of the remainders $\bR_1$, $\bR_\bOmega$, and $\bR_\bU$. 
It follows that
\begin{align*}
\mathrm{tr}\{(\widehat\bOmega - \bOmega_0)(\bOmega_0^{-1} - \bOmega^{-1})\}
& = \vect(\widehat\bOmega - \bOmega_0)\transpose (\bOmega_0^{-1}\otimes\bOmega_0^{-1})D\bSigma(\btheta_{0})(\btheta - \btheta_{0})\\
&\quad + \vect(\widehat\bOmega - \bOmega_0)\transpose\mathbf{r}_{\bOmega^{-1}}(\btheta, \btheta_{0}),
\end{align*}
and
\begin{align*}
\|\bOmega_0^{1/2}(\bOmega^{-1} - \bOmega_0^{-1})\bOmega_0^{1/2}\|_{\mathrm{F}}^2
& = \mathrm{tr}\left\{(\bOmega^{-1} - \bOmega_0^{-1})\bOmega_0(\bOmega^{-1} - \bOmega_0^{-1})\bOmega_0\right\}\\
& = \vect(\bOmega_0^{-1} - \bOmega^{-1})(\bOmega_0\otimes \bOmega_0)\vect(\bOmega_0^{-1} - \bOmega^{-1})\\
& = (\btheta - \btheta_{0})\transpose D\bSigma(\btheta_{0})(\bOmega_0^{-1}\otimes \bOmega_0^{-1}) D\bSigma(\btheta_{0})(\btheta - \btheta_{0})\\
&\quad + 2(\btheta - \btheta_{0})\transpose D\bSigma(\btheta_{0})\transpose \mathbf{r}_{\bOmega^{-1}}(\btheta, \btheta_{0})\\
&\quad + \mathbf{r}_{\bOmega^{-1}}(\btheta, \btheta_{0})\transpose(\bOmega_0\otimes \bOmega_0)\mathbf{r}_{\bOmega^{-1}}(\btheta, \btheta_{0}).
\end{align*}
Denote
\begin{align*}
r_q(\btheta, \btheta_{0})& = 2(\btheta - \btheta_{0})\transpose D\bSigma(\btheta_{0})\transpose \mathbf{r}_{\bOmega^{-1}}(\btheta, \btheta_{0})
% \\&\quad
 +  \mathbf{r}_{\bOmega^{-1}}(\btheta, \btheta_{0})\transpose(\bOmega_0\otimes \bOmega_0)\mathbf{r}_{\bOmega^{-1}}(\btheta, \btheta_{0}).
\end{align*}
By the property of the remainder $\mathbf{r}_{\bOmega^{-1}}$, we see that $|r_q(\btheta, \btheta_{0})|\lesssim \|\btheta - \btheta_{0}\|_1^3$.

\vspace*{2ex}\noindent
Now putting all the above derivations together, we obtain the following expansion of the log-likelihood function:
\begin{align*}
\ell(\bOmega(\btheta)) - \ell(\bOmega(\btheta_{0}))
& = \frac{n}{2}\vect(\widehat\bOmega - \bOmega_0)(\bOmega_0^{-1}\otimes\bOmega_0^{-1})D\bSigma(\btheta_{0})\\
&\quad - \frac{n}{4}(\btheta - \btheta_{0})\transpose D\bSigma(\btheta_{0})\transpose(\bOmega_0^{-1}\otimes \bOmega_0^{-1}) D\bSigma(\btheta_{0})(\btheta - \btheta_{0})\\
&\quad - \frac{n}{2}\sum_{k = 1}^{2r}\int_0^{h_k}\frac{(h_k - s)^2}{(1 - s)^3}\mathrm{d}s - \frac{n}{4}r_q(\btheta, \btheta_{0})\\
&\quad + \frac{n}{2}\vect(\widehat\bOmega - \bOmega_0)\transpose \mathbf{r}_{\bOmega^{-1}}(\btheta, \btheta_{0}).
\end{align*}
The third line of the previous equation is the deterministic remainder and the fourth line is the stochastic remainder. 
For the sum of the integrals in the third line of the above display, since $\|\bOmega(\btheta) - \bOmega_0\|_{\mathrm{F}} \to 0$, we may assume that $\max_k|h_k|\leq 1/2$, and hence, 
\begin{align*}
\sup_{\btheta\in\calA_{n}(S)}\left|
\frac{n}{2}\sum_{k = 1}^{2r}\int_0^{h_k}\frac{(h_k - s)^2}{(1 - s)^3}\mathrm{d}s
\right|
&\lesssim n\sup_{\|\bOmega(\btheta) - \bOmega_0\|_{\mathrm{F}} < M\eps_n}\sum_{k = 1}^{2r}\int_0^{h_k}(h_k - s)^2\mathrm{d}s\\
&\leq n\max_{1\leq k\leq 2r}|h_k|\frac{1}{3}\sum_{k = 1}^{2r}h_k^2
% \\&
\lesssim n\|\bOmega(\btheta) - \bOmega_0\|_{\mathrm{F}}^3\\
&\lesssim n\|\btheta - \btheta_0\|_1^3
\lesssim \sqrt{
	\frac{(r^2s^2\log n + rs^2\log p)^3}{n}
}\to 0.
\end{align*}
The stochastic remainder is given by
\begin{align*}
\frac{n}{2}\vect(\widehat\bOmega - \bOmega_0)\transpose \mathbf{r}_{\bOmega^{-1}}(\btheta, \btheta_{0})
& = n\vect(\widehat\bOmega - \bOmega_0)(\bOmega_0^{-1}\otimes\bOmega_0^{-1})\vect\{\bR_\bU(\bvarphi, \bvarphi_0)\bM_0\bU_0\transpose\}\\
&\quad + \frac{n}{2}\vect(\widehat\bOmega - \bOmega_0)\transpose(\bOmega_0^{-1}\otimes\bOmega_0^{-1})\vect\{\bR_\bOmega(\btheta, \btheta_0)\}\\
&\quad + \frac{n}{2}\vect(\widehat\bOmega - \bOmega_0)\transpose\vect\{\bR_1(\bOmega(\btheta),\bOmega_0)\}.
\end{align*}
By Lemma \ref{lemma:stochastic_remainder}, the supremum of the stochastic remainder over $\btheta\in\calB_{n}$ is also $o_{\prob_0}(1)$, and hence completing the proof. 
\end{proof}

% subsection local_asymptotic_normality (end)

\subsection{Distributional approximation: Proof of Theorem \ref{thm:BvM}} % (fold)
\label{sub:distributional_approximation}

This subsection elaborates on the proof of Theorem \ref{thm:BvM}. We remark that the proof is a generalization of the proof of Theorem 6 in \cite{castillo2015} modulus a local asymptotic normality argument developed in Section \ref{sub:local_asymptotic_normality}. For convenience, we introduce additional notations that will be used to characterize the limit shape of the posterior distribution. 
\begin{align*}
\begin{aligned}
\bZ_{0} & = \sqrt{\frac{n}{2}}(\bOmega_0^{-1/2}\otimes \bOmega_0^{-1/2})D\bSigma(\btheta_{0}),
% \\
\quad
\beps_n   = \sqrt{\frac{n}{2}}(\bOmega_0^{-1/2}\otimes \bOmega_0^{-1/2})\vect(\widehat\bOmega - \bOmega_0).
\end{aligned}
\end{align*}
Let $\bZ_{0S}: = \bZ_0\bF_S$. Then $\widehat{\btheta}_S$, $\eye_S(\btheta_0)$, and $\widehat{w}_S$ can be equivalently written as
\begin{align*}
% \label{eqn:theta_hat}
\widehat\btheta_{S} &= (\bZ_{0S}\transpose\bZ_{0S})^{-1} \bZ_{0S}\transpose(\bZ_{0}\btheta_{0} + \beps_n),\quad \eye_S(\btheta_0) = \bZ_{0S}\transpose\bZ_{0S},\\
% \label{eqn:w_hat}
\widehat{w}_{S} &\propto \frac{\pi_p(|S|)}{{p - r\choose |S|}\gamma(|S|)}\det\{2\pi(\bZ_{0S}\transpose\bZ_{0S})^{-1}\}^{1/2}\exp\left(\frac{1}{2}\|\bZ_{0S}\widehat\btheta_{S}\|_2^2\right),
\end{align*}
where $\gamma(|S|)$ is defined by
\[
\gamma(|S|) = \int_{\|\bA_S\|_2 < 1}\exp(-2\|\vect(\bA_S)\|_1)\mathrm{d}\bA_S.
\]
The proof is based on the following collection of technical lemmas. Recall that the sub-Gaussian norm and the sub-exponential norm of a random variable $X$ is defined by
\begin{align*}
\|X\|_{\psi_2} = \sup_{p\geq 1}p^{-1/2}\{\expect_0(|X|^p)\}^{1/p},\quad
\|X\|_{\psi_1} = \sup_{p\geq 1}p^{-1}\{\expect_0(|X|^p)\}^{1/p}.
\end{align*}
We refer to \cite{vershynin2010introduction} for a detailed review on the concept of these (Orlicz) norms. 

\begin{lemma}\label{lemma:residual_linfty_deviation}
Under the setup in Section \ref{sub:bayesian_sparse_pca_and_non_intrinsic_loss}, there exists some constant $C_0, c_0 > 0$ only depending on the spectra of $\bSigma_0$, such that
\begin{align*}
\prob_0\left\{|(\btheta - \btheta_{0})\transpose\bZ_{0}\transpose\beps_n| > C_0\|\btheta - \btheta_{0}\|_1\sqrt{n\log p}\right\}\leq \frac{2}{p^{c_0}}.
\end{align*}
\end{lemma}
\begin{proof}[Proof of Lemma \ref{lemma:residual_linfty_deviation}]
% Let $\btheta = [\bvarphi\transpose, \bmu\transpose]\transpose$, $\btheta_{0} = [\bvarphi_{0}\transpose, \bmu_0\transpose]\transpose$, and $\bU_{0} = \bU(\bvarphi_{0}) = [\bu_{01},\ldots,\bu_{0r}]$. 
Denote 
\[
\bX_0 = \bX_{\bvarphi_0} = \begin{bmatrix*}
\zero_{r\times r} & -\bA_0\transpose \\
\bA_0 & \zero_{(p - r)\times (p - r)}
\end{bmatrix*}\quad\text{and}\quad
\bC_0 = (\eye_p - \bX_{0})^{-1}.
\]
By definition, we have
\begin{align*}
&\frac{2}{n}(\btheta - \btheta_{0})\transpose\bZ_{0}\transpose\beps_n\\
&\quad = (\bvarphi - \bvarphi_{0})\transpose D\bU(\bvarphi_{0})\transpose(\bM_0\bU_{0}\transpose\otimes\eye_p)(\eye_{p^2} + \bK_{pp})(\bOmega_{0}^{-1}\otimes\bOmega_0^{-1})\vect(\widehat\bOmega - \bOmega_0)\\
&\quad\quad + (\bmu - \bmu_0)\transpose D_\bmu\bSigma(\btheta_{0})\transpose(\bOmega_0^{-1}\otimes\bOmega_0^{-1})\vect(\widehat\bOmega -\bOmega_0)\\
&\quad = 4\vect(\bX_\bvarphi - \bX_{0})\transpose \{(\eye_p - \bX_{0})^{-1}\eye_{p\times r}\bM_0\bU_{0}\transpose\otimes (\eye_p - \bX_{0})\inverseT\}
\vect\{\bOmega_0^{-1}(\widehat\bOmega - \bOmega_0)\bOmega_0^{-1}\}\\
&\quad\quad + (\bmu - \bmu_0)\transpose\mathbb{D}_r\transpose\vect\{\bOmega_0^{-1}(\widehat\bOmega - \bOmega_0)\bOmega_0^{-1}\}
% \begin{bmatrix*}
% \bu_{01}\transpose\bSigma_0^{-1}(\widehat\bSigma - \bSigma_0)\bSigma_0^{-1}\bu_{01}\\
% \vdots\\
% \bu_{0r}\transpose\bSigma_0^{-1}(\widehat\bSigma - \bSigma_0)\bSigma_0^{-1}\bu_{0r}\}
% \end{bmatrix*}
\\
&\quad = 4\vect(\bX_\bvarphi - \bX_{0})\transpose\vect\{\bC_{0}\transpose\bOmega_0^{-1}(\widehat\bOmega - \bOmega_0)\bOmega_0^{-1}\bU_{0}\bM_0\eye_{p\times r}\transpose\bC_{0}\transpose\} + \\
&\quad\quad + \vect(\bM - \bM_0)\transpose\vect\{\bU_0\transpose\bOmega_0^{-1}(\widehat\bOmega - \bOmega_0)\widehat\bOmega_0^{-1}\bU_0\}
\end{align*}
Let $\be_j$ be the standard basis vector along the $j$th coordinate in $\mathbb{R}^p$, i.e., the $j$th coordinate being $1$ and the rest of the coordinates being zeros, $\balpha_i = \bOmega_0^{-1}\bC_{0}\be_i$, $\bbeta_j = \bOmega_0^{-1}\bU_{0}\bM_0\eye_{p\times r}\transpose\bC_{0}\transpose\be_j$, and $\bgamma_k = \bOmega_0^{-1}\bU_0\be_k$. Then by the H\"older's inequality, 
\begin{align*}
|(\btheta - \btheta_{0})\transpose\bZ_{0}\transpose\beps_n|
&\leq 2n\|\vect(\bX_\bvarphi - \bX_{0})\|_1\max_{j,h\in[p]}\left|\balpha_j\transpose(\widehat\bOmega - \bOmega_0)\bbeta_h\right|\\
&\quad + n\|\bmu - \bmu_0\|_1\max_{k,l\in[r]}\left|\bgamma_k\transpose(\widehat\bOmega - \bOmega_0)\bgamma_l\right|\\
&\leq n\|\btheta - \btheta_{0}\|_1\max_{j,h\in[p],k,l\in[r]}\left\{2\left|\balpha_j\transpose(\widehat\bOmega - \bOmega_0)\bbeta_h\right|, \left|\bgamma_k\transpose(\widehat\bOmega - \bOmega_0)\bgamma_l\right|\right\}.
\end{align*}
Observe that 
\begin{align*}
\|\bC_{0}\|_2^2 &= \lambda_{\max}(\bC_{0}\transpose\bC_{0}) = \lambda_{\min}^{-1}\{(\eye + \bX_{0})(\eye - \bX_{0})\} = \lambda_{\min}^{-1}(\eye + \bX_{0}\bX_{0}\transpose)\leq 1, 
\end{align*}
and that
\begin{align*}
\|\balpha_j\|_2&\leq \|\bOmega_0^{-1}\|_2\|\bC_{0}\|_2 \leq 1,\quad
\|\bbeta_j\|_2\leq \|\bOmega_0^{-1}\|_2\|\bLambda_0\|_2\|\bC_{0}\|_2\leq \|\bOmega_0\|_2,\\
\|\bgamma_k\|_2&\leq \|\bOmega_0^{-1}\|_2\leq 1,\quad
\|\bu\transpose\by_i\|_{\psi_2}\lesssim (\bu\transpose\bOmega_0\bu)^{1/2}\leq \|\bOmega_0\|_2^{1/2}\|\bu\|_2\quad\text{for all }\bu \in\mathbb{R}^p. 
\end{align*}
It follows from the properties of Orlicz norms that
\begin{align*}
&\max_{j,h\in[p]}\|\balpha_j\transpose\by_i\by_i\transpose\bbeta_h\|_{\psi_1}\leq \max_{j\in[p]}\|\balpha_j\transpose\by_i\|_{\psi_2}\max_{j\in[p]}\|\bbeta_j\transpose\by_i\|_{\psi_2} = O(1),\\
&\max_{k,l\in[r]}\|\bgamma_k\transpose\by_i\by_i\transpose\bgamma_l\|_{\psi_1}\leq \left(\max_{k\in[r]}\|\bgamma_k\transpose\by_i\|_{\psi_2}\right)^2 = O(1), 
\end{align*}
and hence, by the union bound and the Bernstein-type inequality for sub-exponential random variables (see, for example, Proposition 5.16 in \citealp{vershynin2010introduction}), we have, for any $t > 0$,
\begin{align*}
&\prob_0\left(|(\btheta - \btheta_{0})\transpose\bZ_{0}\transpose\beps_n| > t\|\btheta - \btheta_{0}\|_1\right)\\
&\quad\leq \sum_{j = 1}^p\sum_{h = 1}^p\prob_0\left(\left|\balpha_j\transpose(\widehat\bOmega - \bOmega_0)\bbeta_h\right| > \frac{t}{2n}\right) + \sum_{k = 1}^r\sum_{l = 1}^r\prob_0\left(\left|\bgamma_k\transpose(\widehat\bOmega - \bOmega_0)\bgamma_l\right| > \frac{t}{n}\right)\\
&\quad = \sum_{j = 1}^p\sum_{h = 1}^p\prob_0\left(\left|\frac{1}{n}\sum_{i = 1}^n(\balpha_j\transpose\by_i\by_i\transpose\bbeta_h - \balpha_j\transpose\bOmega_0\bbeta_h)\right| > \frac{t}{2n}\right)\\
&\quad\quad + \sum_{k = 1}^r\sum_{l = 1}^r\prob_0\left(\left|\frac{1}{n}\sum_{i = 1}^n(\bgamma_k\transpose\by_i\by_i\transpose\bgamma_l - \bgamma_k\transpose\bOmega_0\bgamma_l)\right| > \frac{t}{n}\right)\\
&\quad\leq 2(p^2 + r^2)\exp\left\{-C_0\min\left(\frac{t^2}{n}, t\right)\right\}\\
&\quad\leq 4\exp\left\{2\log p-C_0\min\left(\frac{t^2}{n}, t\right)\right\}
\end{align*}
for some constant $C_0 > 0$ (possibly depending on the spectra of $\bOmega_0$). The proof is completed by taking $t = (4/C_0)^{1/2}\sqrt{n\log p}$
\end{proof}

\begin{lemma}\label{lemma:normal_restriction}
Under the prior specification and the setup in Section \ref{sub:bayesian_sparse_pca_and_non_intrinsic_loss} as well as Assumptions A1-A5, there exists some constant $M > 0$ such that
\begin{align*}
\expect_0\left\{\Pi^{\infty}_\btheta\left(\|\btheta - \btheta_0\|_1 > M\sqrt{\frac{r^2s^2\log n + rs^2\log p}{n}}\mathrel{\Bigg|}\bY_n\right)\right\} \to 0.
\end{align*}
\end{lemma}

\begin{proof}[Proof of Lemma \ref{lemma:normal_restriction}]
Denote
\[
\widetilde{\calA}_{n} = \left\{
\btheta:\|\btheta - \btheta_{0}\|_1\leq M\sqrt{\frac{r^2s^2\log n + rs^2\log p}{n}}
\right\}
\]
and
\[
Q_{S}(\mathrm{d}\btheta) = 
\{\phi(\btheta_{S}\mid\widehat\btheta_{S}, (\bZ_{0S}\transpose\bZ_{0S})^{-1})\mathrm{d}\btheta_{S}\}\{\delta_{\zero}(\mathrm{d}\btheta_{S^c})\}.
\]
By definition, we can write
\begin{align*}
&\Pi^\infty_\btheta(\btheta\in\widetilde{\calA}_n^c\mid\bY_n)\\
 % & = \frac{\Pi^\infty(\widetilde{\calA}_n\mid\bY_n)}{\Pi^\infty(\mathbb{R}^{d + r(r + 1)/2}\mid\bY_n)}\\
&\quad = 
% \frac{
\left[
\sum_{S\in \calS_0}\frac{\pi_p(|S|)\exp(\|\bZ_{0S}\widehat{\btheta}_S\|_2^2/2)}{{p - r \choose |S|}\gamma(|S|)}\det\{2\pi(\bZ_{0S}\transpose\bZ_{0S})^{-1}\}^{1/2}Q_S(\widetilde{\calA}_n^c)
% \int_{\calA_n(S)}\exp\{-\frac{1}{2}(\btheta_S - \widehat{\btheta}_S)\transpose\bZ_{0S}\transpose\bZ_{0S}(\btheta_S - \widehat{\btheta}_S)\}\mathrm{d}\btheta_S
\right]\\
&\quad\quad\times
\left[
\sum_{S\in \calS_0}\frac{\pi_p(|S|)\exp(\|\bZ_{0S}\widehat{\btheta}_S\|_2^2/2)}{{p - r \choose |S|}\gamma(|S|)}\det\{2\pi(\bZ_{0S}\transpose\bZ_{0S})^{-1}\}^{1/2}Q_S(\mathbb{R}^{(p - r)r + r(r + 1)/2})
\right]^{-1}.
\end{align*}
Using the fact that for any fixed index set $S\in \calS_0$, and any measurable set $\calB\subset\mathbb{R}^{(p - r)r + r(r + 1)/2}$
\begin{align*}
\det\{2\pi(\bZ_{0S}\transpose\bZ_{0S})^{-1}\}^{1/2}Q_S(\calB)
& = \int_{\calB_S}\exp\left\{-\frac{1}{2}\|\bZ_{0S}(\btheta_S - \widehat{\btheta}_S)\|_2^2\right\}\mathrm{d}\btheta_S,
\end{align*}
where $\calB_S = \{\btheta_S:[\btheta_S\transpose, \zero_{S^c}\transpose]\transpose\in \calB\}$ is the intersection of $\calB$ with the subspace $\mathbb{R}^{|S|r + r(r + 1)/2}$, we write
\begin{align*}
&\Pi^\infty_\btheta(\btheta\in\widetilde{\calA}_n^c\mid\bY_n) \\
&\quad = 
% \frac{
\left[
\sum_{S\in \calS_0}\frac{\pi_p(|S|)\exp(\|\bZ_{0S}\widehat{\btheta}_S\|_2^2/2)}{{p - r \choose |S|}\gamma(|S|)}\int_{(\widetilde{\calA}_n^c)_S}\exp\left\{-({1}/{2})\|\bZ_{0S}(\btheta_S - \widehat{\btheta}_S)\|_2^2\right\}\mathrm{d}\btheta_S
\right]\\
% }{
&\quad\quad\quad\times
\left[
\sum_{S\in \calS_0}\frac{\pi_p(|S|)\exp(\|\bZ_{0S}\widehat{\btheta}_S\|_2^2/2)}{{p - r \choose |S|}\gamma(|S|)}
\int \exp\left\{-({1}/{2})\|\bZ_{0S}(\btheta_S - \widehat{\btheta}_S)\|_2^2\right\}\mathrm{d}\btheta_S
% }
\right]^{-1}
\\
&\quad = 
% \frac{
\left[
\sum_{S\in \calS_0}\frac{\pi_p(|S|)}{{p - r \choose |S|}\gamma(|S|)}\int_{(\widetilde{\calA}_n^c)_S}\exp\left\{\frac{1}{2}\|\bZ_{0S}\widehat{\btheta}_S\|_2^2-\frac{1}{2}\|\bZ_{0S}(\btheta_S - \widehat{\btheta}_S)\|_2^2\right\}\mathrm{d}\btheta_S
\right]\\
% }{
&\quad\quad\quad\times
\left[
\sum_{S\in \calS_0}\frac{\pi_p(|S|)}{{p - r \choose |S|}\gamma(|S|)}
\int \exp\left\{\frac{1}{2}\|\bZ_{0S}\widehat{\btheta}_S\|_2^2-\frac{1}{2}\|\bZ_{0S}(\btheta_S - \widehat{\btheta}_S)\|_2^2\right\}\mathrm{d}\btheta_S
% }
\right]^{-1}
.
\end{align*}
Let $\bt_n = \bZ_0\btheta_0 + \beps_n$. Observe that $\bZ_{0S}\widehat{\btheta}_S$ is the projection of $\bt_n$ onto the subspace spanned by the columns of $\bZ_{0S}$, and by Parseval's identity, we have
\begin{align*}
% &
\frac{1}{2}\|\bZ_{0S}\widehat{\btheta}_S\|_2^2 - \frac{1}{2}\|\bZ_{0S}(\btheta_S - \widehat{\btheta}_S)\|_2^2
% \\
&
% \quad
 = \frac{1}{2}\|\bt_n\|_2^2 - \frac{1}{2}\| \bt_n - \bZ_{0S}\widehat{\btheta}_{0S} \|_2^2 - \frac{1}{2}\|\bZ_{0S}(\btheta_S - \widehat{\btheta}_S)\|_2^2\\
&
% \quad
 = \frac{1}{2}\|\bt_n\|_2^2 - \frac{1}{2}\|\bt_n - \bZ_{0S}\btheta_S\|_2^2
% \\
% &
% \quad
%  = \bt_n\transpose\bZ_{0S}\btheta_S - \frac{1}{2}\|\bZ_{0S}\btheta_S\|_2^2\\
% &
% \quad
%  = (\bt_n - \bZ_{0S}\btheta_{0S})\transpose\bZ_{0S}(\btheta_S - \btheta_{0S}) - \frac{1}{2}\|\bZ_{0S}(\btheta_S - \btheta_{0S})\|_2^2 + \bt_n\transpose\bZ_{0S}\btheta_{0S} - \frac{1}{2}\|\bZ_{0S}\btheta_{0S}\|_2^2
\end{align*}
and that
\begin{align*}
- \frac{1}{2}\|\bt_n - \bZ_{0S}\btheta_S\|_2^2 + \frac{1}{2}\|\bt_n - \bZ_{0}\btheta_{0}\|_2^2
& = (\bt_n - \bZ_0\btheta_0)\transpose\bZ_0(\btheta_S - \btheta_{0S}) - \frac{1}{2}\|\bZ_{0S}(\btheta_S - \btheta_{0S})\|_2^2\\
& = \beps_n\transpose\bZ_{0S}(\btheta_S - \btheta_{0S}) - \frac{1}{2}\|\bZ_{0S}(\btheta_S - \btheta_{0S})\|_2^2
.
\end{align*}
Note that $\|\bt_n\|_2^2$ and $\|\bt_n - \bZ_0\btheta_0\|_2^2$ do not depend on $\btheta$ or the indexing set $S$. It follows that
\begin{align*}
&\Pi^\infty_\btheta(\btheta\in\widetilde{\calA}^c_n\mid\bY_n) \\
&\quad = 
% \frac{
\left[
\sum_{S\in \calS_0}\frac{\pi_p(|S|)}{{p - r \choose |S|}\gamma(|S|)}
\int_{(\widetilde{\calA}^c_n)_S}\exp\left\{\frac{1}{2}\|\bZ_{0S}\widehat{\btheta}_S\|_2^2 - \frac{1}{2}\|\bZ_{0S}(\btheta_S - \widehat{\btheta}_S)\|_2^2\right\}\mathrm{d}\btheta_S
\right]\\
% }{
&\quad\quad\times\left[
\sum_{S\in \calS_0}\frac{\pi_p(|S|)}{{p - r \choose |S|}\gamma(|S|)}
\int \exp\left\{\frac{1}{2}\|\bZ_{0S}\widehat{\btheta}_S\|_2^2- \frac{1}{2}\|\bZ_{0S}(\btheta_S - \widehat{\btheta}_S)\|_2^2\right\}\mathrm{d}\btheta_S
\right]^{-1}
% }
\\
&\quad = 
% \frac{
\left[
\sum_{S\in \calS_0}\frac{\pi_p(|S|)}{{p - r \choose |S|}\gamma(|S|)}
\int_{(\widetilde{\calA}^c_n)_S}\exp\left\{\beps_n\transpose\bZ_{0S}(\btheta_S - \btheta_{0S}) - \frac{1}{2}\|\bZ_{0S}(\btheta_S - \btheta_{0S})\|_2^2\right\}\mathrm{d}\btheta_S
\right]\\
% }{
&\quad\quad\times
\left[
\sum_{S\in \calS_0}\frac{\pi_p(|S|)}{{p - r \choose |S|}\gamma(|S|)}
\int \exp\left\{
\beps_n\transpose\bZ_{0S}(\btheta_S - \btheta_{0S}) - \frac{1}{2}\|\bZ_{0S}(\btheta_S - \btheta_{0S})\|_2^2
\right\}\mathrm{d}\btheta_S
\right]^{-1}
% }
\\
&\quad = \frac{N_n^\infty}{D_n^\infty}.
\end{align*}
% because $\|\bt_n\|_2^2$ and $\|\bt_n - \bZ_0\btheta_0\|_2^2$ does not depend on $S$ and $\btheta$. 
We now analyze the numerator $N_n^\infty$ and the denominator $D_n^\infty$ separately. 

\vspace*{2ex}\noindent
$\blacksquare$ We first analyze the denominator $D_n^\infty$. Denote $U_S(\mathrm{d}\btheta) = (\mathrm{d}\btheta_{S})\{\delta_{\zero}(\mathrm{d}\btheta_{S^c})\}$ for any $S\in\calS_0$. It follows that
\begin{align*}
D_n^\infty
& = \sum_{S\in \calS_0}
\frac{\pi_p(|S|)}{
 {p - r \choose |S|} \gamma(|S|)
 }
\int\exp\left\{
\beps_n\transpose\bZ_0(\btheta - \btheta_0) - \frac{1}{2}\|\bZ_{0}(\btheta - \btheta_0)\|_2^2
\right\}U_S(\mathrm{d}\btheta)\\
& \geq \frac{\pi_p(|S_0|)}{
 {p - r \choose |S_0|} \gamma(|S_0|)
 }
\int\exp\left\{
\beps_n\transpose\bZ_0(\btheta - \btheta_0) - \frac{1}{2}\|\bZ_{0}(\btheta - \btheta_0)\|_2^2
\right\}U_{S_0}(\mathrm{d}\btheta)
\end{align*}
By definition of the multivariate normal distribution, we have,
\begin{align*}
Q_{S_0}(\mathrm{d}\btheta)
% & = \frac{\exp\{-(1/2)\|\bt_{n} - \bZ_{0}\btheta\|_2^2 + (1/2)\|\bt_{n} - \bZ_{0}\btheta_{0}\|_2^2\}U_S(\mathrm{d}\btheta)}
% 	{\int_{\mathbb{R}^{d + r}}\exp\{-(1/2)\|\bt_{n} - \bZ_{0}\btheta\|_2^2 + (1/2)\|\bt_{n} - \bZ_{0}\btheta_{0}\|_2^2\}U_{S_0}(\mathrm{d}\btheta)}\\
& = \frac{\exp\{-(1/2)\|\bZ_{0}(\btheta - \btheta_{0})\|_2^2 + (\btheta - \btheta_{0})\transpose\bZ_{0}\transpose\beps_n\}U_{S_0}(\mathrm{d}\btheta)}
	{\int_{\mathbb{R}^{(p - r)r + r(r + 1)/2}}\exp\{-(1/2)\|\bZ_{0}(\btheta - \btheta_{0})\|_2^2 + (\btheta - \btheta_{0})\transpose\bZ_{0}\transpose\beps_n\}U_{S_0}(\mathrm{d}\btheta)}.
\end{align*}
Define the measures
\begin{align*}
\eta_{S_0}(\mathrm{d}\btheta) & = \exp\left\{-\frac{1}{2}\|\bZ_{0}(\btheta - \btheta_{0})\|_2^2 + (\btheta - \btheta_{0})\transpose\bZ_{0}\transpose\beps_n\right\}U_{S_0}(\mathrm{d}\btheta),\\
\sigma_{S_0}(\mathrm{d}\bbeta_{S_0}) & = \exp\left\{-\frac{1}{2}\|\bZ_{0S_0}\bbeta_{S_0}\|_2^2\right\}\mathrm{d}\bbeta_{S_0},\quad\bbeta_{S_0} = \btheta_{S_0} - \btheta_{0S_0}
\end{align*}
and the probability distribution $\bar{\sigma}_{S_0}(\mathrm{d}\bbeta_{S_0}) = \sigma_{S_0}(\mathrm{d}\bbeta_{S_0})/\sigma_{S_0}(\mathbb{R}^{|S_0|r + r(r + 1)/2})$. 
Then the denominator $\eta(\mathbb{R}^{(p - r)r + r(r + 1)/2})$ can be lower bounded as follows:
\begin{align*}
\eta_{S_0}(\mathbb{R}^{(p - r)r + r(r + 1)/2})
& = \int_{\mathbb{R}^{(p - r)r + r(r + 1)/2}} \exp\left\{-\frac{1}{2}\|\bZ_{0}(\btheta - \btheta_{0})\|_2^2 + (\btheta - \btheta_{0})\transpose\bZ_{0}\transpose\beps_n\right\}U_{S_0}(\mathrm{d}\btheta)\\
& = \int_{\mathbb{R}^{|S_0|r + r(r + 1)/2}} \exp\left\{-\frac{1}{2}\|\bZ_{0S_0}\bbeta_{S_0}\|_2^2 + \bbeta_{S_0}\transpose\bZ_{0{S_0}}\transpose\beps_n\right\}\mathrm{d}\bbeta_{S_0}\\
& = \sigma_{S_0}(\mathbb{R}^{|{S_0}|r  + r(r + 1)/2}) \int_{\mathbb{R}^{|{S_0}|r + r(r + 1)/2}}\exp(\bbeta_{S_0}\transpose\bZ_{0{S_0}}\transpose\beps_n)\bar{\sigma}_{S_0}(\mathrm{d}\bbeta_{S_0})\\
& \geq \sigma_{S_0}(\mathbb{R}^{|S_0|r + r(r + 1)/2}) \exp\left\{\int_{\mathbb{R}^{|S_0|r + r(r + 1)/2}}(\bbeta_{S_0}\transpose\bZ_{0S_0}\transpose\beps_n)\mathrm{d}\bar{\sigma}_{S_0}(\mathrm{d}\bbeta_{S_0})\right\}\\
& = \sigma_{S_0}(\mathbb{R}^{|S_0|r + r(r + 1)/2})
% \\
% & = \int_{\mathbb{R}^{|{S_0}|r + r(r + 1)/2}}\exp\left\{-\frac{1}{2}\|\bZ_{0S_0}\bbeta_{S_0}\|_2^2\right\}\mathrm{d}\bbeta_{S_0}\\
% & 
= \det\{2\pi(\bZ_{0S_0}\transpose\bZ_{0S_0})^{-1}\}^{1/2},
\end{align*}
where we have used the change of variable $\bbeta_{S_0} = \btheta_{S_0} - \btheta_{0{S_0}}$, the Jensen's inequality applied to the distribution $\bar{\sigma}$, and the fact that $\bar{\sigma}_{S_0}$ is symmetric about zero so that the expected value of $\bbeta_{S_0}\transpose\bZ_{0{S_0}}\transpose\beps_n$ with regard to $\bar{\sigma}(\mathrm{d}\bbeta_{S_0})$ is $0$. Hence we obtain the following lower bound for the denominator $D_n^\infty$:
\begin{align*}
D_n^\infty 
& \geq \frac{\pi_p(|S_0|)}{{p - r\choose |S_0|}\gamma(|S_0|)}\int\exp\left\{\beps_n\transpose\bZ_0(\btheta - \btheta_0) - \frac{1}{2}\|\bZ_0(\btheta - \btheta_0)\|_2^2\right\}U_{S_0}(\mathrm{d}\btheta)\\
& = \frac{\pi_p(|S_0|)}{{p - r\choose |S_0|}\gamma_{S_0}(|S_0|)}\eta_{S_0}(\mathbb{R}^{(p - r)r + r(r + 1)/2})\geq \frac{\pi_p(|S_0|)}{{p - r\choose |S_0|}}\det\{2\pi(\bZ_{0S_0}\transpose\bZ_{0S_0})^{-1}\}^{1/2}\\
&\geq \frac{\pi_p(s_0)(2\pi)^{s_0r/2 + r(r + 1)/4}}{\exp(s_0\log p)\det(\bZ_{0S}\transpose\bZ_{0S})^{1/2}}.
\end{align*}
By the geometric-algorithmic mean inequality,
\begin{align*}
\det(\bZ_{0S}\transpose\bZ_{0S})^{1/2}
& = \left\{\prod_{i = 1}^{|S|r + r(r + 1)/2}\lambda_i(\bZ_{0S}\transpose\bZ_{0S})\right\}^{1/2}\\
&\leq \left\{\frac{1}{|S|r + r(r + 1)/2}\sum_{i = 1}^{|S|r + r(r + 1)/2}\lambda_i(\bZ_{0S}\transpose\bZ_{0S})\right\}^{|S|r/2 + r(r + 1)/4}\\
& = \left\{\frac{\mathrm{tr}(\bZ_{0S_0}\transpose\bZ_{0S_0})}{|S|r + r(r + 1)/2}\right\}^{|S_0|r/2 + r(r + 1)/4}
% \\&
\leq (\|\bZ_{0S}\|_2^2)^{s_0r/2 + r(r + 1)/4}\\
&\leq \left(\frac{n}{2}\|D\bSigma(\btheta_0)\|_2^2\right)^{s_0r/2 + r(r + 1)/4}\leq\exp(Crs_0\log n)
\end{align*}
for some constant $C > 0$. Therefore, 
\[
D_n^\infty\geq \pi_p(s_0)\exp\left(-s_0\log p - Crs_0\log n
\right).
\]

\vspace*{2ex}\noindent
$\blacksquare$ We next analyze the numerator $N_n^\infty$. 
Write
\begin{align*}
N_n^\infty & = 
\sum_{S\in \calS_0}\frac{\pi_p(|S|)}{{p - r \choose |S|}\gamma(|S|)}
\int_{(\widetilde{\calA}^c_n)_S}\exp\left\{\beps_n\transpose\bZ_{0S}(\btheta_S - \btheta_{0S}) - \frac{1}{2}\|\bZ_{0S}(\btheta_S - \btheta_{0S})\|_2^2\right\}\mathrm{d}\btheta_S\\
& = 
\sum_{S\in \calS_0}\frac{\pi_p(|S|)}{{p - r \choose |S|}\gamma(|S|)}
\int_{\widetilde{\calA}^c_n}\exp\left\{\beps_n\transpose\bZ_{0}(\btheta - \btheta_{0}) - \frac{1}{2}\|\bZ_{0}(\btheta - \btheta_{0})\|_2^2\right\}U_S(\mathrm{d}\btheta)
\end{align*}
Denote the event 
\[
\Xi_n = \left\{
|(\btheta - \btheta_{0})\transpose\bZ_{0}\transpose\beps_n| \leq \bar{\lambda}\|\btheta - \btheta_{0}\|_1,
\right\},
\]
where $\bar{\lambda} = C_0\sqrt{n\log p}$ is such that $\prob_0(\Xi_n^c)\leq 2/p = o(1)$ by Lemma \ref{lemma:residual_linfty_deviation}. 
By definition of $\calA_n$, for any $\btheta\in\widetilde{\calA}_{n}^c$, we have,
\[
\|\btheta - \btheta_{0}\|_1> M\sqrt{\frac{r^2s^2\log n + rs^2\log p}{n}}.
\]
By Theorem \ref{thm:DSigma_singular_value}, we have
\begin{align*}
\sigma_{\min}^2(\bZ_0)& = {\frac{n}{2}}\lambda_{\min}\left\{D\bSigma(\btheta_0)\transpose(\bSigma_0^{-1}\otimes\bSigma_0^{-1})D\bSigma(\btheta_0)\right\}
\\&
= \frac{n}{2}\min_{\|\btheta\|_2 = 1}\|(\bOmega_0^{-1/2}\otimes\bOmega_0^{-1/2})D\bSigma(\btheta_0)\btheta\|_2^2
\\&
\geq \frac{n}{2}\lambda_{\min}(\bOmega_0^{-1})^2\min_{\|\btheta\|_2 = 1}\|D\bSigma(\btheta_0)\btheta\|_2^2
\\&
\geq \frac{n\sigma_{\min}^2\{D\bSigma(\btheta_0)\}}{2\|\bOmega_0\|_2^2}\gtrsim n.
\end{align*}
Then over the event $\Xi_n$, with $\mathrm{supp}(\bA)\in \calS_0$ and $\btheta = [\vect(\bA)\transpose, \bmu\transpose]\transpose$, we have,
\begin{align*}
(\btheta - \btheta_{0})\transpose\bZ_{0}\transpose\beps_n
& \leq 2\bar{\lambda}\|\btheta - \btheta_{0}\|_1 - \bar{\lambda}\|\btheta - \btheta_{0}\|_1\\
& = 2\left\{\frac{1}{2}\|\bZ_{0}(\btheta - \btheta_{0})\|_2^2\right\}^{1/2}\left\{\frac{\sqrt{2}\bar{\lambda}\|\btheta - \btheta_{0}\|_1}{\|\bZ_{0}(\btheta - \btheta_{0})\|_2}\right\} - \bar{\lambda}\|\btheta - \btheta_{0}\|_1\\
& \leq \frac{1}{2}\|\bZ_{0}(\btheta - \btheta_{0})\|_2^2 + \frac{2\bar{\lambda}^2\|\btheta - \btheta_{0}\|_1^2}{\sigma_{\min}(\bZ_{0})^2\|\btheta - \btheta_{0}\|^2_2} - \bar{\lambda}\|\btheta - \btheta_{0}\|_1\\
&\leq \frac{1}{2}\|\bZ_{0}(\btheta - \btheta_{0})\|_2^2 + \frac{\bar{C}_0rs_0\log p\|\btheta - \btheta_{0}\|_2^2}{\|\btheta - \btheta_{0}\|_2^2} 
 - \bar{\lambda}\|\btheta - \btheta_{0}\|_1\\
& = \frac{1}{2}\|\bZ_{0}(\btheta - \btheta_{0})\|_2^2 + 
\bar{C}_0rs_0\log p
 - \bar{\lambda}\|\btheta - \btheta_{0}\|_1.
\end{align*}
where $\bar{C}_0$ is a constant depending on the spectra of $\bOmega_0$. 
For any $\btheta\in\widetilde{\calA}_{n}^c$, we have,
\begin{align*}
\bar{\lambda}\|\btheta - \btheta_{0}\|_1
& = C_0\sqrt{n\log p}\|\btheta - \btheta_{0}\|_1\geq {C_0M\sqrt{r^2s^2(\log n)(\log p) + rs^2(\log p)^2}}\\
& \geq C_1 M(rs_0\log n + s_0\sqrt{r}\log p)
\end{align*}
for some constant $C_1 > 0$.
Therefore, by choosing a sufficiently large $M > 0$, we have, for any $\btheta\in\widetilde{\calA}_n^c$,
\begin{align*}
&\exp\left\{
- \frac{1}{2}\|\bZ_{0}(\btheta - \btheta_{0})\|_2^2 + 
(\btheta - \btheta_{0})\transpose\bZ_{0}\transpose\beps_n
\right\}\\
&\quad \leq \exp\left\{ - \frac{C_1M(rs_0\log n + s_0\sqrt{r}\log p)}{4} - \frac{\bar{\lambda}}{2}\|\btheta - \btheta_0\|_1\right\}
\end{align*}
for some constant $C_1 > 0$ (possibly depending on the spectra of $\bOmega_0$), which further implies that over the event $\Xi_n$,
\begin{align*}
N_n^\infty & = \sum_{S\in\calS_0}\frac{\pi_p(|S|)}{{p - r\choose |S|}\gamma(|S|)}\int_{(\widetilde{\calA}^c_n)_S}\exp\left\{\beps_n\transpose\bZ_0(\btheta - \btheta_0) - \frac{1}{2}\|\bZ_{0}(\btheta - \btheta_0)\|_2^2\right\}U_S(\mathrm{d}\btheta)\\
& \leq 
\exp\left\{ - \frac{C_1M(rs_0\log n + s_0\sqrt{r}\log p)}{4}\right\}\\
&\quad\times\sum_{S\in\calS_0}\frac{\pi_p(|S|)}{{p - r\choose |S|}\gamma(|S|)} \int_{\widetilde{\calA}^c_n}\exp\left(- \frac{\bar{\lambda}}{2}\|\btheta - \btheta_0\|_1\right)U_S(\mathrm{d}\btheta)\\
& = \exp\left\{ - \frac{C_1M(rs_0\log n + s_0\sqrt{r}\log p)}{4}\right\}\sum_{S\in\calS_0}\frac{\pi_p(|S|)}{{p - r\choose |S|}\gamma(|S|)}\left(\frac{4}{\bar{\lambda}}\right)^{|S|r + r(r + 1)/2}\\
& \leq \exp\left\{ - \frac{C_1M (rs_0\log n + s_0\sqrt{r}\log p)}{4}\right\}\sum_{t = 0}^{\kappa_0s_0}\frac{\pi_p(t)}{\gamma(t)}\\
& = \exp\left\{ - \frac{C_1M (rs_0\log n + s_0\sqrt{r}\log p)}{4}\right\}\exp\left\{\frac{1}{2}\kappa_0s_0r\log(\kappa_0s_0r) + 2\kappa_0s_0r\right\}\\
&\leq \exp\left\{ - \frac{C_1M (rs_0\log n + s_0\sqrt{r}\log p)}{4}\right\}\exp\left(\frac{3}{2}\kappa_0s_0r\log n +2\kappa_0s_0r\right)\\
&\leq \exp\left\{ - \frac{C_1M (rs_0\log n + s_0\sqrt{r}\log p)}{8}\right\}
% \left(\frac{4}{\bar{\lambda}}\right)^{sr + r^2 + r(r + 1)/2}
\end{align*}
for a sufficiently large $M > 0$, where we have used the fact that $\bar{\lambda}\to 0$ as $n\to\infty$ and \eqref{eqn:gamma_S_lower_bound}. 

\vspace*{2ex}\noindent
$\blacksquare$ We are now finally in a position to analyze the ratio $N_n^\infty/D_n^\infty$. Write
\begin{align*}
&\expect_0\left\{
\Pi_\btheta^\infty(\btheta\in\widetilde{\calA}_n^c\mid\bY_n)
\right\}\\
&\quad\leq \prob_0(\Xi_n^c) + \expect_0\left\{\mathbbm{1}(\Xi_n)\frac{N_n^\infty}{D_n^\infty}\right\}\\
% &\quad\leq o(1) + \frac{{p - r\choose |S_0|}}{\pi_p(|S_0|)}\exp\left\{ - \frac{C_1M(rs_0\log n + s_0\sqrt{r}\log p)}{8}\right\}
% \det\{2\pi(\bZ_{0S_0}\transpose\bZ_{0S_0})^{-1}\}^{-1/2}\\
% % \sum_{s = 0}^{p - r}\pi_p(s)\\
% &\quad\leq o(1) + \frac{1}{\pi_p(|S_0|)}\exp\left\{
% s_0\log p - \frac{C_1M (rs_0\log n + s_0\sqrt{r}\log p)}{8}
% \right\}\frac{\det(\bZ_{0S_0}\transpose\bZ_{0S_0})^{1/2}}{(2\pi)^{(s_0 - r)r/2 + r(r + 1)/4}}\\
% &\quad\leq o(1) + \frac{1}{\pi_p(|S_0|)}\exp\left\{
% s_0\log p - \frac{C_1M (rs_0\log n + s_0\sqrt{r}\log p)}{8}
% \right\}\left\{\frac{\mathrm{tr}(\bZ_{0S_0}\transpose\bZ_{0S_0})}{d_0}\right\}^{d_0/2}\\
% &\quad\leq o(1) + \frac{1}{\pi_p(|S_0|)}\exp\left\{
% s_0\log p - \frac{C_1M (rs_0\log n + s_0\sqrt{r}\log p)}{8}
% \right\}\left(\|\bZ_{0S_0}\|_2^2\right)^{d_0/2}\\
% &\quad\leq o(1) + \frac{1}{\pi_p(|S_0|)}\exp\left\{
% s_0\log p - \frac{C_1M (rs_0\log n + s_0\sqrt{r}\log p)}{8}
% \right\}\left(\frac{n}{2}\|D\bSigma(\btheta_0)\|_2^2\right)^{d_0/2}\\
% &\quad\leq o(1) + \frac{1}{\pi_p(|S_0|)}\exp\left\{
% s_0\log p - \frac{C_1M (rs_0\log n + s_0\sqrt{r}\log p)}{8} + Cd_0\log n
% \right\}\\
&\quad\leq o(1) + \frac{1}{\pi_p(s_0)}\exp\left\{
s_0\log p - \frac{C_1M (rs_0\log n + s_0\sqrt{r}\log p)}{8} + Cs_0r\log n
\right\}\\
&\quad\leq o(1) + 2\exp\left\{(C + 1)rs_0\log n + (a + 1)s_0\log p
 - \frac{C_0M(rs_0\log n + s_0\sqrt{r}\log p)}{8}
\right\}\\
&\quad\to 0
\end{align*}
by taking a sufficiently large $M > 0$ again. The proof is thus completed. 
\end{proof}

We are now in a position to present the proof of Theorem \ref{thm:BvM}. 
\begin{proof}[Proof of Theorem \ref{thm:BvM}]
% For any $\btheta$, let $\btheta = [\vect(\bA)\transpose, \bmu\transpose]\transpose$, where $\bA\in\mathbb{R}^{(p - r)\times r}$ with $\|\bA\|_2 < 1$. 
We first claim that Assumption A5 implies that 
\[
\expect_0\{\Pi_\btheta(\btheta:S_0\subset \mathrm{supp}(\bA)\mid\bY_n)\} \to 1.
\]
In fact, if $S_0\cap \mathrm{supp}(\bA)^c\neq \varnothing$, then there exists some $j\in S_0$ such that $j\in \mathrm{supp}(\bA)^c$. Therefore,
\begin{align*}
\|\btheta - \btheta_0\|_2 
& \geq \|\vect(\bA) - \vect(\bA_0)\|_{\mathrm{F}} = \left\{\sum_{l = 1}^p\|\be_l\transpose(\bA - \bA_0)\|^2_2\right\}^{1/2}\\
&\geq \|[\bA]_{j*} - [\bA_0]_{j*}\|_2 = \|[\bA_0]_{j*}\|_2
\geq\min_{j\in S_0}\|[\bA_0]_{j*}\|_2. 
\end{align*}
Using the result from Theorem \ref{thm:contraction_Frobenius_norm}, we have,
\begin{align*}
\expect_0\left[\Pi_\btheta\left\{\btheta:\|\btheta - \btheta_0\|_2 < M\sqrt{\frac{rs\log n + s\log p}{n}}\mathrel{\Bigg|}\bY_n\right\}\right] \to 1
\end{align*}
for some constant $M > 0$. Hence,
\begin{align*}
&
\expect_0[\Pi_\btheta\{\btheta:S_0\subset \mathrm{supp}(\bA)\mid\bY_n\}]
\\
&\quad = 1 - \expect_0[\Pi_\btheta\{\btheta:S_0\cap \mathrm{supp}(\bA)^c\neq \varnothing\mid\bY_n\}]\\
&\quad \geq 1 - \expect_0\left\{\Pi_\btheta\left(\btheta:\|\btheta - \btheta_0\|_2\geq \min_{j\in S_0}\|[\bA_0]_{j*}\|_2 \mid\bY_n\right)\right\}\\
&\quad \geq 1 - \expect_0\left\{\Pi_\btheta\left(\btheta:\|\btheta - \btheta_0\|_2\geq M\sqrt{\frac{rs\log n + s\log p}{n}} \mathrel{\Big|}\bY_n\right)\right\} \to 1.
\end{align*}
Let 
\[
\calE_n := \bigcup_{S\in\calS_0}\calA_n(S) = \bigcup_{S\in\calS_0}\left\{\btheta:\|\btheta - \btheta_0\|_1\leq M\sqrt{\frac{r^2s^2\log n + rs^2\log p}{n}},\mathrm{supp}(\bA) = S\right\}
\]
By Theorem \ref{thm:contraction_Frobenius_norm}, Lemma \ref{lemma:posterior sparsity}, and Assumption A5, we immediately see that 
\[
\expect_0\{\Pi_\btheta(\btheta\in\calE_n\mid\bY_n)\}\to 1.
\] 
% First observe that for any $\btheta = [\vect(\bA)\transpose, \bmu\transpose]\transpose\in\widetilde\calA_n(S)$, we have $\|\bA\|_2 < 1$ and $\bM\in\mathbb{M}_+^r$, where $\bM = \bM(\bmu)$ is the symmetric matrix whose upper diagonal entries are elements of $\bmu$. 
Let $\bt_n = \bZ_0\btheta_0 + \beps_n$, 
\begin{align*}
v_S & \propto 
\frac{\pi_p(|S|)}{{p - r\choose |S|}}\int\exp\{\ell(\bOmega(\btheta)) - \ell(\bOmega_0)\}\Pi_\btheta(\mathrm{d}\btheta)\mathbbm{1}\{S\in\calS_0\},\\
G_S(\mathrm{d}\btheta) & := \{\exp(-2\|\btheta_S\|_1)\mathrm{d}\btheta_S\}\{\delta_\zero(\mathrm{d}\btheta_{S^c})\}, \\
\mu_{S}(\mathrm{d}\btheta) & := 
\frac{\pi_p(|S|)}{{p - r\choose |S|}\gamma(|S|)}\mathbbm{1}\{\btheta\in\calA_{n}(S)\}\exp\{\ell(\bOmega(\btheta)) - \ell(\bOmega_0)\} G_S(\mathrm{d}\btheta),\\
\nu_{S}(\mathrm{d}\btheta) & := 
\frac{\pi_p(|S|)}{{p - r\choose |S|}\gamma(|S|)}
\exp\left(\frac{1}{2}\beps_n\transpose\beps_n - \frac{1}{2}\|\bt_{n} - \bZ_{0S}\widehat{\btheta}_{S}\|_2^2 - 2\|\btheta_{0}\|_1\right)\\
&\quad\times 
\mathbbm{1}\{\btheta\in\calA_{n}(S)\}\exp\left\{
-\frac{1}{2}(\btheta_S - \widehat\btheta_{S})\transpose\bZ_{0S}\transpose\bZ_{0S}(\btheta_S - \widehat\btheta_{S})
\right\}\mathrm{d}\btheta_S\times \delta_{\zero}(\mathrm{d}\btheta_{S^c}),\\
\widetilde{\varpi}_{S} & :=  \frac{\pi_p(|S|)}{{p - r\choose |S|}\gamma(|S|)}
\exp\left(\frac{1}{2}\beps_n\transpose\beps_n - \frac{1}{2}\|\bt_{n} - \bZ_{0S}\widehat{\btheta}_{S}\|_2^2 - 2\|\btheta_{0}\|_1\right)\mathbbm{1}\{S\in\calS_0\}\\
&\quad\times \int_{(\calA_{n}(S))_S}\exp\left\{-\frac{1}{2}(\btheta_S - \widehat\btheta_{S})\transpose\bZ_{0S}\transpose\bZ_{0S}(\btheta_S - \widehat\btheta_{S})\right\}\mathrm{d}\btheta_{S},\\
\widehat{\varpi}_{S} & :=  \frac{\pi_p(|S|)}{{p - r\choose |S|}\gamma(|S|)}
\exp\left(\frac{1}{2}\beps_n\transpose\beps_n - \frac{1}{2}\|\bt_{n} - \bZ_{0S}\widehat{\btheta}_{S}\|_2^2 - 2\|\btheta_{0}\|_1\right)\mathbbm{1}\{S\in\calS_0\}\\
&\quad\times \int_{\mathbb{R}^{|S|r + r(r + 1)/2}}\exp\left\{-\frac{1}{2}(\btheta_S - \widehat\btheta_{S})\transpose\bZ_{0S}\transpose\bZ_{0S}(\btheta_S - \widehat\btheta_{S})\right\}\mathrm{d}\btheta_{S},\\
% \widetilde{w}_{mS} & = \frac{\widetilde{\Omega}_{mS}}{\sum_{S\in\calS(\kappa_n)}\sum_{m = 1}^{2^r}\widetilde{\Omega}_{mS}},\\
% {Q}_{mS}(\mathrm{d}\btheta) & = \phi(\btheta_S\mid\widehat{\btheta}_{mS}, (\bZ_{0mS}\transpose\bZ_{0mS})^{-1})\mathrm{d}\btheta_S\otimes\delta_{\zero}(\mathrm{d}\btheta_{S^c}),\\
\widetilde{Q}_{S}(\mathrm{d}\btheta) & := \frac{\mathbbm{1}\{\btheta\in\calA_{n}(S)\}
Q_{S}(\mathrm{d}\btheta)
}{Q_{S}(\btheta\in\calA_{n}(S))}.
 % = \frac{\mathbbm{1}\{\btheta\in\calA_{n}(S)\}\phi(\btheta_S\mid\widehat{\btheta}_{S}, (\bZ_{0S}\transpose\bZ_{0S})^{-1})\mathrm{d}\btheta_S\times\delta_{\zero}(\mathrm{d}\btheta_{S^c})}{Q_{S}(\btheta\in\calA_{n}(S))}
\end{align*}
By Parseval's identity, we have
\begin{align*}
&\exp\left(-2\|\btheta_{0}\|_1 - \frac{1}{2}\|\bt_{n} - \bZ_{0S}\widehat\btheta_{S}\|_2^2\right)\\
&
\quad
 = \exp\left\{-2\|\btheta_0\|_1 - \frac{1}{2}\|\bt_n - \bZ_{0S}(\bZ_{0S}\transpose\bZ_{0S})^{-1}\bZ_{0S}\bt_n\|_2^2\right\}\\
&\quad = \exp\left\{-2\|\btheta_0\|_1 - \frac{1}{2}\|\bt_n\|_2^2 + \frac{1}{2}\|\bZ_{0S}(\bZ_{0S}\transpose\bZ_{0S})^{-1}\bZ_{0S}\bt_n\|_2^2\right\}\\
&\quad = \exp\left(-2\|\btheta_0\|_1 - \frac{1}{2}\|\bt_n\|_2^2\right)\exp\left(\frac{1}{2}\|\bZ_{0S}\widehat{\btheta}_S\|_2^2\right).
\end{align*}
Note that $\exp\{-2\|\btheta_0\|_1 - (1/2)\|\bt_n\|_2^2\}$ does not depend on the supporting set $S$. 
Therefore, by definiton of $\widehat{w}_{S}$, we have
\begin{align*}
\widehat{w}_{S}&\propto\frac{\pi_p(|S|)}{{p - r\choose |S|}\gamma(|S|)}\det\{2\pi(\bZ_{0S}\transpose\bZ_{0S})^{-1}\}^{1/2}\exp\left(\frac{1}{2}\|\bZ_{0S}\widehat\btheta_{S}\|_2^2\right)\\
&\propto \frac{\pi_p(|S|)}{{p - r\choose |S|}\gamma(|S|)}\det\{2\pi(\bZ_{0S}\transpose\bZ_{0S})^{-1}\}^{1/2}\exp\left(\frac{1}{2}\|\beps_n\|_2^2 - 2\|\btheta_{0}\|_1 - \frac{1}{2}\|\bt_{n} - \bZ_{0S}\widehat\btheta_{S}\|_2^2\right)\\
& = \widehat{\varpi}_S.
\end{align*}
Namely,
% \begin{align*}
$\widehat{w}_{S} = {\widehat\varpi_{S}}/{\sum_{T\in\calS_0}\widehat\varpi_{T}}$, $S\in\calS_0$.
% \end{align*}
For any probability distribution $\prob(\cdot)$ and any event $\calA$, we have, by the law of total probability,
\begin{align*}
\left\|\prob(\cdot) - \frac{\prob(\cdot\cap \calA)}{\prob(\calA)}\right\|_{\mathrm{TV}}
& = \sup_{\calB}\left|\prob(\calB) - \frac{\prob(\calB\cap \calA)}{\prob(\calA)}\right|
% \\&
 = \sup_{\calB}\left|\frac{\prob(\calB\cap \calA^c)\prob(\calA) - \prob(\calB\cap\calA)\prob(\calA^c)}{\prob(\calA)}\right|\\
&\leq \sup_{\calB}\frac{\prob(\calB\cap \calA^c)\prob(\calA) + \sup_{\calB}\prob(\calB\cap\calA)\prob(\calA^c)}{\prob(\calA)}
% \\&
\leq 2\prob(\calA^c).
\end{align*}
For any measurable set $\calB\subset\mathscr{D}(p, r)$, the (exact) posterior probability of $\calB$ given $\calE_n$ and $\bY_n$ can be written as
\begin{align*}
\frac{\Pi_\btheta(\calB\cap\calE_n\mid\bY_n)}{\Pi_\btheta(\calE_n\mid\bY_n)} 
&= \frac{\sum_{S\in\calS_0}\pi_p(|S|){p - r\choose |S|}^{-1}\gamma(|S|)^{-1}\int_{\calB\cap\calE_n}\exp\{\ell(\bOmega(\btheta)) - \ell(\bOmega_0)\}G_S(\mathrm{d}\btheta)}
{\sum_{S\in\calS_0}\pi_p(|S|){p - r\choose |S|}^{-1}\gamma(|S|)^{-1}\int_{\calE_n}\exp\{\ell(\bOmega(\btheta)) - \ell(\bOmega_0)\}G_S(\mathrm{d}\btheta)}\\
&=
\frac{\sum_{S\in\calS_0}
% \sum_{m = 1}^{2^r}
\pi_p(|S|){p - r\choose |S|}^{-1}\gamma(|S|)^{-1}\int_{\calB\cap\calA_{n}(S)}\exp\{\ell(\bOmega(\btheta)) - \ell(\bOmega_0)\}G_S(\mathrm{d}\btheta)}{
\sum_{S\in\calS_0}
% \sum_{m = 1}^{2^r}
\pi_p(|S|){p - r\choose |S|}^{-1}\gamma(|S|)^{-1}\int_{\calA_{n}(S)}\exp\{\ell(\bOmega(\btheta)) - \ell(\bOmega_0)\}G_S(\mathrm{d}\btheta)}\\
& = \sum_{S\in\calS_0}
% \sum_{m = 1}^{2^r}
\frac{\mu_{S}(\calB)}{\|\sum_{S\in\calS_0}
% \sum_{m = 1}^{2^r}
\mu_{S}(\cdot)\|_{\mathrm{TV}}}.
\end{align*}
By the triangle inequality, the total variation distance between $\Pi(\mathrm{d}\btheta\mid\bY_n)$ and $\Pi^\infty(\mathrm{d}\btheta\mid\bY_n)$ can be decomposed as follows:
\begin{align}
&
\|\Pi_\btheta(\btheta\in\cdot\mid\bY_n) - \Pi_\btheta^\infty(\btheta\in\cdot\mid\bY_n)\|_{\mathrm{TV}}\nonumber\\
\label{eqn:total_variation_triangle_I}
&\quad \leq \left\|\Pi_\btheta(\btheta\in\cdot\mid\bY_n) - \frac{\Pi_\btheta(\btheta\in\cdot\cap\calE_n)}{\Pi_\btheta(\btheta\in\calE_n)}\right\|_{\mathrm{TV}}
\\
\label{eqn:total_variation_triangle_II}
&\quad\quad
 + \left\|\sum_{S\in\calS_0}
% \sum_{m = 1}^{2^r}
\frac{\mu_{S}(\cdot)}{\|\sum_{S\in\calS_0}
% \sum_{m = 1}^{2^r}
\mu_{S}(\cdot)\|_{\mathrm{TV}}} - \sum_{S\in\calS(\kappa_n)}
% \sum_{m = 1}^{2^r}
\frac{\nu_{S}(\cdot)}{\|\sum_{S\in\calS_0}
% \sum_{m = 1}^{2^r}
\nu_{S}(\cdot)\|_{\mathrm{TV}}} \right\|_{\mathrm{TV}}\\
\label{eqn:total_variation_triangle_III}
&\quad\quad + \left\|\sum_{S\in\calS_0}
% \sum_{m = 1}^{2^r}
\frac{\nu_{S}(\cdot)}{\|\sum_{S\in\calS_0}
% \sum_{m = 1}^{2^r}
\nu_{S}(\cdot)\|_{\mathrm{TV}}} 
 - \Pi_\btheta^\infty(\btheta\in\cdot\mid\bY_n)\|
\right\|_{\mathrm{TV}}.
% & \quad:= I + II + III.
\end{align}
The first term on the right-hand side is upper bounded by $2\Pi_\btheta(\btheta\in\calE_n^c\mid\bY_n)$, which is $o_{\prob_0}(1)$ by Theorem \ref{thm:contraction_Frobenius_norm}, Lemma \ref{lemma:posterior sparsity}, and Assumption A5. It suffices to focus on the second and the third term. For the second term on line \eqref{eqn:total_variation_triangle_II}, write
\begin{align*}
&\left\|\sum_{S} \frac{\mu_{S}}{\|\sum_{S}\mu_{S}\|_{\mathrm{TV}}} - \sum_{S}\frac{\nu_{S}}{\|\sum_{S}\nu_{S}\|_{\mathrm{TV}}} \right\|_{\mathrm{TV}}\\
&\quad = \left\|\frac{\|\sum_{S}\nu_{S}\|_{\mathrm{TV}}\sum_S(\mu_{S} - \nu_{S}) + (\|\sum_{S}\nu_{S}\|_{\mathrm{TV}} - \|\sum_{S}\mu_{S}\|_{\mathrm{TV}})\sum_{S}\nu_{S}}{\|\sum_{S}\mu_{S}\|_{\mathrm{TV}}\|\sum_{S}\nu_{S}\|_{\mathrm{TV}}}\right\|_{\mathrm{TV}}\\
&\quad\leq \frac{2\sum_{S}\|\mu_{S} - \nu_{S}\|_{\mathrm{TV}}}{\|\sum_{S}\mu_{S}\|_{\mathrm{TV}}}\\
&\quad = \frac{2}{\|\sum_{S}\mu_{S}\|_{\mathrm{TV}}}\sum_{S}\sup_{\calB}\left|\int_\calB\mathrm{d}\mu_{S} - \int_\calB\left(\frac{\mathrm{d}\nu_{S}}{\mathrm{d}\mu_{S}}\right)\mathrm{d}\mu_{S}\right|\\
&\quad\leq \frac{2}{\|\sum_{S}\mu_{S}\|_{\mathrm{TV}}}\sum_{S}\sup_{\calB}\int_\calB\left|1 - \left(\frac{\mathrm{d}\nu_{S}}{\mathrm{d}\mu_{S}}\right)\right|\mathrm{d}\mu_{S}\\
&\quad\leq \frac{2}{\|\sum_{S}\mu_{S}\|_{\mathrm{TV}}}\sum_{S}\sup_{\calB}\mu_{S}(\calB)\left\|1 - \left(\frac{\mathrm{d}\nu_{S}}{\mathrm{d}\mu_{S}}\right)\right\|_{L_\infty(\calA_n(S))}\\
&\quad = 2\sup_{S\in\calS_0}\left\|1 - \left(\frac{\mathrm{d}\nu_{S}}{\mathrm{d}\mu_{S}}\right)\right\|_{L_\infty(\calA_n(S))}.
\end{align*}
By definition of $\nu_S$ and $\mu_S$, for any $\btheta\in\calA_{n}(S)$, we have
\begin{align*}
-\log\frac{\mathrm{d}\nu_{S}}{\mathrm{d}\mu_{S}}(\btheta)
& = \ell(\bOmega(\btheta)) - \ell(\bOmega_0) -\frac{1}{2}\|\beps_n\|_2^2 + \frac{1}{2}\|\bt_{n} - \bZ_{0S}\widehat\btheta_{S}\|_2^2 + \frac{1}{2}\|\bZ_{0S}(\btheta_S - \widehat\btheta_{S})\|_2^2\\
&\quad + 2\|\btheta_{0S}\|_1 - 2\|\btheta_{S}\|_1.
\end{align*}
Observe that $\bZ_{0S}(\widehat\btheta_{S} - \btheta_S)$ is inside the column space of $\bZ_{0S}$, and $\bt_{n} - \bZ_{0S}\widehat\btheta_{S}$ lies in the orthogonal complement of the column space of $\bZ_{0S}$, it follows from the Parseval's identity that for $\btheta\in\calA_{n}(S)$,
\begin{align*}
\|\bt_{n} - \bZ_{0}\btheta\|_2^2
&=\|\bt_{n} - \bZ_{0S}\btheta_{S}\|_2^2
% \\&
= \|\bt_{n} - \bZ_{0S}\widehat\btheta_{S} + \bZ_{0S}(\btheta_S - \widehat\btheta_{S})\|_2^2\\
& = \|\bt_{n} - \bZ_{0S}\widehat\btheta_{S}\|_2^2 + \|\bZ_{0S}(\btheta_S - \widehat\btheta_{S})\|_2^2.
\end{align*}
Using the fact that
\begin{align*}
(\btheta - \btheta_{0})\transpose\bZ_{0}\transpose\beps_n - \frac{1}{2}\|\bZ_{0}(\btheta - \btheta_{0})\|_2^2
& = 
-\frac{1}{2}\|\bt_{n} - \bZ_{0}\btheta\|_2^2 + \frac{1}{2}\|\bt_{n} - \bZ_{0}\btheta_{0}\|_2^2,
\end{align*}
we can further obtain
\begin{align*}
-\log\frac{\mathrm{d}\nu_{S}}{\mathrm{d}\mu_{S}}(\btheta)
& = \ell(\bOmega(\btheta)) - \ell(\bOmega_0) - \frac{1}{2}\|\bt_{n} - \bZ_{0}\btheta_{0}\|_2^2 + \frac{1}{2}\|\bt_{n} - \bZ_{0}\btheta\|_2^2\\
&\quad + 2(\|\btheta_{0S}\|_1 - \|\btheta_S\|_1)\\
& = \ell(\bOmega(\btheta)) - \ell(\bOmega_0) - (\btheta - \btheta_{0})\transpose\bZ_{0}\transpose\beps_n + \frac{1}{2}\|\bZ_{0}(\btheta - \btheta_{0})\|_2^2\\
&\quad + 2(\|\btheta_{0S}\|_1 - \|\btheta_S\|_1)\\
& = R_n(\btheta, \btheta_{0}) + 2(\|\btheta_{0S}\|_1 - \|\btheta_S\|_1),
\end{align*}
where the remainder $R_n$ satisfies
\[
\sup_{S\in\calS_0}\sup_{\btheta\in\calA_{n}(S)}|R_n(\btheta, \btheta_{0})|\leq \sup_{\btheta\in\calB_n}|R_n(\btheta, \btheta_{0})| = o_{\prob_0}(1)
\]
by Theorem \ref{thm:LAN}. In addition, we also have 
\begin{align*}
\sup_{S\in\calS_0}\sup_{\btheta\in\calA_{n}(S)}2|\|\btheta_{0S}\|_1 - \|\btheta_S\|_1|
&\leq \sup_{S\in\calS_0}\sup_{\btheta\in\calA_{n}(S)}2\|\btheta_{S} - \btheta_{0S}\|_1\\
&\lesssim\sqrt{\frac{r^2s^2\log n + rs^2\log p}{n}} = o(1)
\end{align*}
by definition.
Therefore, the term on line \eqref{eqn:total_variation_triangle_II} is upper bounded by
\begin{align*}
&2\sup_{S\in\calS_0}\left\|1 - \frac{\mathrm{d}\nu_{S}}{\mathrm{d}\mu_{S}}\right\|_{L_\infty(\calA_n(S))}
\\&
\quad = 2\sup_{S\in\calS_0}\left\|1 - \exp\{-R_n(\btheta, \btheta_{0}) - 2(\|\btheta_{0S}\|_1 - \|\btheta_S\|_1)\}\right\|_{L_\infty(\calA_n(S))}\\
&\quad\leq 2
% \sup_{S\in\calS(\kappa_n)}
\left[\exp\left\{\sup_{S\in\calS_0}\sup_{\btheta\in\calA_{n}(S)}(|R_n(\btheta, \btheta_{0})| + 2(\|\btheta_{0S}\|_1 - \|\btheta_S\|_1))\right\} - 1\right]
% \\&
 = o_{\prob_0}(1).
\end{align*}
This shows that the term on line \eqref{eqn:total_variation_triangle_II} is $o_{\prob_0}(1)$. We now finally focus on the term on line \eqref{eqn:total_variation_triangle_III}. Using the fact that $\widehat{w}_S = \widehat{\varpi}_S/\sum_{T\in\calS_0}\widehat{\varpi}_T$, we have
\begin{align*}
\Pi_\btheta^\infty(\mathrm{d}\btheta\mid\bY_n) = 
% \sum_{m = 1}^{2^r}
\sum_{S\in\calS_0}\left(\frac{\widehat\varpi_{S}}{\sum_{S\in\calS_0}\widehat\varpi_{S}}\right)Q_{S}(\mathrm{d}\btheta)
\end{align*}
and
\begin{align*}
\sum_{S\in\calS_0}\frac{\nu_{S}(\mathrm{d}\btheta)}{\|\sum_{S\in\calS_0}\nu_{S}(\cdot)\|_{\mathrm{TV}}}
& = \sum_{S\in\calS_0}\left(\frac{\widetilde\varpi_{S}}{\sum_{S\in\calS_0}\widetilde\varpi_{S}}\right)\widetilde{Q}_{S}(\mathrm{d}\btheta)
\end{align*}
because by construction, $\nu_S(\mathrm{d}\btheta) = \widetilde{\varpi}_S\widetilde{Q}_S(\mathrm{d}\btheta)$. 
Note that for any measurable set $\calB$, $Q_S(\calB\cap\calE_n) = Q_S(\calB\cap \calA_n(S))$, and by definition, $\widetilde{Q}_S(\calB) = Q_S(\calB\cap \calA_n(S))/Q_S(\calA_n(S))$. Therefore,
\begin{align*}
\frac{\Pi^{\infty}_\btheta(\btheta\in\calB\cap\calE_n\mid\bY_n)}{\Pi^{\infty}_\btheta(\btheta\in\calE_n\mid\bY_n)}
& = \frac{\sum_{S\in\calS_0}\widehat{w}_SQ_S(\calB\cap \calE_n)
}{\sum_{S\in\calS_0}\widehat{w}_SQ_S({\calE}_n)}
% \\&
 = \frac{\sum_{S\in\calS_0}\widehat{\varpi}_SQ_S(\calB\cap {\calA}_n(S))
}{\sum_{S\in\calS_0}\widehat{\varpi}_SQ_S({\calA}_n(S))}\\
& = \frac{\sum_{S\in\calS_0}\widehat{\varpi}_SQ_S({\calA}_n(S))\widetilde{Q}_S(\calB)
}{\sum_{S\in\calS_0}\widehat{\varpi}_SQ_S({\calA}_n(S))}.
\end{align*}
Furthermore, we have
\begin{align*}
&\widehat{\varpi}_SQ_S(\calA_n(S))\\
&\quad \propto  \frac{\pi_p(|S|)}{{p - r\choose |S|}\gamma(|S|)}\exp\left(\frac{1}{2}\|\bZ_{0S}\widehat{\btheta}_S\|_2^2\right)\det\{2\pi(\bZ_{0S}\transpose\bZ_{0S})^{-1}\}^{1/2}Q_S(\calA_n(S))\\
&\quad = \frac{\pi_p(|S|)}{{p - r\choose |S|}\gamma(|S|)}\exp\left(\frac{1}{2}\|\bZ_{0S}\widehat{\btheta}_S\|_2^2\right)
\int_{\calA_n}\exp\left\{ -\frac{1}{2}\|\bZ_{0S}(\btheta_S - \widehat{\btheta}_S)\|_2^2 \right\}\mathrm{d}\btheta_S\times \delta_{\zero}(\mathrm{d}\btheta_{S^c})\\
&\quad = \frac{\pi_p(|S|)}{{p - r\choose |S|}\gamma(|S|)}\exp\left(\frac{1}{2}\|\bZ_{0S}\widehat{\btheta}_S\|_2^2\right)\int_{(\calA_n(S))_S}\exp\left\{-\frac{1}{2}\|\bZ_{0S}(\btheta_S - \widehat{\btheta}_S)\|_2^2\right\}\mathrm{d}\btheta_S\\
&\quad\propto \widetilde{\varpi}_S,
\end{align*}
implying that
\begin{align*}
\frac{\Pi^{\infty}_\btheta(\btheta\in\calB\cap{\calE}_n\mid\bY_n)}{\Pi^{\infty}_\btheta({\calE}_n\mid\bY_n)}
& = \frac{\sum_{S\in\calS_0}\widehat{\varpi}_SQ_S({\calA}_n(S))\widetilde{Q}_S(\calB)
}{\sum_{S\in\calS_0}\widehat{\varpi}_SQ_S({\calA}_n(S))}
 = \sum_{S\in\calS_0}\left(\frac{\widetilde{\varpi}_S
}{\sum_{S\in\calS_0}\widetilde{\varpi}_S}\right)\widetilde{Q}_S(\calB)\\
& = \frac{\sum_{S\in \calS_0}\nu_S(\calB)}{\|\sum_{S\in \calS_0}\nu_S(\cdot)\|_{\mathrm{TV}}}
.
\end{align*}
Hence by Lemma \ref{lemma:normal_restriction} we know that the term on line \eqref{eqn:total_variation_triangle_III} is upper bounded by
\begin{align*}
\left\|
\frac{\sum_{S\in\calS_0}\nu_S(\cdot)}{\|\sum_{S\in\calS_0}\nu_S(\cdot)\|_{\mathrm{TV}}} - \Pi^\infty_\btheta(\btheta\in \cdot\mid\bY_n)
\right\|_{\mathrm{TV}}
&= \left\|
\frac{\Pi^\infty_\btheta(\btheta\in \cdot\cap\calE_n\mid\bY_n)}{\Pi^\infty_\btheta(\btheta\in{\calE}_n\mid\bY_n)} - \Pi^\infty_\btheta(\btheta\in \cdot\mid\bY_n)
\right\|_{\mathrm{TV}}\\
&\leq 2\Pi^\infty_\btheta(\btheta\in {\calE}^c_n\mid\bY_n) = o_{\prob_0}(1).
\end{align*}
The proof is thus completed.
\end{proof}

% subsection distributional_approximation (end)

\subsection{Posterior contraction under spectral norm} % (fold)
\label{sub:posterior_contraction_under_spectral_norm}

\begin{proof}[Proof of Theorem \ref{thm:contraction_spectral_norm}]
The proof of Theorem \ref{thm:contraction_spectral_norm} is based on Theorem \ref{thm:BvM} together with a discretization trick for the spectral norm loss. 
By Davis-Kahan theorem, $\|\sin\Theta(\bU(\bvarphi), \bU_0)\|_2\lesssim \|\bSigma(\btheta) - \bSigma_0\|_2/\lambda_r(\bSigma_0)$, it suffices to consider the posterior contraction under $\|\bOmega(\btheta) - \bOmega_0\|_2 = \|\bSigma(\btheta) - \bSigma_0\|_2$. Because
\begin{align*}
&\expect_0\left[\Pi_\btheta\left\{\|\bOmega(\btheta) - \bOmega_0\|_2 > M\sqrt{\frac{s\log p}{n}}\mathrel{\Bigg|}\bY_n\right\}\right]\\
&\quad\leq \expect_0\left[\Pi^\infty_\btheta\left\{\|\bSigma(\btheta) - \bSigma_0\|_2 > M\sqrt{\frac{s\log p}{n}}\mathrel{\Bigg|}\bY_n\right\}\right]\\
&\quad\quad + \expect_0\left\{\|\Pi_\btheta(\mathrm{d}\btheta\mid\bY_n) - \Pi^\infty_\btheta(\mathrm{d}\btheta\mid\bY_n)\|_{\mathrm{TV}}\right\},
\end{align*}
and
% \[
$\expect_0\left\{\|\Pi_\btheta(\mathrm{d}\btheta\mid\bY_n) - \Pi^\infty_\btheta(\mathrm{d}\btheta\mid\bY_n)\|_{\mathrm{TV}}\right\} = o(1)$
% \]
by Lebesgue dominated convergence theorem and Theorem \ref{thm:BvM}. Therefore it suffices to focus on 
\[
\expect_0\left[\Pi^\infty_\btheta\left\{\|\bSigma(\btheta) - \bSigma_0\|_2 > M\sqrt{\frac{s\log p}{n}}\mathrel{\Bigg|}\bY_n\right\}\right].
\]
% In particular, the limiting distribution can be written as
% \begin{align*}
% \Pi^\infty\left(\mathrm{d}\btheta\mid\bY_n\right) = \sum_{S\in\calS_0}\widehat{w}_SQ_{S}(\mathrm{d}\btheta),
% \end{align*}
% where
% where $\kappa_n = \kappa_0s_0$ for some large constant $\kappa_0 > 0$, and 
% \begin{align*}
% \widehat{w}_S&\propto \frac{\pi_p(|S|)}{{p - r\choose |S|}\gamma(|S|)}\det\{2\pi(\bZ_{0S}\transpose\bZ_{0S})^{-1}\}^{-1/2}\exp\left(\frac{1}{2}\|\bZ_{0S}\widehat\btheta_{0S}\|_2^2\right)\mathbbm{1}\{S\in\calS(\kappa_n),S_0\subset S\},\\
Denote $Q_S(\mathrm{d}\btheta) = \{\phi(\btheta_S\mid\widehat{\btheta}_S, (\bZ_{0S}\transpose\bZ_{0S})^{-1})\mathrm{d}\btheta_S\}\{\delta_{\zero}(\mathrm{d}\btheta_{S^c})\}$.
% \end{align*}
For any $S\in\calS_0$, 
% denote $\bU_S(\bvarphi_S)$ the Cayley transform from $\mathbb{R}^{d(S)}$ to $\mathbb{O}(p(S)\times r)$, where $d(S) = |S|r$, $p(S) = (|S| + r)r$, and $\bvarphi_S = \vect(\bA_S)$ for a $|S|\times r$ matrix $\bA_S$. Namely,
% \[
% \bU_S(\bvarphi_S) = \begin{bmatrix*}
% \zero_{r\times r} & -\bA_S\transpose \\ \bA_S & \zero_{|S|\times |S|}
% \end{bmatrix*}.
% \]
% We further 
denote
% \[
$\bSigma_S(\btheta_S) = \bU(\bvarphi_S)\bM\bU(\bvarphi_S)$,
% \]
where $\btheta_S = [\bvarphi_S\transpose, \bmu\transpose]\transpose = [\vect(\bA_S)\transpose, \bmu\transpose]\transpose$, and $\bU(\bvarphi_S)$ denotes the Cayley	parameterization of $\bvarphi_S = \vect(\bA_S)$ from $\mathbb{R}^{|S|r}$ to $\mathbb{O}(|S| + r, r)$. Then from the proof of Lemma \ref{lemma:stochastic_remainder}, we see that for any $\btheta = [\vect(\bA)\transpose, \bmu\transpose]\transpose = [\bvarphi\transpose, \bmu\transpose]\transpose$ with $\mathrm{supp}(\bA) = S\supset S_0$, there exists a permutation matrix $\bP_S$ such that
\[
\bA = \bP_S\begin{bmatrix*}
\bA_S \\ \zero
\end{bmatrix*},\quad
\bA_0 = \bP_S\begin{bmatrix*}
\bA_{0S}\\ \zero
\end{bmatrix*},
\]
which further implies that
\[
\bU(\bvarphi) = \bQ_S\begin{bmatrix*}
\bU(\bvarphi_S) \\ \zero
\end{bmatrix*},\quad
\bU_0 = \bQ_S\begin{bmatrix*}
\bU(\bvarphi_{0S}) \\ \zero
\end{bmatrix*},
\]
where $\bQ_S = \mathrm{diag}(\eye_r, \bP_S)$, and $\bvarphi_{0S} = \vect(\bA_{0S})$ for an appropriate $\bA_{0S}\in\mathbb{R}^{|S|\times r}$. Therefore, for any $\btheta = [\vect(\bA)\transpose, \bmu\transpose]\transpose$ with $\mathrm{supp}(\bA) = S\supset S_0$, we have
\begin{align*}
&\|\bSigma(\btheta) - \bSigma_0\|_2\\
&\quad = \|\bU(\bvarphi)\bM\bU(\bvarphi)\transpose - \bU_0\bM_0\bU_0\transpose\|_2\\
&\quad = \left\|
\bQ_S\begin{bmatrix*}
\bU(\bvarphi_S) \\ \zero
\end{bmatrix*}\bM \begin{bmatrix*}
\bU(\bvarphi_S)\transpose & \zero
\end{bmatrix*}\bQ_S\transpose - \bQ_S\begin{bmatrix*}
\bU(\bvarphi_{0S}) \\ \zero
\end{bmatrix*}\bM_0 \begin{bmatrix*}
\bU(\bvarphi_{0S})\transpose & \zero
\end{bmatrix*}\bQ_S\transpose
\right\|_2\\
&\quad = \left\|
\bQ_S\begin{bmatrix*}
\bU(\bvarphi_S)\bM\bU(\bvarphi_S) - \bU(\bvarphi_{0S})\bM_0\bU(\bvarphi_{0S})\transpose & \zero \\ \zero & \zero 
\end{bmatrix*}
\bQ_S\transpose
\right\|_2
% \\&
 = \|\bSigma_S(\btheta_S) - \bSigma_S(\btheta_{0S})\|_2,
\end{align*}
where $\btheta_{0S} = [\bvarphi_{0S}\transpose, \bmu_0\transpose]\transpose$. 

\vspace*{2ex}\noindent
Now we proceed to analyze the probability of the event $\{\|\bSigma(\btheta) - \bSigma_0\|_2> M\sqrt{(s_0\log p)/n}\}$ under the $\Pi^\infty_\btheta(\mathrm{d}\btheta\mid\bY_0)$ distribution. By Lemma \ref{lemma:normal_restriction}, there exists some constants $M_1 > 0$, such that $\Pi_\btheta^\infty(\btheta\in\widetilde{\calA}_n^c\mid\bY_n) = o_{\prob_0}(1)$, where
\[
\widetilde{\calA}_n = \left\{\btheta:\|\btheta - \btheta_0\|_1\leq M_1\sqrt{\frac{r^2s^2\log n + rs^2\log p}{n}}\right\}.
\]
Therefore,
\begin{align*}
&\Pi^\infty_\btheta\left\{\|\bSigma(\btheta) - \bSigma_0\|_2 > M\sqrt{\frac{s\log p}{n}}\mathrel{\Bigg|}\bY_n\right\}\\
&\quad\leq \Pi^\infty_\btheta\left\{\|\bSigma(\btheta) - \bSigma_0\|_2 > M\sqrt{\frac{s\log p}{n}}, \btheta\in\widetilde{\calA}_n\mathrel{\Bigg|}\bY_n\right\} + \Pi^\infty_\btheta\left(\btheta\in\widetilde{\calA}_n^c\mid\bY_n\right)\\
&\quad = \sum_{S\in \calS_0}\widehat{w}_SQ_S\left\{\btheta: \|\bSigma(\btheta) - \bSigma_0\|_2 > M\sqrt{\frac{s\log p}{n}},\btheta\in\widetilde{\calA}_n\right\} + o_{\prob_0}(1)\\
&\quad = \sum_{S\in \calS_0}\widehat{w}_SQ_S\left\{\btheta: \|\bSigma_S(\btheta_S) - \bSigma_{S}(\btheta_{0S})\|_2 > M\sqrt{\frac{s\log p}{n}},\btheta\in\widetilde{\calA}_n\right\} + o_{\prob_0}(1)\\
&\quad = \sum_{S\in \calS_0}\widehat{w}_SQ_S(\calB_n(S)),
\end{align*}
where
\begin{align*}
\calB_n(S)& = \left\{\btheta: \|\bSigma_S(\btheta_S) - \bSigma_S(\btheta_{0S})\|_2 > M\sqrt{\frac{s_0\log p}{n}},\btheta\in\widetilde{\calA}_n, \mathrm{supp}(\bA) = S\right\}.
\end{align*}
Now let $p(S):= |S| + r$, $\calS^{p(S) - 1} := \{\bv\in\mathbb{R}^{p(S)}:\|\bv\|_2 = 1\}$ be the unit sphere in $\mathbb{R}^{p(S)}$, and let $\calS^{p - 1}(1/5)$ be a $1/5$-net of $\calS^{p(S) - 1}$ with smallest cardinality, namely, for any $\bv\in \calS^{p(S) - 1}$, there exists some $\bu(\bv)\in\calS^{p(S) - 1}(1/5)$, such that $\|\bu(\bv) - \bv\|_2 < 1/5$. It follows that
\begin{align*}
&\|\bOmega_S(\btheta_S) - \bOmega_S(\btheta_{0S})\|_2\\
&\quad = \max_{\bv\in \calS^{p(S) - 1}}|\bv\transpose\{\bOmega_S(\btheta_S) - \bOmega_S(\btheta_{0S})\}\bv|\\
&\quad = \max_{\bv\in \calS^{p(S) - 1}}|\{\bv - \bu(\bv) + \bu(\bv)\}\transpose\{\bOmega_S(\btheta_S) - \bOmega_S(\btheta_{0S})\}\{\bv - \bu(\bv) + \bu(\bv)\}|\\
&\quad\leq \max_{\bv\in \calS^{p(S) - 1}}\{2\|\bv - \bu(\bv)\|_2 + \|\bv - \bu(\bv)\|_2^2\}\|\bOmega_S(\btheta_S) - \bOmega_S(\btheta_{0S})\|_2\\
&\quad\quad + \max_{\bu\in \calS^{p(S) - 1}(1/5)}|\bu\transpose\{\bOmega_S(\btheta_S) - \bOmega_S(\btheta_{0S})\}\bu|\\
&\quad\leq \frac{1}{2}\|\bOmega_S(\btheta_S) - \bOmega_S(\btheta_{0S})\|_2 + \max_{\bu\in \calS^{p(S) - 1}(1/5)}|\bu\transpose\{\bOmega_S(\btheta_S) - \bOmega_S(\btheta_{0S})\}\bu|,
\end{align*}
implying that
\begin{align*}
\|\bOmega_S(\btheta_S) - \bOmega_S(\btheta_{0S})\|_2
\leq 2\max_{\bu\in \calS^{p(S) - 1}(1/5)}|\bu\transpose\{\bOmega_S(\btheta_S) - \bOmega_S(\btheta_{0S})\}\bu|.
\end{align*}
In addition, we also observe that there exists some constant $c > 0$ such that
\[
\log |\calS^{p(S) - 1}(1/5)| \leq c p(S) = |S| + r\leq c\kappa_0s_0.
\]
Clearly, for any $\btheta\in \calB_n(S)$, we have
\[
\|\btheta_S - \btheta_{0S}\|_2\leq \|\btheta_S - \btheta_{0S}\|_1\leq M_1(\btheta)\eps_n(\btheta)\to 0.
\]
Then by Theorem \ref{thm:CT_deviation_Sigma}, we have
\begin{align*}
\vect\{\bSigma_S(\btheta_S) - \bSigma_{S}(\btheta_{0S})\} = D\bSigma_S(\btheta_{0S})(\btheta_S - \btheta_{0S}) + \vect\{\bR_S(\btheta_S, \btheta_{0S})\},
\end{align*}
where 
\[
\|\bR_S(\btheta_S, \btheta_{0S})\|_{\mathrm{F}}\lesssim \|\btheta_S - \btheta_{0S}\|_2^2\leq \|\btheta_S - \btheta_{0S}\|_1^2\lesssim \frac{r^2s^2\log n + rs^2\log p}{n}
\]
whenever $\btheta\in \calB_n(S)$. Hence, for all $\btheta\in\calB_n(S)$ and all $\bu\in\calS^{p(S) - 1}$,
\begin{align*}
&|\bu\transpose\{\bSigma_S(\btheta_S) - \bSigma_S(\btheta_{0S})\}\bu|\\
&\quad = |(\bu\otimes\bu)\transpose\vect\{\bSigma_S(\btheta_S) - \bSigma_S(\btheta_{0S})\}|\\
&\quad \leq |(\bu\otimes\bu)\transpose D\bSigma_S(\btheta_{0S})(\btheta_S - \btheta_{0S})| + |(\bu\otimes\bu)\transpose\vect\{\bR_S(\btheta_S, \btheta_{0S})\}|\\
&\quad = |(\bu\otimes\bu)\transpose D\bSigma_S(\btheta_{0S})(\btheta_S - \btheta_{0S})| + |\bu\transpose\bR_S(\btheta_S, \btheta_{0S})\bu|\\
&\quad\leq |(\bu\otimes\bu)\transpose D\bSigma_S(\btheta_{0S})(\btheta_S - \btheta_{0S})| + \|\bR_S(\btheta_S, \btheta_{0S})\|_{\mathrm{F}}.
\end{align*}
Note that
\begin{align*}
\frac{r^2s^2\log n + rs^2\log p}{n}
% &  \lesssim \sqrt{\frac{s\log p}{n}}\sqrt{\max\left\{\frac{r^4s^3(\log n)^2}{n\log p},\frac{r^2s^3\log p}{n}\right\}}\\
&  \lesssim \sqrt{\frac{s\log p}{n}}\sqrt{\max\left\{\frac{r^4s^3(\log n)}{n},\frac{r^2s^3\log p}{n}\right\}}\\
&  \lesssim \sqrt{\frac{s\log p}{n}}\sqrt{\max\left\{\frac{(r^2s^2\log n)^3}{n},\frac{(rs^2\log p)^3}{n}\right\}} = o(1) \sqrt{\frac{s_0\log p}{n}}
\end{align*}
because of Assumption A4. This implies that $\|\bR_S(\btheta_S, \btheta_{0S})\|_{\mathrm{F}} = o(\sqrt{(s\log p)/n})$ whenever $\btheta\in\widetilde{\calA}_n$. 
Hence, by the union	bound, we further write
\begin{align*}
Q_S(\calB_n(S)) &
\leq Q_S\left\{
\max_{\bu\in \calS^{p(S) - 1}(1/5)}|\bu\transpose\{\bSigma_S(\btheta_S) - \bSigma_S(\btheta_{0S})\}\bu| > \frac{M}{2}\sqrt{\frac{s\log p}{n}}, \btheta\in\widetilde{\calA}_n
\right\}
\\
&\leq \sum_{\bu\in \calS^{p(S) - 1}(1/5)}
Q_S\left\{
|\bu\transpose\{\bSigma_S(\btheta_S) - \bSigma_S(\btheta_{0S})\}\bu| > \frac{M}{2}\sqrt{\frac{s\log p}{n}}, 
\btheta\in\widetilde{\calA}_n
\right\}\\
% & \leq \sum_{\bu\in \calS^{p(S) - 1}(1/5)}
% Q_S\left\{|(\bu\otimes\bu)\transpose D\bSigma_S(\btheta_{0S})(\btheta_S - \btheta_{0S})| + M_1(\btheta)^2\beps_n(\btheta)^2 > \frac{M}{2}\sqrt{\frac{s_0\log p}{n}}\right\}\\
& \leq  \sum_{\bu\in \calS^{p(S) - 1}(1/5)}
Q_S\left\{|(\bu\otimes\bu)\transpose D\bSigma_S(\btheta_{0S})(\btheta_S - \btheta_{0S})| > \frac{M}{4}\sqrt{\frac{s\log p}{n}}\right\}\\
&\leq \sum_{\bu\in \calS^{p(S) - 1}(1/5)}
Q_S\left\{|(\bu\otimes\bu)\transpose D\bSigma_S(\btheta_{0S})(\btheta_S - \widehat{\btheta}_S)| > \frac{M}{8}\sqrt{\frac{s\log p}{n}}\right\}\\
&\quad + \sum_{\bu\in \calS^{p(S) - 1}(1/5)}
Q_S\left\{|(\bu\otimes\bu)\transpose D\bSigma_S(\btheta_{0S})(\widehat{\btheta}_S - \btheta_{0S})| > \frac{M}{8}\sqrt{\frac{s\log p}{n}}\right\}\\
&= \sum_{\bu\in \calS^{p(S) - 1}(1/5)}
Q_S\left\{|(\bu\otimes\bu)\transpose D\bSigma_S(\btheta_{0S})(\btheta_S - \widehat{\btheta}_S)| > \frac{M}{8}\sqrt{\frac{s\log p}{n}}\right\}\\
&\quad + \sum_{\bu\in \calS^{p(S) - 1}(1/5)}
\mathbbm{1}\left\{|(\bu\otimes\bu)\transpose D\bSigma_S(\btheta_{0S})(\widehat{\btheta}_S - \btheta_{0S})| > \frac{M}{8}\sqrt{\frac{s\log p}{n}}\right\}.
\end{align*}
Therefore, we obtain
\begin{align}
&\sum_{S\in\calS_0}\expect_0\{\widehat{w}_S Q_S(\calB_n(S))\}\nonumber\\
\label{eqn:spectral_norm_contraction_Pi_infinity_I}
&\quad\leq \sum_{S\in\calS_0}\sum_{\bu\in \calS^{p(S) - 1}(1/5)}\expect_0\left[Q_S\left\{|(\bu\otimes\bu)\transpose D\bSigma_S(\btheta_{0S})(\btheta_S - \widehat{\btheta}_S)| > \frac{M}{8}\sqrt{\frac{s\log p}{n}}\right\}\right]
\\
\label{eqn:spectral_norm_contraction_Pi_infinity_II}
&\quad\quad
 + 
 \sum_{S\in\calS_0}\sum_{\bu\in \calS^{p(S) - 1}(1/5)}\prob_0\left\{|(\bu\otimes\bu)\transpose D\bSigma_S(\btheta_{0S})(\widehat{\btheta}_S - \btheta_{0S})| > \frac{M}{8}\sqrt{\frac{s\log p}{n}}\right\}.
\end{align}
% For any $S\in \calS(\kappa_n)$, define $\bE_S$ to be the matrix such that $\bZ_{0S} = \bZ_0\bE_{S}$. 
We analyze the two terms on line \eqref{eqn:spectral_norm_contraction_Pi_infinity_I} and line \eqref{eqn:spectral_norm_contraction_Pi_infinity_II} separately. 

\vspace*{2ex}\noindent
$\blacktriangle$
For the term on line \eqref{eqn:spectral_norm_contraction_Pi_infinity_I}, we use the fact that $Q_S$ is a (degenerate) multivariate normal distribution and write
\begin{align*}
&Q_S\left\{|(\bu\otimes\bu)\transpose D\bSigma_S(\btheta_{0S})(\btheta_S - \widehat{\btheta}_S)| > \frac{M}{8}\sqrt{\frac{s_0\log p}{n}}\right\}
% \\&\quad
 = \prob_{\omega_S}\left(|\omega_S| > \frac{M}{8}\sqrt{\frac{s_0\log p}{n}}\right),
\end{align*}
where condition on the data $\bY_n$, and hence, $\widehat{\btheta}_S$,
\begin{align*}
\omega_S &: = (\bu\otimes\bu)\transpose D\bSigma_S(\btheta_{0S})(\btheta_S - \widehat{\btheta}_S)
\overset{Q_S}{\sim} \mathrm{N}\left(0, V_S\right),\\
V_S &: = (\bu\otimes\bu)\transpose D\bSigma_S(\btheta_{0S})
(\bZ_{0S}\transpose\bZ_{0S})^{-1} D\bSigma_S(\btheta_{0S})\transpose(\bu\otimes \bu),
\end{align*}
Note that by Theorem \ref{thm:DSigma_singular_value} and Assumption A2,
\begin{align*}
\|(\bZ_{0S}\transpose\bZ_{0S})^{-1}\|_2
& = \frac{2}{n}\|\{
\bF_S\transpose D\bSigma(\btheta_{0})\transpose(\bOmega_{0}^{-1}\otimes \bOmega_{0}^{-1})D\bSigma(\btheta_{0})\bF_S
\}^{-1}\|_2\\
& = \frac{2}{n}\sigma_{\min}^{-1}\{\bF_S\transpose 
D\bSigma(\btheta_{0})\transpose(\bOmega_{0}^{-1}\otimes \bOmega_{0}^{-1})D\bSigma(\btheta_{0})\bF_S
\}\\
& = \frac{2}{n}\left[\min_{\|\btheta_S\|_2 = 1}
\btheta\transpose_S \bF_S\transpose D\bSigma(\btheta_{0})\transpose(\bOmega_{0}^{-1}\otimes \bOmega_{0}^{-1})D\bSigma(\btheta_{0})\bF_S
\btheta_S\right]^{-1}\\
% &\leq \frac{2}{n}\left[\min_{\|\btheta_S\|_2 = 1}\lambda_{\min}(\bOmega_{0}^{-1}\otimes \bOmega_{0}^{-1})\|D\bSigma(\btheta_{0})\bF_S\btheta_S\|_2^2\right]^{-1}\\
&\leq \frac{2}{n}\left[ \lambda_{\min}(\bOmega_{0}^{-1}\otimes \bOmega_{0}^{-1})\sigma^2_{\min}\{D\bSigma(\btheta_{0})
\}\sigma_{\min}^2(\bF_S)\right]^{-1} = O(1/n),
\end{align*}
implying that
% \begin{align*}
$V_S \leq  \|D\bSigma_S(\btheta_{0S})\|_2^2\|(\bZ_{0S}\transpose\bZ_{0S})^{-1}\|_2\leq {C}/{n}$
% \end{align*}
for some constant $C > 0$.
By Chernoff bound and the fact that $\omega_S\overset{\calL}{ = }-\omega_S$ under $Q_S$, we further have
\begin{align*}
\prob_{\omega_S}\left(|\omega_S| > \frac{M}{8}\sqrt{\frac{s\log p}{n}}\right)
& = \prob_{\omega_S}\left(\omega_S > \frac{M}{8}\sqrt{\frac{s\log p}{n}}\right) + \prob_{\omega_S}\left(\omega_S < -\frac{M}{8}\sqrt{\frac{s\log p}{n}}\right)\\
& = 2\prob_{\omega_S}\left(\omega_S\sqrt{ns\log p} > \frac{Ms\log p}{8}\right)
 % + \prob_{\omega_S}\left(-\sqrt{ns_0\log p}\omega_S > \frac{M{s_0\log p}}{8}\right)
 \\
&\leq \frac{2\expect_{\omega_S}\{\exp(\omega_S\sqrt{ns\log p})\}}{\exp\{(Ms\log p)/8\}}
 = \frac{2\exp\{(V_Sns\log p)/2\}}{\exp\{(Ms\log p)/8\}}
 \\&
 \leq \frac{2\exp\{(Cs\log p)/2\}}{\exp\{(Ms\log p)/8\}}
 % \\&
 = 2\exp\left\{ -\left(\frac{M}{8} - \frac{C}{2}\right)s\log p \right\}.
\end{align*}
Therefore, the term on line \eqref{eqn:spectral_norm_contraction_Pi_infinity_I} is upper bounded by
\begin{align*}
% &\sum_{S\in\calS(\kappa_n):S_0\subset S}\sum_{\bu\in \calS^{p(S) - 1}(1/5)}\expect_0\left[Q_S\left\{|(\bu\otimes\bu)\transpose D\bSigma_S(\btheta_{0S})(\btheta_S - \widehat{\btheta}_S)| > \frac{M}{8}\sqrt{\frac{s_0\log p}{n}}\right\}\right]\\
& 2\sum_{S\in\calS_0}\sum_{\bu\in \calS^{p(S) - 1}(1/5)}\exp\left\{ -\left(\frac{M}{8} - \frac{C}{2}\right)s\log p \right\}\\
&\quad\leq 2\sum_{t = s_0}^{\kappa_0s_0}{p - r\choose t}\exp(c\kappa_0s_0)\exp\left\{ -\left(\frac{M}{8} - \frac{C}{2}\right)s\log p \right\}\\
&\quad\leq 2\kappa_0s_0\exp\left\{\kappa_0 s_0\log p + c\kappa_0s_0 -\left(\frac{M}{8} - \frac{C}{2}\right)s\log p \right\}\to 0
\end{align*}
by taking a sufficiently large $M > 0$ because $s = s_0 + r$. 

\vspace*{2ex}\noindent
$\blacktriangle$ 
We are now left with the concentration of $\widehat{\btheta}_S - \btheta_{0S}$ on line \eqref{eqn:spectral_norm_contraction_Pi_infinity_II}. Since $S_0\subset S$, it follows that
\begin{align*}
\widehat{\btheta}_S - \btheta_{0S}
& = (\bZ_{0S}\transpose\bZ_{0S})^{-1}\bZ_{0S}\transpose(\bZ_{0S}\btheta_{0S} + \beps_n) - \btheta_{0S} = (\bZ_{0S}\transpose\bZ_{0S})^{-1}\bZ_{0S}\transpose\beps_n\\
& = \sqrt{\frac{n}{2}}(\bZ_{0S}\transpose\bZ_{0S})^{-1}\bF_S\transpose\bZ_{0}\transpose\vect\{\bOmega_{0}^{-1/2}(\widehat{\bOmega} - \bOmega_0)\bOmega_0^{-1/2}\}\\
& = {\frac{n}{2}}(\bZ_{0S}\transpose\bZ_{0S})^{-1}\bF_S\transpose D\bSigma(\btheta_0)\transpose\vect\{\bOmega_{0}^{-1}(\widehat{\bOmega} - \bOmega_0)\bOmega_0^{-1}\}\\
& =\left\{\bF_S\transpose D\bSigma(\btheta_0)\transpose(\bOmega_0^{-1}\otimes\bOmega_0^{-1}) D\bSigma(\btheta_0)\bF_S \right\}^{-1}\bF_S\transpose D\bSigma(\btheta_0)\transpose\\
&\quad\times \vect\{\bOmega_{0}^{-1}(\widehat{\bOmega} - \bOmega_0)\bOmega_0^{-1}\}.
\end{align*}
Therefore,
\begin{align*}
(\bu\otimes\bu)\transpose D\bSigma_S(\btheta_{0S})(\widehat{\btheta}_S - \btheta_{0S})
& = \bbeta_S\transpose \vect\{\bOmega_{0}^{-1}(\widehat{\bOmega} - \bOmega_0)\bOmega_0^{-1}\},
\end{align*}
where
\begin{align*}
\bbeta_S\transpose = 
(\bu\otimes\bu)\transpose D\bSigma_S(\btheta_{0S})\left\{\bF_S\transpose D\bSigma(\btheta_0)\transpose(\bOmega_0^{-1}\otimes\bOmega_0^{-1}) D\bSigma(\btheta_0)\bF_S \right\}^{-1}\bF_S\transpose D\bSigma(\btheta_0)\transpose.
\end{align*}
Consider a $p\times p$ matrix $\bB_S$ such that $\vect(\bB_S) = \bbeta_S$. It follows that
\begin{align*}
\rho_S
& = (\bu\otimes\bu)\transpose D\bSigma_S(\btheta_{0S})(\widehat{\btheta}_S - \btheta_{0S})
% \\&
 = \vect(\bB_S)\transpose\vect\{\bSigma_0^{-1}(\widehat{\bSigma} - \bSigma_0)\bSigma_0^{-1}\}\\
& = \mathrm{tr}\left\{\bB_S\transpose\bSigma_0^{-1}
(\widehat{\bSigma} - \bSigma_0)
\bSigma_0^{-1}
\right\}
% \\&
 = \mathrm{tr}\left\{\widetilde{\bB}_S\bSigma_0^{-1}
(\widehat{\bSigma} - \bSigma_0)
\bSigma_0^{-1}
\right\}\\
& = \mathrm{tr}\left\{\bSigma_0^{-1}
\widetilde{\bB}_S\bSigma_0^{-1}
(\widehat{\bSigma} - \bSigma_0)\right\}
% \\&
 = - \mathrm{tr}(\bSigma_0^{-1}\widetilde{\bB}_S) + \frac{1}{n}\mathrm{tr}\left(\bSigma_0^{-1}\widetilde{\bB}_S\bSigma_0^{-1}\sum_{i = 1}^n\by_i\by_i\transpose
\right),
\end{align*}
where $\widetilde{\bB}_S$ is the symmetrization of $\bB_S$ defined by $\widetilde{\bB}_S = (\bB_S + \bB_S\transpose)/2$. Also note that
\begin{align*}
\|\widetilde{\bB}_S\|_{\mathrm{F}} & \leq \|\bB_S\|_{\mathrm{F}} = \|\bbeta_S\|_2\\
& \leq \|D\bSigma(\btheta_{0})\|_2\|D\bSigma_S(\btheta_{0S})\|_2\|\{\bF_S\transpose D\bSigma(\btheta_0)\transpose(\bOmega_0^{-1}\otimes\bOmega_0^{-1})D\bSigma(\btheta_0)\bF_S\}^{-1}\|_2\\
& = O(1).
\end{align*}
Since $\sum_{i = 1}^n\by_i\by_i\sim\mathrm{Wishart}(n, \bSigma_0)$, it follows from the moment-generating function of the Wishart distribution that (see, e.g., Chapter 8 of \cite{eaton1983multivariate}) for any  $u\in\mathbb{R}$ with $u/n\to 0$ and sufficiently large $n$, 
\[
\expect_0\left[\exp\left\{\mathrm{tr}\left(\frac{u}{n}\bSigma_0^{-1}\widetilde{\bB}_S\bSigma_0^{-1}\sum_{i = 1}^n\by_i\by_i\transpose\right)\right\}\right]
= \exp\left\{-\frac{n}{2}\log\det\left(\eye - \frac{2u}{n}\bSigma_0^{-1}\widetilde{\bB}_S\right)\right\}.
\]
Observe that $\bSigma_0^{-1}\widetilde{\bB}_S$ and $\bSigma_0^{-1/2}\widetilde{\bB}_S\bSigma_0^{-1/2}$ are similar matrices having the same set of eigenvalues, that $(2u/n)\lambda_j(\bSigma_0^{-1/2}\widetilde{\bB}_S\bSigma_0^{-1/2}) = o(1)$, and that $\log(1 + x)\geq x - x^2$for sufficiently small $|x|$, we further write
\begin{align*}
\log\det\left(\eye - \frac{2u}{n}\bSigma_0^{-1}\widetilde{\bB}_S\right)
& = \sum_{j = 1}^p\log\lambda_j\left(\eye - \frac{2u}{n}\bSigma_0^{-1}\widetilde{\bB}_S\right)
  = \sum_{j = 1}^p\log\left\{1 - \frac{2u}{n}\lambda_j(\bSigma_0^{-1}\widetilde{\bB}_S)\right\}\\
& = \sum_{j = 1}^p\log\left\{1 - \frac{2u}{n}\lambda_j(\bSigma_0^{-1/2}\widetilde{\bB}_S\bSigma_0^{-1/2})\right\}\\
&\geq -\sum_{j = 1}^p\frac{2u}{n}\lambda_j(\bSigma_0^{-1/2}\widetilde{\bB}_S\bSigma_0^{-1/2}) - \sum_{j = 1}^p\left\{\frac{2u}{n}\lambda_j(\bSigma_0^{-1/2}\widetilde{\bB}_S\bSigma_0^{-1/2})\right\}^2\\
& = - \frac{2u}{n}\mathrm{tr}(\bSigma_0^{-1/2}\widetilde{\bB}_S\bSigma_0^{-1/2}) - \frac{4u^2}{n^2}\|\bSigma_0^{-1}\widetilde{\bB}_S\|_{\mathrm{F}}^2\\
& \geq - \frac{2u}{n}\mathrm{tr}(\bSigma_0^{-1}\widetilde{\bB}_S) - \frac{4u^2}{n^2}\|\bSigma_0^{-1}\|_2^2\|\widetilde{\bB}_S\|_{\mathrm{F}}^2\\
&\geq - \frac{2u}{n}\mathrm{tr}(\bSigma_0^{-1}\widetilde{\bB}_S) - \frac{Cu^2}{n^2}.
\end{align*}
for some constant $C > 0$. Therefore, with $u/n = o(1)$, for sufficiently large $n$, we have
\begin{align*}
\expect_0\{\exp(u\rho_S)\}
& = \exp\{-u\mathrm{tr}(\bSigma_0^{-1}\widetilde{\bB}_S)\}\expect_0\left[ \exp\left\{\frac{u}{n}
\mathrm{tr}\left(\bSigma_0^{-1}\widetilde{\bB}_S\bSigma_0^{-1}\sum_{i = 1}^n\by_i\by_i\transpose\right)
\right\} \right]\\
& = \exp\{-u\mathrm{tr}(\bSigma_0^{-1}\widetilde{\bB}_S)\}\exp\left\{-\frac{n}{2}\log\det\left(\eye - \frac{2u}{n}\bSigma_0^{-1}\widetilde{\bB}_S\right)\right\}\\
&\leq \exp\left\{-u\mathrm{tr}(\bSigma_0^{-1}\widetilde{\bB}_S) + \frac{n}{2}\frac{2u}{n}\mathrm{tr}(\bSigma_0^{-1}\widetilde{\bB}_S) + \frac{Cu^2}{2n}\right\}\\
&\leq \exp(Cu^2/n).
\end{align*}
Hence, by the Chernoff bound for normal, we obtain
\begin{align*}
&\prob_0\left\{|(\bu\otimes\bu)\transpose D\bSigma_S(\btheta_{0S})(\widehat{\btheta}_S - \btheta_{0S})| > \frac{M}{8}\sqrt{\frac{s\log p}{n}}\right\}\\
&\quad = \prob_0\left(\rho_S > \frac{M}{8}\sqrt{\frac{s\log p}{n}}\right) + \prob_0\left(\rho_S <- \frac{M}{8}\sqrt{\frac{s\log p}{n}}\right)\\
&\quad\leq \prob_0\left(\rho_S\sqrt{ns\log p} > \frac{Ms\log p}{8}\right) + \prob_0\left(-\rho_S\sqrt{ns\log p} > \frac{Ms\log p}{8}\right)\\
&\quad\leq \frac{\expect_0\{\exp(\rho_S\sqrt{ns\log p})\}}{\exp\{(Ms\log p)/8\}} + \frac{\expect_0\{\exp(-\rho_S\sqrt{ns\log p})\}}{\exp\{(Ms\log p)/8\}}\\
&\quad\leq 2\exp\left\{-\left(\frac{1}{8}M - C\right)s\log p  \right\}.
\end{align*}
Finally, the above bound leads to the following upper bound for the term on line \eqref{eqn:spectral_norm_contraction_Pi_infinity_II}:
\begin{align*}
&\sum_{S\in\calS_0}\sum_{\bu\in \calS^{p(S) - 1}(1/5)}\prob_0\left\{|(\bu\otimes\bu)\transpose D\bSigma_S(\btheta_{0S})(\widehat{\btheta}_S - \btheta_{0S})| > \frac{M}{8}\sqrt{\frac{s\log p}{n}}\right\}\\
&\quad\leq 2\sum_{S\in\calS_0}\sum_{\bu\in \calS^{p(S) - 1}(1/5)}\exp\left\{-\left(\frac{1}{8}M - C\right)s\log p\right\}\\
&\quad\leq 2\sum_{t = s_0}^{\kappa_0s_0}{p - r \choose t}|S^{(t + r) - 1}(1/5)|\exp\left\{-\left(\frac{1}{8}M - C\right)s\log p\right\}\\
&\quad\leq 2\sum_{t = s_0}^{\kappa_0s_0}(\kappa_0s_0)^{p - r}\exp\left\{c\kappa_0s_0 - \left(\frac{1}{8}M - C\right)s\log p\right\}\\
&\quad\leq 2\kappa_0s_0\exp\left\{(c + 1)\kappa_0s_0\log p-\frac{1}{8}Ms\log p + {Cs\log p}\right\}\to 0
\end{align*}
by taking a sufficiently large $M > 0$. The proof is thus completed.
\end{proof}

% subsection posterior_contraction_in_spectral_norm (end)

% section proof_of_section_sub:bayesian_sparse_pca_and_non_intrinsic_loss (end)

\section{Proofs for Section \ref{sub:stochastic_block_model}} % (fold)
\label{sec:proofs_for_section_sub:stochastic_block_model}

In this section we provide the proof Theorem \ref{thm:OSE_SBM}. The proof is a modification of the asymptotic normality for classical M-estimators in a parametric model established in Theorem 5.45 in \cite{van2000asymptotic} but also relies on two technical lemmas established in Section \ref{sub:technical_lemmas_for_section_sub:stochastic_block_model}. These technical Lemmas are also useful in the proofs involved in Appendix \ref{sec:proofs_for_section_sub:biclustering}. 

\subsection{Technical lemmas for Section \ref{sub:stochastic_block_model}} % (fold)
\label{sub:technical_lemmas_for_section_sub:stochastic_block_model}

\begin{lemma}\label{lemma:matrix_differential_lemma}
Let $\Theta\subset\mathbb{R}^q$ be open, $\bF:\Theta\to\mathbb{R}^{p\times u_1u_2}$, $\bG:\Theta\to\mathbb{R}^{u_1\times v_1}, \bH:\Theta\to\mathbb{R}^{u_2\times v_2}$ be continuously differentiable matrix-valued functions, and $\bJ\in\mathbb{R}^{v_1\times v_2}$ be a constant matrix. If
\[
\sup_{\btheta\in\Theta}\left\|\frac{\partial\vect\{\bF(\btheta)\}}{\partial\btheta}\right\|_{\mathrm{F}},\quad
\sup_{\btheta\in\Theta}\left\|\frac{\partial\vect\{\bG(\btheta)\}}{\partial\btheta}\right\|_{\mathrm{F}},\quad\text{and}\quad
\sup_{\btheta\in\Theta}\left\|\frac{\partial\vect\{\bH(\btheta)\}}{\partial\btheta}\right\|_{\mathrm{F}}
\]
are bounded, then 
\[
\sup_{\btheta\in \Theta}\left\| \frac{\partial\bF(\btheta)\{\bG(\btheta)\otimes\bH(\btheta)\}\vect(\bF)}{\partial\btheta\transpose}\right\|_{\mathrm{F}} < \infty.
\]
\begin{proof}[Proof of Lemma \ref{lemma:matrix_differential_lemma}]
The proof is a straightforward matrix differential calculus computation. 
Following Theorem 9 in \cite{MAGNUS1985474}, we have
\begin{align*}
&\frac{\partial\bF(\btheta)\{\bG(\btheta)\otimes\bH(\btheta)\}\vect(\bF)}{\partial\btheta\transpose}\\
&\quad = \frac{\partial\bF(\btheta)\vect\{\bH(\btheta)\bJ\bG(\btheta)\transpose\}}{\partial\btheta\transpose}\\
&\quad = \left[\vect\{\bH(\btheta)\bJ\bG(\btheta)\transpose\}\transpose\otimes\eye_{p}\right]\frac{\partial\vect\{\bF(\btheta)\}}{\partial\btheta\transpose} + 
	\bF(\btheta)\frac{\partial\vect\{\bH(\btheta)\bJ\bG(\btheta)\transpose\}}{\partial\btheta\transpose}\\
&\quad = \left[\vect\{\bH(\btheta)\bJ\bG(\btheta)\transpose\}\transpose\otimes\eye_{p}\right]\frac{\partial\vect\{\bF(\btheta)\}}{\partial\btheta\transpose}\\
&\quad\quad + \bF(\btheta)\left[\{\bG(\btheta)\otimes \eye_{u_2}\}(\bJ\transpose\otimes\eye_{u_2})\frac{\partial\vect\{\bH(\btheta)\}}{\partial\btheta\transpose{}} + \{\eye_{u_1}\otimes\bH(\btheta)\bJ\}\frac{\partial\vect\{\bG(\btheta)\transpose{}\}}{\partial\btheta}\right].
\end{align*}
Therefore,
\begin{align*}
\left\|\frac{\partial\bF(\btheta)\{\bG(\btheta)\otimes\bH(\btheta)\}\vect(\bF)}{\partial\btheta\transpose}\right\|_{\mathrm{F}}
&\leq \|\bH(\btheta)\|_{\mathrm{F}}\|\bJ\|_2\|\bG(\btheta)\transpose\|_{2}\|\left\|\frac{\partial\vect\{\bF(\btheta)\}}{\partial\btheta\transpose}\right\|_{\mathrm{F}}\\
&\quad + \|\bF(\btheta)\|_{\mathrm{F}}\|\bG(\btheta)\|_2\|\bJ\|_2\left\|\frac{\partial\vect\{\bH(\btheta)\}}{\partial\btheta\transpose{}}\right\|_{\mathrm{F}}\\
&\quad + \|\bF(\btheta)\|_{\mathrm{F}}\|\bH(\btheta)\|_{2}\|\bJ\|_2\left\|\frac{\partial\vect\{\bG(\btheta)\}}{\partial\btheta\transpose{}}\right\|_{\mathrm{F}},
\end{align*}
and hence,
\[
\sup_{\btheta\in\Theta}\left\|\frac{\partial\bF(\btheta)\{\bG(\btheta)\otimes\bH(\btheta)\}\vect(\bF)}{\partial\btheta\transpose}\right\|_{\mathrm{F}} < \infty.
\]
\end{proof}
\end{lemma}

\begin{lemma}\label{lemma:DDB_theta_bound}
Under the setup and notations in Section \ref{sub:stochastic_block_model}, for every choice $\bar{\btheta}_0$ such that $\bSigma(\bar\btheta_0)\in \mathbb{M}(r)\cap (0, 1)^{K\times K}$, there exists some $\eps > 0$ such that the Jacobian
	\[
	\frac{\partial}{\partial\btheta}\{\vect(\bE_{st})\transpose D\bSigma(\btheta)\}
	\]
	is Lipschitz continuous for all $\btheta\in B_2(\bar{\btheta}_0, \eps)$ for all $s,t\in [K]$.
\end{lemma}

\begin{proof}[Proof of Lemma \ref{lemma:DDB_theta_bound}]
We consider the coordinates of $\vect(\bE_{st})\transpose D\bSigma(\btheta)$. Recall that
\[
D\bSigma(\btheta) = \begin{bmatrix*}
(\eye_{K^2} + \bK_{KK})\{\bU(\bvarphi)\bM\otimes\eye_K\}D\bU(\bvarphi) & \{\bU(\bvarphi)\otimes\bU(\bvarphi)\}\mathbb{D}_r
\end{bmatrix*}.
\]
% where $\bGamma_\bmu$ is such that $\bGamma_\bmu\bmu = \vect(\bM)$, and the entries of $\bmu$ are the upper diagonal entries of $\bM$. 
Denote $\be_j(m)$ be the standard basis vector in $\mathbb{R}^{m}$, where the $j$th coordinate of $\be_j(m)$ is $1$, and the rest of the coordinates of $\be_j(m)$ are zeros. Denote
\[
\bvarphi_{kl} = \vect\{\be_k(K - d)\be_l(d)\transpose\}
\]
 for any $k\in [K - r]$ and $l\in [r]$, $\bv_{kl} = [\bvarphi_{kl}, \zero_{r(r + 1)/2}\transpose]\transpose$, $\btheta = [\bvarphi\transpose, \bmu\transpose]\transpose$, $\bvarphi = \vect(\bA)$ for $\bA\in\mathbb{R}^{(K - r)\times r}$,
\[
\bX_\bvarphi = \begin{bmatrix*}
\zero_{r\times r} & -\bA\transpose\\ \bA & \zero_{(K - r)\times (K - r)}
\end{bmatrix*},
\]
and $\bC(\bvarphi) = (\eye_K - \bX_\bvarphi)^{-1}$. Then
\begin{align*}
% &
\vect(\bE_{st})\transpose D\bSigma(\btheta)
\bv_{kl}
% \begin{bmatrix*}
% \bvarphi_{kl}\\
% \zero_{d(d + 1)/2}
% \end{bmatrix*}
% \\
% &\quad = 
&
=
\vect(\bE_{st})\transpose 
(\eye_{K^2} + \bK_{KK})\{\bU(\bvarphi)\bM\otimes\eye_K\}D\bU(\bvarphi) \bvarphi_{kl}
\\
&
% =2 \vect(\bE_{rs} + \bE_{sr})\transpose\vect\{
% (\eye_K - \bX_\bvarphi)^{-1}\bX_{\bvarphi_{kl}}(\eye_K - \bX_\bvarphi)^{-\mathrm{T}}\bU(\bvarphi)\bM\bU(\bvarphi)\transpose
% \}\\
% &
= 2\vect(\bE_{st} + \bE_{ts})\transpose\vect\{\bC(\bvarphi)\bX_{\bvarphi_{kl}}\bC(\bvarphi)\transpose\bSigma(\btheta)\}.
% &\quad 
\end{align*}
Therefore, using Theorem 9 in \cite{MAGNUS1985474}, we have
\begin{align*}
\frac{\partial}{\partial\btheta\transpose}\{\vect(\bE_{st})\transpose D\bSigma(\btheta)
\bv_{kl}\}
 & = 2\vect(\bE_{st} + \bE_{ts})\transpose \frac{\partial\vect\{\bC(\bvarphi)\bX_{\bvarphi_{kl}}\bC(\bvarphi)\transpose\bSigma(\btheta)\}}{\partial\btheta\transpose}\\
 % & = 2\vect(\bE_{rs} + \bE_{sr})\transpose \left[
 % \{\bSigma(\btheta)\otimes \eye_K\}\frac{\partial\vect\{\bC(\bvarphi)\bX_{\bvarphi_{kl}}\bC(\bvarphi)\transpose\}}{\partial\btheta\transpose}
 % \right]\\
 % &\quad + 2\vect(\bE_{rs} + \bE_{sr})\transpose \left[
 % \{\eye_K\otimes \bC(\bvarphi)\bX_{\bvarphi_{kl}}\bC(\bvarphi)\transpose\}\frac{\partial\vect\{\bSigma(\btheta)\}}{\partial\btheta\transpose}
 % \right]\\
 & = 2\vect(\bE_{st} + \bE_{ts})\transpose 
 % \left[
 \{\bSigma(\btheta)\otimes \eye_K\}\frac{\partial\vect\{\bC(\bvarphi)\bX_{\bvarphi_{kl}}\bC(\bvarphi)\transpose\}}{\partial\btheta\transpose}
 % \right]
 \\
 &\quad + 2\vect(\bE_{st} + \bE_{ts})\transpose 
 % \left[
 \{\eye_K\otimes \bC(\bvarphi)\bX_{\bvarphi_{kl}}\bC(\bvarphi)\transpose\}D\bSigma(\btheta)
 % \right]
 .
\end{align*}
Similarly, for all $e,f\in[r]$ with $e\leq f$, 
\begin{align*}
&\frac{\partial}{\partial\btheta\transpose}\left\{\vect(\bE_{st})\transpose D\bSigma(\btheta)\begin{bmatrix*}
\zero_{(K - r)r}\\ \vect(\bE_{ef})
\end{bmatrix*}\right\}\\
&\quad = \vect({\bE_{st}})\transpose\left[\frac{\partial\vect\{\bU(\bvarphi)\bT_{ef}\bU(\bvarphi)\transpose\}}{\partial\btheta\transpose}\right]\\
&\quad = \begin{bmatrix*}
\vect(\bE_{st})\transpose (\eye_{K^2} + \bK_{KK})\{\bU(\bvarphi)\bT_{ef}\otimes \eye_K\}D\bU(\bvarphi) & \zero_{K^2\times d(d + 1)/2}
\end{bmatrix*}.
\end{align*}
By Lemma \ref{lemma:matrix_differential_lemma}, to show that 
\[
\frac{\partial}{\partial\btheta\transpose}\{\vect(\bE_{st})\transpose D\bSigma(\btheta)\}
\]
is Lipschitz continuous, it suffices show that 
\[
\sup_{\btheta\in B_2(\bar\btheta_0, \eps)}\max\{D_1(\btheta), D_2(\btheta), D_3(\btheta), D_4(\btheta), D_5(\btheta), D_6(\btheta)\} <\infty,
\]
for some sufficiently small $\eps > 0$, where
\begin{align*}
&D_1(\btheta) := \|D\bSigma(\btheta)\|_{F},\quad
D_2(\btheta): = \left\|\frac{\partial\vect\{\bC(\bvarphi)\bX_{\bvarphi_{kl}}\bC(\bvarphi)\transpose\}}{\partial\btheta\transpose}\right\|_{\mathrm{F}}\\
&D_3(\btheta) := \left\|\frac{\partial}{\partial\btheta\transpose}\vect\left[
\frac{\partial\vect\{\bC(\bvarphi)\bX_{\bvarphi_{kl}}\bC(\bvarphi)\transpose\}}{\partial\btheta\transpose}
\right]\right\|_{\mathrm{F}},\quad
D_4(\btheta):=\left\|\frac{\partial\vect\{D\bSigma(\btheta)\}}{\partial\btheta\transpose}\right\|_{\mathrm{F}},\\
&D_5(\btheta):= \|D\bU(\bvarphi)\|,\quad D_6(\btheta) := \left\|\frac{\partial \vect\{D\bU(\bvarphi)\}}{\partial\btheta\transpose}\right\|.
\end{align*}
Note that for any $\btheta\in B_2(\bar\btheta_0, \eps)$,
\begin{align*}
\sup_{\btheta\in B_2(\bar\btheta_0,\eps)} D_1(\btheta) & = \sup_{\btheta\in B_2(\bar\btheta_0,\eps)}\|D\bSigma(\btheta)\|_{\mathrm{F}}
\leq \sup_{\btheta\in B_2(\bar\btheta_0,\eps)}\|\bU(\bvarphi)\otimes\bU(\bvarphi)\|_{\mathrm{F}}\|\mathbb{D}_r\|_2\\
&\quad + 
2\sup_{\btheta\in B_2(\bar\btheta_0,\eps)}\|\bU(\bvarphi)\bM\otimes\eye_K\|_{\mathrm{F}}\|D\bU(\bvarphi)\|_2
% \\&
 < \infty,
\end{align*}
it suffices to focus on the remaining derivative matrices, and we consider them separately. 

\vspace*{2ex}\noindent
$\blacksquare$ We first consider $D_2(\btheta)$. Using Theorem 9 in \cite{MAGNUS1985474}, we have,
\begin{align*}
&\frac{\partial\vect\{\bC(\bvarphi)\bX_{\bvarphi_{kl}}\bC(\bvarphi)\transpose\}}{\partial\btheta\transpose}\\
&\quad = \{\bC(\bvarphi)\otimes \eye_K\}\frac{\partial\vect\{\bC(\bvarphi)\bX_{\bvarphi_{kl}}\}}{\partial\btheta\transpose} + \{\eye_K\otimes \bC(\bvarphi)\bX_{\bvarphi_{kl}}\}\frac{\partial\vect\{\bC(\bvarphi)\transpose\}}{\partial\btheta\transpose}\\
&\quad = \{\bC(\bvarphi)\otimes \eye_K\}(\bX_{\bvarphi_{kl}}\transpose\otimes\eye_K)\frac{\partial\vect\{\bC(\bvarphi)\}}{\partial\btheta\transpose} + \{\eye_K\otimes \bC(\bvarphi)\bX_{\bvarphi_{kl}}\}\bK_{KK}\frac{\partial\vect\{\bC(\bvarphi)\}}{\partial\btheta\transpose}\\
&\quad = (\bK_{KK} - \eye_{K^2})\{\bC(\bvarphi)\bX_{\bvarphi_{kl}}\otimes\eye_K\}
\begin{bmatrix*}
\displaystyle\frac{\partial\vect\{\bC(\bvarphi)\}}{\partial\bvarphi\transpose} & \zero_{K^2\times r(r + 1)/2}
\end{bmatrix*}.
\end{align*}
Denote $[\bvarphi]_t$ be the $t$th coordinate of $\bvarphi$. By matrix differential calculus, for any $t\in [(K - r)r]$, we have
\begin{align*}
\frac{\partial\vect\{\bC(\bvarphi)\}}{\partial[\bvarphi]_t}
& = \vect\left\{\frac{\partial(\eye - \bX_\bvarphi)^{-1}}{\partial[\bvarphi]_t}\right\}
  = -\vect\left\{(\eye_K - \bX_\bvarphi)^{-1}\frac{\partial(\eye - \bX_\bvarphi)}{\partial[\bvarphi]_t}(\eye_K - \bX_\bvarphi)^{-1}\right\}\\
& = \vect\left\{(\eye_K - \bX_\bvarphi)^{-1}\frac{\partial\bX_\bvarphi}{\partial[\bvarphi]_t}(\eye_K - \bX_\bvarphi)^{-1}\right\}\\
& = \{\bC(\bvarphi)\transpose\otimes\bC(\bvarphi)\}\vect\left\{\frac{\partial\bX_\bvarphi}{\partial[\bvarphi]_t}\right\}
% \\&
 = \{\bC(\bvarphi)\transpose\otimes\bC(\bvarphi)\}\bGamma_\bvarphi\be_{t}\{(K - r)r\}.
\end{align*}
Therefore,
\begin{align*}
\frac{\partial\vect\{\bC(\bvarphi)\}}{\partial\bvarphi\transpose} = \{\bC(\bvarphi)\transpose\otimes\bC(\bvarphi)\}\bGamma_\bvarphi,
\end{align*}
and hence,
\begin{align*}
&\frac{\partial\vect\{\bC(\bvarphi)\bX_{\bvarphi_{kl}}\bC(\bvarphi)\transpose\}}{\partial\btheta\transpose}\\
&\quad = 
\begin{bmatrix*}
(\bK_{KK} - \eye_{K^2})\{\bC(\bvarphi)\bX_{\bvarphi_{kl}}\otimes\eye_K\}\{\bC(\bvarphi)\transpose\otimes\bC(\bvarphi)\}\bGamma_\bvarphi
&
\zero_{K^2\times r(r + 1)/2}
\end{bmatrix*}\\
&\quad = 
\begin{bmatrix*}
(\bK_{KK} - \eye_{K^2})\{\bC(\bvarphi)\bX_{\bvarphi_{kl}}\bC(\bvarphi)\transpose\otimes\bC(\bvarphi)\}\bGamma_\bvarphi
&
\zero_{K^2\times r(r + 1)/2}
\end{bmatrix*}.
\end{align*}
These results show that
\begin{align*}
% &\sup_{\btheta:n\|\btheta - \btheta_0\|_2\leq M}\left\|\frac{\partial\vect\{\bC(\bvarphi)\}}{\partial\btheta\transpose}\right\|_{\mathrm{F}}
% \leq \sup_{\btheta:n\|\btheta - \btheta_0\|_2\leq M}\|\bC(\bvarphi)\|_2^2\|\bGamma_\bvarphi\|_{\mathrm{F}} < \infty,\\
\sup_{\btheta\in B_2(\bar\btheta_0, \eps)} D_2(\btheta)
& = \sup_{\btheta\in B_2(\bar\btheta_0, \eps)}
\left\|\frac{\partial\vect\{\bC(\bvarphi)\bX_{\bvarphi_{kl}}\bC(\bvarphi)\transpose\}}{\partial\btheta\transpose}\right\|_{\mathrm{F}} \\
&\leq \sup_{\btheta\in B_2(\bar\btheta_0, \eps)}2\|\bC(\bvarphi)\|_2^3\|\bX_{\bvarphi_{kl}}\|_2\|\bGamma_\bvarphi\|_{\mathrm{F}}
< \infty.
\end{align*}

\vspace*{2ex}\noindent
$\blacksquare$ We next consider $D_3(\btheta)$. We leverage the previous result. 
For any $a\in[K - r], b\in [r]$,
\begin{align*}
\frac{\partial\vect\{\bC(\bvarphi)\bX_{\bvarphi_{kl}}\bC(\bvarphi)\transpose\}}{\partial\btheta\transpose}
\begin{bmatrix*}
\bvarphi_{ab}\\\zero_{r(r + 1)/2}
\end{bmatrix*}
% \\
&
% \quad 
= (\bK_{KK} - \eye_{K^2})\{\bC(\bvarphi)\bX_{\bvarphi_{kl}}\bC(\bvarphi)\transpose\otimes \bC(\bvarphi)\}\bGamma_\bvarphi\bvarphi_{ab}\\
% & = (\bK_{KK} - \eye_{K^2})\{\bC(\bvarphi)\bX_{\bvarphi_{kl}}\bC(\bvarphi)\transpose\otimes \bC(\bvarphi)\}\vect(\bX_{\bvarphi_{ab}})\\
& = (\bK_{KK} - \eye_{K^2})\vect\{\bC(\bvarphi)\bX_{\bvarphi_{ab}}\bC(\bvarphi)\bX_{\bvarphi_{kl}}\transpose\bC(\bvarphi)\transpose\}.
\end{align*}
Using Theorem 9 in \cite{MAGNUS1985474} again, we have,
\begin{align*}
&\left\|\frac{\partial\vect\{\bC(\bvarphi)\bX_{\bvarphi_{ab}}\bC(\bvarphi)\bX_{\bvarphi_{kl}}\transpose\bC(\bvarphi)\transpose\}}{\partial\btheta\transpose{}}\right\|_{\mathrm{F}}\\
&\quad\leq \|\bC(\bvarphi)\bX_{\bvarphi_{kl}}\bC(\bvarphi)\transpose\otimes\eye_K\|_2\left\|\frac{\partial\vect\{\bC(\bvarphi)\bX_{\bvarphi_{ab}}\}}{\partial\btheta}\right\|_{\mathrm{F}}
\\&\quad
 + \|\eye_K\otimes\bC(\bvarphi)\bX_{\bvarphi_{ab}}\|_2\left\|\frac{\partial\vect\{\bC(\bvarphi)\bX_{\bvarphi_{kl}}\bC(\bvarphi)\transpose{}\}}{\partial\btheta\transpose{}}\right\|_{\mathrm{F}}\\
&\quad\leq \|\bC(\bvarphi)\|_2^2\|\bX_{\bvarphi_{kl}}\|_2\|\bX_{\bvarphi_{ab}}\|_2\left\|\frac{\partial\vect\{\bC(\bvarphi)\}}{\partial\btheta}\right\|_{\mathrm{F}}\\
&\quad\quad + \|\bC(\bvarphi)\|_2\|\bX_{\bvarphi_{ab}}\|_2\left\|\frac{\partial\vect\{\bC(\bvarphi)\bX_{\bvarphi_{kl}}\bC(\bvarphi)\transpose{}\}}{\partial\btheta\transpose{}}\right\|_{\mathrm{F}},
\end{align*}
implying that
\begin{align*}
&\sup_{\btheta\in B_2(\bar\btheta_0, \eps)}\left\|
\frac{\partial}{\partial\btheta\transpose}\vect\left[\frac{\partial\vect\{\bC(\bvarphi)\bX_{\bvarphi_{kl}}\bC(\bvarphi)\transpose\}}{\partial\btheta\transpose}\right]
\right\|_{\mathrm{F}}\\
&\quad \leq\sum_{a\in [K - r]}\sum_{b\in [r]}\sup_{\btheta\in B_2(\bar\btheta_0, \eps)}\left\|\frac{\partial}{\partial\btheta\transpose}\vect\left[
\frac{\partial\vect\{\bC(\bvarphi)\bX_{\bvarphi_{kl}}\bC(\bvarphi)\transpose\}}{\partial\btheta\transpose}
\begin{bmatrix*}
\bvarphi_{ab}\\\zero_{r(r + 1)/2}
\end{bmatrix*}
\right]
\right\|_2\\
&\quad\leq \sum_{a\in [K - r]}\sum_{b\in [r]}\|\bK_{KK} - \eye_{K^2}\|_2\sup_{\btheta\in B_2(\bar\btheta_0, \eps)}
\left\|\frac{\partial\vect\{\bC(\bvarphi)\bX_{\bvarphi_{ab}}\bC(\bvarphi)\bX_{\bvarphi_{kl}}\transpose\bC(\bvarphi)\transpose\}}{\partial\btheta\transpose{}}\right\|_{\mathrm{F}}
 < \infty. 
\end{align*}

\vspace*{2ex}\noindent
$\blacksquare$ We finally turn to $D_4(\btheta)$. Recall that
\begin{align*}
D\bSigma(\btheta)
& = \begin{bmatrix*}
(\eye_{K^2} + \bK_{KK})\{\bU(\bvarphi)\bM\otimes\eye_K\}D\bU(\bvarphi) & \{\bU(\bvarphi)\otimes\bU(\bvarphi)\}\mathbb{D}_r
\end{bmatrix*}
\end{align*}
Then for any $c\in [K - r],d\in[r]$, we have
\begin{align*}
D\bSigma(\btheta)\begin{bmatrix*}
\bvarphi_{cd}\\\zero_{r(r + 1)/2}
\end{bmatrix*}
& = (\eye_{K^2} + \bK_{KK})\{\bU(\bvarphi)\bM\otimes \eye_K\}D\bU(\bvarphi)\bvarphi_{cd}\\
& = 2(\eye_{K^2} + \bK_{KK})\vect\{\bC(\bvarphi)\bX_{\bvarphi_{cd}}\bC(\bvarphi)\transpose\bSigma(\btheta)\},
\end{align*}
and for any $e,f\in [r]$, $e\leq f$, 
\begin{align*}
D\bSigma(\btheta)\begin{bmatrix*}
\zero_{(K - r)r}\\
\vect(\bE_{ef})
\end{bmatrix*}
& = \{\bU(\bvarphi)\otimes\bU(\bvarphi)\}\mathbb{D}_r\bE_{ef} = \{\bU(\bvarphi)\otimes\bU(\bvarphi)\}\vect\{\bT_{ef}\}\\
& = \vect\{\bU(\bvarphi)\bT_{ef}\bU(\bvarphi)\transpose\},
\end{align*}
where $\bE_{ef} = \be_e(r)\be_f(r)\transpose$, and $\bT_{ef} = \bE_{ef} + \bE_{fe}$ if $e\neq f$ and $\bT_{ee} = \bE_{ee}$.
In addition, we use Theorem 9 in \cite{MAGNUS1985474} again to compute matrix derivatives
\begin{align*}
&
\sup_{\btheta\in B_2(\bar\btheta_0,\eps)}\left\|\frac{\partial\vect\{\bC(\bvarphi)\bX_{\bvarphi_{cd}}\bC(\bvarphi)\transpose\bSigma(\btheta)\}}{\partial\btheta\transpose}\right\|_{\mathrm{F}}
\\
&
\quad
\leq \sup_{\btheta\in B_2(\bar\btheta_0,\eps)}\|\bSigma(\btheta)\otimes\eye_K\|_2\left\|\frac{\partial\vect\{\bC(\bvarphi)\bX_{\bvarphi_{cd}}\bC(\bvarphi)\transpose\}}{\partial\btheta\transpose}\right\|_{\mathrm{F}}\\
&
\quad
\quad + \sup_{\btheta\in B_2(\bar\btheta_0,\eps)}\|\eye_K\otimes\bC(\bvarphi)\bX_{\bvarphi_{cd}}\bC(\bvarphi)\transpose\|_2\|D\bSigma(\btheta)\|_{\mathrm{F}}\\
&\quad\leq K\sup_{\btheta\in B_2(\bar\btheta_0,\eps)}\left\|\frac{\partial\vect\{\bC(\bvarphi)\bX_{\bvarphi_{cd}}\bC(\bvarphi)\transpose\}}{\partial\btheta\transpose}\right\|_{\mathrm{F}}\\
&\quad\quad + \sup_{\btheta\in B_2(\bar\btheta_0,\eps)}\|\bC(\bvarphi)\|_2^2\|\bX_{\bvarphi_{cd}}\transpose\|_2\|D\bSigma(\btheta)\|_{\mathrm{F}}
% \\&
<\infty,\\
&\sup_{\btheta\in B_2(\bar\btheta_0,\eps)}\left\|\frac{\partial\vect\{\bU(\bvarphi)\bT_{ef}\bU(\bvarphi)\transpose\}}{\partial\btheta\transpose}\right\|_{\mathrm{F}}\\
&\quad = \sup_{\btheta\in B_2(\bar\btheta_0,\eps)}\left\|\frac{\partial\vect\{\bU(\bvarphi)\bT_{ef}\bU(\bvarphi)\transpose\}}{\partial\bvarphi\transpose}\right\|_{\mathrm{F}}\\
&\quad\leq \sup_{\btheta\in B_2(\bar\btheta_0,\eps)}\|(\eye_{K^2} + \bK_{KK})\{\bU(\bvarphi)\bT_{ef}\otimes\eye_K\}D\bU(\bvarphi)\|_{\mathrm{F}}
% \\&
% \quad
 <\infty,
\end{align*}
which further implies that
\begin{align*}
\sup_{\btheta\in B_2(\bar\btheta_0, \eps)}\left\|\frac{\partial\vect\{D\bSigma(\btheta)\}}{\partial\btheta\transpose{}}\right\|_{\mathrm{F}} < \infty.
\end{align*}
% Hence, we conclude that
% \begin{align*}
% \sup_{\btheta\in B_2(\bar\btheta_0, \eps)}\left\|\frac{\partial^2}{\partial\btheta\partial\btheta\transpose}\{\vect(\bE_{st})\transpose D\bSigma(\btheta)\bv_{kl}\}
% \transpose\right\|_{\mathrm{F}}
%  < \infty
% \end{align*}
% for all $k\in[K - r]$ and $l\in [r]$ by Lemma \ref{lemma:matrix_differential_lemma}. 

\vspace*{2ex}\noindent
$\blacksquare$ Finally, we consider $D_5(\btheta)$ and $D_6(\btheta)$ by showing that
\begin{align}
\label{eqn:Lipschitz_continuity_Jacobian_I}
&\sup_{\btheta\in B_2(\bar\btheta_0, \eps)}\left\|\frac{\partial\vect\{\bU(\bvarphi)\}}{\partial\btheta\transpose}\right\|_{\mathrm{F}} = \sup_{\btheta\in B_2(\bar\btheta_0,\eps)}\|D\bU(\bvarphi)\|_{\mathrm{F}} < \infty,\\
\label{eqn:Lipschitz_continuity_Jacobian_II}
&\sup_{\btheta\in B_2(\bar\btheta_0, \eps)}\left\|\frac{\partial\vect\{D\bU(\bvarphi)\}}{\partial\btheta\transpose}\right\|_{\mathrm{F}} = \sup_{\btheta\in B_2(\bar\btheta_0, \eps)}\left\|\frac{\partial\vect\left[2\eye_{p\times r}\transpose\bC(\bvarphi)\transpose\otimes \bC(\bvarphi)\}\bGamma_\bvarphi\right]}{\partial\btheta\transpose}\right\|_{\mathrm{F}} < \infty
\end{align}
for a sufficiently small $\eps > 0$. Equation \eqref{eqn:Lipschitz_continuity_Jacobian_I} follows directly from the fact that 
\[
\|D\bU(\bvarphi)\|_{\mathrm{F}}\lesssim 2\|\bC(\bvarphi)\|_2^2\|\bGamma_\bvarphi\|_{\mathrm{F}} < \infty
\]
 for all $\bvarphi$. Equation \eqref{eqn:Lipschitz_continuity_Jacobian_II} follows from Lemma \ref{lemma:matrix_differential_lemma} and the fact that
 \[
\sup_{\btheta\in B_2(\bar\btheta_0, \eps)}\left\|\frac{\partial\vect\{\bC(\bvarphi)\}}{\partial\btheta\transpose}\right\|_{\mathrm{F}} = \sup_{\btheta\in B_2(\bar\btheta_0, \eps)}\|\{\bC(\bvarphi)\transpose\otimes \bC(\bvarphi)\}\bGamma_\bvarphi\|_{\mathrm{F}} < \infty. 
 \]
% and for any $a\in [K - r], b\in [r]$, we have
% \begin{align*}
% &\frac{\partial}{\partial\btheta\transpose}\vect(\bE_{rs})\transpose (\eye_{K^2} + \bK_{KK})\{\bU(\bvarphi)\bT_{ef}\otimes \eye_K\}D\bU(\bvarphi)\bvarphi_{ab}\\
% &\quad = 2\frac{\partial}{\partial\btheta\transpose}\vect(\bE_{rs})\transpose (\eye_{K^2} + \bK_{KK})\vect\{\bC(\bvarphi)\bX_{\bvarphi_{ab}}\bC(\bvarphi)\transpose\bSigma(\btheta)\}\\
% &\quad = 2\vect(\bE_{rs})\transpose (\eye_{K^2} + \bK_{KK})\frac{\partial\vect\{\bC(\bvarphi)\bX_{\bvarphi_{ab}}\bC(\bvarphi)\transpose\bSigma(\btheta)\}}{\partial\btheta\transpose}.
% \end{align*}
% According to the previous derivation, we have
% \begin{align*}
% &\sup_{\btheta:n\|\btheta - \btheta_0\|_2\leq M}\left\|\frac{\partial}{\partial\btheta\transpose}\vect(\bE_{rs})\transpose (\eye_{K^2} + \bK_{KK})\{\bU(\bvarphi)\bT_{ef}\otimes \eye_K\}D\bU(\bvarphi)\bvarphi_{ab}\right\|\\
% &\quad\leq 
% 4\sup_{\btheta:n\|\btheta - \btheta_0\|_2\leq M}\left\|\frac{\partial\vect\{\bC(\bvarphi)\bX_{\bvarphi_{ab}}\bC(\bvarphi)\transpose\bSigma(\btheta)\}}{\partial\btheta\transpose}\right\|_{\mathrm{F}} < \infty.
% \end{align*}
The proof is thus completed.
\end{proof}

% subsection technical_lemmas_for_section_sub:stochastic_block_model (end)

\subsection{Proof of Theorem \ref{thm:OSE_SBM}} % (fold)
\label{sub:proof_of_theorem_thm:ose_sbm}

We first consider the case where the cluster assignment function $\tau_0$ is known up to a permutation, in the sense that we are aware of an oracle cluster assignment function $\sigma_0:[n]\to [K]$ such that $\tau_0 = \omega\circ\sigma_0$ for some permutation $\omega:[K]\to [K]$. Under the notations and setup in Section \ref{sub:stochastic_block_model}, we consider the following oracle estimators:
\begin{itemize}
	\item[(i)] Let $\bSigma^*_{n,\sigma_0}$ be the $\sigma_0$-dependent oracle estimator for $\bSigma_0$, whose $(s, t)$ entry is given by
	\[
	[\bSigma^*_{n,\sigma_0}]_{st} = \frac{\sum_{i = 1}^n\sum_{j = 1}^nA_{ij}\mathbbm{1}\{\sigma_0(i) = s, \sigma_0(j) = t\}}{\sum_{i = 1}^n\sum_{j = 1}^n\mathbbm{1}\{\sigma_0(i) = s, \sigma_0(j) = t\}}
	\]
	\item[(ii)] Compute the $\sigma_0$-dependent oracle least-squares estimator for $\btheta_0$ by solving
	\[
	\btheta_{n,\sigma_0}^* = \argmin_{\btheta\in\mathscr{D}(K, r)}\|\bSigma^*_{n,\sigma_0} - \bSigma(\btheta)\|_{\mathrm{F}}^2
	\]
	\item[(iii)] Compute the $\sigma_0$-dependent oracle one-step estimator for $\btheta_0$:
	\[
	\widehat{\btheta}^*_{n,\sigma_0} = \btheta^*_{n,\sigma_0} - \eye(\btheta^*, \sigma_0)^{-1}\frac{\partial\ell}{\partial\btheta}(\btheta^*(\sigma_0), \sigma_0),
	\]
	where $\partial\ell/\partial\btheta$ is the score function given by \eqref{eqn:score_function_SBM} and $\eye(\cdot, \cdot)$ is the Fisher information matrix given by \eqref{eqn:Fisher_infor_SBM}. 
\end{itemize}
\begin{lemma}\label{lemma:OSE_oracle}
Under the notations and setup in Section \ref{sub:stochastic_block_model}, there exists a $K\times K$ permutation matrix $\bPi_0$ depending on $\sigma_0$, such that
\[
{n}\bJ_{\bPi_0}(\btheta_{0\bPi_0})^{1/2}(\widehat{\btheta}^*_{n,\sigma_0} - \btheta_{0\bPi_0})\overset{\calL}{\to}\mathrm{N}(\zero_d, \eye_d),
\]
where, for any permutation matrix $\bPi$ and $\btheta\in\mathscr{D}(K, r)$, $\bJ_{\bPi}(\btheta)$ is defined in Theorem \ref{thm:OSE_SBM}. 
\end{lemma}

\begin{proof}[Proof of Lemma \ref{lemma:OSE_oracle}]
We mimic the proof of Theorem 5.45 in \cite{van2000asymptotic} and show that there exists a permutation matrix $\bPi_0$ depending on $\sigma_0$, such that:
\begin{itemize}
	\item[(a)] ${\btheta}^*_{n,\sigma_0}$ is ${n}$-consistent for $\btheta_{0\bPi_0}$, \emph{i.e.}, ${n}({\btheta}^*_{n,\sigma_0} - \btheta_{0\bPi_0}) = O_{\prob_0}(1)$.
	\item[(b)] $(1/n^2)\eye({\btheta}^*_{n,\sigma_0}, \sigma_0)\overset{\prob_0}{\to}\bJ_{\bPi_0}(\btheta_{0\bPi_0})$;
	\item[(c)] For any constant $M > 0$, 
	\[
	\sup_{{n}\|\btheta - \btheta_{0\bPi_0}\|_2 < M}\left\|\frac{1}{n}\left\{\frac{\partial\ell}{\partial\btheta}(\btheta, \sigma_0) - \frac{\partial\ell}{\partial\btheta}(\btheta_{0\bPi_0}, \sigma_0)\right\} - {n}\bJ_{\bPi_0}(\btheta_{0\bPi_0})(\btheta - \btheta_{0\bPi_0})\right\|_2 = o_{\prob_0}(1).
	\]
	% $\blacktriangle$
\end{itemize}
Since there exists a permutation $\omega:[K]\to [K]$ such that $\sigma_0 = \omega\circ\tau_0$, then correspondingly, there exists a permutation matrix $\bPi_0$, such that
\begin{align*}
\bPi_0\begin{bmatrix*}
1\\\vdots\\K
\end{bmatrix*} = \begin{bmatrix*}
\omega^{-1}(1)\\\vdots \\\omega^{-1}(K)
\end{bmatrix*},\quad
\bPi_0\begin{bmatrix*}
\pi_1\\\vdots\\\pi_K
\end{bmatrix*} = \begin{bmatrix*}
\omega^{-1}(\pi_1)\\\vdots \\\omega^{-1}(\pi_K)
\end{bmatrix*}.
\end{align*}
It follows that
\[
[\bPi_0\transpose\bSigma^*_{n,\sigma_0}\bPi_0]_{st} = \frac{\sum_{i = 1}^n\sum_{j = 1}^nA_{ij}\mathbbm{1}\{\tau_0(i) = s, \tau_0(j) = t\}}{\sum_{i = 1}^n\sum_{j = 1}^n\mathbbm{1}\{\tau_0(i) = s, \tau_0(j) = t\}},\quad s,t\in [K].
\]
\begin{itemize}
	\item[$\blacksquare$] We first consider (a). By Lemma 1 in \cite{bickel2013}, we know that 
	\[
	{n}\|\bPi_0\transpose\bSigma^*_{n,\sigma_0}\bPi_0 - \bSigma_0\|_{\mathrm{F}} = O_{\prob_0}(1).
	\]
	Note that neither $\bSigma_0$ nor $\bPi_0\transpose\bSigma^*_{n,\sigma_0}\bPi_0$ is necessarily positive semidefinite. However, $\bSigma_0$ and $(\bPi_0\transpose\bSigma^*_{n,\sigma_0}\bPi_0)^2$ are positive semidefinite so that we can apply Theorem \ref{thm:intrinsic_deviation_Sigma} appropriately. 
	By construction, 
	\begin{align*}
	\|(\bSigma^*_{n,\sigma_0})^2 - \bPi_0\bSigma_0^2\bPi_0\transpose\|_{\mathrm{F}}
	& = \|(\bPi_0\transpose\bSigma^*_{n,\sigma_0}\bPi_0)^2 - \bSigma_0^2\|_{\mathrm{F}}\\
	&\leq (\|\bPi_0\transpose\bSigma^*_{n,\sigma_0}\bPi_0\|_2 + \|\bSigma_0\|_2)\|\bPi_0\transpose\bSigma^*_{n,\sigma_0}\bPi_0 - \bSigma_0\|_{\mathrm{F}}\\
	& \leq 2K\|\bPi_0\transpose\bSigma^*_{n,\sigma_0}\bPi_0 - \bSigma_0\|_{\mathrm{F}} = O_{\prob_0}\left(\frac{1}{{n}}\right).
	\end{align*}
	Let $\btheta^*_{n,\sigma_0}: = [(\bvarphi^*_{n,\sigma_0})\transpose, (\bmu^*_{n,\sigma_0})\transpose]\transpose$ and $\btheta_{0\bPi_0} = [\bvarphi_{0\bPi_0}\transpose, \bmu_{0\bPi_0}\transpose]\transpose$ such that $\bPi_0\bSigma_0\bPi_0\transpose = \bSigma(\btheta_{0\bPi_0}) = \bU(\bvarphi_{0\bPi_0})\bM(\bmu_{0\bPi_0})\bU(\bvarphi_{0\bPi_0})\transpose$. Note that $\bSigma_{n,\sigma_0}^*$ can be written as
	\[
	\bSigma^*_{n,\sigma_0} = \bU(\bvarphi^*_{n,\sigma_0})\bM(\bmu^*_{n,\sigma_0})\bU(\bvarphi^*_{n,\sigma_0})\transpose + \bU_\perp\bLambda_\perp\bU_\perp\transpose,
	\]
	where $\bU_{\perp}$ is the orthogonal complement of $\bU(\bvarphi^*_{n,\sigma_0})$ such that $[\bU(\bvarphi^*_{n,\sigma_0}),\bU_{\perp}]\in\mathbb{O}(K)$, and $\bLambda_\perp = \mathrm{diag}\{\lambda_{r+1}(\bSigma^*_{n,\sigma_0}),\ldots,\lambda_K(\bSigma^*_{n,\sigma_0})\}$. 
	By Davis-Kahan theorem,
	\[
	% \lesssim
	\|\sin\Theta\{\bU(\bvarphi^*_{n,\sigma_0}), \bU(\bvarphi_{0\bPi_0})\}\|_{\mathrm{F}}
	\lesssim \|(\bSigma_{n,\sigma_0}^*)^2 - \bPi_0\bSigma_0^2\bPi_0\transpose\|_{\mathrm{F}} = O_{\prob_0}\left(\frac{1}{{n}}\right).
	\]
	Since any $r$ columns of $\bU(\bvarphi_{0\bPi_0})$ are linearly independent, it follows from Corollary \ref{corr:intrinsic_deviation_Projection} and Theorem \ref{thm:alignment_free_Davis_Kahan} that
	\[
	\|\bU(\bvarphi^*_{n,\sigma_0}) - \bU(\bvarphi_{0\bPi_0})\|_{\mathrm{F}}\lesssim 
	\|\bvarphi^*_{n,\sigma_0} - \bvarphi_{0\bPi_0}\|_2 = O_{\prob_0}\left(\frac{1}{{n}}\right).
	\]
	Therefore, 
	\begin{align*}
	\|\bmu^*_{n,\sigma_0} - \bmu_{0\bPi_0}\|_2
	&\leq \|\bU(\bvarphi^*_{n,\sigma_0})\transpose \bSigma(\btheta^*_{n,\sigma_0}) \bU(\bvarphi^*_{n,\sigma_0}) - \bU(\bvarphi_{0\bPi_0})\transpose \bSigma(\btheta_{0\bPi_0})\bU(\bvarphi_{0\bPi_0})\|_{\mathrm{F}}\\
	& = \|\bU(\bvarphi^*_{n,\sigma_0})\transpose \bSigma^*_{n,\sigma_0} \bU(\bvarphi^*_{n,\sigma_0}) - \bU(\bvarphi_{0\bPi_0})\transpose \bSigma(\btheta_{0\bPi_0})\bU(\bvarphi_{0\bPi_0})\|_{\mathrm{F}}\\
	&\leq \|\bU(\bvarphi^*_{n,\sigma_0})\transpose\bSigma^*_{n,\sigma_0} \{\bU(\bvarphi^*_{n,\sigma_0}) - \bU(\bvarphi_{0\bPi_0})\}\|_{\mathrm{F}}\\
	&\quad + \|\bU(\bvarphi^*_{n,\sigma_0})\transpose\{\bSigma^*_{n,\sigma_0} - \bSigma(\btheta_{0\bPi_0})\}\bU(\bvarphi_{0\bPi_0})\|_{\mathrm{F}}\\
	&\quad + \|\{\bU(\bvarphi^*_{n,\sigma_0}) - \bU(\bvarphi_{0\bPi_0})\}\transpose \bSigma(\btheta_{0\bPi_0})\bU(\bvarphi_{0\bPi_0})\|_{\mathrm{F}}\\
	&\leq \{\|\bSigma_{n,\sigma_0}^*\|_{\mathrm{F}} + \|\bSigma(\btheta_{0\bPi_0})\|_{\mathrm{F}}\}\|\bU(\bvarphi_{n,\sigma_0}^*) - \bU(\bvarphi_{0\bPi_0})\|_{\mathrm{F}}\\
	&\quad + \|\bSigma_{n,\sigma_0}^* - \bSigma(\btheta_{0\bPi_0})\|_{\mathrm{F}} = O_{\prob_0}(n^{-1}),
	\end{align*}
	and hence, $\|\btheta_{n,\sigma_0}^* - \btheta_{0\bPi_0}\|_2\leq \|\bvarphi_{n,\sigma_0}^* - \bvarphi_{0\bPi_0}\|_2 + \|\bmu_{n,\sigma_0}^* - \bmu_{0\bPi_0}\|_2 = O_{\prob_0}(n^{-1})$. \\

	\item[$\blacksquare$] We next consider (b). By the result (b), we know that
	\[
	\prob_0\left\{
	\|\bSigma_{n,\sigma_0} - \bSigma(\btheta_{0\bPi_0})\|_{\max} > \frac{\log n}{n}
	\right\}\leq
	\prob_0\left\{
	\|\bSigma_{n,\sigma_0} - \bSigma(\btheta_{0\bPi_0})\|_{\mathrm{F}} > \frac{\log n}{n}
	\right\}\to 0.
	\]
	Let $\Xi_n = \{\|\bSigma_{n,\sigma_0} - \bSigma(\btheta_{0\bPi_0})\|_{\max} \leq {(\log n)}/{n}\}$. By Theorem \ref{thm:second_order_deviation_CT}, Corollary \ref{corr:intrinsic_deviation_Projection}, and Davis-Kahan theorem, over the event $\Xi_n$, we have,
	\begin{align*}
	\|\bSigma_{n,\sigma_0}^2 - \bSigma(\btheta_{0\bPi_0})^2\|_{\mathrm{F}}
	&\leq 2K\|\bSigma_{n,\sigma_0} - \bSigma(\btheta_{0\bPi_0})\|_{\max}\leq\frac{2K\log n}{n},\\
	\|\bU(\bvarphi_{n,\sigma_0}^*) - \bU(\bvarphi_{0\bPi_0})\|_{2\to\infty}
	&\leq \|\bU(\bvarphi_{n,\sigma_0}^*) - \bU(\bvarphi_{0\bPi_0})\|_{\mathrm{F}}\lesssim 
	\|\bvarphi_{n,\sigma_0}^* - \bvarphi_{0\bPi_0}\|_2\\
	&\lesssim
	% \|\bU()\bU(\widehat{\bvarphi}(\tau_0))\transpose - \bU_0\bU_0\transpose\|_{\mathrm{F}}
	 \|\bSigma_{n,\sigma_0}^2 - \bSigma(\btheta_{0\bPi_0})^2\|_{\mathrm{F}}\lesssim\frac{\log n}{n},\\
	\|\bM(\bmu_{n,\sigma_0}^*) - \bM(\bmu_{0\bPi_0})\|_{\mathrm{F}}
	&\leq \|\bU(\bvarphi^*_{n,\sigma_0})\transpose\bSigma^*_{n,\sigma_0} \{\bU(\bvarphi^*_{n,\sigma_0}) - \bU(\bvarphi_{0\bPi_0})\}\|_{\mathrm{F}}\\
	&\quad + \|\bU(\bvarphi^*_{n,\sigma_0})\transpose\{\bSigma^*_{n,\sigma_0} - \bSigma(\btheta_{0\bPi_0})\}\bU(\bvarphi_{0\bPi_0})\|_{\mathrm{F}}\\
	&\quad + \|\{\bU(\bvarphi^*_{n,\sigma_0}) - \bU(\bvarphi_{0\bPi_0})\}\transpose \bSigma(\btheta_{0\bPi_0})\bU(\bvarphi_{0\bPi_0})\|_{\mathrm{F}}\\
	&\leq \{\|\bSigma_{n,\sigma_0}^*\|_{\mathrm{F}} + \|\bSigma(\btheta_{0\bPi_0})\|_{\mathrm{F}}\}\|\bU(\bvarphi_{n,\sigma_0}^*) - \bU(\bvarphi_{0\bPi_0})\|_{\mathrm{F}}\\
	&\quad + \|\bSigma_{n,\sigma_0}^* - \bSigma(\btheta_{0\bPi_0})\|_{\mathrm{F}}
	% \\&
	\lesssim \frac{\log n}{n}.
	\end{align*}
	Hence, over the event $\Xi_n$,
	\begin{align*}
	&\|\bSigma(\btheta_{n,\sigma_0}^*) - \bSigma(\btheta_{0\bPi_0})\|_{\max}\\
	&\quad = \max_{s,t\in[K]}|\be_s\transpose\{\bSigma(\btheta_{n,\sigma_0}^*) - \bSigma(\btheta_{0\bPi_0})\}\be_t|\\
	&\quad = \max_{s,t\in[K]}|\be_s\transpose\{\bU(\bvarphi_{n,\sigma_0}^*)\bM(\bmu_{n,\sigma_0}^*)\bU(\bvarphi_{n,\sigma_0}^*)\transpose - \bU(\bvarphi_{0\bPi_0})\bM(\bmu_{0\bPi_0})\bU(\bvarphi_{0\bPi_0})\transpose\}\be_t|\\
	&\quad\leq \max_{s,t\in[K]}|\be_s\transpose[\bU(\bvarphi_{n,\sigma_0}^*)\bM(\bmu_{n,\sigma_0}^*)\{\bU(\bvarphi_{n,\sigma_0}^*) - \bU(\bvarphi_{0\bPi_0})\}\transpose]\be_t|\\
	&\quad\quad + \max_{s,t\in[K]}|\be_s\transpose[\bU(\bvarphi_{n,\sigma_0}^*)\{\bM(\bmu_{n,\sigma_0}^*) - \bM(\bmu_{0\bPi_0})\}\bU(\bvarphi_{0\bPi_0})\transpose]\be_t|\\
	&\quad\quad + \max_{s,t\in[K]}|\be_s\transpose[\{\bU(\bvarphi_{n,\sigma_0}^*) - \bU(\bvarphi_{0\bPi_0})\}\bM(\bmu_{0\bPi_0})\bU(\bvarphi_{0\bPi_0})\transpose]\be_t|\\
	&\quad\lesssim \|\bU(\bvarphi_{n,\sigma_0}^*) - \bU(\bvarphi_{0\bPi_0})\|_{2\to\infty} + \|{\bM}(\bmu_{n,\sigma_0}^*) - \bM(\bmu_{0\bPi_0})\|_{\mathrm{F}}\lesssim \frac{\log n}{n},
	\end{align*}
	and hence,
	\begin{align*}
	[\bSigma(\btheta_{n,\sigma_0}^*)]_{st}&\geq [\bSigma(\btheta_{0\bPi_0})]_{st} - \|\bSigma(\btheta_{n,\sigma_0}^*) - \bSigma(\btheta_{0\bPi_0})\|_{\max}\geq \frac{1}{2}[\bSigma(\btheta_{0\bPi_0})]_{st},\\
	1 - [\bSigma(\btheta_{n,\sigma_0}^*)]_{st}&\geq 1 - [\bSigma(\btheta_{0\bPi_0})]_{st} - \|\bSigma(\btheta_{n,\sigma_0}^*) - \bSigma(\btheta_{0\bPi_0})\|_{\max}\geq \frac{1 - [\bSigma(\btheta_{0\bPi_0})]_{st}}{2}.
	\end{align*}
	Now consider a parameter space 
	\[
	\Theta = \left\{\btheta: [\bSigma(\btheta)]_{st} \geq \frac{[\bSigma(\btheta_{0\bPi_0})]_{st}}{2}, 1 - [\bSigma(\btheta)]_{st} \geq \frac{1 - [\bSigma(\btheta_{0\bPi_0})]_{st}}{2},s,t\in [K]\right\}.
	\]
	% and let $n_s(\sigma_0): = \sum_{i = 1}^n\mathbbm{1}\{\sigma_0(i) = s\}$, $s\in [K]$. 
	Then we have
	\begin{align*}
	&\sup_{\btheta\in \Theta}\left\|\frac{1}{n^2}\eye(\btheta, \sigma_0) - \bJ_{\bPi_0}(\btheta)\right\|_{\mathrm{F}}\\
	&\quad\leq \sup_{\btheta\in \Theta}\sum_{s = 1}^K\sum_{t = 1}^K\left|\frac{n_{st}(\sigma_0)}{n^2} - \frac{[\bPi_0\bpi]_s[\bPi_0\bpi]_t}{2}\right|\frac{\|D\bSigma(\btheta)\|_2}{[\bSigma(\btheta)]_{st}\{1 - [\bSigma(\btheta)]_{st}\}}\\
	&\quad = \sup_{\btheta\in \Theta}\sum_{s = 1}^K\sum_{t = 1}^K\left|\frac{1}{2n^2}\sum_{i\neq j}\mathbbm{1}\{\tau_0(i) = \omega^{-1}(s)\}\mathbbm{1}\{\tau_0(j) = \omega^{-1}(t)\} - \frac{\pi_{\omega^{-1}(s)}\pi_{\omega^{-1}(t)}}{2}\right|\\
	&\quad\quad\qquad\qquad\qquad\times\frac{\|D\bSigma(\btheta)\|_2}{[\bSigma(\btheta)]_{st}\{1 - [\bSigma(\btheta)]_{st}\}}\\
	&\quad\lesssim \max_{s,t\in[K]}\frac{4K^2}{[\bSigma(\btheta_{0\bPi_0})]_{st}\{1 - [\bSigma(\btheta_{0\bPi_0})]_{st}\}}\left|\frac{1}{2n^2}\sum_{i\neq j}\mathbbm{1}\{\tau(i) = s\}\mathbbm{1}\{\tau(j) = t\} - \frac{\pi_s\pi_t}{2}\right|
	\\&\quad
	\to 0.
	\end{align*}
	Namely, $(1/n^2)\eye(\cdot, \sigma_0)$ converges to $\bJ_{\bPi_0}(\cdot)$ uniformly over $\btheta\in\Theta$. Therefore, for all $\eps > 0$, there exists some $N(\eps)$, such that 
	\[
	\sup_{\btheta\in \Theta}\left\|\frac{1}{n^2}\eye(\btheta, \sigma_0) - \bJ_{\bPi_0}(\btheta)\right\|_{\mathrm{F}} < \eps/2
	\]
	for all $n\geq N(\eps)$. Also note that $\bJ_{\bPi_0}(\btheta_{n,\sigma_0}^*) \overset{\prob_0}{\to} \bJ_{\bPi_0}(\btheta_{0\bPi_0})$ by the continuous mapping theorem. Hence, for all $\eps > 0$ and all $n \geq N(\eps)$, 
	\begin{align*}
	&\prob_0\left\{\left\|\frac{1}{n^2}\eye(\btheta_{n,\sigma_0}^*, \sigma_0) - \bJ_{\bPi_0}(\btheta_{0\bPi_0})\right\|_{\mathrm{F}} > \eps\right\}\\
	&\quad \leq \prob_0\left\{\left\|\frac{1}{n^2}\eye(\btheta_{n,\sigma_0}^*, \sigma_0) - \bJ_{\bPi_0}(\btheta_{0\bPi_0})\right\|_{\mathrm{F}} > \eps, \bA\in \Xi_n\right\}
	% \\&\quad\quad
	 + \prob_0\left(\Xi_n^c\right)\\
	&\quad \leq \prob_0\left\{\left\|\frac{1}{n^2}\eye(\btheta_{n,\sigma_0}^*, \sigma_0) - \bJ_{\bPi_0}(\btheta_{n,\sigma_0}^*)\right\|_{\mathrm{F}} + \|\bJ_{\bPi_0}(\btheta_{n,\sigma_0}^*) - \bJ_{\bPi_0}(\btheta_{0\bPi_0})\|_{\mathrm{F}} > \eps, \bA\in\Xi_n\right\}\\
	&\quad\quad + o(1)\\
	&\quad \leq \prob_0\left\{\sup_{\btheta\in \Theta}\left\|\frac{1}{n^2}\eye(\btheta, \sigma_0) - \bJ_{\bPi_0}(\btheta)\right\|_{\mathrm{F}} + \|\bJ_{\bPi_0}(\btheta_{n,\sigma_0}^*) - \bJ_{\bPi_0}(\btheta_{0\bPi_0})\|_{\mathrm{F}} > \eps\right\} + o(1)\\
	&\quad \leq \prob_0\left\{ \|\bJ_{\bPi_0}(\btheta_{n,\sigma_0}^*) - \bJ_{\bPi_0}(\btheta_{0\bPi_0})\|_{\mathrm{F}} > \eps/2\right\} + o(1)\to 0.
	\end{align*}

	\item[$\blacksquare$] We finally verify condition (c). By definition,
	\begin{align*}
	&\frac{1}{n^2}\frac{\partial^2\ell}{\partial\btheta\partial\btheta\transpose}(\btheta, \sigma_0)\\
	&\quad = -\frac{\eye(\btheta, \sigma_0)}{n^2} + \frac{1}{n^2}\sum_{s,t\in [K]}\frac{\partial}{\partial\btheta}\{\vect(\bE_{st})\transpose D\bSigma(\btheta)\}\frac{m_{st}(\sigma_0) - n_{st}(\sigma_0)[\bSigma(\btheta)]_{st}}{[\bSigma(\btheta)]_{st}\{1 - [\bSigma(\btheta)]_{st}\}}\\
	&\quad\quad - \frac{1}{n^2}\sum_{s,t\in [K]}\frac{m_{st}(\sigma_0) - n_{st}(\sigma_0)[\bSigma(\btheta)]_{st}}{[\bSigma(\btheta)]_{st}^2\{1 - [\bSigma(\btheta)]_{st}\}^2}\frac{\partial[\bSigma(\btheta)]_{st}\{1 - [\bSigma(\btheta)]_{st}\}\vect(\bE_{st})\transpose D\bSigma(\btheta)}{\partial\btheta}\\
	&\quad = -\frac{\eye(\btheta, \sigma_0)}{n^2} + \frac{1}{n^2}\sum_{s,t\in [K]}\frac{\partial}{\partial\btheta}\{\vect(\bE_{st})\transpose D\bSigma(\btheta)\}\frac{m_{st}(\sigma_0) - n_{st}(\sigma_0)[\bSigma(\btheta)]_{st}}{[\bSigma(\btheta)]_{st}(\btheta)\{1 - [\bSigma(\btheta)]_{st}\}}\\
	&\quad\quad - \sum_{s,t\in [K]}\frac{\{m_{st}(\sigma_0) - n_{st}(\sigma_0)[\bSigma(\btheta)]_{st}\}
	}{n^2[\bSigma(\btheta)]_{st}^2\{1 - [\bSigma(\btheta)]_{st}\}^2}\\
	&\quad\quad\qquad\qquad\times \{1 - 2[\bSigma(\btheta)]_{st}\} D\bSigma(\btheta)\transpose\vect(\bE_{st})
	% \frac{\partial}{\partial\btheta}[B_{rs}(\btheta)\{1 - B_{rs}(\btheta)\}]
	\vect(\bE_{st})\transpose D\bSigma(\btheta)
	.
	\end{align*}
	By Lemma \ref{lemma:DDB_theta_bound}, $D\bSigma(\btheta)$ is Lipschitz continuous over $B_2(\btheta_{0\bPi_0}, \eps)$ for some small $\eps > 0$, and 
	\[
	\frac{\partial}{\partial\btheta}\{\vect(\bE_{rs})\transpose D\bSigma(\btheta)\}
	\]
	is also Lipschitz continuous over $B_2(\btheta_{0\bPi_0}, \eps)$. Therefore, the function class
	\[
	\left\{
	\frac{1}{n^2}\frac{\partial^2\ell}{\partial\btheta\partial\btheta\transpose}(\btheta, \sigma_0)
	\right\}_{n = 1}^\infty
	\]
	is uniformly Lipschitz continuous on $B_2(\btheta_{0\bPi_0}, \eps)$. Hence, by Taylor's theorem, there exists a constant $L_0 > 0$, such that
	\[
	\left\|\left\{\frac{1}{n}\frac{\partial\ell}{\partial\btheta}(\btheta, \sigma_0) - \frac{1}{n}\frac{\partial\ell}{\partial\btheta}(\btheta_{0\bPi_0}, \sigma_0)\right\}
	 - \frac{1}{n}\frac{\partial^2\ell(\btheta_{0\bPi_0}, \sigma_0)}{\partial\btheta\partial\btheta\transpose}(\btheta - \btheta_{0\bPi_0})
	\right\|_{\mathrm{F}}\leq L_0n\|\btheta - \btheta_{0\bPi_0}\|_2^2
	\]
	for all $\btheta\in\{\btheta:n\|\btheta - \btheta_{0\bPi_0}\|_2\leq M\}\subset B_2(\btheta_{0\bPi_0}, \eps)$. Hence,
	\begin{align*}
	&\sup_{\btheta:n\|\btheta - \btheta_{0\bPi_0}\|_2\leq M}
	\left\|\left\{\frac{1}{n}\frac{\partial}{\partial\btheta}(\btheta, \sigma_0) - \frac{1}{n}\frac{\partial}{\partial\btheta}(\btheta_{0\bPi_0}, \sigma_0)\right\}
	 - \frac{1}{n}\frac{\partial^2\ell}{\partial\btheta\partial\btheta\transpose}(\btheta_{0\bPi_0}, \sigma_0)(\btheta - \btheta_{0\bPi_0})
	\right\|_{\mathrm{F}}\\
	&\quad\leq \sup_{\btheta:n\|\btheta - \btheta_{0\bPi_0}\|_2\leq M}L_0n\|\btheta - \btheta_{0\bPi_0}\|_2^2 = o(1).
	\end{align*}
	By triangle inequality, it suffices to show
	\begin{align*}
	\sup_{\btheta:n\|\btheta - \btheta_{0\bPi_0}\|_2\leq M}\left\|\frac{1}{n^2}\frac{\partial^2\ell}{\partial\btheta\partial\btheta\transpose}(\btheta_{0\bPi_0}, \sigma_0) + \bJ_{\bPi_0}(\btheta_{0\bPi_0})\right\|_{\mathrm{F}}n\|\btheta - \btheta_{0\bPi_0}\|_{2} = o_{\prob_0}(1).
	\end{align*}
	By construction, the result (b), and the law of large numbers,
	\begin{align*}
	&\frac{1}{n^2}\frac{\partial^2\ell}{\partial\btheta\partial\btheta\transpose}(\btheta_{0\bPi_0}, \sigma_0) + \bJ_{\bPi_0}(\btheta_{0\bPi_0})\\
	&\quad = \bJ_{\bPi_0}(\btheta_{0\bPi_0}) - \frac{\eye(\btheta_{0\bPi_0}, \sigma_0)}{n^2}\\
	&\quad\quad + \frac{1}{n^2}\sum_{s,t\in [K]}\left[\frac{\partial}{\partial\btheta}\{\vect(\bE_{st})\transpose D\bSigma(\btheta)\}\mathrel{\Bigg|}_{\btheta = \btheta_{0\bPi_0}}\right]\frac{m_{st}(\sigma_0) - n_{st}(\sigma_0)[\bSigma(\btheta_{0\bPi_0})]_{st}}{[\bSigma(\btheta_{0\bPi_0})]_{st}\{1 - [\bSigma(\btheta_{0\bPi_0})]_{st}\}}\\
	&\quad\quad - \frac{1}{n^2}\sum_{s,t\in [K]}\frac{\{m_{st}(\sigma_0) - n_{st}(\sigma_0)[\bSigma(\btheta_{0\bPi_0})]_{st}\}}{[\bSigma(\btheta_{0\bPi_0})]_{st}^2\{1 - [\bSigma(\btheta_{0\bPi_0})]_{st}\}^2}\\
	&\quad\quad\quad\quad\quad\quad\quad\times
	\{1 - 2[\bSigma(\btheta_{0\bPi_0})]_{st}\} D\bSigma(\btheta_{0\bPi_0})\transpose\vect(\bE_{st})
	% \frac{\partial}{\partial\btheta}[B_{rs}(\btheta)\{1 - B_{rs}(\btheta)\}]
	\vect(\bE_{st})\transpose D\bSigma(\btheta_{0\bPi_0})\\
	&\quad = \bJ_{\bPi_0}(\btheta_{0\bPi_0}) - \frac{\eye(\btheta_{0\bPi_0}, \sigma_0)}{n^2} + o_{\prob_0}(1) = o_{\prob_0}(1).
	\end{align*}
	Hence,
	\begin{align*}
	&\sup_{\btheta:n\|\btheta - \btheta_{0\bPi_0}\|_2\leq M}\left\|\frac{1}{n^2}\frac{\partial^2\ell}{\partial\btheta\partial\btheta\transpose}(\btheta_{0\bPi_0}, \sigma_0) + \bJ_{\bPi_0}(\btheta_{0\bPi_0})\right\|_{\mathrm{F}}n\|\btheta - \btheta_{0\bPi_0}\|_{2}\\
	&\quad\leq M\left\|\frac{1}{n^2}\frac{\partial^2\ell}{\partial\btheta\partial\btheta\transpose}(\btheta_{0\bPi_0}, \sigma_0) + \bJ_{\bPi_0}(\btheta_{0\bPi_0})\right\|_{\mathrm{F}} = o_{\prob_0}(1).
	\end{align*}
\end{itemize}

\vspace*{2ex}\noindent
We now use the theory for the one-step estimator to prove that
\[
{n}\bJ_{\bPi_0}(\btheta_{0\bPi_0})^{1/2}(\widehat{\btheta}_{n,\sigma_0}^* - \btheta_{0\bPi_0})\overset{\calL}{\to}\mathrm{N}(\zero_d, \eye_{d}).
\]
By definition, the results (a), (b), and (c), we have
\begin{align*}
&\frac{1}{n^2}\eye({\btheta}_{n,\sigma_0}^*, \sigma_0)n(\widehat{\btheta}_{n,\sigma_0}^* - \btheta_{0\bPi_0})\\
&\quad = \frac{1}{n^2}\eye({\btheta}_{n,\sigma_0}^*, \sigma_0)n({\btheta}_{n,\sigma_0}^* - \btheta_{0\bPi_0})
% \\&\quad\quad
 - n\left\{
\frac{1}{n^2}\frac{\partial\ell}{\partial\btheta}({\btheta}_{n,\sigma_0}^*, \sigma_0) - \frac{1}{n^2}\frac{\partial\ell}{\partial\btheta}(\btheta_{0\bPi_0}, \sigma_0)
\right\}
\\&\quad\quad
 + \frac{1}{n}\frac{\partial \ell}{\partial\btheta }(\btheta_{0\tau_0}, \sigma_0)\\
&\quad = \left\{
\frac{1}{n^2}\eye(\btheta_{n,\sigma_0}^*, \sigma_0) - \bJ_{\bPi_0}(\btheta_{0\bPi_0})
\right\}n( {\btheta}_{n,\sigma_0}^* - \btheta_{0\bPi_0})
\\&\quad\quad
 - \left[
n\left\{
\frac{1}{n^2}\frac{\partial\ell}{\partial\btheta}(\btheta_{n,\sigma_0}^*, \sigma_0) - \frac{1}{n^2}\frac{\partial\ell}{\partial\btheta}(\btheta_{0\bPi_0}, \sigma_0)
\right\} - n\bJ_{\bPi_0}(\btheta_{0\bPi_0})(\btheta_{n,\sigma_0}^* - \btheta_{0\bPi_0})
\right]
\\&\quad\quad
 + \frac{1}{n}\frac{\partial\ell}{\partial\btheta}(\btheta_{0\bPi_0}, \sigma_0)\\
&\quad = \frac{1}{n}\frac{\partial\ell}{\partial\btheta}(\btheta_{0\bPi_0}, \sigma_0) + o_{\prob_0}(1).
\end{align*}
By the central limit theorem, for any $s,t\in [K]$,
\begin{align*}
Z_{st}& = \frac{m_{rs}(\sigma_0) - n_{rs}(\sigma_0)[\bSigma(\btheta_{0\bPi_0})]_{st}}{n[\bSigma(\btheta_{0\bPi_0})]_{st}\{1 - [\bSigma(\btheta_{0\bPi_0})]_{st}\}}\\
& = \frac{\sqrt{n_{rs}(\sigma_0)}}{n[\bSigma(\btheta_{0\bPi_0})]_{st}\{1 - [\bSigma(\btheta_{0\bPi_0})]_{st}\}}\frac{1}{\sqrt{n_{rs}(\sigma_0)}}\sum_{i < j:\sigma_0(i) = s,\sigma_0(j) = t}\{A_{ij} - \expect_0(A_{ij})\}\\
&\overset{\calL}{\to}\mathrm{N}\left(
0, \frac{[\bPi_0\bpi]_r[\bPi_0\bpi]_s}{2[\bSigma(\btheta_{0\bPi_0})]_{st}\{1 - [\bSigma(\btheta_{0\bPi_0})]_{st}\}}
\right).
\end{align*}
Note that $\{Z_{st}\}_{s,t\in [K]}$ are independent random variables, we immediately obtain
\begin{align*}
\frac{1}{n}\frac{\partial\ell}{\partial\btheta}(\btheta_{0\bPi_0}, \sigma_0)
& = \sum_{s,t\in [K]}\frac{m_{st}(\sigma_0) - n_{st}(\sigma_0)[\bSigma(\btheta_{0\bPi_0})]_{st}}{n[\bSigma(\btheta_{0\bPi_0})]_{st}\{1 - [\bSigma(\btheta_{0\bPi_0})]_{st}\}}D\bSigma(\btheta_{0\bPi_0})\transpose\vect(\bE_{st})\\
& = \sum_{s,t\in [K]}Z_{st}D\bSigma(\btheta_{0\bPi_0})\transpose\vect(\bE_{st})\\
&\overset{\calL}{\to}\mathrm{N}\left(\zero_d, \sum_{s,t\in [K]}\frac{[\bPi_0\bpi]_r[\bPi_0\bpi]_s
D\bSigma(\btheta_{0\bPi_0})\transpose\vect(\bE_{st})\vect(\bE_{st})\transpose D\bSigma(\btheta_{0\bPi_0})
}{2[\bSigma(\btheta_{0\bPi_0})]_{st}\{1 - [\bSigma(\btheta_{0\bPi_0})]_{st}\}}\right)\\
& = \mathrm{N}(\zero_d, \bJ_{\bPi_0}(\btheta_{0\bPi_0})).
\end{align*}
By result (b) and the Slutsky's theorem, we conclude that
\begin{align*}
n\bJ_{\bPi_0}(\btheta_{0\bPi_0})^{1/2}(\widehat{\btheta}_{n,\sigma_0}^* - \btheta_{0\bPi_0})\overset{\calL}{\to}\mathrm{N}(\zero_d, \eye_d).
\end{align*}
\end{proof}

\begin{proof}[Proof of Theorem \ref{thm:OSE_SBM}]
By the strong consistency of $\widehat{\tau}$, there exists a sequence of permutations $(\omega_n)_{n = 1}^\infty$ such that
\[
\prob_0\left[\sum_{i = 1}^n\mathbbm{1}\{\widehat{\tau}(i) \neq \omega_n\circ\tau_0(i)\} = 0\right]\to 1.
\]
Also, for any permutation $\omega:[K]\to[K]$ and the associated permutation matrix $\bPi_\omega$ such that
\[
\bPi_\omega\begin{bmatrix*}
1\\\vdots\\K
\end{bmatrix*} = \begin{bmatrix*}
\omega^{-1}(1)\\\vdots\\\omega^{-1}(K)
\end{bmatrix*},
\]
denote 
\[
Z_{n,\omega} := n\bJ_{\bPi_\omega}(\btheta_{0\bPi_{\omega}})^{1/2}({\btheta}_{n,\omega\circ\tau_0}^* - \btheta_{0\bPi_{\omega}}).
\]
By Lemma \ref{lemma:OSE_oracle}, $Z_{n,\omega}\overset{\calL}{\to}\mathrm{N}(\zero_d, \eye_d)$. Therefore, for any measurable set $A\subset\mathbb{R}^d$, 
\begin{align*}
\min_{\omega:[K]\to [K]}\prob_0(Z_{n,\omega}\in A)\to \int_A\phi(\bx\mid\zero_d,\eye_d)\mathrm{d}\bx,\\
\max_{\omega:[K]\to [K]}\prob_0(Z_{n,\omega}\in A)\to \int_A\phi(\bx\mid\zero_d,\eye_d)\mathrm{d}\bx,
\end{align*}
where the minimum and maximum are taken with regard to all permutations $\omega:[K]\to [K]$. 
Now let $A\subset\mathbb{R}^{d}$ be measurable. First note that
\begin{align*}
0&\leq \prob_0\left[n\bJ_{\bPi_{\omega_n}}(\btheta_{0\bPi_{\omega_n}})^{1/2}(\widehat{\btheta}_n - \btheta_{0\bPi_{\omega_n}})\in A, 
\sum_{i = 1}^n\mathbbm{1}\{\widehat{\tau}(i) \neq \omega_n\circ\tau_0(i)\}
 > 0\right]\\
&\leq \prob_0\left[Z_{n,\omega_n}\in A, 
\sum_{i = 1}^n\mathbbm{1}\{\widehat{\tau}(i) \neq \omega_n\circ\tau_0(i)\} > 0\right]\\
&\leq \prob_0\left[
\sum_{i = 1}^n\mathbbm{1}\{\widehat{\tau}(i) \neq \omega_n\circ\tau_0(i)\}
 > 0\right] \to 0.
\end{align*}
Let $\bPi_n := \bPi_{\omega_n}$. Using the asymptotic normality of $Z_{n,\omega}$, we have
\begin{align*}
&\prob_0\left[n\bJ_{\bPi_n}(\btheta_{0\bPi_n})^{1/2}(\widehat{\btheta}_n - \btheta_{0\bPi_n})\in A, 
\sum_{i = 1}^n\mathbbm{1}\{\widehat{\tau}(i) \neq \omega_n\circ\tau_0(i)\}
 = 0\right]\\
&\quad = \prob_0\left[n\bJ_{\bPi_n}(\btheta_{0\bPi_n})^{1/2}({\btheta}_{n, \omega_n\circ\tau_0}^* - \btheta_{0\bPi_n})\in A, 
\sum_{i = 1}^n\mathbbm{1}\{\widehat{\tau}(i) \neq \omega_n\circ\tau_0(i)\} = 0
\right]\\
&\quad \leq \prob_0(Z_{n,\omega_n}\in A)\leq \max_{\omega:[K]\to [K]}\prob_0(Z_{n,\omega}\in A) \to \int_A\phi(\bx\mid\zero_d,\eye_d)\mathrm{d}\bx,\\
&\prob_0\left[n\bJ_{\bPi_n}(\btheta_{0\bPi_n})^{1/2}(\widehat{\btheta}_n - \btheta_{0\bPi_n})\in A, 
\sum_{i = 1}^n\mathbbm{1}\{\widehat{\tau}(i) \neq \omega_n\circ\tau_0(i)\} = 0
\right]\\
&\quad = \prob_0\left[n\bJ_{\bPi_n}(\btheta_{0\bPi_n})^{1/2}(\btheta_{n,\omega_n\circ\tau_0}^* - \btheta_{0\bPi_n})\in A, 
\sum_{i = 1}^n\mathbbm{1}\{\widehat{\tau}(i) \neq \omega_n\circ\tau_0(i)\}
 = 0\right]\\
&\quad = \prob_0(Z_{n,\omega_n}\in A) + \prob_0\left[\sum_{i = 1}^n\mathbbm{1}\{\widehat{\tau}(i) \neq \omega_n\circ\tau_0(i)\} = 0\right]\\
&\quad\quad - \prob_0\left[\left\{Z_{n,\omega_n}\in A\right\}\cup
\left\{
\sum_{i = 1}^n\mathbbm{1}\{\widehat{\tau}(i) \neq \omega_n\circ\tau_0(i)\} = 0
\right\}\right]\\
&\quad \geq \prob_0(Z_{n,\omega_n}\in A) + \prob_0\left[\sum_{i = 1}^n\mathbbm{1}\{\widehat{\tau}(i) \neq \omega_n\circ\tau_0(i)\} = 0
\right] - 1\\
&\quad \geq \min_{\omega:[K]\to [K]}\prob_0(Z_{n,\omega_n}\in A) + 1 - o(1) - 1
% \\&\quad
 \to \int_A\phi(\bx\mid\zero_d,\eye_d)\mathrm{d}\bx.
\end{align*}
Namely, 
\[
\prob_0\left[n\bJ_{\bPi_n}(\btheta_{0\bPi_n})^{1/2}(\widehat{\btheta}_n - \btheta_{0\bPi_n})\in A, \sum_{i = 1}^n\mathbbm{1}\{\widehat{\tau}(i) \neq \omega_n\circ\tau_0(i)\} = 0\right]
 \to\int_A\phi(\bx\mid\zero_d,\eye_d)\mathrm{d}\bx.
\]
Hence, for any measurable set $A\subset\mathbb{R}^{d}$, 
\begin{align*}
&\prob_0\left[n\bJ_{\bPi_n}(\btheta_{0\bPi_n})^{1/2}(\widehat{\btheta}_n - \btheta_{0\bPi_n}) \in A\right]\\
&\quad = \prob_0\left[n\bJ_{\bPi_n}(\btheta_{0\bPi_n})^{1/2}(\widehat{\btheta}_n - \btheta_{0\bPi_n})\in A, 
\sum_{i = 1}^n\mathbbm{1}\{\widehat{\tau}(i) \neq \omega_n\circ\tau_0(i)\} = 0\right]\\
&\quad\quad + \prob_0\left[n\bJ_{\bPi_n}(\btheta_{0\bPi_n})^{1/2}(\widehat{\btheta}_n - \btheta_{0\bPi_n})\in A, \sum_{i = 1}^n\mathbbm{1}\{\widehat{\tau}(i) \neq \omega_n\circ\tau_0(i)\} > 0\right]\\
&\quad = \prob_0\left[n\bJ_{\bPi_n}(\btheta_{0\bPi_n})^{1/2}(\widehat{\btheta}_n - \btheta_{0\bPi_n})\in A, 
\sum_{i = 1}^n\mathbbm{1}\{\widehat{\tau}(i) \neq \omega_n\circ\tau_0(i)\} = 0\right] + o(1)\\
% & = \prob_0\left[ n\{\widehat{\btheta}^{(\mathrm{OS})}(\tau_0) - \btheta_0\}\in A\right] + o(1)\\
&\quad\to \int_A\phi(\bx\mid\zero_d,\eye_d)\mathrm{d}\bx.
\end{align*}
This completes the proof of the claim that $n\bJ_{\bPi_n}(\btheta_{0\bPi_n})^{1/2}(\widehat{\btheta}_n - \btheta_{0\bPi_n})\overset{\calL}{\to}\mathrm{N}(\zero_d, \eye_d)$.
\end{proof}

% subsection proof_of_theorem_thm:ose_sbm (end)

% section proofs_for_section_sub:stochastic_block_model (end)

\section{Proofs for Section \ref{sub:biclustering}} % (fold)
\label{sec:proofs_for_section_sub:biclustering}

This section provides the proofs of Theorems \ref{thm:Kmeans_biclustering} and Theorems \ref{thm:asymptotic_normality_biclustering}. In preparation for doing so, we need several technical lemmas that will be established in Section \ref{sub:technical_lemmas_for_section_sub:biclustering}. Subsequently, Sections \ref{sub:proof_of_theorem_thm:kmeans_biclustering} and \ref{sub:proof_of_theorem_thm:asymptotic_normality_biclustering} present the proofs of two main theorems of interest. We begin this section with the introduction of several notations that are designed for the proofs in this section. Let $\bU\bS\bV\transpose$ be the singular value decomposition of $\bY^*_0$, where $\bU\in\mathbb{O}(m, r)$, $\bV\in\mathbb{O}(n, r)$, and $\bS = \mathrm{diag}\{\sigma_1(\bY_0^*),\ldots,\sigma_r(\bY_0^*)\}$. For any cluster assignment functions $\tau:[m]\to [p_1],\gamma:[n]\to [p_2]$, let $m_s(\tau) := \sum_{i = 1}^m\mathbbm{1}\{\tau(i) = s\}$, $n_t(\gamma):= \sum_{j = 1}^n\mathbbm{1}\{\gamma(j) = t\}$, $s\in [p_1],t\in [p_2]$. For any $s\in [p_1],t\in [p_2]$, let $\bE_{st}$ be the $p_1\times p_2$ matrix of all zeros except $1$ at the $(s,t)$th element. Denote
 % the singular value decomposition of $\expect_0\bY$ as $\expect_0\bY
 % = \bP\bSigma\bQ\transpose = \bU\bS\bV\transpose$, where $\bS\in\mathbb{R}^{r\times r}$ is diagonal with diagonal entries being singular values of $\expect_0\bY$. Also, denote 
 the singular value decomposition of $\bP_0\bW_1 = \bU_\bP\bS_\bP\bV_\bP\transpose$, $\bQ_0\bW_2 = \bU_\bQ\bS_\bQ\bV_\bQ\transpose$, where $\bV_\bP, \bV_\bQ\in\mathbb{O}(r)$, $\bU_\bP\in\mathbb{O}(m, r)$, $\bU_\bQ\in\mathbb{O}(n, r)$, $\bS_\bP$ and $\bS_\bQ$ are diagonal matrix of singular values: $\bS_\bP = \mathrm{diag}\{\sigma_1(\bP_0\bW_1),\ldots,\sigma_r(\bP_0\bW_1)\}$, and $\bS_\bQ = \mathrm{diag}\{\sigma_1(\bQ_0\bW_2),\ldots,\sigma_r(\bQ_0\bW_2)\}$. Denote $\be_j(m)$ be the standard basis vector in $\mathbb{R}^{m}$, where the $j$th coordinate of $\be_j(m)$ is $1$, and the rest of the coordinates of $\be_j(m)$ are zeros. 

\subsection{Technical lemmas for Section \ref{sub:biclustering}} % (fold)
\label{sub:technical_lemmas_for_section_sub:biclustering}

\begin{lemma}\label{thm:two_to_infinity_biclustering}
Uthe notations and setup in Sections \ref{sub:euclidean_representation_of_subspaces}, \ref{sub:extension_to_general_rectangular_matrices}, and \ref{sub:biclustering}, there exists $\bW_\bU,\bW_\bV\in\mathbb{O}(r)$ such that
\begin{align*}
&\|\widehat{\bU} - \bU\bW_\bU\|_{2\to\infty} = O_{\prob_0}\left\{
\frac{1}{\sqrt{m}}\left(\frac{\log m}{\sqrt{n}} + \frac{1}{\sqrt{m}}\right)
\right\},\\
&\|\widehat{\bV} - \bV\bW_\bV\|_{2\to\infty} = O_{\prob_0}\left\{
\frac{1}{\sqrt{n}}\left(\frac{\log n}{\sqrt{m}} + \frac{1}{\sqrt{n}}\right)
\right\}.
\end{align*}
\end{lemma}

\begin{proof}[Proof of Lemma \ref{thm:two_to_infinity_biclustering}]
% Suppose $\bSigma$ has singular value decomposition $\bSigma = \bW_1\bD\bW_2\transpose$, where $\bW_1\in\mathbb{O}(p_1, r)$, $\bW_2 \in\mathbb{O}(p_2, r)$, and $\bD = \mathrm{diag}\{\sigma_1(\bSigma),\ldots,\sigma_r(\bSigma)\}$ with $\sigma_1(\bSigma)\geq\ldots\geq\sigma_r(\bSigma) > 0$. 
First write
 \[
 \bU\bS\bV\transpose = \bP_0\bSigma_0\bQ_0\transpose = \bP_0\bW_1\bD\bW_2\transpose\bQ_0\transpose = \bU_\bP\bS_\bP\bV_\bP\transpose\bD\bV_\bQ\bS_\bQ\bU_\bQ\transpose.
 \]
 Note that $\bS_\bP\bV_\bP\transpose\bD\bV_\bQ\bS_\bQ$ is a $r\times r$ full rank matrix. This implies that $\mathrm{Span}(\bU) = \mathrm{Span}(\bU_\bP)$ and $\mathrm{Span}(\bV) = \mathrm{Span}(\bV_\bP)$. In addition, 
\begin{align*}
\max_{s\in[p_1]}m_s(\tau_0)\eye_r& = \max_{s\in[p_1]}m_s(\tau_0)\bW_1\transpose\bW_1\succeq \bW_1\transpose\mathrm{diag}\{m_1(\tau_0),\ldots,m_{p_1}(\tau_0)\}\bW_1\\
& = \bW_1\transpose\bP_0\transpose\bP_0\bW_1 = \bV_\bP\bS_\bP^2\bV_\bP\transpose\succeq
% \min_{s\in[p_1]}m_s\bV_\bP\bV_\bP\transpose = 
\min_{s\in[p_1]}m_s(\tau_0)\eye_r,\\
\max_{t\in[p_2]}n_t(\gamma_0)\eye_r& = \max_{t\in[p_2]}n_t(\gamma_0)\bW_2\transpose\bW_2\succeq \bW_2\transpose\mathrm{diag}\{n_1(\gamma_0),\ldots,n_{p_2}(\gamma_0)\}\bW_2\\
& = \bW_2\transpose\bQ_0\transpose\bQ_0\bW_2 = \bV_\bQ\bS_\bQ^2\bV_\bQ\transpose\succeq
% \min_{s\in[p_1]}m_s\bW_1\transpose\bW_1 = 
\min_{t\in[p_2]}n_t(\gamma_0)\eye_r.
\end{align*}
 Therefore,
 \begin{align*}
& \sqrt{m} \asymp\sqrt{\min_{s\in [p_1]}m_s(\tau_0)}\leq \sigma_{\min}(\bS_\bP)\leq \sigma_{\max}(\bS_\bP) \leq 
 				\sqrt{\max_{s\in [p_1]}m_s(\tau_0)} \asymp \sqrt{m},\\
& \sqrt{n} \asymp\sqrt{\min_{t\in [p_2]}n_t(\gamma_0)}\leq \sigma_{\min}(\bS_\bQ)\leq \sigma_{\max}(\bS_\bQ) \leq 
 				\sqrt{\max_{t\in [p_2]}n_t(\gamma_0)} \asymp \sqrt{n},
 \end{align*}
 and hence,
 \begin{align*}
 &\|\bU\|_{2\to\infty} = \|\bU_\bP\|_{2\to\infty} = \|\bP_0\bW_1\bV_\bP\bS_\bP^{-1}\|_{2\to\infty}
 \leq \|\bP_0\|_{2\to\infty}\|\bS_\bP^{-1}\|_2\lesssim \frac{1}{\sqrt{m}},\\
 &\|\bV\|_{2\to\infty} = \|\bU_\bQ\|_{2\to\infty} = \|\bQ_0\bW_2\bV_\bQ\bS_\bQ^{-1}\|_{2\to\infty}
 \leq \|\bQ_0\|_{2\to\infty}\|\bS_\bQ^{-1}\|_2\lesssim \frac{1}{\sqrt{n}}.
 \end{align*}
 Also observe that for any square matrix $\bA,\bB\in\mathbb{R}^{r\times r}$, 
 \begin{align*}
 \sigma_r(\bA\bB) &= \min_{\|\bx\|_2 = 1}\|\bA\bB\bx\|_2\geq \min_{\|\bx\|_2 = 1}\sigma_r(\bA)\|\bB\bx\|_2\\
 &\geq \min_{\|\bx\|_2 = 1}\sigma_r(\bA)\sigma_r(\bB)\|\bx\|_2 = \sigma_r(\bA)\sigma_r(\bB).
 \end{align*}
 Note that $\mathrm{Span}(\bU) = \mathrm{Span}(\bU_\bP)$ and $\mathrm{Span}(\bV) = \mathrm{Span}(\bU_\bQ)$, so that there exists orthogonal matrices $\bR_\bP,\bR_\bQ\in\mathbb{O}(r)$, such that $\bU = \bU_\bP\bR_\bP$ and $\bV = \bU_\bQ\bR_\bQ$. Hence,
 \begin{align*}
&\bU\bS\bV\transpose = \bU_\bP\bS_\bP\bV_\bP\transpose\bD\bV_\bQ\bS_\bQ\bU_\bQ\transpose\\
&\quad
\Longrightarrow
\quad
\bS = (\bU\transpose\bU_\bP\bS_\bP\bV_\bP\transpose)\bD(\bV_\bQ\bS_\bQ\bU_\bQ\transpose\bV)
 = (\bR_\bP\transpose\bS_\bP\bV_\bP\transpose)\bD(\bV_\bQ\bS_\bQ\bR_\bQ),
\end{align*}
and hence,
\begin{align*}
\sigma_r(\bP_0\bSigma_0\bQ_0\transpose)
& = \sigma_r(\bS)
\geq \sigma_r(\bR_\bP\transpose\bS_\bP\bV_\bP\transpose)\sigma_r(\bD)\sigma_r(\bV_\bQ\bS_\bQ\bR_\bQ)
% \\
% &\geq \sigma_r(\bSigma)\sqrt{\min_{s\in[p_1]}m_s\min_{t\in [p_2]}n_t}
\gtrsim \sqrt{mn}
\end{align*}
because $\sigma_r(\bD) = \sigma_r(\bSigma_0) > 0$ and $\sigma_r(\bS_\bP)\asymp \sqrt{m}$, $\sigma_r(\bS_\bQ)\asymp\sqrt{n}$. By Corollary 3.3 in \cite{bandeira2016}, we see that $\|\bE\|_2 = O_{\prob_0}(\sqrt{m} + \sqrt{n})$ with $\sqrt{m} + \sqrt{n} = o\{\sigma_r(\bP_0\bSigma_0\bQ_0\transpose)\}$. Therefore, by Theorem 3.7 in \cite{cape2019}, there exists orthogonal matrices $\bW_\bU,\bW_\bV\in\mathbb{O}(r)$, such that
\begin{equation}
\label{eqn:two_to_infinity_bound_I_biclustering}
\begin{aligned}
\|\widehat\bU - \bU\bW_\bU\|_{2\to\infty}
&\leq 2\left\{\frac{\|(\eye_r - \bU\bU\transpose)\bE\bV\bV\transpose\|_{2\to\infty}}{\sigma_r(\bP_0\bSigma_0\bQ_0\transpose)}\right\}\\
&\quad + 2\left\{\frac{\|(\eye_r - \bU\bU\transpose)\bE(\eye_r - \bV\bV\transpose)\|_{2\to\infty}}{\sigma_r(\bP_0\bSigma_0\bQ_0\transpose)}\right\}\|\sin\Theta(\widehat{\bV}, \bV)\|_2\\
&\quad + \|\sin\Theta(\widehat{\bU}, \bU)\|_2^2\|\bU\|_{2\to\infty},
\end{aligned}
\end{equation}
\begin{equation}
\label{eqn:two_to_infinity_bound_II_biclustering}
\begin{aligned}
\|\widehat\bV - \bV\bW_\bV\|_{2\to\infty}
&\leq 2\left\{\frac{\|(\eye_r - \bV\bV\transpose)\bE\transpose\bU\bU\transpose\|_{2\to\infty}}{\sigma_r(\bP_0\bSigma_0\bQ_0\transpose)}\right\}\\
&\quad + 2\left\{\frac{\|(\eye_r - \bV\bV\transpose)\bE\transpose(\eye_r - \bU\bU\transpose)\|_{2\to\infty}}{\sigma_r(\bP_0\bSigma_0\bQ_0\transpose)}\right\}\|\sin\Theta(\widehat{\bU}, \bU)\|_2\\
&\quad + \|\sin\Theta(\widehat{\bV}, \bV)\|_2^2\|\bV\|_{2\to\infty}.
\end{aligned}
\end{equation}
The proof proceeds by working on the individual terms on the right-hand side of \eqref{eqn:two_to_infinity_bound_I_biclustering} and \eqref{eqn:two_to_infinity_bound_II_biclustering}. To this end, we collect the following facts:
\begin{itemize}
	\item[(i)] $\|\bE\|_2 = O_{\prob_0}(\sqrt{m} + \sqrt{n})$. This is a consequence of Corollary 3.3 in \cite{bandeira2016}. 
	\item[(ii)] $\|\sin\Theta(\widehat{\bU}, \bU)\|_2 = O_{\prob_0}(1/\sqrt{n})$ and $\|\sin\Theta(\widehat{\bV}, \bV)\|_2 = O_{\prob_0}(1/\sqrt{m})$. This is a consequence of the unilateral perturbation bound of singular subspace for random matrices due to \cite{cai2018} (see Theorem 3 there). 
	\item[(iii)] $\|\bE\bV\|_{2\to\infty} = o_{\prob_0}(\log m)$ and $\|\bE\transpose\bU\|_{2\to\infty} = o_{\prob_0}(\log n)$. 
	% Denote the $(j, k)$th entry of $\bV$ by $v_{ik}$ for $i\in [m]$ and $k \in [r]$. 
	Then by definition of $\|\cdot\|_{2\to\infty}$, 
	\begin{align*}
	\|\bE\bV\|_{2\to\infty}
	&\leq \sqrt{r}\max_{i\in [m]}\max_{k\in [r]}\left|\sum_{j = 1}^m[\bE]_{ij}[\bV]_{jk}\right|.
	\end{align*}
	Since $[\bE]_{ij}$'s are sub-Gaussian random variables, then by the Hoeffding-type inequality (see, e.g., Proposition 5.10 in \citealp{vershynin2010introduction}) and the union bound, there exists some constant $c > 0$ such that
	\begin{align*}
	\prob_{0}\left(\|\bE\bV\|_{2\to\infty} > \log m\right)
	&\leq \prob_0\left(\sqrt{r}\max_{i\in [m]}\max_{k\in [r]}\left|\sum_{j = 1}^m[\bE]_{ij}[\bV]_{jk}\right| > \log m\right)\\
	&\leq \sum_{i = 1}^m\sum_{k = 1}^r\prob_0\left(\left|\sum_{j = 1}^m[\bE]_{ij}[\bV]_{jk}\right| > \frac{\log m}{\sqrt{r}}\right)\\
	&\leq rm\exp\left\{1 - \frac{c(\log m)^2}{r\sum_{j = 1}^m[\bV]_{jk}^2}\right\} = o(1).
	\end{align*}
	The result for $\|\bE\transpose\bU\|_{2\to\infty}$ follows from a symmetric argument applied to the transpose of $\bY$ and $\bY_0^*$. 
\end{itemize}
Then applying the above facts together with the two-to-infinity norm bound obtained in \eqref{eqn:two_to_infinity_bound_I_biclustering} and \eqref{eqn:two_to_infinity_bound_II_biclustering}, we obtain
\begin{align*}
\|\widehat\bU - \bU\bW_\bU\|_{2\to\infty}
&\lesssim \frac{(1 + \|\bU\bU\transpose\|_{\infty})\|\bE\bV\|_{2\to\infty}}{\sqrt{mn}}\\
&\quad + \frac{(1 + \|\bU\bU\transpose\|_{2\to\infty})\|\bE\|_2\|\eye_r - \bV\bV\transpose\|_2}{\sqrt{mn}}\|\sin\Theta(\widehat{\bV}, \bV)\|_2\\
&\quad + \frac{1}{\sqrt{m}}\|\sin\Theta(\widehat{\bU}, \bU)\|_2^2\\
&\lesssim \frac{\|\bE\bV\|_{2\to\infty}}{\sqrt{mn}} + \frac{\|\bE\|_2\|\sin\Theta(\widehat{\bV}, \bV)\|_2}{\sqrt{mn}} + \frac{1}{\sqrt{m}}\|\sin\Theta(\widehat{\bU}, \bU)\|_2^2\\
& = o_{\prob_0}\left(\frac{\log m}{\sqrt{mn}}\right) + O_{\prob_0}\left\{\frac{(\sqrt{m} + \sqrt{n})}{\sqrt{mn}}\frac{1}{\sqrt{m}}\right\} + O_{\prob_0}\left(\frac{1}{n\sqrt{m}}\right)\\
% & = o_{\prob_0}\left\{\frac{1}{\sqrt{m}}\left(\frac{\log m}{\sqrt{n}} + \frac{\sqrt{m} + \sqrt{n}}{\sqrt{mn}} + \frac{1}{n}\right)\right\}\\
& = O_{\prob_0}\left\{\frac{1}{\sqrt{m}}\left(\frac{\log m}{\sqrt{n}} + \frac{1}{\sqrt{m}}\right)\right\}.
\end{align*}
A symmetric argument applied to the right singular vector $\widehat{\bV}$ and $\bV$ implies
\[
\|\widehat{\bV} - \bV\bW_\bV\|_{2\to\infty} = O_{\prob_0}\left\{\frac{1}{\sqrt{n}}\left(\frac{\log n}{\sqrt{m}} + \frac{1}{\sqrt{n}}\right)\right\}.
\]
The proof is thus completed. 
\end{proof}

\begin{lemma}\label{lemma:DDB_theta_bound_biclustering}
Under the notations and setup in Sections \ref{sub:euclidean_representation_of_subspaces}, \ref{sub:extension_to_general_rectangular_matrices}, and \ref{sub:biclustering}, for every choice $\bar{\btheta}_0$, there exists some $\eps > 0$ such that the Jacobian
	\[
	\frac{\partial}{\partial\btheta}\{\vect(\bE_{st})\transpose D\bSigma(\btheta)\}
	\]
	is Lipschitz continuous for all $\btheta\in B_2(\bar{\btheta}_0, \eps)$ for all $s\in[p_1]$ and $t\in [p_2]$.
\end{lemma}

\begin{proof}[Proof of Lemma \ref{lemma:DDB_theta_bound_biclustering}]
The proof is similar to that of Lemma \ref{lemma:DDB_theta_bound} and is included for completeness. We consider the coordinates of $\vect(\bE_{st})\transpose D\bB(\btheta)$. Recall that
\[
D\bSigma(\btheta) = \begin{bmatrix*}
\bK_{p_2p_1}\{\bM\otimes\eye_{p_2}\}D\bU(\bvarphi) & \bU(\bvarphi)\otimes\eye_{p_1}
\end{bmatrix*},
\]
% where $\bGamma_\bmu$ is such that $\bGamma_\bmu\bmu = \vect(\bM)$, and the entries of $\bmu$ are the upper diagonal entries of $\bM$. 
Denote
\[
\bvarphi_{kl} = \vect\{\be_k(p_2 - r)\be_l(r)\transpose\}
\]
 for any $k\in [p_2 - r]$ and $l\in [r]$, $\bv_{kl} = [\bvarphi_{kl}, \zero_{p_1r}\transpose]\transpose$, $\btheta = [\bvarphi\transpose, \bmu\transpose]\transpose$, $\bvarphi = \vect(\bA)$ for $\bA\in\mathbb{R}^{(p_2 - r)\times r}$,
\[
\bX_\bvarphi = \begin{bmatrix*}
\zero_{r\times r} & -\bA\transpose\\ \bA & \zero_{(p_2 - r)\times (p_2 - r)}
\end{bmatrix*},
\]
and $\bC(\bvarphi) = (\eye_{p_2} - \bX_\bvarphi)^{-1}$. It follows immediately that $\|\bC(\bvarphi)\|_2\leq 1$ for all $\bvarphi$. Furthermore, for any $e\in [p_1]$ and $f\in [r]$, denote
% \[
$\bF_{ef} = \be_e(p_1)\be_f(r)\transpose$.
% \]
Then to show that
\[
	\frac{\partial}{\partial\btheta}\{\vect(\bE_{st})\transpose D\bSigma(\btheta)\}
\]
is Lipschitz continuous within a small neigborhood of $\bar{\btheta}_0$, it suffices to show that
\begin{align}
\label{eqn:Lipschitz_continuity_DSigma_rectangular_I}
\sup_{\btheta:\|\btheta - \bar\btheta_0\|_2 < \eps}\left\|\frac{\partial^2}{\partial\btheta\partial\btheta\transpose}\vect(\bE_{st})\transpose D\bSigma(\btheta)
\bv_{kl}\right\| < \infty
\end{align}
and
\begin{align}
\label{eqn:Lipschitz_continuity_DSigma_rectangular_II}
\sup_{\btheta:\|\btheta - \bar\btheta_0\|_2 < \eps} 
\left\|
\frac{\partial^2}{\partial\btheta\partial\btheta\transpose}\vect(\bE_{st})\transpose D\bSigma(\btheta)
\begin{bmatrix*}
\zero_{(p_1 - r)r} \\ \vect(\bF_{ef})
\end{bmatrix*}
\right\| < \infty
\end{align}
for any $k\in [p_2 - r],l\in [r],e\in [p_1], f\in [r]$. 
We begin the proof by collecting several facts:
\begin{itemize}
	\item[(i)] $(\partial/\partial\bvarphi\transpose)\vect\{\bC(\bvarphi)\} = \{\bC(\bvarphi)\transpose\otimes \bC(\bvarphi)\}\bGamma_\bvarphi$ and 
	\[
	\sup_{\btheta:\|\btheta - \bar\btheta_0\|_2 < \eps}\left\|\frac{\partial\vect\{\bC(\bvarphi)\}}{\partial\bvarphi\transpose}\right\|_{\mathrm{F}} \leq \sup_{\btheta:\|\btheta - \bar\btheta_0\|_2 < \eps}\|\bC(\bvarphi)\|_2^2\|\bGamma_\bvarphi\|_{\mathrm{F}} < \infty.
	\]
	This is an immediate consequence of the proof of Lemma \ref{lemma:DDB_theta_bound}. 
	\item[(ii)] For any $k \in [p_2 - r]$ and $l\in [r]$, 
	\begin{align*}
	&\frac{\partial\vect\{\bC(\bvarphi)\bX_{\bvarphi_{kl}}\bC(\bvarphi)\transpose\}}{\partial\btheta\transpose}	
	=
	\begin{bmatrix*}
	(\bK_{p_2p_2} - \eye_{p_2^2})\{\bC(\bvarphi)\bX_{\bvarphi_{kl}}\bC(\bvarphi)\transpose\otimes\bC(\bvarphi)\}\bGamma_\bvarphi
	&
	\zero_{p_2^2\times p_1r}
	\end{bmatrix*},\\
	&\sup_{\btheta:\|\btheta - \bar\btheta_0\|_2 < \eps}\left\|\frac{\partial\vect\{\bC(\bvarphi)\bX_{\bvarphi_{kl}}\bC(\bvarphi)\transpose\}}{\partial\btheta\transpose}
	\right\|_{\mathrm{F}} < \infty.
	\end{align*}
	This is also a direct consequence of the proof of Lemma \ref{lemma:DDB_theta_bound}. 
	\item[(iii)] For any $k \in [p_2 - r]$ and $l\in [r]$,
	\[
	\sup_{\btheta:\|\btheta - \bar\btheta_0\|_2 < \eps}\left\|\frac{\partial}{\partial\btheta\transpose}\vect\left[
	\frac{\partial\vect\{\bC(\bvarphi)\bX_{\bvarphi_{kl}}\bC(\bvarphi)\transpose\}}{\partial\btheta\transpose}
	\right]
	\right\|_{\mathrm{F}} < \infty.
	\]
	This fact has already been established in the proof of Lemma \ref{lemma:DDB_theta_bound} as well. 
	\item[(iv)] $\|D\bSigma(\btheta)\|_{\mathrm{F}}$ is finite within a small neigborhood of $\bar\btheta_0$. This follows directly from the following computation:
	\begin{align*}
	\sup_{\btheta:\|\btheta - \bar\btheta_0\|_2 < \eps}\|D\bSigma(\btheta)\|_{\mathrm{F}}
	&\leq \sup_{\btheta:\|\btheta - \bar\btheta_0\|_2 < \eps}\|\bK_{p_2p_1}(\bM\otimes \eye_{p_2})D\bU(\bvarphi)\|_{\mathrm{F}} + \|\bU(\bvarphi)\otimes\eye_{p_1}\|_{\mathrm{F}}\\
	&\leq \sup_{\btheta:\|\btheta - \bar\btheta_0\|_2 < \eps}2\|\bM\|_2\|\bC(\bvarphi)\|_2^2\|\bGamma_\bvarphi\|_{\mathrm{F}} + \sup_{\btheta:\|\btheta - \bar\btheta_0\|_2 < \eps}\|\bU(\bvarphi)\|_{\mathrm{F}}\\
	& < \infty.
	\end{align*}
	\item[(v)] $\vect\{D\bSigma(\btheta)\}$ is Lipschitz continuous in a neighborhood of $B_2(\bar{\btheta}_0, \eps)$. Recall that
\begin{align*}
D\bSigma(\btheta)
& = \begin{bmatrix*}
\bK_{p_2p_1}(\bM\otimes\eye_{p_2})D\bU(\bvarphi) & \bU(\bvarphi)\otimes\eye_{p_1}
\end{bmatrix*}
\end{align*}
Then for any $a\in [p_2 - r],b\in[r]$, we have
\begin{align*}
D\bSigma(\btheta)\begin{bmatrix*}
\bvarphi_{ab}\\\zero_{p_2r}
\end{bmatrix*}
& = \bK_{p_2p_1}(\bM\otimes \eye_{p_2})D\bU(\bvarphi)\bvarphi_{ab}\\
& = 2\bK_{p_2p_1}\vect\{\bC(\bvarphi)\bX_{\bvarphi_{ab}}\bC(\bvarphi)\transpose\bSigma(\btheta)\transpose\},
\end{align*}
and for any $e\in [p_1] $, $f\in [r]$, 
\begin{align*}
D\bSigma(\btheta)\begin{bmatrix*}
\zero_{(p_2 - r)r}\\
\vect(\bF_{ef})
\end{bmatrix*}
& = \vect\{\bF_{ef}\bU(\bvarphi)\transpose\}.
\end{align*}
In addition, we use Theorem 9 in \cite{MAGNUS1985474} again to compute matrix derivatives
\begin{align*}
&
\sup_{\btheta\in B_2(\bar\btheta_0,\eps)}\left\|\frac{\partial\vect\{\bC(\bvarphi)\bX_{\bvarphi_{ab}}\bC(\bvarphi)\transpose\bSigma(\btheta)\transpose\}}{\partial\btheta\transpose}\right\|_{\mathrm{F}}
\\
&
\quad
\leq \sup_{\btheta\in B_2(\bar\btheta_0,\eps)}\|\bSigma(\btheta)\otimes\eye_{p_2}\|_2\left\|\frac{\partial\vect\{\bC(\bvarphi)\bX_{\bvarphi_{ab}}\bC(\bvarphi)\transpose\}}{\partial\btheta\transpose}\right\|_{\mathrm{F}}\\
&
\quad
\quad + \sup_{\btheta\in B_2(\bar\btheta_0,\eps)}\|\eye_{p_1}\otimes\bC(\bvarphi)\bX_{\bvarphi_{ab}}\bC(\bvarphi)\transpose\|_2\|D\bSigma(\btheta)\|_{\mathrm{F}}\\
&\quad\leq  \sup_{\btheta\in B_2(\bar\btheta_0,\eps)}\left\|\frac{\partial\vect\{\bC(\bvarphi)\bX_{\bvarphi_{ab}}\bC(\bvarphi)\transpose\}}{\partial\btheta\transpose}\right\|_{\mathrm{F}}\\
&\quad\quad +  \sup_{\btheta\in B_2(\bar\btheta_0,\eps)}\|\bC(\bvarphi)\|_2^2\|\bX_{\bvarphi_{ab}}\transpose\|_2\|D\bSigma(\btheta)\|_{\mathrm{F}}
% \\&
<\infty,
% \\
% &\sup_{\btheta\in B_2(\bar\btheta_0,\eps)}\left\|\frac{\partial\vect\{\bU(\bvarphi)\bT_{ef}\bU(\bvarphi)\transpose\}}{\partial\btheta\transpose}\right\|_{\mathrm{F}}\\
% &\quad = \sup_{\btheta\in B_2(\bar\btheta_0,\eps)}\left\|\frac{\partial\vect\{\bU(\bvarphi)\bT_{ef}\bU(\bvarphi)\transpose\}}{\partial\bvarphi\transpose}\right\|_{\mathrm{F}}\\
% &\quad\leq \sup_{\btheta\in B_2(\bar\btheta_0,\eps)}\|(\eye_{K^2} + \bK_{KK})\{\bU(\bvarphi)\bT_{ef}\otimes\eye_K\}D\bU(\bvarphi)\|_{\mathrm{F}}
% % \\&
% % \quad
%  <\infty,
\end{align*}
which further implies that
\begin{align*}
\sup_{\btheta\in B_2(\bar\btheta_0, \eps)}\left\|\frac{\partial\vect\{D\bSigma(\btheta)\}}{\partial\btheta\transpose{}}\right\|_{\mathrm{F}} < \infty.
\end{align*}
because $\|D\bU(\bvarphi)\|_{\mathrm{F}}$ is always upper bounded by an absolute constant. 
\end{itemize}
By matrix differential calculus,
\begin{align*}
% &
\vect(\bE_{st})\transpose D\bSigma(\btheta)
\bv_{kl}
% \begin{bmatrix*}
% \bvarphi_{kl}\\
% \zero_{d(d + 1)/2}
% \end{bmatrix*}
% \\
% &\quad = 
&
=
\vect(\bE_{st})\transpose 
\bK_{p_2p_1}\{\bM\otimes\eye_{p_2}\}D\bU(\bvarphi) \bvarphi_{kl}
\\
&
=2 \vect(\bE_{st}\transpose)\transpose\vect\{
(\eye_{p_2} - \bX_\bvarphi)^{-1}\bX_{\bvarphi_{kl}}(\eye_{p_2} - \bX_\bvarphi)^{-\mathrm{T}}\bU(\bvarphi)\bM\transpose
\}\\
&
= 2\vect(\bE_{st}\transpose)\transpose\vect\{\bC(\bvarphi)\bX_{\bvarphi_{kl}}\bC(\bvarphi)\transpose\bSigma(\btheta)\transpose\}.
% &\quad 
\end{align*}
Therefore, using Theorem 9 in \cite{MAGNUS1985474}, we have
\begin{align*}
\frac{\partial}{\partial\btheta\transpose}\{\vect(\bE_{ts})\transpose D\bSigma(\btheta)
\bv_{kl}\}
 & = \vect(\bE_{st}\transpose)\transpose \frac{\partial\vect\{\bC(\bvarphi)\bX_{\bvarphi_{kl}}\bC(\bvarphi)\transpose\bSigma(\btheta)\transpose\}}{\partial\btheta\transpose}\\
 & = \vect(\bE_{st}\transpose)\transpose \left[
 \{\bSigma(\btheta)\otimes \eye_{p_2}\}\frac{\partial\vect\{\bC(\bvarphi)\bX_{\bvarphi_{kl}}\bC(\bvarphi)\transpose\}}{\partial\btheta\transpose}
 \right]\\
 &\quad + \vect(\bE_{st}\transpose)\transpose \left[
 \{\eye_{p_1}\otimes \bC(\bvarphi)\bX_{\bvarphi_{kl}}\bC(\bvarphi)\transpose\}\frac{\partial\vect\{\bSigma(\btheta)\transpose\}}{\partial\btheta\transpose}
 \right]\\
 & = \vect(\bE_{st}\transpose)\transpose \left[
 \{\bSigma(\btheta)\otimes \eye_K\}\frac{\partial\vect\{\bC(\bvarphi)\bX_{\bvarphi_{kl}}\bC(\bvarphi)\transpose\}}{\partial\btheta\transpose}
 \right]\\
 &\quad + \vect(\bE_{st}\transpose)\transpose \left[
 \{\eye_K\otimes \bC(\bvarphi)\bX_{\bvarphi_{kl}}\bC(\bvarphi)\transpose\}\bK_{p_1p_2}D\bSigma(\btheta)
 \right],
\end{align*}
or equivalently,
\begin{align*}
\frac{\partial}{\partial\btheta}\{\vect(\bE_{ts})\transpose D\bSigma(\btheta)
\bv_{kl}\}
 & = \frac{\partial\vect\{\bC(\bvarphi)\bX_{\bvarphi_{kl}}\bC(\bvarphi)\transpose\}\transpose{}}{\partial\btheta}
 % \left[
 \{\bSigma(\btheta)\transpose\otimes \eye_K\}
 % \right]
 \vect(\bE_{st}\transpose)\\
 &\quad +  D\bSigma(\btheta)\transpose\bK_{p_2p_1}
 % \left[
 \{\eye_K\otimes \bC(\bvarphi)\bX_{\bvarphi_{kl}}\transpose\bC(\bvarphi)\transpose\}
 % \right]
 \vect(\bE_{st}\transpose).
\end{align*}
Applying Lemma \ref{lemma:matrix_differential_lemma} with 
\[
\bF(\btheta) = \frac{\partial\vect\{\bC(\bvarphi)\bX_{\bvarphi_{kl}}\bC(\bvarphi)\transpose\}\transpose{}}{\partial\btheta},\quad
\bG(\btheta) = \bSigma(\btheta)\transpose,\quad \bH(\btheta) = \eye_{p_2},\quad \bJ = \bE_{st}\transpose,
\]
together with fact (iii) and fact (iv), we see that
\[
\sup_{\btheta:\|\btheta - \bar\btheta_0\|_2 < \eps}\left\|\frac{\partial}{\partial\btheta\transpose}\vect\left[
\frac{\partial\vect\{\bC(\bvarphi)\bX_{\bvarphi_{kl}}\bC(\bvarphi)\transpose\}\transpose{}}{\partial\btheta}
 % \left[
 \{\bSigma(\btheta)\transpose\otimes \eye_K\}
 % \right]
 \vect(\bE_{st}\transpose)
\right]\right\|_{\mathrm{F}} < \infty.
\]
Applying Lemma \ref{lemma:matrix_differential_lemma} with 
\[
\bF(\btheta) = D\bSigma(\btheta)\transpose\bK_{p_2p_1},\quad
\bG(\btheta) = \eye_{p_2},\quad \bH(\btheta) = \bC(\bvarphi)\bX_{\bvarphi_{kl}}\bC(\bvarphi)\transpose,\quad \bJ = \bE_{st}\transpose,
\]
together with fact (ii) and fact (v), we see that
\[
\sup_{\btheta:\|\btheta - \bar\btheta_0\|_2 < \eps}\left\|\frac{\partial}{\partial\btheta\transpose}\vect\left[
D\bSigma(\btheta)\transpose\bK_{p_2p_1}
 % \left[
 \{\eye_K\otimes \bC(\bvarphi)\bX_{\bvarphi_{kl}}\transpose\bC(\bvarphi)\transpose\}
 % \right]
 \vect(\bE_{st}\transpose)
\right]\right\|_{\mathrm{F}} < \infty.
\]
Therefore,
\begin{align*}
&\sup_{\btheta:\|\btheta - \bar\btheta_0\|_2 < \eps}\left\|\frac{\partial^2}{\partial\btheta\partial\btheta\transpose}\{\vect(\bE_{ts})\transpose D\bSigma(\btheta)
\bv_{kl}\}\right\|_{\mathrm{F}}\\
&\quad\leq\sup_{\btheta:\|\btheta - \bar\btheta_0\|_2 < \eps}\left\|\frac{\partial}{\partial\btheta\transpose}\vect
\left[
\frac{\partial}{\partial\btheta}\{\vect(\bE_{ts})\transpose D\bSigma(\btheta)
\bv_{kl}\}
\right]
\right\|_{\mathrm{F}} < \infty.
\end{align*}
This completes the proof of \eqref{eqn:Lipschitz_continuity_DSigma_rectangular_I}. 
Similarly, for any $e\in [p_1]$ and $f\in [r]$,
\begin{align*}
&\frac{\partial}{\partial\btheta\transpose}\vect(\bE_{st})\transpose D\bSigma(\btheta)
\begin{bmatrix*}
\zero_{(p_1 - r)r} \\ \vect(\bF_{ef})
\end{bmatrix*}\\
&\quad = \frac{\partial}{\partial\btheta\transpose} \vect(\bE_{st})\transpose \vect\{\bF_{ef}\bU(\bvarphi)\transpose\}\\
&\quad = \frac{\partial}{\partial\btheta\transpose} \vect(\bE_{st})\transpose (\eye_{p_2}\otimes \bF_{ef})\vect\{\bU(\bvarphi)\transpose\}\\
&\quad = \frac{\partial}{\partial\btheta\transpose} \vect(\bE_{st})\transpose (\eye_{p_2}\otimes \bF_{ef})\bK_{p_2r}\vect\{\bU(\bvarphi)\}\\
&\quad = \vect(\bE_{st})\transpose (\eye_{p_2}\otimes \bF_{ef})\bK_{p_2r}D\bU(\bvarphi)\\
&\quad = 2\vect(\bE_{st})\transpose(\eye_{p_2}\otimes \bF_{ef})\bK_{p_2r}\{\eye_{p_2\times r}\transpose\bC(\bvarphi)\transpose\otimes \bC(\bvarphi)\}\bGamma_\bvarphi.
\end{align*}
Namely,
\[
\frac{\partial}{\partial\btheta}\vect(\bE_{st})\transpose D\bSigma(\btheta)
\begin{bmatrix*}
\zero_{(p_1 - r)r} \\ \vect(\bF_{ef})
\end{bmatrix*}
 = 2\bGamma_\bvarphi\transpose(\eye_{p_2}\otimes \bF_{ef}\transpose)\bK_{rp_2}\{\bC(\bvarphi)\eye_{p_2\times r}\otimes \bC(\bvarphi)\transpose\}
 \vect(\bE_{st}).
\]
Then applying Lemma \ref{lemma:matrix_differential_lemma} with
\[
\bF(\btheta) = 2\bGamma_\bvarphi\transpose,\quad \bG(\btheta) = \bC(\bvarphi)\eye_{p_2\times r},\quad \bH(\btheta) = \bC(\bvarphi),\quad
\bJ = \bK_{rp_2}(\eye_{p_2}\otimes \bF_{ef}\transpose)\vect(\bE_{st}),
\]
together with fact (i), we conclude that
\begin{align*}
&\sup_{\btheta:\|\btheta - \bar\btheta_0\|_2 < \eps}
\left\|
\frac{\partial^2}{\partial\btheta\partial\btheta\transpose}
\vect(\bE_{st})\transpose D\bSigma(\btheta)
\begin{bmatrix*}
\zero_{(p_1 - r)r} \\ \vect(\bF_{ef})
\end{bmatrix*}\right\|_{\mathrm{F}}\\
&
\quad
= 
\sup_{\btheta:\|\btheta - \bar\btheta_0\|_2 < \eps}
\left\|
\frac{\partial}{\partial\btheta\transpose}\vect\left\{\frac{\partial}{\partial\btheta}\vect(\bE_{st})\transpose D\bSigma(\btheta)
\begin{bmatrix*}
\zero_{(p_1 - r)r} \\ \vect(\bF_{ef})
\end{bmatrix*}
\right\}
\right\|_{\mathrm{F}} < \infty.
\end{align*}
Hence, the proof of \eqref{eqn:Lipschitz_continuity_DSigma_rectangular_II} is completed. 
\end{proof}

% subsection technical_lemmas_for_section_sub:biclustering (end)

\subsection{Proof of Theorem \ref{thm:Kmeans_biclustering}} % (fold)
\label{sub:proof_of_theorem_thm:kmeans_biclustering}

\begin{proof}[\bf Proof of Theorem \ref{thm:Kmeans_biclustering}]
% Suppose $\bSigma$ has singular value decomposition $\bSigma = \bW_1\bD\bW_2\transpose$, where $\bW_1\in\mathbb{O}(p_1, r)$, $\bW_2 \in\mathbb{O}(p_2, r)$, and $\bD = \mathrm{diag}\{\sigma_1(\bSigma),\ldots,\sigma_r(\bSigma)\}$ with $\sigma_1(\bSigma)\geq\ldots\geq\sigma_r(\bSigma) > 0$. 
% Denote the singular value decomposition of $\bP\bSigma\bQ\transpose$ as
% $\bP\bSigma\bQ\transpose = \bU\bS\bV\transpose$, where $\bS\in\mathbb{R}^{r\times r}$ is diagonal with diagonal entries being singular values of $\bP\bSigma\bQ\transpose$. Also, denote the singular value decomposition of $\bP\bW_1 = \bU_\bP\bS_\bP\bV_\bP\transpose$, $\bQ\bW_2 = \bU_\bQ\bS_\bQ\bV_\bQ\transpose$, where $\bV_\bP, \bV_\bQ\in\mathbb{O}(r)$, $\bU_\bP\in\mathbb{O}(m, r)$, $\bU_\bQ\in\mathbb{O}(n, r)$, $\bS_\bP = \mathrm{diag}\{\sigma_1(\bP\bW_1),\ldots,\sigma_r(\bP\bW_1)\}$, and $\bS_\bQ = \mathrm{diag}\{\sigma_1(\bQ\bW_2),\ldots,\sigma_r(\bQ\bW_2)\}$. 
The key to the proof of Theorem \ref{thm:Kmeans_biclustering} lies in the two-to-infinity norm control of $\bU$ and $\bV$ via Lemma \ref{thm:two_to_infinity_biclustering}. 
We begin the proof by establishing the following results:
\begin{itemize}
	\item[(i)] The spectra of $\bY_0^*$ have the following bounds:
	\begin{align*}
	&\sigma_r(\bSigma_0)\left\{\min_{s\in[p_1]}m_s(\tau_0)\min_{t\in [p_2]}n_t(\gamma_0)\right\}^{1/2}\leq \sigma_r(\bS),\\
	&\sigma_1(\bS)\leq \sigma_1(\bSigma_0)\left\{\max_{s\in[p_1]}m_s(\tau_0)\max_{t\in [p_2]}n_t(\gamma_0)\right\}^{1/2}.
	\end{align*}
	The result follows by exploiting the proof of Lemma \ref{thm:two_to_infinity_biclustering}.
	\item[(ii)] The spectra of $\bP\bW_1$ and $\bQ\bW_2$ have the following bounds:
	\begin{align*}
	&	\left\{\min_{s\in[p_1]}m_s(\tau_0)\right\}^{1/2}\leq \sigma_r(\bP_0\bW_1)\leq \sigma_1(\bP_0\bW_1)\leq \left\{\max_{s\in[p_1]}m_s(\tau_0)\right\}^{1/2},\\
	& 	\left\{\min_{t\in [p_2]}n_t(\gamma_0)\right\}^{1/2}\leq \sigma_r(\bQ_0\bW_2)\leq \sigma_1(\bQ_0\bW_2)\leq\left\{\max_{t\in [p_2]}n_t(\gamma_0)\right\}^{1/2}.
	\end{align*}
	This result is also a by-product of the proof of Lemma \ref{thm:two_to_infinity_biclustering}.
	\item[(iii)] For any $i_1,i_2\in [m]$ and $j_1,j_2\in [n]$ such that $\tau_0(i_1)\neq \tau_0(i_2)$ and $\gamma_0(j_1)\neq \gamma_0(j_2)$, 
	\begin{align*}
	&\left\|[\bU]_{i_1*} - [\bU]_{i_2*}\right\|_2\geq \frac{\delta\{\sigma_r(\bSigma_0)\min_{t\in [p_2]}n_t(\gamma_0)\}^{1/2}}{\sigma_1(\bSigma_0)\{\max_{s\in [p_1]}m_s(\tau_0)\max_{t\in [p_2]}n_t(\gamma_0)\}^{1/2}},\\
	&\left\|[\bV]_{j_1*} - [\bV]_{j_2*}\right\|_2\geq \frac{\delta\{\sigma_r(\bSigma_0)\min_{s\in [p_1]}m_s(\tau_0)\}^{1/2}}{\sigma_1(\bSigma_0)\{\max_{s\in [p_1]}m_s(\tau_0)\max_{t\in [p_2]}n_t(\gamma_0)\}^{1/2}}.
	\end{align*}
	Let $\bD^{1/2}\bW_1\transpose\bP_0\transpose\bP_0\bW_1\bD^{1/2} = \bZ_\bP\bD_\bP^2\bZ_\bP\transpose$ and $\bD^{1/2}\bW_2\transpose\bQ_0\transpose\bQ_0\bW_2\bD^{1/2} = \bZ_\bQ\bD_\bQ^2\bZ_\bQ\transpose$ be the spectral decompositions of $\bD^{1/2}\bW_1\transpose\bP_0\transpose\bP_0\bW_1\bD^{1/2}$ and $\bD^{1/2}\bW_2\transpose\bQ_0\transpose\bQ_0\bW_2\bD^{1/2}$, respectively, where $\bD_\bP,\bD_\bQ$ are diagonal matrices, and $\bZ_\bP,\bZ_\bQ\in\mathbb{O}(r)$. 
	% Then 
	% \[
	% \bW_1\transpose\bP\transpose\bP\bW_1 = \bV_\bP\bS_\bP^2\bV_\bP\transpose,\quad
	% \bW_2\transpose\bQ\transpose\bQ\bW_2 = \bV_\bQ\bS_\bQ^2\bV_\bQ\transpose.
	% \]
	Further define the following matrices:
	$\bG = \bQ_0\bW_2\bD^{1/2}\bZ_\bP$, $\widetilde{\bG} = \bG\bD_\bP$, $\widetilde{\bV} = \bV\bS$, $\bH = \bP_0\bW_1\bD^{1/2}\bZ_\bQ$, $\widetilde{\bH} = \bH\bD_\bQ$, and $\widetilde{\bU} = \bU\bS$. It follows that
	\begin{align*}
	\widetilde{\bG}\widetilde{\bG}\transpose
	& = \bG\bD_\bP^2\bG\transpose = (\bQ_0\bW_2\bD^{1/2})(\bZ_\bP\bD_\bP^2\bZ_\bP\transpose)(\bD^{1/2}\bW_2\transpose\bQ_0\transpose)\\
	&  = (\bQ_0\bW_2\bD^{1/2}) (\bD^{1/2}\bW_1\transpose\bP_0\transpose\bP_0\bW_1\bD^{1/2}) (\bD^{1/2}\bW_2\transpose\bQ_0\transpose)\\
	& = \bQ_0\bSigma_0\transpose\bP_0\transpose\bP_0\bSigma_0\bQ_0\transpose = \bV\bS^2\bV\transpose = \widetilde{\bV}\widetilde{\bV}\transpose,\\
	\widetilde{\bH}\widetilde{\bH}\transpose
	& = \bH\bD_\bQ^2\bH\transpose = (\bP_0\bW_1\bD^{1/2})(\bZ_\bQ\bD_\bQ^2\bZ_\bQ\transpose)(\bD^{1/2}\bW_1\transpose\bP_0\transpose)\\
	&  = (\bP_0\bW_1\bD^{1/2}) (\bD^{1/2}\bW_2\transpose\bQ_0\transpose\bQ_0\bW_2\bD^{1/2}) (\bD^{1/2}\bW_1\transpose\bP_0\transpose)\\
	& = \bP_0\bSigma_0\bQ_0\transpose\bQ_0\bSigma_0\transpose\bP_0\transpose = \bU\bS^2\bU\transpose = \widetilde{\bU}\widetilde{\bU}\transpose.
	\end{align*}
	Denote $\be_{i_1i_2}(m) = \be_{i_1}(m) - \be_{i_2}(m)\in\mathbb{R}^m$ and $\be_{j_1j_2}(n) = \be_{j_1}(n) - \be_{j_2}(n)\in\mathbb{R}^n$. With $\tau_0(i_1)\neq \tau_0(i_2)$ and $\gamma_0(j_1)\neq \gamma_0(j_2)$, we see that
	\begin{align*}
	\|\be_{i_1}\transpose(m){\bH} - \be_{i_2}\transpose(m){\bH}\|_2
	& = \|\{\be_{i_1}\transpose(m)\bP_0\bW_1\bD^{1/2} - \be_{i_2}\transpose(m)\bP_0\bW_1\bD^{1/2}\}\bZ_\bQ\|_2\\
	& = \|\be_{i_1}\transpose(m)\bP_0\bSigma_{01} - \be_{i_2}\transpose(m)\bP_0\bSigma_{01}\|_2\\
	& = \|\be_{\tau_0(i_1)}\transpose(p_1)\bSigma_{01} - \be_{\tau_0(i_2)}\transpose(p_1)\bSigma_{01}\|_2 \geq \delta,\\
	\|\be_{j_1}\transpose(n){\bG} - \be_{j_2}\transpose(n){\bG}\|_2
	& = \|\{\be_{j_1}\transpose(n)\bQ_0\bW_2\bD^{1/2} - \be_{j_2}\transpose(n)\bQ_0\bW_2\bD^{1/2}\}\bZ_\bP\|_2\\
	& = \|\be_{j_1}\transpose(n)\bQ_0\bSigma_{02} - \be_{j_2}\transpose(n)\bQ_0\bSigma_{02}\|_2\\
	& = \|\be_{\gamma_0(j_1)}\transpose(p_2)\bSigma_{02} - \be_{\gamma_0(j_2)}\transpose(p_2)\bSigma_{02}\|_2 \geq \delta.
	\end{align*}
	Therefore,
	\begin{align*}
	\delta^2 
	& \leq \|\be_{i_1}\transpose(m){\bH} - \be_{i_2}\transpose(m){\bH}\|_2^2 
	\leq \|\be_{i_1}\transpose(m)\widetilde{\bH} - \be_{i_2}\transpose(m)\widetilde{\bH}\|_2^2\|\bD_\bQ^{-1}\|_2^2\\
	&\leq \left\{\sigma_r(\bSigma_0)\min_{t\in [p_2]}n_t(\gamma_0)\right\}^{-1}\be_{i_1i_2}(m)\transpose\widetilde{\bH}\widetilde{\bH}\transpose\be_{i_1i_2}(m)\\
	& =  \left\{\sigma_r(\bSigma_0)\min_{t\in [p_2]}n_t(\gamma_0)\right\}^{-1}\be_{i_1i_2}(m)\transpose\widetilde{\bU}\widetilde{\bU}\transpose\be_{i_1i_2}(m)\\
	& = \left\{\sigma_r(\bSigma_0)\min_{t\in [p_2]}n_t(\gamma_0)\right\}^{-1}\|\be_{i_1}\transpose(m)\widetilde{\bU} - \be_{i_2}\transpose(m)\widetilde{\bU}\|_2^2\\
	& \leq \left\{\sigma_r(\bSigma_0)\min_{t\in [p_2]}n_t(\gamma_0)\right\}^{-1}\|\bS\|_2^2\|[\bU]_{i_1*} - [\bU]_{i_2*}\|_2^2\\
	& \leq \frac{\max_{s\in [p_1]}m_s(\tau_0)\max_{t\in [p_2]}n_t(\gamma_0)}{\min_{t\in [p_2]}n_t(\gamma_0)}\frac{\sigma_1^2(\bSigma_0)}{\sigma_r(\bSigma_0)}\|[\bU]_{i_1*} - [\bU]_{i_2*}\|_2^2,\end{align*}
	and similarly,
	\begin{align*}
	\delta^2 
	& \leq \|\be_{j_1}\transpose(n){\bG} - \be_{j_2}\transpose(n){\bG}\|_2^2 
	\leq \|\be_{j_1}\transpose(n)\widetilde{\bG} - \be_{j_2}\transpose(n)\widetilde{\bG}\|_2^2\|\bD_\bP^{-1}\|_2^2\\
	&\leq\left\{\sigma_r(\bSigma_0)\min_{s\in [p_1]}m_s(\tau_0)\right\}^{-1}\be_{j_1j_2}(n)\transpose\widetilde{\bG}\widetilde{\bG}\transpose\be_{j_1j_2}(n)\\
	& 
	= \left\{\sigma_r(\bSigma_0)\min_{s\in [p_1]}m_s(\tau_0)\right\}^{-1}\be_{j_1j_2}(n)\transpose\widetilde{\bV}\widetilde{\bV}\transpose\be_{j_1j_2}(n)\\
	& = \left\{\sigma_r(\bSigma_0)\min_{s\in [p_1]}m_s(\tau_0)\right\}^{-1}\|\be_{j_1}\transpose(n)\widetilde{\bV} - \be_{j_2}\transpose(n)\widetilde{\bV}\|_2^2\\
	& \leq \left\{\sigma_r(\bSigma_0)\min_{s\in [p_1]}m_s(\tau_0)\right\}^{-1}\|\bS\|_2^2\|[\bV]_{i_1*} - [\bV]_{i_2*}\|_2^2\\
	& \leq \frac{\max_{s\in [p_1]}m_s(\tau_0)\max_{t\in [p_2]}n_t(\gamma_0)}{\min_{s\in [p_1]}m_s(\tau_0)}\frac{\sigma_1^2(\bSigma)}{\sigma_r(\bSigma_0)}\|[\bV]_{i_1*} - [\bV]_{i_2*}\|_2^2.
	\end{align*}
	This completes the proof of result (iii).
\end{itemize}
From the result (iii) above, we see that there exists some constant $\delta' > 0$, such that
\[
\|[\bU]_{i_1*} - [\bU]_{i_2*}\|_2 \geq \frac{\delta'}{\sqrt{m}},\quad
\|[\bV]_{j_1*} - [\bV]_{j_2*}\|_2 \geq \frac{\delta'}{\sqrt{n}}.
\]
Observe that 
\begin{align*}
\bP_0\bSigma_0\bQ_0\transpose
& = \bP_0\bW_1\bD\bW_2\transpose\bQ_0 = \bU_\bP\bS_\bP\bV_\bP\transpose\bD\bV_\bQ\bS_\bQ\bU_\bQ\transpose = \bU\bS\bV\transpose,
\end{align*}
and we know that $\mathrm{Span}(\bU) = \mathrm{Span}(\bU_\bP)$, $\mathrm{Span}(\bV) = \mathrm{Span}(\bU_\bQ)$. Therefore, there exists orthogonal matrices $\bR_\bP,\bR_\bQ$ such that $\bU = \bU_\bP\bR_\bP$ and $\bV = \bU_\bQ\bR_\bQ$, and hence,
\begin{align*}
&\bU\bW_\bU = \bU_\bP\bR_\bP\bW_\bU = \bP_0\bW_1\bV_\bP\bS_\bP^{-1}\bR_\bP\bW_\bU,\\
&\bV\bW_\bV = \bU_\bQ\bR_\bQ\bW_\bV = \bQ_0\bW_2\bV_\bQ\bS_\bQ^{-1}\bR_\bQ\bW_\bV.
\end{align*}
This shows that $\bU\bW_\bU$ has $p_1$ unique rows and $\bV\bW_\bV$ has $p_2$ unique rows because of the cluster assignment structure of $\bP_0$ and $\bQ_0$. Then by Theorem \ref{thm:two_to_infinity_biclustering} and the definition of $\bC(\widehat{\bU})$ and $\bC(\widehat{\bV})$, 
\begin{align*}
\|\bC(\widehat{\bU}) - \widehat{\bU}\|_{\mathrm{F}}
& \leq \|\bU\bW_\bU - \widehat{\bU}\|_{\mathrm{F}}\leq \sqrt{m}\|\widehat{\bU} - \bU\bW_\bU\|_{2\to\infty}
%  = O_{\prob_0}\left(
% \frac{\log m}{\sqrt{n}} + \frac{1}{\sqrt{m}}
% \right)
 = o_{\prob_0}(1),\\
\|\bC(\widehat{\bV}) - \widehat{\bV}\|_{\mathrm{F}}
& \leq \|\bV\bW_\bV - \widehat{\bV}\|_{\mathrm{F}}\leq \sqrt{n}\|\widehat{\bV} - \bV\bW_\bV\|_{2\to\infty}
%  = O_{\prob_0}\left(
% \frac{\log n}{\sqrt{m}} + \frac{1}{\sqrt{n}}
% \right)
 = o_{\prob_0}(1).
\end{align*}
Now we establish the strong consistency of our spectral clustering method. 
% upper bounds for $\|\bC(\widehat{\bU}) - \bU\bW_\bV\|_{2\to\infty}$ and $\|\bC(\widehat{\bV}) - \bV\|_{2\to\infty}$, respectively. 
Let 
\[
\eps \equiv \eps_{mn} := \max\left(\frac{M_{mn}\log m}{\sqrt{n}} + \frac{1}{\sqrt{m}},
\frac{M_{mn}\log n}{\sqrt{m}} + \frac{1}{\sqrt{n}}\right),
\]
where $M_{mn}$ is a slowly growing sequence such that $\eps = o(1)$ and $M_{mn}\to\infty$ as $\min(m, n)\to\infty$. Let $\calA_1,\ldots,\calA_{p_1}$ be the $\ell_2$-balls with radii $\eps\{\min_{s\in[p_1]}m_s(\tau_0)\}^{-1/2}$ centered at the $p_1$ distinct rows of $\bU$, and $\calB_1,\ldots,\calB_{p_2}$ be the $\ell_2$-balls with radii $\eps\{\min_{t\in[p_2]}n_t(\gamma_0)\}^{-1/2}$ centered at the $p_2$ distinct rows of $\bV$. By the result (iii), $(\calA_s)_{s = 1}^{p_1}$ are disjoint, and $(\calB_t)_{t = 1}^{p_2}$ are disjoint, if we take $\min(m, n)$ to be sufficiently large. Define events $\Xi_{\bU} = \{\bY:\|\widehat{\bU} - \bU\bW_\bU\|_{2\to\infty} < \eps/(2\sqrt{m})\}$ and $\Xi_{\bV} = \{\bY:\|\widehat{\bV} - \bV\bW_\bV\|_{2\to\infty} < \eps/(2\sqrt{n})\}$. Clearly, $P(\Xi_\bU\cap \Xi_\bV)\to 1$. 
\begin{enumerate}
	\item[(a)] We claim that each $\calA_s$ contains at least one row of $\bC(\widehat{\bU})$ and each $\calB_t$ contains at least one row of $\bC(\widehat{\bV})$ over the event $\Xi_\bU\cap \Xi_\bV$. We prove it by contradiction and assume otherwise. Then it is either the case that there exists a index $s\in [p_1]$ such that $\calA_s$ does not contain any rows of $\bC(\widehat{\bU})$, or that there exists a index $t\in [p_2]$ such that $\calB_t$ does not contain any rows of $\bC(\widehat{\bU})$. Consequently, 
	\begin{align*}
	\|\bC(\widehat{\bU}) - \bU\bW_\bU\|_{\mathrm{F}}^2 > \min_{s\in [p_1]}m_s(\tau_0)\{\eps^2\min_{s\in [p_1]}m_s(\tau_0)^{-1}\} = \eps^2
	\end{align*} or 
	\begin{align*}
	\|\bC(\widehat{\bV}) - \bV\bW_\bV\|_{\mathrm{F}}^2 > \min_{t\in [p_2]}n_t(\gamma_0)\{\eps^2 \min_{t\in [p_2]}n_t(\gamma_0)^{-1}\} = \eps^2.
	\end{align*} 
	In the former case, 
	\begin{align*}
	\|\bC(\widehat{\bU}) - \widehat{\bU}\|_{\mathrm{F}}
	&\geq \|\bC(\widehat{\bU}) - \bU\bW_\bU\|_{\mathrm{F}} - \|\bU\bW_\bU - \widehat{\bU}\|_{\mathrm{F}}> \eps - \frac{\eps}{2} = \frac{\eps}{2},
	\end{align*}
	and in the latter case,
	\begin{align*}
	\|\bC(\widehat{\bV}) - \widehat{\bV}\|_{\mathrm{F}}
	&\geq \|\bC(\widehat{\bV}) - \bV\bW_\bV\|_{\mathrm{F}} - \|\bV\bW_\bV - \widehat{\bV}\|_{\mathrm{F}}> \eps - \frac{\eps}{2} = \frac{\eps}{2},
	\end{align*}
	where we have used the fact that $\|\widehat{\bU} - \bU\bW_\bU\|_{\mathrm{F}} < \eps/2$ and $\|\widehat{\bV} - \bV\bW_\bV\|_{\mathrm{F}} < \eps/2$ over the event $\Xi_\bU\cap\Xi_\bV$. In either case, contradiction is resulted with the fact that 
	\[
	\|\bC(\widehat{\bU}) - \widehat{\bU}\|_{\mathrm{F}} \leq \|\bU\bW_\bU - \widehat{\bU}\|_{\mathrm{F}} < \eps/2
	\]
	and
	\[
	\|\bC(\widehat{\bV}) - \widehat{\bV}\|_{\mathrm{F}} \leq \|\bV\bW_\bV - \widehat{\bV}\|_{\mathrm{F}} < \eps/2
	\]
	 over the event $\Xi_\bU\cap\Xi_\bV$. The claim is therefore proved. \\

	\item[(b)] We claim that each $\calA_s$ contains exactly one unique row of $\bC(\widehat{\bU})$ and each $\calB_t$ contains exactly one unique row of $\bC(\widehat{\bV})$ over the event $\Xi_\bU\cap\Xi_\bV$. In fact, since each $\calA_s$ contains at least one row of $\bC(\widehat{\bU})$, which by itself has $p_1$ distinct rows, and there are $p_1$ disjoint balls $(\calA_s)_{s = 1}^{p_1}$, then by the pigeonhole principle, each $\calA_s$ contains exactly one unique row of $\bC(\widehat{\bU})$. The same reasoning applied to $\bC(\widehat{\bV})$ and $(\calB_t)_{t = 1}^{p_2}$ implies that each $\calB_t$ contains exactly one unique row of $\bC(\widehat{\bV})$. \\

	\item[(c)] Suppose $\widehat{\tau}:[m]\to [p_1]$ is the estimated cluster assignment function. We claim that for any $s\in [p_1]$, $\widehat{\tau}^{-1}(s):=\{i\in [m]:\widehat{\tau}(i) = s\}$ has at least two element over the event $\Xi_\bU$. The proof proceeds by showing the weak consistency of the $k$-means clustering, namely, the ratio between number of mis-clustered rows and $m$ is $o(1)$. Now consider the row index set
	\[
	\calI = \left\{i\in [m]:\|[\bC(\widehat{\bU})]_{i*} - [\bU\bW_\bU]_{i*}\|_2 \geq \frac{\delta'}{2\sqrt{m}}\right\}.
	\]
	By definition of $\calI$,
	\[
	\|\bC(\widehat{\bU}) - \bU\bW_\bU\|_{\mathrm{F}}\geq |\calI|^{1/2}\frac{\delta'}{2\sqrt{m}}.
	\]
	On the other hand, over the event $\Xi_\bU$,
	\begin{align*}
	\|\bC(\widehat{\bU}) - \bU\bW_\bU\|_{\mathrm{F}}
	&\leq \|\bC(\widehat{\bU}) - \widehat{\bU}\|_{\mathrm{F}} + \|\widehat{\bU} - \bU\bW_\bU\|_{\mathrm{F}}\\
	&\leq \|\bU\bW_\bU - \widehat{\bU}\|_{\mathrm{F}} + \|\widehat{\bU} - \bU\bW_\bU\|_{\mathrm{F}}\\
	& = 2 \|\widehat{\bU} - \bU\bW_\bU\|_{\mathrm{F}}\leq 2\sqrt{m}\|\widehat{\bU} - \bU\bW_\bU\|_{2\to\infty}
	% \\&
	\leq \eps.
	\end{align*}
	We conclude with the previous two inequalities that
	% \[
	$|\calI|\leq 4m(\eps/\delta')^2 = o(m)$ since $\eps = o(1)$.
	% \]
	For any $i_1,i_2\in \calI^c$ with $[\bC(\widehat{\bU})]_{i_1*} = [\bC(\widehat{\bU})]_{i_2*}$, we see that
	\begin{align*}
	\|[\bU]_{i_1*} - [\bU]_{i_2*}\|_2
	& = \|[\bU\bW_\bU]_{i_1*} - [\bU\bW_\bU]_{i_2*}\|_2\\
	& \leq \|[\bC(\widehat{\bU})]_{i_1*} - [\bU\bW_\bU]_{i_1*}\|_2 + \|[\bC(\widehat{\bU})]_{i_2*} - [\bU\bW_\bU]_{i_2*}\|_2
	% \\& \leq \frac{\eps}{2\sqrt{m}}
	 < \frac{\delta'}{\sqrt{m}},
	\end{align*}
	which implies that $[\bU]_{i_1*} - [\bU]_{i_2*}$, \emph{i.e.}, $\tau_0(i_1) = \tau_0(i_2)$, by the result (iii).

	\vspace*{1ex}	
	\noindent
	Note that $m_s \geq |\calI|$ for all $s\in [p_1]$, namely, $\{[\bU\bW_\bU]_{i*}:i\in \calI^c\}$ has exactly $p_1$ distinct rows because $|\calI| = o(m)$ but $m_s \asymp m$. Let $\calU_1,\ldots,\calU_{p_1}$ be $\ell_2$-balls with radii $\delta'/(2\sqrt{m})$ that are centered at the unique rows of $\bU\bW_\bU$. Clearly, $\calU_s$'s are disjoint. By definition of $\calI$, each $\calU_s$ contains at least one element of $\{[\bC(\widehat{\bU})]_{i*}:i\in \calI^c\}$. By the pigeonhole principle and the fact that $\calU_1,\ldots,\calU_{p_1}$ are disjoint, we also conclude that each $\calU_s$ contains exactly one element of $\{[\bC(\widehat{\bU})]_{i*}:i\in \calI^c\}$. 

	\vspace*{1ex}	
	\noindent
	Consequently, for any $i_1,i_2\in \calI^c$ with $\tau_0(i_1) = \tau_0(i_2)$, this implies $[\bU\bW_\bU]_{i_1*} = [\bU\bW_\bU]_{i_2*}$. Also, $[\bC(\widehat{\bU})]_{i_1*}$ and $[\bC(\widehat{\bU})]_{i_2*}$ both lie inside one of the balls among $\calU_1,\ldots,\calU_{p_1}$ that is centered at $[\bU\bW_\bU]_{i_1*}$. By the uniqueness of the row in $\bC(\widehat{\bU})$ that is contained in this ball, we conclude that $[\bC(\widehat{\bU})]_{i_1*} = [\bC(\widehat{\bU})]_{i_2*}$, and hence, $\widehat{\tau}(i_1) = \widehat{\tau}(i_2)$. 

	\vspace*{1ex}
	\noindent
	The above reasoning implies that every element in $\{[\widehat{\bU}]_{i*}:i\in \calI^c\}$ is correctly clustered by the $K$-means method. Since the number of mis-clustered rows is upper bounded by a constant, we conclude that the number of correctly clustered rows is at least $m - |\calI| \asymp m$, and hence, for each cluster $s\in [p_1]$, $\{i\in [m]:\widehat{\tau}(i) = s\}\geq m_s - |\calI|\geq 2$. \\

	\item[(d)] Suppose $\widehat{\gamma}:[n]\to [p_2]$ is the estimated cluster assignment function. The same reasoning as above also implies that for any $t\in [p_1]$, $\widehat{\gamma}^{-1}(t):=\{j\in [n]:\widehat{\gamma}(j) = t\}$ has at least two element over the event $\Xi_\bV$. \\

	\item[(e)] We claim that 
	\begin{align*}
	&\|\bC(\widehat{\bU}) - \bU\bW_\bU\|_{2\to\infty} < \eps\{\min_{s\in[p_1]}m_s(\tau_0)\}^{-1/2},\\
	% \quad\text{and}\quad
	&\|\bC(\widehat{\bV}) - \bV\bW_\bV\|_{2\to\infty} < \eps\{\min_{t\in[p_2]}n_t(\gamma_0)\}^{-1/2}
	\end{align*}
	over the event $\Xi_\bU\cap\Xi_\bV$, and prove it by contradiction. 

	\vspace*{1ex}\noindent
	Suppose the $p_1$ unique rows of $\bU\bW_\bU$ are denoted by $\bu_{1}^*,\ldots,\bu_{p_1}^*$, and the cluster assignment function $\tau$ is arranged such that $[\bU\bW_\bU]_{i*} = \bu_{\tau_0(i)}^*$. Also, for each $s \in [p_1]$, by the results (a) and (b), there exists exactly one unique row of $\bC(\widehat{\bU})$, denoted by $\widehat{\bc}_s^*$, such that $\|\widehat{\bc}_s^* - \bu_s^*\|_2 < \eps\{\min_{s\in[p_1]}m_s(\tau_0)\}^{-1/2}$. Then an estimated cluster assignment function $\widehat{\tau}$ associated with 
	$\bC(\widehat{\bU})$ can be taken such that $[\bC(\widehat{\bU})]_{i*} = \widehat{\bc}_{\widehat{\tau}(i)}^*$ for all $i \in [m]$. 

	\vspace*{1ex}\noindent
	Assume that $\|\bC(\widehat{\bU}) - \bU\bW_\bU\|_{2\to\infty} \geq \eps\{\min_{s\in[p_1]}m_s(\tau_0)\}^{-1/2}$. Then there exists some row index $i\in [m]$ such that 
	\[
	\|[\bC(\widehat{\bU})]_{i*} - [\bU\bW_\bU]_{i*}\|_2 = \|\widehat{\bc}_{\widehat{\tau}(i)}^* - \bu_{\tau(i)}^*\|> \eps\left\{\min_{s\in[p_1]}m_s(\tau_0)\right\}^{-1/2}.
	\]
	Fix the index $i$. This immediately implies that $\widehat{\tau}(i)\neq \tau_0(i)$ according to the above analysis. 
	% Suppose $\tau(i) = s\in [p_1]$ and $\widehat{\tau}(i) = \widehat{s}\in [p_1]$ for some $s,\widehat{s}$. 
	Hence, by letting $\min(m, n)$ be sufficiently large,
	\begin{align*}
	\|\widehat{\bc}_{\widehat{\tau}(i)}^* - [\widehat{\bU}]_{i*}\|_2
	% & \geq \|[\widehat{\bU}]_{i*} -  \bu_{\widehat\tau(i)}^*\|_2 - \|\widehat{\bc}_{\widehat{\tau}(i)}^* - \bu_{\widehat\tau(i)}^*\|_2\\
	&\geq \|\bu_{\tau_0(i)}^* - \bu_{\widehat\tau(i)}^*\|_2 - \|[\widehat{\bU}]_{i*} - \bu_{\tau_0(i)}^*\|_2 - \|\widehat{\bc}_{\widehat{\tau}(i)}^* - \bu_{\widehat\tau(i)}^*\|_2\\
	&\geq \frac{\delta'}{\sqrt{m}} - \|\widehat{\bU} - \bU\bW_\bU\|_{2\to\infty} - \frac{\eps}{\{\min_{s\in[p_1]}m_s(\tau_0)\}^{1/2}}\\
	& > \frac{\delta'}{\sqrt{m}} - \frac{\eps}{2\sqrt{m}} - \frac{\eps}{\{\min_{s\in[p_1]}m_s(\tau_0)\}^{1/2}} > \frac{\delta'}{2\sqrt{m}},\\
	\|\widehat{\bc}_{{\tau_0}(i)}^* - [\widehat{\bU}]_{i*}\|_2
	&\leq \|\widehat{\bc}_{{\tau_0}(i)}^* - \bu_{\tau_0(i)}^*\|_2 + \|\bu_{\tau_0(i)}^* - [\widehat{\bU}]_{i*}\|_2\\
	&\leq \|\widehat{\bc}_{{\tau_0}(i)}^* - \bu_{\tau_0(i)}^*\|_2 + \|\bU\bW_\bU - \widehat{\bU}\|_{2\to\infty}\\
	&\leq \frac{\eps}{\{\min_{s\in[p_1]}m_s(\tau_0)\}^{1/2}} + \frac{\eps}{2\sqrt{m}} < \frac{\delta'}{2\sqrt{m}}.
	\end{align*}
	Namely, 
	% \[
	$\|\widehat{\bc}_{\widehat{\tau}(i)}^* - [\widehat{\bU}]_{i*}\|_2 > \|\widehat{\bc}_{{\tau_0}(i)}^* - [\widehat{\bU}]_{i*}\|_2$.
	% \]
	However, if one instead define another cluster assignment function 
	\[
	\widetilde{\tau}(j) = \left\{
	\begin{aligned}
	&\widehat{\tau}(j),\quad & \text{if }j\neq i,\\
	&\tau_0(i),\quad & \text{if }j = i,
	\end{aligned}
	\right.
	\]
	and define a matrix $\widetilde{\bC}$ with its $j$th row defined as follows:
	\[
	[\widetilde{\bC}]_{j*} = \widehat{\bc}_{\widetilde{\tau}(j)}^* = \left\{
	\begin{aligned}
	&\widehat{\bc}_{\widehat{\tau}(j)}^*,\quad & \text{if }j\neq i\\
	&\widehat{\bc}_{\tau_0(i)},\quad & \text{if }j = i
	\end{aligned}
	\right.,
	\]
	Then we see immediately that
	\begin{align*}
	\|\widetilde{\bC} - \widehat{\bU}\|_{\mathrm{F}}^2
	& = \sum_{j\neq i}\|\widehat{\bc}_{\widehat{\tau}(j)}^* - [\widehat\bU]_{j*}\|_2^2 + \| \widehat{\bc}_{\tau_0(i)}^* - [\widehat{\bU}]_{i*} \|_2^2\\
	&< \sum_{j\neq i}\|\widehat{\bc}_{\widehat{\tau}(j)}^* - [\bU]_{j*}\|_2^2 + \| \widehat{\bc}_{\widehat{\tau}(i)} -  [\widehat{\bU}]_{i*}\|_2^2
	= \|\bC(\widehat{\bU}) - \widehat{\bU}\|_{\mathrm{F}}^2. 
	\end{align*}
	We also know that $\widetilde{\bC}$ has $p_1$ distinct rows because $\{j\in [m]:\widehat{\tau}(j) = \widehat{\tau}(i)\} \geq 2$ according to the result (c). Namely, changing $\widehat{\tau}(i)$ to $\tau_0(i)$ does not reduce the number of unique rows of $\widetilde{\bC}$. Hence, the above result contradicts with the fact that $\bC(\widehat{\bU})$ is the minimizer of the $K$-means criterion function. Namely, over the event $\Xi_\bU$, 
	\[
	\|\bC(\widehat{\bU}) - \bU\bW_\bU\|_{2\to\infty} < \eps\{\min_{s\in [p_1]} m_s(\tau_0)\}^{-1/2}.
	\] 

	\vspace*{1ex}\noindent
	The same reasoning applied to $\widehat{\bV}$ and $\bV\bW_\bV$ yields 
	\[
	\|\bC(\widehat{\bV}) - \bV\bW_\bV\|_{2\to\infty} < \eps\{\min_{t\in [p_2]}n_t(\gamma_0)\}^{-1/2}.
	\] 

	\item[(f)] Using the result (e), over the event $\Xi_\bU\cap \Xi_\bV$, we have
	\begin{align*}
	&\|\bC(\widehat{\bU}) - \bU\bW_\bU\|_{2\to\infty} < \eps\{\min_{s\in [p_1]} m_s(\tau_0)\}^{-1/2},\\
	&\|\bC(\widehat{\bV}) - \bV\bW_\bV\|_{2\to\infty} < \eps\{\min_{t\in [p_2]}n_t(\gamma_0)\}^{-1/2},
	\end{align*}
	it follows immediately that the sets
	\begin{align*}
	\calI = \left\{i\in [m]:\|[\bC(\widehat{\bU})]_{i*} - [\bU\bW_\bU]_{i*}\|_2 \geq \frac{\delta'}{2\sqrt{m}}\right\},\\
	\calJ = \left\{j\in [n]:\|[\bC(\widehat{\bV})]_{i*} - [\bV\bW_\bV]_{j*}\|_2 \geq \frac{\delta'}{2\sqrt{n}}\right\}
	\end{align*}
	are empty by letting $\min(m, n)$ be sufficiently large. We have also proved in the result (c) that the rows with indices $i\in \calI^c$ are correctly clustered, and the exactly same argument also leads to the result that the columns with indices $j\in \calJ^c$ are correctly clustered. Since $\calI = [m]$ and $\calJ = [n]$, we complete the proof of the strong consistency over the set $\Xi_\bU\cap \Xi_\bV$, which has probability going to $1$. 
\end{enumerate}
\end{proof}

% subsection proof_of_theorem_thm:kmeans_biclustering (end)

\subsection{Proof of Theorem \ref{thm:asymptotic_normality_biclustering}} % (fold)
\label{sub:proof_of_theorem_thm:asymptotic_normality_biclustering}
The proof is similar to Appendix \ref{sub:proof_of_theorem_thm:ose_sbm}. We first consider the case where $\tau_0$ and $\gamma_0$ are known. Define an oracle matrix
$\widetilde\bSigma$ whose $(s, t)$ entry is given by
\[
[\widetilde\bSigma]_{st} = \frac{1}{m_s(\tau_0)n_t(\gamma_0)}\sum_{i = 1}^m\sum_{j = 1}^ny_{ij}\mathbbm{1}\{\tau_0(i) = s, \gamma_0(j) = t\}.
\]
For any permutation matrices $\bPi_1\in\mathbb{O}(p_1)$ and $\bPi_2\in\mathbb{O}(p_2)$, define the following oracle least-squares estimator
\[
\widetilde{\btheta}_{\bPi_1\bPi_2} = \argmin_{\btheta\in\mathscr{T}(p_1, p_2, r)}\|\bPi_1\widetilde{\bSigma}\bPi_2\transpose - \bSigma(\btheta)\|_{\mathrm{F}}^2.
\]
We begin the proof with the following lemma addressing the $\sqrt{mn}$-consistency of the oracle least-squares estimator
\begin{lemma}\label{lemma:consistency_biclustering}
Under the notations and setup in Sections \ref{sub:euclidean_representation_of_subspaces}, \ref{sub:extension_to_general_rectangular_matrices}, and \ref{sub:biclustering}, for any two permutation matrices $\bPi_1\in\mathbb{O}(m)$ and $\bPi_2\in\mathbb{O}(n)$,
\begin{align*}
\|\bSigma(\widetilde\btheta_{\bPi_1\bPi_2}) - \bPi_1\bSigma(\btheta_0)\bPi_2\transpose\|_{\mathrm{F}} = O_{\prob_0}\left(\frac{1}{\sqrt{mn}}\right),\quad
\|\widetilde{\btheta}_{\bPi_1\bPi_2} - \btheta_{0\bPi_1\bPi_2}\|_2 = O_{\prob_0}\left(\frac{1}{\sqrt{mn}}\right).
\end{align*}
\end{lemma}

\begin{proof}[Proof of Lemma \ref{lemma:consistency_biclustering}]
For convenience denote $\widetilde{\btheta} = \widetilde{\btheta}_{\bPi_1\bPi_2}$ and $\btheta^* = \btheta_{0\bPi_1\bPi_2}$. 
The proof is based on a ``basic inequality'' and the tools from empirical processes. 
Since $\widetilde{\btheta}$ is the minimizer of the loss function $\btheta\mapsto \|\bPi_1\widetilde{\bSigma}\bPi_2\transpose - \bSigma(\btheta)\|_{\mathrm{F}}^2$, it follows that
\begin{align*}
\|\bPi_1\widetilde{\bSigma}\bPi_2\transpose - \bSigma(\widetilde{\btheta})\|_{\mathrm{F}}^2\leq \|\bPi_1\widetilde{\bSigma}\bPi_2\transpose - \bSigma(\btheta_0)\|_{\mathrm{F}}^2.
\end{align*}
Write $\bPi_1\widetilde{\bSigma}\bPi_2\transpose - \bSigma(\widetilde{\btheta}) = \bPi_1\widetilde{\bSigma}\bPi_2\transpose - \bPi_1\bSigma(\btheta_0)\bPi_2\transpose + \bPi_1\bSigma(\btheta_0)\bPi_2\transpose - \bSigma(\widetilde{\btheta})$ and expand the squared Frobenius norm:
\begin{align*}
&\|\bPi_1\widetilde{\bSigma}\bPi_2\transpose - \bPi_1\bSigma(\btheta_0)\bPi_2\transpose\|_{\mathrm{F}}^2 + 2\langle \bPi_1\widetilde{\bSigma}\bPi_2\transpose - \bPi_1\bSigma_0\bPi_2\transpose, \bPi_1\bSigma_0\bPi_2\transpose - \bSigma(\widetilde{\btheta})\rangle_{\mathrm{F}}\\
&\quad + \|\bPi_1\bSigma_0\bPi_2\transpose - \bSigma(\widetilde{\btheta})\|_{\mathrm{F}}^2\\
&\quad\leq \|\bPi_1\widetilde{\bSigma}\bPi_2\transpose - \bPi_1\bSigma(\btheta_0)\bPi_2\transpose\|_{\mathrm{F}}^2,
\end{align*}
where $\langle\cdot,\cdot\rangle_{\mathrm{F}}$ is the Frobenius inner product induced by the Frobenius norm. Therefore,
\begin{align*}
&\|\bPi_1\bSigma_0\bPi_2\transpose - \bSigma(\widetilde{\btheta})\|_{\mathrm{F}}^2\\
&\quad\leq 2\langle \bPi_1\widetilde{\bSigma}\bPi_2\transpose - \bPi_1\bSigma_0\bPi_2\transpose, \bSigma(\widetilde{\btheta}) - \bPi_1\bSigma_0\bPi_2\transpose\rangle\\
&\quad\leq 2\left|\left\langle \bPi_1\widetilde{\bSigma}\bPi_2\transpose - \bPi_1\bSigma_0\bPi_2\transpose, 
	\frac{\bPi_1\bSigma(\widetilde{\btheta})\bPi_2\transpose - \bPi_1\bSigma_0\bPi_2\transpose}{\|\bPi_1\bSigma(\widetilde{\btheta})\bPi_2\transpose - \bPi_1\bSigma_0\bPi_2\transpose\|_{\mathrm{F}}} \right\rangle_{\mathrm{F}}\right|\\
&\quad\quad\times \|\bPi_1\bSigma(\widetilde{\btheta})\bPi_2\transpose  - \bPi_1\bSigma_0\bPi_2\transpose\|_{\mathrm{F}}\\
&\quad\leq 2\sup_{\bDelta\in\mathbb{R}^{p_1\times p_2}:\|\bDelta\|_{\mathrm{F}} = 1}|\langle\bPi_1\widetilde{\bSigma}\bPi_2\transpose - \bPi_1\bSigma_0\bPi_2\transpose, \bDelta\rangle_{\mathrm{F}}|\|\bPi_1\bSigma(\widetilde{\btheta})\bPi_2\transpose - \bPi_1\bSigma_0\bPi_2\transpose\|_{\mathrm{F}}
\end{align*}
whenever $\|\bPi_1\bSigma(\widetilde{\btheta})\bPi_2\transpose - \bPi_1\bSigma_0\bPi_2\transpose\|_{\mathrm{F}} > 0$. Then
\begin{align*}
\|\bPi_1\bSigma_0\bPi_2\transpose - \bSigma(\widetilde{\btheta})\|_{\mathrm{F}}
&\leq 2\sup_{\bDelta\in\mathbb{R}^{p_1\times p_2}:\|\bDelta\|_{\mathrm{F}} = 1}|\langle\bPi_1(\widetilde{\bSigma} - \bSigma_0)\bPi_2\transpose, \bDelta\rangle_{\mathrm{F}}|\\
&= 2\sup_{\bDelta\in\mathbb{R}^{p_1\times p_2}:\|\bDelta\|_{\mathrm{F}} = 1}|\langle\widetilde{\bSigma} - \bSigma_0, \bPi_1\transpose\bDelta\bPi_2\rangle_{\mathrm{F}}|\\
&\leq  2\sup_{\bDelta\in\mathbb{R}^{p_1\times p_2}:\|\bDelta\|_{\mathrm{F}} \leq 1}|\langle\widetilde{\bSigma} - \bSigma_0, \bDelta\rangle_{\mathrm{F}}|.
\end{align*}
Note that this inequality also holds when $\|\bPi_1\bSigma_0\bPi_2\transpose - \bSigma(\widetilde{\btheta})\|_{\mathrm{F}} = 0$. To bound the supremum of the collection of random variables $(|\langle\widetilde{\bSigma} - \bSigma_0, \bDelta\rangle_{\mathrm{F}}|)_{\bDelta}$, we use a maximum inequality for empirical processes. Define a stochastic process $J(\bDelta) = \langle\widetilde{\bSigma} - \bSigma_0, \bDelta\rangle_{\mathrm{F}}$ indexed by $\bDelta\in\mathbb{R}^{p_1\times p_2},\|\bDelta\|_{\mathrm{F}} \leq 1$. It follows from the sub-Gaussian inequality (see, \emph{e.g.}, Proposition 5.10 in \citealp{vershynin2010introduction}) that for any $\bDelta_1,\bDelta_2\in\mathbb{R}^{p_1\times p_2}$ and any $u > 0$,
\begin{align*}
&\prob_0\left\{
|J(\bDelta_1) - J(\bDelta_2)| > u
\right\}\\
&\quad = \prob_0\left\{
|\langle\widetilde{\bSigma} - \bSigma_0, \bDelta_1 - \bDelta_2\rangle_{\mathrm{F}}| > u
\right\}\\
&\quad = \prob_0\left\{
\left|
\sum_{s = 1}^{p_1}\sum_{t = 1}^{p_2}\sum_{i = 1}^m\sum_{j = 1}^n\frac{\{y_{ij} - \expect_0(y_{ij})\}\mathbbm{1}\{\tau_0(i) = s,\gamma_0(j) = t\}[\bDelta_1 - \bDelta_2]_{st}}{m_s(\tau_0)n_t(\gamma_0)}
\right| > u
\right\}\\
&\quad\leq e\cdot\exp\left[-cu^2\left\{\sum_{s = 1}^{p_1}\sum_{t = 1}^{p_2}\sum_{i = 1}^m\sum_{j = 1}^n\left(\frac{\mathbbm{1}\{\tau_0(i) = s,\gamma_0(j) = t\}[\bDelta_1 - \bDelta_2]_{st}}{m_s(\tau_0)n_t(\gamma_0)}\right)^2\right\}^{-1}\right]\\
&\quad = e\cdot\exp\left[-cu^2\left\{\sum_{s = 1}^{p_1}\sum_{t = 1}^{p_2}\frac{([\bDelta_1 - \bDelta_2]_{st})^2}{m_s(\tau_0)n_t(\gamma_0)}\right\}^{-1}\right]\\
&\quad = e\cdot\exp\left[-cu^2\left\{\sum_{s = 1}^{p_1}\sum_{t = 1}^{p_2}\frac{([\bDelta_1 - \bDelta_2]_{st})^2}{m_s(\tau_0)n_t(\gamma_0)}\right\}^{-1}\right]\\
&\quad\leq e\cdot\exp\left[-cu^2\left\{\sum_{s = 1}^{p_1}\sum_{t = 1}^{p_2}\frac{([\bDelta_1 - \bDelta_2]_{st})^2}{\min_{s\in[p_1]}m_s\min_{t\in [p_2]}n_t}\right\}^{-1}\right]\\
&\quad\leq e\cdot\exp\left(-\frac{c'u^2mn}{\|\bDelta_1 - \bDelta_2\|_{\mathrm{F}}^2}\right).
\end{align*}
This shows that the stochastic process $\{J(\bDelta)\}_{\bDelta}$ is a sub-Gaussian process with respect to the metric $d(\bDelta_1, \bDelta_2) = C\|\bDelta_1 - \bDelta_2\|_{\mathrm{F}}/\sqrt{mn}$ for some constant $C > 0$. For any metric space $(\calT, d)$, let $\calN(\eps, \calT, d)$ be the $\eps$-covering number of $\calT$, i.e., the minimum number of balls of the form $B_d(x, \eps):=\{y\in \calT:d(x, y) < \eps\}$ that are needed to cover $\calT$.  It follows from the covering number for Euclidean balls \citep{10.2307/4153175} that
\begin{align*}
\calN\left(\frac{\eps}{2}, \{\bDelta\in\mathbb{R}^{p_1\times p_2}:\|\bDelta\|_{\mathrm{F}} \leq 1\}, d\right)
% \\
&
% \quad
\leq \calN\left(\frac{\eps\sqrt{mn}}{2C}, \{\bx\in\mathbb{R}^{p_1p_2}:\|\bx\|_2\leq 1\}, \|\cdot\|_2\right)\\
&
% \quad
\leq \left(\frac{6C}{\eps\sqrt{mn}}\right)^{p_1p_2}.
\end{align*}
Also, the diameter of $\calT := \{\bDelta\in\mathbb{R}^{p_1p_2}:\|\bDelta\|_{\mathrm{F}} \leq 1\}$ can be upper bounded by
\[
\mathrm{diam}(\calT) = \sup_{\bDelta_1,\bDelta_2\in\calT}\frac{C}{\sqrt{mn}}\|\bDelta_1 - \bDelta_2\|_{\mathrm{F}}\leq \frac{2C}{\sqrt{mn}}.
\]
Hence, the Dudley's integral can be further computed:
\begin{align*}
&\int_0^{\mathrm{diam}(\calT)}\sqrt{\log\calN\left(\frac{\eps}{2}, \{\bDelta\in\mathbb{R}^{p_1\times p_2}:\|\bDelta\|_{\mathrm{F}} \leq 1\}, d\right)}\mathrm{d}\eps\\
&\quad\leq \int_0^{2C/\sqrt{mn}}\sqrt{p_1p_2\log\left(\frac{6C}{\eps\sqrt{mn}}\right)}\mathrm{d}\eps\\
&\quad\leq 2C\sqrt{\frac{p_1p_2\log 3}{mn}} + \sqrt{p_1p_2}\int_0^{2C/\sqrt{mn}}\sqrt{\log\left(\frac{2C}{\eps\sqrt{mn}}\right)}\mathrm{d}\eps\\
&\quad = 2C\sqrt{\frac{p_1p_2\log 3}{mn}} + 2C\sqrt{\frac{p_1p_2}{mn}}\int_0^1\sqrt{\log(1/t)}\mathrm{d}t\asymp \frac{1}{\sqrt{mn}}.
\end{align*}
It follows from the maximal inequality (see, \emph{e.g.}, Corollary 8.5 in \citealp{kosorok2007introduction}) that
\begin{align*}
\|\bSigma(\widetilde{\btheta}) - \bPi_1\bSigma_0\bPi_2\transpose\|_{\mathrm{F}}\leq 2\sup_{\|\bDelta\|_{\mathrm{F}}\leq 1}\left|\langle \widetilde{\bSigma} - \bSigma_0, \bDelta\rangle_{\mathrm{F}}\right| = O_{\prob_0}\left(\frac{1}{\sqrt{mn}}\right).
\end{align*}
Denote $\widetilde\btheta = [\widetilde\bvarphi\transpose, \widetilde\bmu\transpose]\transpose$, $\btheta^* = [(\bvarphi^*)\transpose, (\bmu^*)\transpose]\transpose$, $\bSigma^* = \bSigma(\btheta^*) = \bPi_1\bSigma_0\bPi_2\transpose = \bM^*(\bU^*)\transpose = \bM(\bmu^*)\bU(\bvarphi^*)\transpose$. 
For the second assertion, we first observe that by a variant of the Wedin's $\sin\Theta$ theorem (see Theorem 3 in \citealp{10.1093/biomet/asv008}), 
\begin{align*}
% \|\bU(\widetilde\bvarphi)\bU(\widetilde\bvarphi)\transpose - \bU_*\bU_*\transpose\|_{\mathrm{F}}
% &\asymp
\|\sin\Theta\{\bU(\widetilde\bvarphi), \bU^*\}\|_{\mathrm{F}}
% \\
&\leq \frac{2(2\|\bSigma^*\|_2 + \|\bSigma(\widetilde\btheta) - \bSigma^*\|_{\mathrm{F}}) \|\bSigma(\widetilde\btheta) - \bSigma^*\|_{\mathrm{F}}}
{\sigma_r(\bSigma^*)}
% \\&
 = O_{\prob_0}\left(\frac{1}{\sqrt{mn}}\right).
\end{align*}
Since any $r$ columns of $\bU_0$ are linearly independent, then by Corollary \ref{corr:intrinsic_deviation_Projection}, 
\[
\|\widetilde\bvarphi - \bvarphi^*\|_2\lesssim \|\sin\Theta\{\bU(\widetilde\bvarphi),\bU^*\}\|_{\mathrm{F}} = O_{\prob_0}\left(\frac{1}{\sqrt{mn}}\right).
\]
In addition, by Theorem \ref{thm:second_order_deviation_CT}, we see that
\[
\|\bU(\widetilde\bvarphi) - \bU^*\|_{\mathrm{F}}\leq \|D\bSigma(\btheta^*)\|_2\|\widetilde\bvarphi - \bvarphi^*\|_2 + C\|\widetilde\bvarphi - \bvarphi^*\|_2^2 = O_{\prob_0}\left(\frac{1}{\sqrt{mn}}\right).
\]
Therefore,
\begin{align*}
\|\widetilde\bmu - \bmu^*\|_2 & = \|\bM(\widetilde\bmu) - \bM^*\|_{\mathrm{F}}
= \|\{\bSigma(\widetilde\btheta) - \bSigma(\btheta^*)\}\bU(\widetilde\bvarphi) - \bSigma(\btheta^*)\{\bU(\widetilde\bvarphi) - \bU^*\}\|_{\mathrm{F}}\\
&\leq \|\bSigma(\widetilde\btheta) - \bSigma(\btheta^*)\|_{\mathrm{F}} + \|\bSigma(\btheta^*)\|_2\|\bU(\widetilde\bvarphi) - \bU^*\|_{\mathrm{F}}
% \\&
 = O_{\prob_0}\left(\frac{1}{\sqrt{mn}}\right),
\end{align*}
and hence,
\[
\|\widetilde\btheta - \btheta^*\|_2\leq \|\widetilde\bvarphi - \bvarphi^*\|_2 + \|\widetilde\bmu - \bmu^*\|_2 = O_{\prob_0}\left(\frac{1}{\sqrt{mn}}\right).
\]
The proof is thus completed. 
\end{proof}

\begin{lemma}\label{thm:asymptotic_normality_biclustering_oracle}
Under the notations and setup in Sections \ref{sub:euclidean_representation_of_subspaces}, \ref{sub:extension_to_general_rectangular_matrices}, and \ref{sub:biclustering}, for any two permutation matrices $\bPi_1\in\mathbb{O}(m)$ and $\bPi_2\in\mathbb{O}(n)$, for any two permutation matrices $\bPi_1\in\mathbb{O}(p_1)$ and $\bPi_2\in\mathbb{O}(p_2)$,
% \[
% \widetilde{\btheta} - \btheta = \{D\bSigma(\btheta)\transpose D\bSigma(\btheta)\}^{-1}D\bSigma(\btheta)\transpose\vect\{\widetilde{\bSigma} - \bSigma(\btheta)\} + o_{\prob_0}\left(\frac{1}{\sqrt{mn}}\right).
% \]
% In addition,
\begin{align*}
\sqrt{mn}\bG(\btheta_{\bPi_1\bPi_2})(\widetilde{\btheta}_{\bPi_1\bPi_2} - \btheta_{0\bPi_1\bPi_2}) \overset{\calL}{\to}\mathrm{N}(\zero_d, \eye_d).
\end{align*}
% where
% \begin{align*}
% \bG(\btheta^*)
% & = 
% \{D\bSigma(\btheta^*)\transpose(\bPi_2\otimes\bPi_1)
% \bLambda
% (\bPi_1\transpose\otimes\bPi_2\transpose)
% D\bSigma(\btheta^*)\}^{-1/2}
% \{D\bSigma(\btheta^*)\transpose D\bSigma(\btheta^*)\}
% ,\\
% \bLambda
% & = \mathrm{diag}\{\sigma^2\vect(\bw\bpi\transpose)\}
% \end{align*}
% and $d = p_1r + (p_2 - r)r$ is the dimension of the parameter space $\mathscr{T}$. 
\end{lemma}

\begin{proof}[Proof of Lemma \ref{thm:asymptotic_normality_biclustering_oracle}]
For convenience still denote $\widetilde{\btheta} = \widetilde{\btheta}_{\bPi_1\bPi_2}$ and $\btheta^* = \btheta_{0\bPi_1\bPi_2}$. 
By construction, $\mathscr{T}(p_1, p_2, r)$ is open and $\btheta^*$ is in the interior of $\mathscr{T}(p_1, p_2, r)$. Also, $\|\widetilde\btheta - \btheta^*\|_2 = O_{\prob_0}\{(mn)^{-1/2}\}$ by Lemma \ref{lemma:consistency_biclustering}, namely, $\widetilde\btheta$ is also in the interior of $\mathscr{T}(p_1, p_2, r)$ with probability going to one. Assume such an event occurs. Consider the function
\[
\bpsi(\btheta) := D\bSigma(\btheta)\transpose\vect\{\bPi_1\widetilde\bSigma\bPi_2\transpose - \bSigma(\btheta)\}.
\]
Clearly, $\bpsi(\btheta)$ is the gradient of the function $\btheta\mapsto \|\bPi_1\widetilde\bSigma\bPi_2\transpose - \bSigma(\btheta)\|_{\mathrm{F}}^2$, i.e.,
\[
\bpsi(\btheta) = \frac{\partial}{\partial\btheta}\|\bPi_1\widetilde\bSigma\bPi_2\transpose - \bSigma(\btheta)\|_{\mathrm{F}}^2.
\]
Since $\widetilde\btheta$ is the minimizer of the function $\btheta\mapsto \|\bPi_1\widetilde{\bSigma}\bPi_2\transpose - \bSigma(\btheta)\|_{\mathrm{F}}^2$ and is in the interior of $\mathscr{T}(p_1, p_2, r)$, it follows that $\bpsi(\widetilde\btheta) = \zero$. 
% Denote $d = p_1r + (p_2 - r)r$ the dimension of the parameter space $\mathscr{T}$. 
By the matrix differential calculus (see, \emph{e.g.}, Theorem 9 in \citealp{MAGNUS1985474}), 
\begin{align*}
\frac{\partial\bpsi(\btheta)}{\partial\btheta\transpose}
& = D\bSigma(\btheta)\transpose D\bSigma(\btheta) + \frac{\partial\vect\{D\bSigma(\btheta)\}\transpose}{\partial\btheta}[\eye_{d}\otimes \vect\{\bSigma(\btheta) - \bPi_1\widetilde\bSigma\bPi_2\transpose\}].
\end{align*}
By Lemma \ref{lemma:DDB_theta_bound_biclustering}, we see that
\[
D\bSigma(\btheta) = \frac{\partial\vect\{\bSigma(\btheta)\}}{\partial\btheta\transpose},\quad\frac{\partial\vect\{D\bSigma(\btheta)\}}{\partial\btheta\transpose}
\]
are both Lipschitz continuous for $\btheta\in \{\btheta:\|\btheta - \btheta^*\|_2 < \eps\}$ for some $\eps > 0$. Namely, the vector valued function
\[
\frac{\partial\bpsi(\btheta)}{\partial\btheta\transpose}
\]
is also Lipschitz continuous for $\btheta\in \{\btheta:\|\btheta - \btheta^*\|_2 < \eps\}$ for some $\eps > 0$, and hence, by Taylor's theorem, 
\[
\bpsi(\widetilde\btheta) = \zero = \bpsi(\btheta^*) + \frac{\partial\bpsi(\btheta^*)}{\partial\btheta\transpose}(\widetilde\btheta - \btheta^*) + \mathbf{r}(\widetilde\btheta, \btheta^*),
\]
where $\|\mathbf{r}(\widetilde\btheta, \btheta^*)\|_2\lesssim \|\widetilde\btheta - \btheta^*\|_2^2 = O_{\prob_0}\{(mn)^{-1}\}$. 
Note that
\begin{align*}
&\left\|\frac{\partial\vect\{D\bSigma(\btheta^*)\}\transpose}{\partial\btheta}[\eye_{d}\otimes \vect\{\bSigma(\btheta^*) - \bPi_1\widetilde\bSigma\bPi_2\transpose\}]\right\|_{\mathrm{F}}\\
&\quad\leq \left\|\frac{\partial\vect\{D\bSigma(\btheta^*)\}\transpose}{\partial\btheta}\right\|_{\mathrm{F}}\left\|[\eye_{d}\otimes \vect\{\bSigma(\btheta^*) - \bPi_1\widetilde\bSigma\bPi_2\transpose\}]\right\|_{2}\\
&\quad\leq \left\|\frac{\partial\vect\{D\bSigma(\btheta^*)\}\transpose}{\partial\btheta}\right\|_{\mathrm{F}}\|\bSigma(\btheta^*) - \bPi_1\widetilde\bSigma\bPi_2\transpose\|_{\mathrm{F}}.
\end{align*}
In addition, 
\begin{align*}
&\expect_0\|\bSigma(\btheta^*) - \widetilde\bSigma\|_{\mathrm{F}}^2\\
&\quad = \expect_0\|\bPi_1(\bSigma_0 - \widetilde\bSigma)\bPi_2\transpose\|_{\mathrm{F}}^2
% \\&
 = \expect_0\|\bSigma_0 - \widetilde\bSigma\|_{\mathrm{F}}^2\\
&\quad = \sum_{s = 1}^{p_1}\sum_{t = 1}^{p_2}\expect_0\left[\frac{1}{m_s(\tau_0)n_t(\gamma_0)}\sum_{i = 1}^m\sum_{j = 1}^n\mathbbm{1}\{\tau_0(i) = s,\gamma_0(j) = t\}\{y_{ij} - \expect_0(y_{ij})\}\right]^2\\
&\quad = \sum_{s = 1}^{p_1}\sum_{t = 1}^{p_2}\frac{1}{m_s(\tau_0)n_t(\gamma_0)}\var(y_{ij}) \lesssim \frac{1}{mn}.
\end{align*}
Therefore,
\[
\frac{\partial\bpsi(\btheta^*)}{\partial\btheta\transpose}
 = D\bSigma(\btheta^*)\transpose D\bSigma(\btheta^*) + O_{\prob_0}\left(\frac{1}{\sqrt{mn}}\right)
\]
Using the fact that $\|\widetilde\btheta - \btheta^*\|_2 = O_{\prob_0}\{(mn)^{-1/2}\}$ from Lemma \ref{lemma:consistency_biclustering}, we further write
\begin{align*}
-\psi(\btheta^*)
& = -D\bSigma(\btheta^*)\transpose\vect\{\bPi_1(\widetilde\bSigma - \bSigma_0)\bPi_2\transpose\}\\
& = \frac{\partial\bpsi(\btheta^*)}{\partial\btheta\transpose}(\widetilde\btheta - \btheta^*) + O_{\prob_0}\left(\frac{1}{mn}\right)\\
& = D\bSigma(\btheta^*)\transpose D\bSigma(\btheta^*)(\widetilde\btheta - \btheta^*) + O_{\prob_0}\left(\frac{1}{mn}\right),
\end{align*}
which implies
\[
\widetilde{\btheta} - \btheta^* = -\{D\bSigma(\btheta^*)\transpose D\bSigma(\btheta^*)\}^{-1}D\bSigma(\btheta^*)(\bPi_2\otimes\bPi_1)\vect(\widetilde{\bSigma} - \bSigma_0) + O_{\prob_0}\left(\frac{1}{mn}\right).
\]
Observe that the entries of $\sqrt{mn}(\widetilde{\bSigma} - \bSigma_0)$ are independent mean-zero random variables, and for all $s\in [p_1]$, $t\in [p_2]$,
\begin{align*}
[\widetilde{\bSigma} - \bSigma_0]_{st} & = \sum_{i = 1}^m\sum_{j = 1}^n\frac{\mathbbm{1}\{\tau_0(i) = s,\gamma_0(j) = t\}\{y_{ij} - \expect_0(y_{ij})\}}{m_s(\tau_0)n_t(\gamma_0)},\\
\var_0(\sqrt{mn}[\widetilde{\bSigma} - \bSigma_0]_{st}) & = \sum_{i = 1}^m\sum_{j = 1}^n\frac{mn\mathbbm{1}\{\tau_0(i) = s,\gamma_0(j) = t\}\var_0(y_{ij})}{m_s(\tau_0)^2n_t(\gamma_0)^2}\\
& = \frac{mn}{m_s(\tau_0)n_t(\gamma_0)}\var_0(y_{ij})\to \sigma^2w_s\pi_t,
\end{align*}
and
\begin{align*}
&\sum_{i = 1}^m\sum_{j = 1}^n\expect_0\left|\frac{\sqrt{mn}}{m_s(\tau_0)n_t(\gamma_0)}\mathbbm{1}\{\tau_0(i) = s,\gamma_0(j) = t\}\{y_{ij} - \expect_0(y_{ij})\}\right|^3\\
&\quad = \frac{(mn)^{3/2}}{m_s(\tau_0)^3n_t(\gamma_0)^3}\sum_{i = 1}^m\sum_{j = 1}^n\mathbbm{1}\{\tau_0(i) = s,\gamma_0(j) = t\}\expect_0\{|y_{ij} - \expect_0(y_{ij})|^3\}\\
&\quad = \frac{(mn)^{3/2}}{m_s(\tau_0)^2n_t(\gamma_0)^2}\expect_0\{|y_{ij} - \expect_0(y_{ij})|^3\} \to 0.
\end{align*}
It follows from the Lyapunov's central limit theorem that
\begin{align*}
\sqrt{mn}[\widetilde{\bSigma} - \bSigma_0]_{st}\overset{\calL}{\to}\mathrm{N}(0, w_s\pi_t\sigma_0^2),
\end{align*}
and hence,
\begin{align*}
\sqrt{mn}\vect(\widetilde\bSigma - \bSigma_0)\overset{\calL}{\to}\mathrm{N}(\zero, \sigma_0^2\mathrm{diag}\{\vect(\bw\bpi)\}).
\end{align*}
The proof is completed by using the fact that
\begin{align*}
&\sqrt{mn}(\widetilde\btheta - \btheta^*)\\
&\quad = -\{D\bSigma(\btheta^*)\transpose D\bSigma(\btheta^*)\}^{-1}D\bSigma(\btheta^*)\transpose(\bPi_2\otimes\bPi_1)\{\sqrt{mn}\vect(\widetilde\bSigma - \bSigma^*)\}
% \\&\quad\quad
 + o_{\prob_0}(1)\\
&\quad\overset{\calL}{\to}\mathrm{N}(\zero_d, \bG(\bPi_1,\bPi_2)),
\end{align*}
where by definition,
\begin{align*}
\bG(\bPi_1,\bPi_2) & = \{D\bSigma(\btheta^*)\transpose D\bSigma(\btheta^*)\}^{-1}D\bSigma(\btheta^*)\transpose\\
&\quad\times (\bPi_2\otimes\bPi_1)\sigma^2\mathrm{diag}\{\vect(\bw\bpi\transpose)\}(\bPi_2\transpose\otimes\bPi_1\transpose)\\
&\quad\times D\bSigma(\btheta^*)\{D\bSigma(\btheta^*)\transpose D\bSigma(\btheta^*)\}^{-1}.
\end{align*}
The proof is thus completed. 
\end{proof}

\begin{proof}[Proof of Theorem \ref{thm:asymptotic_normality_biclustering}]
By the strong consistency of $\widehat{\tau}$ and $\widehat{\gamma}$, there exists two sequences of permutations $(\omega_m)_{n = 1}^\infty$, $(\iota_n)_{n = 1}^\infty$, such that
\[
\prob_0\left[\sum_{i = 1}^m\mathbbm{1}\{\widehat{\tau}(i) \neq \omega_m\circ\tau_0(i)\} = 0, 
\sum_{j = 1}^n\mathbbm{1}\{\widehat{\gamma}(i) \neq \iota_n\circ\gamma_0(i)\} = 0
\right]\to 1.
\]
Denote the event
\[
\Xi_{mn} = \left\{
\sum_{i = 1}^m\mathbbm{1}\{\widehat{\tau}(i) \neq \omega_m\circ\tau_0(i)\} = 0, 
\sum_{j = 1}^n\mathbbm{1}\{\widehat{\gamma}(i) \neq \iota_n\circ\gamma_0(i)\} = 0
\right\}.
\]
Also, for any permutations $\omega:[p_1]\to[p_1],\iota:[p_2]\to[p_2]$, and the associated permutation matrices $\bPi_{1\omega},\bPi_{2\iota}$ such that
\[
\bPi_{1\omega}\begin{bmatrix*}
1\\\vdots\\p_1
\end{bmatrix*} = \begin{bmatrix*}
\omega^{-1}(1)\\\vdots\\\omega^{-1}(p_2)
\end{bmatrix*},\quad
\bPi_{2\iota}\begin{bmatrix*}
1\\\vdots\\p_2
\end{bmatrix*} = \begin{bmatrix*}
\iota^{-1}(1)\\\vdots\\\iota^{-1}(p_2)
\end{bmatrix*},
\]
denote 
\[
Z_{mn}(\omega,\iota) := \sqrt{mn}\bG(\bPi_{1\omega},\bPi_{2\iota})^{-1/2}(\widetilde{\btheta}_{\bPi_{1\omega}\bPi_{2\iota}} - \btheta_{0\bPi_{1\omega}\bPi_{2\iota}}).
\]
By Lemma \ref{thm:asymptotic_normality_biclustering_oracle}, $Z_{mn}(\omega,\iota)\overset{\calL}{\to}\mathrm{N}(\zero_d, \eye_d)$ for all permutations $\omega,\iota$. Therefore, for any measurable set $A\subset\mathbb{R}^d$, 
\begin{align*}
\min_{\omega:[p_1]\to [p_1],\iota:[p_2]\to [p_2]}\prob_0\{Z_{mn}(\omega,\iota)\in A\}\to \int_A\phi(\bx\mid\zero_d,\eye_d)\mathrm{d}\bx,\\
\max_{\omega:[p_1]\to [p_1],\iota:[p_2]\to [p_2]}\prob_0\{Z_{mn}(\omega,\iota)\in A\}\to \int_A\phi(\bx\mid\zero_d,\eye_d)\mathrm{d}\bx,
\end{align*}
where the minimum and maximum are taken with regard to all permutations $\omega:[K]\to [K]$. 
Now let $A\subset\mathbb{R}^{d}$ be measurable. First note that
\begin{align*}
0&\leq \prob_0\left\{\sqrt{mn}\bG(\bPi_{1\omega_m},\bPi_{2\iota_n})^{1/2}(\widehat{\btheta}_{mn} - \btheta_{0\bPi_{1\omega_m}\bPi_{2\iota_n}})\in A, \bY\in\Xi_{mn}^c\right\}\\
&\leq \prob_0\left\{Z_{mn}(\omega_m,\iota_n)\in A, \bY\in \Xi_{mn}^c\right\}
% \\&
\leq \prob_0(\Xi_{mn}^c) \to 0.
\end{align*}
Let $\bPi_{1m} := \bPi_{1\omega_m}$ and $\bPi_{2n} := \bPi_{2\iota_n}$. Note that over the event $\Xi_{mn}$, $\widehat{\btheta}_{mn} = \widetilde{\btheta}_{\bPi_{1\omega_m}\bPi_{2\iota_n}} = \widetilde{\btheta}_{\bPi_{1m}\bPi_{2n}}$. Using the asymptotic normality of $Z_{mn}(\omega, \iota)$ for all $\omega,\iota$, we have
\begin{align*}
&\prob_0\left[\sqrt{mn}\bG(\bPi_{1m},\bPi_{2n})^{-1/2}(\widehat{\btheta}_{mn} - \btheta_{0\bPi_{1m}\bPi_{2n}})\in A, \bY\in \Xi_{mn}\right]\\
&\quad = \prob_0\left[
\sqrt{mn}\bG(\bPi_{1m},\bPi_{2n})^{-1/2}(\widetilde{\btheta}_{\bPi_{1\omega_m}\bPi_{2\iota_n}} - \btheta_{0\bPi_{1m}\bPi_{2n}})\in  A, \bY\in \Xi_{mn}
\right]\\
&\quad \leq \prob_0\{Z_{mn}(\omega_m,\iota_n)\in A\}\leq \max_{\omega,\iota}\prob_0(Z_{mn}(\omega,\iota)\in A) \to \int_A\phi(\bx\mid\zero_d,\eye_d)\mathrm{d}\bx,\\
&\prob_0\left[
\sqrt{mn}\bG(\bPi_{1m},\bPi_{2n})^{-1/2}(\widehat{\btheta}_{mn} - \btheta_{0\bPi_{1m}\bPi_{2n}})\in A, \bY\in \Xi_{mn}
\right]\\
&\quad = \prob_0\left[
\sqrt{mn}\bG(\bPi_{1m},\bPi_{2n})^{-1/2}(\widetilde{\btheta}_{\bPi_{1\omega_m}\bPi_{2\iota_n}} - \btheta_{0\bPi_{1m}\bPi_{2n}})\in  A, \bY\in \Xi_{mn}
 \right]\\
&\quad = \prob_0\left\{Z_{mn}(\omega_m,\iota_n)\in A, \bY\in \Xi_{mn}\right\}\\
&\quad = \prob_0\left\{Z_{mn}(\omega_m,\iota_n)\in A\right\} + \prob_0(\bY\in\Xi_{mn})
% \\&\quad\quad
 - \prob_0\left[\left\{Z_{mn}(\omega_m,\iota_n)\in A\right\}\cup
\left\{\bY\in \Xi_{mn}
\right\}\right]\\
&\quad \geq \prob_0\left\{Z_{mn}(\omega_m,\iota_n)\in A\right\} + \prob_0(\bY\in\Xi_{mn}) - 1\\
&\quad \geq \min_{\omega,\iota}\prob_0\{Z_{mn}(\omega,\iota)\in A\} + 1 - o(1) - 1
% \\&\quad
 \to \int_A\phi(\bx\mid\zero_d,\eye_d)\mathrm{d}\bx.
\end{align*}
Namely, 
\[
\prob_0\left[
\sqrt{mn}\bG(\bPi_{1m},\bPi_{2n})^{-1/2}(\widehat{\btheta}_{mn} - \btheta_{0\bPi_{1m}\bPi_{2n}})\in A, \bY\in \Xi_{mn}
\right]
 \to\int_A\phi(\bx\mid\zero_d,\eye_d)\mathrm{d}\bx.
\]
Hence, for any measurable set $A\subset\mathbb{R}^{d}$, 
\begin{align*}
&\prob_0\left[
\sqrt{mn}\bG(\bPi_{1m},\bPi_{2n})^{-1/2}(\widehat{\btheta}_{mn} - \btheta_{0\bPi_{1m}\bPi_{2n}})\in A
\right]\\
&\quad = \prob_0\left[
\sqrt{mn}\bG(\bPi_{1m},\bPi_{2n})^{-1/2}(\widehat{\btheta}_{mn} - \btheta_{0\bPi_{1m}\bPi_{2n}})\in A, 
\bY\in\Xi_{mn}
\right]\\
&\quad\quad + \prob_0\left[
\sqrt{mn}\bG(\bPi_{1m},\bPi_{2n})^{-1/2}(\widehat{\btheta}_{mn} - \btheta_{0\bPi_{1m}\bPi_{2n}})\in A, \bY\in\Xi_{mn}^c
\right]\\
&\quad = \prob_0\left[
\sqrt{mn}\bG(\bPi_{1m},\bPi_{2n})^{-1/2}(\widehat{\btheta}_{mn} - \btheta_{0\bPi_{1m}\bPi_{2n}})\in A, 
\bY\in \Xi_{mn}
\right] + o(1)\\
% & = \prob_0\left[ n\{\widehat{\btheta}^{(\mathrm{OS})}(\tau_0) - \btheta_0\}\in A\right] + o(1)\\
&\quad\to \int_A\phi(\bx\mid\zero_d,\eye_d)\mathrm{d}\bx.
\end{align*}
The proof is thus completed.
\end{proof}

\bibliographystyle{apalike}
\bibliography{reference}
\end{document}